\documentclass[12pt,reqno, oneside]{amsart}

\usepackage[utf8]{inputenc}
\usepackage[normalem]{ulem}
\usepackage{textcomp}
\usepackage[english]{babel}
\usepackage{amsmath, amsfonts, amssymb,amsthm}
\usepackage[dvipsnames]{xcolor}
\usepackage{upgreek}
\usepackage{float}
\usepackage{changepage}
\usepackage{array}
\usepackage{enumerate}
\usepackage{youngtab}
\usepackage{hyperref}
\usepackage{enumitem}

\usepackage{verbatim}
\usepackage{tikz-cd}
\usepackage{tikz}
\usepackage{tikz-cd}
\usetikzlibrary{decorations.shapes}
\usepackage{fancyhdr}
\usepackage{caption}

\usepackage{mathtools}
\mathtoolsset{showonlyrefs}

\usepackage{tikz}
\usetikzlibrary{ decorations.markings}
\usetikzlibrary{calc, matrix, arrows,decorations.pathmorphing, positioning}
\numberwithin{equation}{section}

\newtheorem{theorem}{Theorem}[section]
\newtheorem{lemma}[theorem]{Lemma}
\newtheorem{prop}[theorem]{Proposition}
\newtheorem{conj}[theorem]{Conjecture}
\newtheorem{cor}[theorem]{Corollary}
\newtheorem{defi}[theorem]{Definition}
\newtheorem{remark}[theorem]{Remark}
\newtheorem{example}[theorem]{Example}
\newtheorem{algorithm}[theorem]{Algorithm}

\newcommand{\mbC}{\mathbb{C}}
\newcommand{\mbZ}{\mathbb{Z}}

\newcommand{\sL}{\mathfrak{sl}_2}
\newcommand{\Uqa}{U_q\widehat{\mathfrak{sl}}_2}
\newcommand{\Uq}{U_q\mathfrak{sl}_2}
\newcommand{\Ker}{\mathrm{Ker}}

\newcommand{\Ygn}{\mathrm{Y}\mathfrak{sl}_2}
\newcommand{\Ygngl}{\mathrm{Y}\mathfrak{gl}_2}
\newcommand{\ev}{\mathrm{ev}}
\newcommand{\drin}{\Delta^{(D)}}
\newcommand{\uR}{\mathcal{R}}
\newcommand{\ri}{\mathrm{i}}

\newcommand{\rH}{\mathrm{H}}
\newcommand{\rM}{\mathrm{M}}
\newcommand{\rN}{\mathrm{N}}
\newcommand{\rR}{\mathrm{R}}
\newcommand{\rV}{\mathrm{V}}
\newcommand{\rW}{\mathrm{W}}
\newcommand{\rU}{\mathrm{U}}
\newcommand{\Hom}{\mathrm{Hom}}

\newcommand{\arc}{\mathrm{Arc}}
\newcommand{\uarc}{\mathrm{UArc}}
\newcommand{\conf}{\mathrm{Conf}}
\newcommand{\uconf}{\mathrm{UConf}}
\newcommand{\cconf}{\mathrm{CatConf}}
\newcommand{\ucconf}{\mathrm{UCatConf}}
\newcommand{\iconf}{\mathrm{IConf}}
\newcommand{\nconf}{\mathrm{NConf}}
\newcommand{\rconf}{\mathrm{RConf}}
\newcommand{\sconf}{\mathrm{SConf}}
\newcommand{\cont}{\mathrm{Cont}}

\newcommand{\mcl}{\mathcal}
\newcommand{\mclf}{\mathcal{F}}
\newcommand{\mcls}{\mathcal{S}}
\newcommand{\ren}{\mathrm{RE}}
\newcommand{\len}{\mathrm{LE}}
\newcommand{\Id}{\mathrm{Id}}
\newcommand{\supp}{\mathrm{supp}}
\newcommand{\rL}{\mathrm{L}}
\newcommand{\pvt}{\mathrm{Pvt}}
\newcommand{\walg}{\mathcal{A}_{\{0,2\}}}
\newcommand{\soc}{\mathrm{Soc}}

\newcommand{\lb}{\left(}
\newcommand{\rb}{\right)}

\newcommand{\al}{\alpha}
\newcommand{\bt}{\beta}
\newcommand{\la}{\lambda}
\newcommand{\si}{\sigma}

\newcommand{\be}{\begin{equation}}
\newcommand{\ee}{\end{equation}}
\newcommand{\benonum}{\begin{equation*}}
\newcommand{\eenonum}{\end{equation*}}
\newcommand{\bse}{\begin{subequations}}
\newcommand{\ese}{\end{subequations}}
\newcommand{\ba}{\begin{aligned}}
\newcommand{\ea}{\end{aligned}}
\newcommand{\beanonum}{\begin{align*}}
\newcommand{\eeanonum}{\end{align*}}

\DeclareMathSymbol{\shortminus}{\mathbin}{AMSa}{"39}

\newlength\mylen
\newlist{mycases}{enumerate}{1}
\setlist[mycases,1]{label=\textbf{Case~\arabic*.}, 
  labelwidth=\dimexpr-\mylen-\labelsep\relax,leftmargin=0pt,align=right}

\newcolumntype{M}[1]{>{\centering\arraybackslash}m{#1}}

\makeatletter
\def\env@cases{%
  \let\@ifnextchar\new@ifnextchar
  \left\lbrack
  \def\arraystretch{1.2}%
  \array{@{}l@{\quad}l@{}}%
}
\makeatother

\begin{document}

\title{On representations of quantum affine $\mathfrak{sl}_2$.}

\author{Andrei Grigorev and Evgeny Mukhin} 
\address{EM: Department of Mathematical Sciences,
Indiana University Indianapolis,
402 N. Blackford St., LD 270, 
Indianapolis, IN 46202, USA;}
\email{emukhin@iu.edu} 

\address{AG: Department of Mathematical Sciences,
Indiana University Indianapolis,
402 N. Blackford St., LD 270, 
Indianapolis, IN 46202, USA;\hfill \break
National Research University Higher School of Economics, Department of Mathematics, Russian Federation, 6 Usacheva st., Moscow 119048;}
\email{aagrigor@iu.edu, andrei.al.grigorev@gmail.com}

\begin{abstract}
    We study tensor products of two-dimensional evaluation $\Uqa$-modules at generic values of $q$, $\Uqa$-homomorphisms between them, and closely related subjects.
    \end{abstract}
\maketitle

\centerline{
{\textit{Keywords:\ }}{Quantum affine $\sL$, evaluation modules, arc configurations, socle.}}

  \centerline{
{\textit{AMS Classification numbers:\ }}{17B37, 05E10 (primary), 81R10.}}

\tableofcontents
\section{Introduction.}
The representation theory of quantum groups has been intensively studied for the last four decades with many amazing findings and beautiful results.  Many efforts have been made to understand irreducible finite-dimensional modules over affine quantum groups for generic values of $q$ and to study the root of unity phenomenon. Many authors study super-symmetric versions, shifted quantum groups, quantum toroidal algebras, etc. However, many questions are still unanswered even in the simplest cases.

The current project was started by the question posed by C. Stroppel to the second author a couple of years ago: {\it Given a tensor product $w$ of several two-dimensional evaluation modules of $\Uqa$, what is  $h(w)=\dim\big(\Hom_{\Uqa}(\mbC,w)\big)?$} 
Although there are many other amusing questions one can ask about  tensor products $w$ of evaluation modules of $\Uqa$, see Section \ref{subsec:questions}, the problem of finding ways to compute $h(w)$ is an overarching motif of our work. 

One can restrict to evaluation parameters of the form $q^a$, $a\in 2\mathbb{Z}$, see Proposition \ref{prop:hom_factorize_lattice}. We denote the tensor product of $2$-dimensional evaluation modules with evaluation parameters $q^{a_i}$ simply by the word $w=a_1a_2...a_{k}$, where $a_i$ are even integers. We call $k$ the length of $w$. We call the word $w$ connected if the set of all distinct letters $a_i$ is of the form $\{a,a+2,\dots,a+2m\}$ for some $a,m$.
One can further reduce to the case of connected words of even length, see Proposition \ref{prop:no_gaps}. 
Then for each connected word of even length we have the number $h(w)\in \mathbb{Z}_{\geq 0}$. For example, one can compute
$$
h(220200242424)=4.
$$
The number $h(w)$ depends on the order of the letters in $w$. For example, if $a_1\geq a_i$ for all $i$ then $h(w)=0$ (for our choice of the coproduct). On the other hand $h(w)$ for a word of length $2n$ can be as large as $(2-\epsilon)^n$, see Proposition \ref{prop:max_bound}.

The statistics $h(w)$ contains information about dimensions of all homomorphism spaces between any tensor products of $2$-dimensional evaluation modules $\Hom_{\Uqa} (w_1,w_2)$, see Lemma \ref{lemma:hom_move} and Lemma \ref{lemma:dual}. 

The singular vectors of weight zero in a tensor product of $2n$ copies of the $\Uq$ 2-dimensional module have a basis which is identified with Catalan arc configurations - the sets of non-intersecting arcs in half-plane connecting $2n$ points on a line. Such objects are related to Temperley-Lieb algebras and used in the categorification of $\Uq$, see \cite{Bernstein1999},  \cite{Frenkel2007}.
One could ask if we can construct a basis of $\Uqa$ singular vectors in $w$ in a similar fashion. Clearly, in the affine case the arcs should connect letter $a$ and $a+2$ since the only words of length 2 which have a trivial submodule are of the form $a(a+2)$.
The non-intersecting property also has to be relaxed as we have too few arcs.  Here is an example of an arc configuration.
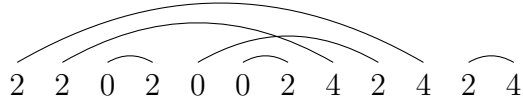
\begin{center} \begin{tikzpicture}
    \node at (0, 0) {$2$};
    \node at (0.6,0) {$2$};
    \node at (1.2,0) {$0$};
    \node at (1.8,0) {$2$};
    \node at (2.4,0) {$0$};
    \node at (3,0)   {$0$};
    \node at (3.6,0) {$2$};
    \node at (4.2,0) {$4$};
    \node at (4.8,0) {$2$};
    \node at (5.4,0) {$4$};
    \node at (6.0,0) {$2$};
    \node at (6.6,0) {$4$};
    \draw[-] (0 ,0.3) to [out=30,in=150] (5.4 ,0.3);
    \draw[-] (0.6,0.3) to [out=30,in=150] (4.2,0.3);
    \draw[-] (1.2 ,0.3) to [out=30,in=150] (1.8,0.3);
    \draw[-] (2.4 ,0.3) to [out=30,in=150] (4.8,0.3);
    \draw[-] (3.0,0.3) to [out=30,in=150] (3.6,0.3);
    \draw[-] (6.0,0.3) to [out=30,in=150] (6.6,0.3);
\end{tikzpicture}
\end{center}

 Indeed, we prove a number of results which support this idea: 
\begin{enumerate}
    \item $h(w)\neq 0$ if and only if there exists an arc configuration, see Theorem \ref{thm:supp};
    \item $h(w)$ is at least the number of irreducible arc configurations, see Theorem \ref{thm:lower_bound_irr};
    \item $h(w)$ is at most the number of steady arc configurations, see Theorem \ref{thm:upper_bound}. 
\end{enumerate}
The irreducible arc configurations do not have intersecting arcs connecting the same letters $a$ and $a+2$, and if an arc connecting $a$ with $a+2$ intersects with an arc connecting $a+2$ with $a+4$ then the latter should be on the left, see Definition \ref{def:irr_arc_conf}. The example of the arc configuration we give above is irreducible. The definition of the steady arc configurations is slightly more technical, see Definition \ref{def:steady_arc_conf}. For the example of the word above, the number of irreducible and steady arc configurations are both equal to $4$ which shows $h(w)=4$.

Alternatively, one can show that numbers $h(w)$ respect several symmetries:
$$
h(a_1a_2...a_{k})=h(a_2...a_{k}(a_1+4))=h((-a_k)...(-a_1))=h((a_1+b)...(a_k+b)).
$$
It is also easy to find the number of trivial components of the Jordan-Holder series of a word by the $q$-character method, see Lemma \ref{lemma:upper_bound_qchar} and Proposition \ref{prop:char_bound_computed}. Finally, for a short exact sequence 
$
\rU \hookrightarrow \rV \twoheadrightarrow \rW
$
we always have $h(\rV)\leq h(\rU)+ h(\rW)$.
Such tools are sufficient to find $h(w)=4$ in our example and many others, see Section \ref{subsec:exact_comps}.  One can treat even more cases if the two approaches are combined and enhanced by using the isomorphisms between words, see Section \ref{subsec:implicit_bounds}.

Neither irreducible nor steady arc configurations provide sharp bounds in general.
We expect that the condition on arcs which produces exactly $h(w)$ arc configurations is not described by pairwise interactions of arcs, and that it is not local. As an illustration of the non-local nature of the problem, we study the isomorphisms between different words. We certainly have $ab=ba$ unless $|a-b|=2$. However, there are other isomorphism which cannot be deduced from pairs of letters. For example, $0020\cong 0200$, $000200\cong 002000$. We expect $0^{k+1}2^k0^k\cong 0^k2^k0^{k+1}$ for all $k$, see Section \ref{subsec:comm}. It is a very interesting and challenging question to classify words up to isomorphisms. We conjecture such a classification for words whose letters are $0$ and $2$, see Conjecture \ref{conj:02isoms}.

Since the number of trivial composition factors in $w$ is easy to compute, the main difficulty in computing $h(w)$ is that the words are not semi-simple modules.  We try to develop a theory of indecomposable modules introducing graphs which enhance the socle filtration, see Section \ref{sec:mod_graphs}. 

Essentially $h(w)$ for a word $w$ is a nullity of linear map $f_0\in\Uqa$ sending the subspace of $\Uq$ singular vectors of weight zero to the subspace of $\Uq$ singular vectors of weight two.  For a word of length $2n$, given a basis, $f_0$ becomes a matrix of size $\frac{3}{n+2} \binom{2n}{n-1} \times \frac{1}{n+1}\binom{2n}{n}$. A straightforward Mathematica program computes $h(w)$ for words up to length 12 and even 14 (using the Yangian version). Moreover, a computer can  easily calculate the upper and lower bounds up to length 22 and often these bounds coincide. So, one can generate data and try to approach the problem completely combinatorially or using machine learning. We give all essential data (up to symmetries) for words up to length 10 in Appendix \ref{app:tables}.

\medskip

We believe that $h(w)$ can be computed directly from the combinatorial set of all arc configurations, see Conjecture \ref{conj:drop_adm}. However, we did not manage to find such an algorithm; there may exists a more suitable combinatorics.

\medskip

If letters $a_i$ in $w$ are non-decreasing, the word $w$ is known as $\Uqa$ Weyl module. The Weyl modules have been studied for a long time, see \cite{chari2002braid}. Our bounds recover the results of \cite{BritoChari} on maps between Weyl modules in the case of $\mathfrak{sl}_2$, see Proposition \ref{prop:weyl_mod_homs}.
Arbitrary words are also called $\Uqa$ mixed Weyl modules, much less is known about mixed case. 

\medskip 

We also hope that this paper can be of some educational value as we have collected many facts and methods scattered in the literature in the simplest non-trivial case of $\Uqa$. For that purpose, we provide a lot of examples and give detailed proofs. At the end we give a list of questions for possible further work, see Section \ref{subsec:questions}.

\medskip 

The paper is constructed as follows. In 
Sections  \ref{sec:prelim} and \ref{sec:basic_repth} we collect the known facts about $\Uqa$ and its finite-dimensional representations. In Section \ref{sec:triv_submod} we start the study of numbers of $h(w)$. 
We develop the language of arc configurations and prove a few properties of $h(w)$ statistics. In Section \ref{sec:bounds} we prove the upper and lower bounds for $h(w)$ and prove that $h(w)\neq 0$ if and only if the set of arc configurations is non-empty. In Section \ref{sec:degs_limits} we study our problem from the point of view of hyperplane arrangements. In this section we study numbers $h(w)$ using degenerations from generic words. In Section \ref{sec:classification} we address the problem of classifying words up to isomorphism. In Section \ref{sec:examples} we show how our theory works in various explicit examples. In particular, the example in the introduction can be rather easily computed by hand. Section \ref{sec:extensions} is devoted to the study of extensions. In particular, we classify all extensions between two evaluation modules. In Section \ref{sec:mod_graphs} we attempt to visualize words using graphs build on socle filtration, which we call submodule graphs.  In Section \ref{sec:further} we discuss combinatorics corresponding to singular vectors and list a number of unanswered questions related to words.

\bigskip

\section{Preliminaries.}\label{sec:prelim}
 In this section we recall well-known facts about the affine quantum group $\Uqa$ (at level zero) and the corresponding Yangian $\Ygn$.

\subsection{The quantum group $\Uqa$.} Fix a non-zero complex number $q$ which is not a root of unity. We fix a value of $\log q$, so that $q^a$ is defined for all $a\in\mbC$.

The quantum group $\Uq$ is the unital associative algebra with generators $e, f, K^{\pm 1}$ and relations
\begin{equation*}
        KK^{-1} = K^{-1}K = 1,\;\; KfK^{-1} = q^{-2}f, \;\;KeK^{-1} = q^{2}e,\;\; [e,f] = \frac{K - K^{-1}}{q -q^{-1}}.
    \end{equation*}

    The affine quantum group $\mathfrak{sl}_2$ in Drinfeld-Jimbo realization is given as follows.

\begin{defi} The affine quantum group 
    $\Uqa$  at level zero is the unital associative algebra with generators $e_i, f_i, K^{\pm1}$, $i=0,1$, and defining relations
    \begin{subequations}\label{eq:JD_pres_rels}
    \begin{equation}
        K K^{-1}=K^{-1}K = 1,
    \end{equation}
    \begin{equation}
    Kf_{1}K^{-1} = q^{-2}f_{1},\;\; Ke_{1}K^{-1} = q^{2}e_{1},\;\; Kf_{0}K^{-1} = q^{2}f_0,\;\;Ke_{0}K^{-1} = q^{-2}e_0,
    \end{equation}
    \begin{equation}
    [e_1 , f_1] = -[e_0,f_0]=\frac{K-K^{-1}}{q - q^{-1}},
    \end{equation}
    \begin{equation}
    [e_1, f_0] = [e_0, f_1] = 0,
    \end{equation}
    \begin{align}
        e_{i}^3e_{j} - (q^2 + 1 + q^{-2})e_{i}^2e_{j}e_{i} + (q^2 + 1 + q^{-2})e_{i}e_{j}e_{i}^2 - e_{j}e_{i}^{3} &= 0,\\
        f_{i}^3f_{j} - (q^2 + 1 + q^{-2})f_{i}^2f_{j}f_{i} + (q^2 + 1 + q^{-2})f_{i}f_{j}f_{i}^2 - f_{j}f_{i}^{3}&= 0,\;\; i,j\in\{0,1\},\;i\neq j.
    \end{align}
    \end{subequations}
\end{defi}

We use the Hopf algebra structure on $\Uqa$ with comultiplication $\Delta$, antipode $S$, and counit $\epsilon$ given by
\begin{equation}\label{eq:Uqa_Com_An_Coun}
    \begin{aligned}[c]
    \Delta: \Uqa &\longrightarrow \Uqa \otimes \Uqa, \\
    K^{\pm1} &\longmapsto K^{\pm1}\otimes K^{\pm1},\\ 
    e_0 &\longmapsto e_{0}\otimes K^{-1} + 1\otimes e_{0},\\
    f_0 &\longmapsto f_0\otimes 1 + K \otimes f_0,\\
    e_1 &\longmapsto e_{1}\otimes K + 1\otimes e_{1},\\
    f_1 &\longmapsto f_1\otimes 1 + K^{-1} \otimes f_1,\\
\end{aligned}\qquad
\begin{aligned}[c]
        S: \Uqa &\longrightarrow \Uqa,\\
    K^{\pm1} &\longmapsto K^{\mp1},\\ 
    e_0 &\longmapsto -e_0K,\\
    f_0 &\longmapsto -K^{-1}f_0,\\
    e_1 &\longmapsto -e_1K^{-1},\\
    f_1 &\longmapsto -Kf_1,\\
    \end{aligned}\qquad
    \begin{aligned}[c]
    \varepsilon: \Uqa &\longrightarrow \mbC,\\
    K^{\pm1} &\longmapsto 1,\\ 
    e_0 &\longmapsto 0,\\
    f_0 &\longmapsto 0,\\
    e_1 &\longmapsto 0,\\
    f_1 &\longmapsto 0.
    \end{aligned}
\end{equation}

There are two embeddings of $\Uq$ to $\Uqa$ given by
\begin{equation}
    \begin{aligned}[c]
    \iota_{1}: \Uq &\longrightarrow \Uqa ,\\
    K^{\pm1} &\longmapsto K^{\pm1},\\ 
    e &\longmapsto e_{1},\\
    f &\longmapsto f_{1},
\end{aligned}\qquad
\begin{aligned}[c]
    \iota_{0}: \Uq &\longrightarrow \Uqa, \\
    K^{\pm1} &\longmapsto K^{\mp1},\\ 
    e &\longmapsto e_{0},\\
    f &\longmapsto f_{0}.
\end{aligned}
\end{equation}

Denote the images of embeddings $\iota_{j}$ by $U_{q}^{j} \sL$, $j=0,1$. The algebras $U_{q}^{j} \sL$ are Hopf subalgebras of $\Uqa$. We often write simply $\Uq \subset \Uqa$ instead of $U_{q}^{1}\sL$.

Conversely, there exists a family of evaluation  surjective  homomorphisms from $\Uqa$ to $\Uq$ parameterized by complex number $a\in \mbC$,
    \begin{equation}\label{eq:Uqa_eval_hom}
    \begin{aligned}
        &\hspace{0pt}\ev_{a}:\Uqa  \longrightarrow \Uq,\\
    &\hspace{30pt}K^{\pm 1} \longmapsto K^{\pm 1},\\
      &\hspace{7.5pt}e_0 \longmapsto q^{a}f,\hspace{10pt} f_{0}\longmapsto q^{-a}e,\\
      &\hspace{7.5pt}e_1 \longmapsto e,\hspace{22pt}  f_{1}\longmapsto f. 
    \end{aligned}
    \end{equation}

The evaluation map $ev_a$ is not a homomorphism of Hopf algebras. 
    
Clearly, $\ev_a\circ \iota_1=id$. 

\medskip

Note that in our notation we use simply $a$ instead of $q^a$. As a result, $\ev_{a} = \ev_{a + \frac{2\pi \ri}{\log(q)}}$. A similar ambiguity will occur often. We hope that it does not lead to any confusion.

\medskip

The algebra $\Uqa$ can alternatively be written in a new Drinfeld form. In the new Drinfeld realization the algebra $\Uqa$ has generators $\{x_{r}^{\pm}, h_{n}, K^{\pm1},\}$ collected in series $\psi^\pm(z), x^\pm(z)$ of a formal variable $z$
    \begin{subequations}
    \begin{equation*}
        \psi^{\pm}(z) = \sum_{r = 0}^{\infty}\psi_{\pm r}^{\pm}z^{\pm r} = K^{\pm 1}\exp\left(\pm (q-q^{-1})\sum_{s=1}^{\infty}h_{\pm s}z^{\pm s}\right),\qquad
        x^{\pm}(z) = \sum_{r \in \mathbb{Z}}x_{r}^{\pm}z^{r}.
    \end{equation*}
    \end{subequations}    

The defining relations are 

\begin{align}
    &\psi^{\pm}(z)\psi^{\pm}(w) = \psi^{\pm}(w)\psi^{\pm}(z),\hspace{10pt}\psi^{\pm}(z)\psi^{\mp}(w) = \psi^{\mp}(w)\psi^{\pm}(z),\\
    &(w - zq^{\pm 2})x^{\pm}(z)x^{\pm}(w) + (z - wq^{\pm 2})x^{\pm}(w)x^{\pm}(z) = 0,\\
    &[x^{+}(z) , x^{-}(w)] = \frac{1}{q-q^{-1}}\left(\delta(z/w)\psi^{+}(z) - \delta(z/w)\psi^{-}(w)\right),\label{eq:def_new_drinf}\\
    &\psi^{+}(z)x^{\pm}(w) = \left(q^2\frac{w-q^{-2}z}{w-q^2z}\right)^{\pm 1}x^{\pm}(w)\psi^{+}(z),\\
    &\psi^{-}(z)x^{\pm}(w) = \left(q^2\frac{w-q^{-2}z}{w-q^2z}\right)^{\pm 1}x^{\pm}(w)\psi^{-}(z).
\end{align}

\text{Here }$\delta(z/w) = \sum_{n \in \mathbb{Z}}\left(\frac{z}{w}\right)^{n}.$

An isomorphism between the Drinfeld-Jimbo and new Drinfeld realizations is given on generators by
    \begin{equation}
    \begin{aligned}
      &\hspace{35pt}K^{\pm1} \longmapsto K^{\pm1},\\
      &\hspace{0pt}e_0 \longmapsto x_{1}^{-},\hspace{30pt} f_{0}\longmapsto x_{-1}^{+},\\
      &\hspace{0pt}e_1 \longmapsto K^{-1}x_{0}^{+},\hspace{8pt}  f_{1}\longmapsto x_{0}^{-}K. 
    \end{aligned}
    \end{equation}

The comultipilication \eqref{eq:Uqa_Com_An_Coun} written in the new Drinfeld realization has the property
\begin{equation}\label{eq:psi_triangular}
    \Delta(\psi^\pm(z))=\psi^\pm(z)\otimes \psi^\pm(z)  \operatorname{mod}(U^{\pm}\otimes U^{\mp}), 
\end{equation}
where $U^\pm = \underset{k>0}{\bigoplus}U_{\pm 2k}$ with $U_{l} = \{u\in \Uqa : KuK^{-1} = q^{l}u\}$ being the subspace of elements of weight $l$.

{There is a family of shift automorphisms $\tau_a$ of Hopf algebra $\Uqa$ parameterized by a complex number $a$,
\begin{equation}\label{eq:Uqa_shift}
    \begin{aligned}
        &\hspace{0pt}\tau_{a}:\Uqa  \longrightarrow \Uqa,\\
    &\hspace{25pt}K^{\pm 1} \longmapsto K^{\pm 1},\\
      &\hspace{0pt}e_0 \longmapsto q^{a}e_0,\hspace{10pt} f_{0}\longmapsto q^{-a}f_{0},\\
      &\hspace{0pt}e_1 \longmapsto e_{1},\hspace{22pt}  f_{1}\longmapsto f_{1}, 
    \end{aligned}
    \end{equation}

    or, in new Drinfeld realization, 
    \begin{equation}
        \begin{aligned}
            \psi^{\pm}(z) &\longmapsto \psi^{\pm}(q^{a}z),\\
            x^{\pm}(z) &\longmapsto x^{\pm}(q^{a}z).
        \end{aligned}
    \end{equation}

Clearly,  $\ev_b\circ\tau_a =\ev_{a+b}$.}

\subsection{The $\Uqa$-modules and $q$-characters.}\label{subsec:uqamod_and_q-char}
The category of finite-dimensional $\Uq$-modules is semi-simple. We denote the irreducible $\Uq$-module of type 1 and of dimension $m+1$ by $\mathrm{L}_m$, $m\in\mbZ_{\geq 0}$. The module $\mathrm{L}_m$ has a basis $\{v_0,\dots,v_m\}$ such that 
\begin{equation*}
K\,v_i=q^{2(m-i)}v_i,\qquad f\,v_i=[i+1]_qv_{i+1},\qquad e\,v_i=[m-i+1]_qv_{i-1}.
\end{equation*}
Here, $v_{-1}=v_{m+1}=0$ and 
$$
[i]_q=\frac{q^{i}-q^{-i}}{q-q^{-1}}.
$$
In general, given a $\Uq$-module $\rV$, a vector $v\in \rV$ is called a weight vector of weight $a$ if $K\,v=q^a v$. A weight vector $v\in\rV$ is called singular if $e\,v=0$. The module $\rm{L}_m$ is generated by a singular vector $v_0$ of weight $m$. 

\medskip

The category of finite-dimensional $\Uqa$-modules is not semi-simple. All $\Uqa$-modules in this paper are finite-dimensional and are of type $1$. In all statements we assume the finite-dimensionality and type $1$ without mentioning.

For $a\in \mbC$, define a $\Uqa$-module structure on $\mathrm{L}_{m}$ by a map
\begin{equation}\label{eq:eval_mod_def}
    \rho_{\mathrm{L}_{m}} \circ \ev_{a}: \Uqa \longrightarrow \mathrm{End}_{\mbC}\left(\mathrm{L}_{m}\right).
\end{equation}
The resulting $\Uqa$-module is called the evaluation module with the evaluation parameter $a$. Motivated by $q$-characters (see below) we denote the module defined by \eqref{eq:eval_mod_def} by the string $[\alpha, \beta]$, where $\alpha=a-m+1$, $\beta=a+m-1$. 

We note that in the notation $[\alpha,\beta]$ we always assume that $\beta-\alpha\in 2\mbZ_{\geq 0}$, then $m=\frac{1}{2}(\beta-\alpha)+1$ is the highest weight and $m+1$ is the dimension. The evaluation parameter is $\frac{1}{2}(\al + \bt)$.

For a tensor product of evaluation modules instead of $[\alpha_{1}{,}\beta_{1}] {\otimes} [\alpha_{2}{,}  \beta_{2}]$ we write simply $[\alpha_{1}{,}\beta_{1}][\alpha_{2}{,}\beta_{2}]$. 

Two strings $[\alpha_{1}, \beta_{1}]$ and $[\alpha_{2}, \beta_{2}]$ are in general position if either one string contains the other,
\bse\label{eq:general_position}
\begin{equation*}
        \{\alpha_{1},\dots, \beta_{1}\}\subseteq \{\alpha_{2},\dots,\beta_{2}\} \textit{, or } \{\alpha_{2},\dots, \beta_{2}\} \subseteq \{\alpha_{1},\dots, \beta_{1}\},
\end{equation*}
or the union of the strings is not a string,
\begin{equation*}
        \{\alpha_{1},\dots, \beta_{1}\} \cap \{\alpha_{2}-2,\alpha_{2} ,\dots, \beta_{2},\beta_{2} + 2\} = \varnothing.
\end{equation*}
\ese
A tensor product of evaluation modules is irreducible if and only if the corresponding strings are in pairwise general position.

Two irreducible products of evaluation modules are isomorphic if and only if they differ only by a permutation of factors.

Every irreducible $\Uqa$-module is a tensor product of evaluation modules, see \cite{chari1991quantum}.

A vector $v$ in a  $\Uqa$-module $\rV$ is called an $\ell$-weight vector of $\ell$-weight $\phi^\pm(z)\in\mbC[[z^{\pm 1}]]$ if $\psi^\pm(z)v=\phi^\pm(z)v$.  

A vector $v$ is called a generalized $\ell$-weight vector of $\ell$-weight $\phi^\pm(z)$ if $(\psi^\pm(z)-\phi^\pm(z))^{\dim \rV}v=0$.

An $\ell$-weight vector $v$ is called $\ell$-singular if $x^+(z)v=0$. 

If $\phi^\pm(z)$ is an $\ell$-weight of a vector in a  $\Uqa$-module then there exists a rational function $\phi(z)$ such that $\phi(0)\phi(\infty)=1$ and $\phi^+(z)$ and $\phi^-(z)$ are expansions of $\phi(z)$ at $z=0$ and $z=\infty$, respectively.

For a rational function $\phi(z)$ denote the space of generalized vectors of $\ell$-weight $\phi$ in $\rV$ by $\rV[\phi]$.

Any $\Uqa$-module $\rV$ admits decomposition $\rV =\mathop{\oplus}\limits_{\phi}\rV[\phi]$.

The $q$-character of $\rV$ is the formal sum of $\ell$-weights of $\rV$ given by
\begin{equation*}
\chi_q(\rV)=\sum_{\phi} \dim\lb \rV[\phi]\rb \phi.
\end{equation*}
We use the notation $1_a$, $a\in\mbC$, for a rational function
 \begin{equation*}
 1_a= q\,\frac{1 - q^{a-1}z}{1 - q^{a+1}z}.
 \end{equation*}
 The $\ell$-weights in  $\Uqa$-modules are monomials in $1_a^{\pm 1}$. Thus, the $q$-character is a Laurent polynomial in  commutative variables $1_a$ with non-negative integer 
 coefficients.

Denote $A_a = 1_{a-1}1_{a+1}$.

The rational function $1_a$ is an affine analogue of fundamental weight and $A_a$ as an affine analogue of a root, see Proposition \ref{prop:ellroot_action}.

For a short exact triple of $\Uqa$ modules $\rV\hookrightarrow \rU\twoheadrightarrow \rW$, 
\begin{equation*}
\chi_q(\rU)=\chi_q(\rV)+\chi_q(\rW).
\end{equation*}
Due to triangularity \eqref{eq:psi_triangular}, the current $\Delta(\psi)(z) - \psi(z)\otimes \psi(z)$ is nilpotent in a product $\rV\otimes \rW$ of  $\Uqa$-modules. Thus
\begin{equation*}    \chi_{q}(\rV\otimes \rW) = \chi_{q}(\rV)\chi_{q}(\rW).
\end{equation*}
The $q$-character is an injective ring homomorphism from the Grothendieck ring of finite-dimensional $\Uqa$-modules to the ring of Laurent polynomials in variables $1_a$, $a\in\mbC$. 
A monomial of a $q$-character  $1_{a_1}^{n_1}\dots1_{a_{k}}^{n_{k}}$ is called dominant if all $n_{k}$ are non-negative integers. An irreducible $\Uqa$-module is generated by an $\ell$-singular vector.
An $\ell$-weight of an $\ell$-singular vector is a dominant monomial.

\medskip 

We call a $\Uqa$-module $\rV$ thin if all $\ell$-weight subspaces of $\rV$ are one-dimensional. In other words $\rV$ is thin if coefficients of all monomials in $\chi_q(\rV)$ are one.

The following proposition is useful.
\begin{prop}\label{prop:ellroot_action}
Let $\rV$ be an $\Uqa$-module. Then for any $\ell$-weight $\phi$, vector $v\in \rV[\phi]$ and $r\in\mbZ$,
\begin{equation}
x^{\pm}_rv \in \mathop{\oplus}\limits_{a}\rV[A_{a}^{\pm 1}\phi].
\end{equation}
Here $a$ runs over $\mbC/\big(\frac{2\pi\ri}{\log(q)}\mbZ\big)$.
\end{prop}
Proposition \ref{prop:ellroot_action} is well-known, cf. \cite{mukhin2014affinization}, Proposition 3.9. However, we did not find the precise statement we need in the literature. We provide a proof in Appendix \ref{app:etc}.

\medskip

An evaluation module $\rV=[\alpha, \beta]$ is thin and $\ell$-weights are given by 
\begin{equation}\label{eq:eval_ell_weights}
    q^{m-2j}\frac{\lb 1 - q^{\alpha-1}z\rb\lb 1 - q^{\beta + 3}z\rb}{\lb 1 - q^{\beta + 1 - 2j}z\rb \lb 1 - q^{\beta + 3 - 2j}z \rb}, \;\;j=0,\dots,m.
\end{equation}

In particular, the $q$-character of an evaluation module $\rV=[\alpha, \beta]$ is given by 
    \begin{equation}\label{eq:eval_q_char}
        \chi_q(\rV) = \sum_{j=0}^{m}q^{m-2j}\frac{\lb 1 - q^{\alpha-1}z\rb\lb 1 - q^{\beta + 3}z\rb}{\lb 1 - q^{\beta + 1 - 2j}z\rb \lb 1 - q^{\beta + 3 - 2j}z \rb}= \sum_{j = 0}^{m}1_{\alpha}1_{\alpha + 2}\dots1_{\beta - 2j}1_{\beta - 2j+4}^{-1}1_{\beta - 2j+6}^{-1}\dots1_{\beta}^{-1}1_{\beta + 2}^{-1},
    \end{equation}
    where $m  = (\beta-\alpha)/2+1$.    For example, the $q$-character of two-dimensional evaluation module $[a,a]$ is
    \begin{equation*}        \chi_q([a,a]) = 1_{a} + 1_{a+2}^{-1}.
    \end{equation*}The dominant monomial of $\chi_q([\al,\beta])$,  $1_\al1_{\al+2}\dots 1_{\bt-2}1_\bt$, is reflected in our notation for evaluation modules.
    
In particular, $q$-characters of irreducible  $\Uqa$-modules are explicit.
Moreover, given $\chi_q(\rV)$, one can explicitly find the composition factors of $\rV$. In other words, we can find the class of $\rV$ in the Grothendieck ring of the category of finite-dimensional representations of $\Uqa$.

A monomial in a $q$-character is called right-negative if it contains $1_a^{-m}$, for some $a\in\mbC$, $m\in\mbZ_{>0}$, and does not contain $1_{a+2k}$ with $k\in \mbZ_{>0}$ in any power.  

A product of two right negative monomials is right negative. All non-dominant monomials in the $q$-character \eqref{eq:eval_q_char} of an evaluation module are right negative. In particular, if a product of two monomials in $q$-characters of evaluation modules is dominant, then at least one of the monomials is dominant.

\begin{example}\label{ex:q_char_mult_1}
        Let $\al_1,\bt_1$  $\al_2 ,\bt_2$ be four even integers such that $\al_1 < \al_2\leq \bt_1+2\leq \bt_2$. Then the dominant monomials in $\chi_{q}([\al_1,\bt_1][\al_2,\bt_2])$ are
        \begin{multline*}
            1_{\al_1}\dots 1_{\beta_1} 1_{\al_2 } \dots 1_{\bt_2},\ 1_{\al_1}\dots 1_{\beta_1-2} 1_{\al_2 } \dots \widehat{1_{\bt_{1}+2}} \dots 1_{\bt_2}, \ \\ 1_{\al_1}\dots 1_{\beta_1-4} 1_{\al_2 } \dots \widehat{1_{\bt_{1}}}\widehat{1_{\bt_{1}+2}} \dots 1_{\bt_2},
            \ \dots \ ,\ 1_{\al_1}\dots 1_{\alpha_2-4} 1_{\beta_1 + 4},\dots, 1_{\beta_2}.
        \end{multline*}
\end{example}

\subsection{Drinfeld coproduct and universal $R$-matrix.} The quantum group $\Uqa$ is known to be quasi-triangular. It means that there exists an invertible $\uR\in \Uqa\hat\otimes \Uqa$ such that
$$\uR \Delta(x) \uR^{-1} = \Delta^{op}(x), \qquad \forall x\in\Uqa,
$$
and 
\begin{equation}\label{eqn:Rmat_comult}
    (\Delta\otimes 1)(\uR) = \uR_{13}\uR_{23},\;\; (1\otimes\Delta)(\uR) = \uR_{13}\uR_{12}.
\end{equation}
Here $\Delta^{op}(x) = {\mathop\sum_{i=1}^{m}} b_{i} \otimes a_i$, whenever $\Delta(x) = {\mathop\sum_{i=1}^{m}} a_{i} \otimes b_i$.  The element $\uR$ is called the universal $R$-matrix.

If the series defining $\uR$ converges in $\mathrm{End}(\rV \otimes \rW)$, then $\check{\uR} :\rV \otimes \rW \to \rW\otimes \rV$, where $\check{\uR}=P\uR$, is a $\Uqa$-homomorphism. Here $P:\rV \otimes \rW \rightarrow \rW\otimes \rV$ is the permutation operator.

Equations \eqref{eqn:Rmat_comult}  allow to compute the action of $\uR$ on $\rV \otimes (\rW_1\otimes \rW_2)$ (respectively, $(\rW_1\otimes \rW_2)\otimes \rV$) as long as the action of $\uR$ on $\rV \otimes \rW_1$, $\rV\otimes \rW_2$ (respectively, $\rW_1 \otimes \rV$, $\rW_2 \otimes \rV$) is defined.

The universal $\uR$-matrix admits the decomposition
\begin{equation}\label{eq:Rmat_gauss}
    \uR = q^{h_{0}\otimes h_0}\uR_{-}\uR_0\uR_{+}.
\end{equation}
Here $\uR_{\pm}-1$ is in a completion of $U^{\pm}\otimes U^{\mp}$  and $h_0 = \log(K)/\log(q)$. We list explicit formulas for $\uR_{+}, \uR_{-}, \uR_0$ in Appendix \ref{app:etc}. The (normalized) action of $\check\uR$ in $[\al_1,\beta_1]\otimes [\al_2,\beta_2]$ is given in Proposition \ref{prop:Rmat_lambdas}, cf. also \eqref{eq:Rmat2}.

Drinfeld coproduct $\drin$ is an algebra homomorphism from $\Uqa$ to a completion of $\Uqa\otimes \Uqa$ given by 
\begin{align}\label{eq:dr_comult}
     \psi^{\pm}(z)&\mapsto \psi^{\pm}(z){\otimes}\psi^{\pm}(z),\notag\\ 
    x^{+}(z)&\mapsto x^{+}(z)\otimes \psi^{-}(z) + 1\otimes x^{+}(z),\\
    x^{-}(z)& \mapsto x^{-}(z)\otimes 1 + \psi^{+}(z)\otimes x^{-}(z). \notag
\end{align}
If $\rV,\rW$ are $\Uqa$-modules and the series converge in $\rV\otimes \rW$, then the Drinfeld coproduct gives a $\Uqa$-module structure on $\rV\otimes \rW$. We denote this module $\rV\otimes_D\rW$. Note that in some cases $\rV\otimes_D\rW$ is not defined.

It is known that the Drinfeld coproduct is related to the comultiplication \eqref{eq:Uqa_Com_An_Coun} by 
\begin{equation}\label{eq:Rp_vs_drinf}
    \drin(x) = \uR_{+}\Delta(x)\uR_{+}^{-1}, \qquad \forall x\in\Uqa,
\end{equation}
see \cite{enriquez2007weight}.

    Thus, if the action of $\uR_+$ on $\rV \otimes \rW$ and $\rV\otimes_{D}\rW$ is well-defined, then $\rV \otimes \rW \cong \rV \otimes_D \rW$. The isomorphism is given by the action of $\uR_{+}^{-1}$ on $\rV \otimes \rW$. If $\rV\otimes_{D}\rW$ is well-defined and irreducible, then $\rV \otimes \rW \cong \rV \otimes_D \rW$. We expect that this is a general fact.

    \begin{conj}\label{conj:drinf_prod}
    If $\rV\otimes_{D}\rW$ is well-defined, then $\rV \otimes \rW \cong \rV \otimes_D \rW$.
    \end{conj}
We prove the conjecture in a special case of thin modules.

\begin{defi}\label{def:qchar_sep}
    A pair of $\Uqa$-modules is called $q$-character separated if for any pair of $\ell$-weights $1_{a_1}^{m_1}\dots 1_{a_k}^{m_k}$, $1_{b_1}^{n_1}\dots 1_{b_l}^{n_l}$ in q-characters $\chi_q(\rV)$ and $\chi_q(\rW)$ respectively, $a_i \neq b_j$ for any $i,j$.
\end{defi}

\begin{prop}\label{prop:equiv_drinf}
    Let $\Uqa$-modules $\rV$ and $\rW$ be thin and $q$-character separated.

    Let the action of universal R-matrix $\uR$ be defined on $\rV\otimes \rW$.

    Then $\rV \otimes_D \rW$ is well-defined and $\rV \otimes_D \rW \cong \rV \otimes \rW$.
\end{prop}
The proof is given in Appendix $\ref{app:etc}$. We expect that the condition of $\rV$ and $\rW$ being thin in Proposition \ref{prop:equiv_drinf} can be dropped.

Particular case of our interest is the following
    \begin{cor}\label{cor:sep_prod}
    Let $\rV = \otimes_{i=1}^{k} [\alpha_i,\beta_i],\; \rW = \otimes_{j=1}^{l}[\alpha^{\prime}_j, \bt^{\prime}_j]$ be such that $\{\alpha_i,\dots \bt_i\} \cap \{\alpha_j,\dots \bt_j\} = \varnothing$ and $\{\alpha_i^{\prime},\dots \bt_i^{\prime}\} \cap \{\alpha_j^{\prime},\dots \bt_j^{\prime}\} = \varnothing$ whenever $i\neq j$. Assume further that for any $i,j$ and  for any $a\in \{\al_i,\al_i+2,\dots, \bt_i-2,\bt_i\},\;\; b\in \{\al^{\prime}_j, \al_j^{\prime}+2, \dots, \bt^{\prime}_{j}-2,\bt^{\prime}_j\}$, we have $a-b\notin \{-2,0,2\}$. Then $\rV \otimes_{D} \rW  \cong \rV \otimes \rW$.
\end{cor}
\begin{proof}
The first two conditions guarantee that $\rV$ and $\rW$ are thin. The last two conditions make sure $\rV$ and $\rW$ are $q$-character separated and that the $R$-matrix is well-defined. Corollary \ref{cor:sep_prod} follows from Proposition \ref{prop:equiv_drinf}.
\end{proof}

\subsection{$\check{R}$-matrix homomorphisms.}
The space of $\Uqa$-homomorphisms between tensor products of two evaluation modules in different orders is always one-dimensional. We describe it explicitly.

As $\Uq$-modules,
\begin{equation}
    [\alpha_{1},\beta_{1}][\alpha_{2},\beta_{2}] \underset{\Uq}{\cong} \mathrm{L}_m\otimes \mathrm{L}_n \underset{\Uq}{\cong}  \mathrm{L}_{m+n} \oplus \mathrm{L}_{m+n-2} \oplus \dots\oplus \mathrm{L}_{|m-n|+2} \oplus \mathrm{L}_{|m-n|},
\end{equation}
where $m=(\beta_1-\al_1)/2+1$, $n=(\beta_2-\al_2)/2+1$.

Embedding $\mathrm{L}_{m+n-2k} \subseteq \mathrm{L}_{m}\otimes \mathrm{L}_{n}$ is defined up to a constant. We fix the constant by choosing a singular vector
\begin{equation}\label{eq:ev_mod_prod_sl2_hw}
u_{k} = \sum_{i,j:\ i+j = k}(-1)^jq^{(n-j+1)j}[m-i]_{q}![n-j]_{q}!(f^{(i)}v_0)\otimes (f^{(j)}w_0)\in \mathrm{L}_m\otimes \mathrm{L}_n.
\end{equation}Here, $[n]_{q}! = \prod_{j=1}^{n}[j]_{q}$, $f^{(l)} = \frac{f^{l}}{[l]_q!}$, $v_0$ and $w_0$ are singular vectors in $\mathrm{L}_{m}$ and $\mathrm{L}_{n}$, respectively. 

Let $u_k^{op}\in \mathrm{L}_n\otimes \mathrm{L}_m$ be the vector of the form \eqref{eq:ev_mod_prod_sl2_hw} 
where $i$ is interchanged with $j$, and $m$ with $n$. 
We identify $\mathrm{L}_m\otimes \mathrm{L}_n$ and  $\mathrm{L}_n\otimes \mathrm{L}_m$ as vector spaces by identifying $f^{l}u_k$ with $f^{l}u_k^{\text{op}}$ for all $k,l$.

Any $\Uqa$-homomorphism is of the form
\begin{equation}\label{eq:interwinner_decomp}
    \check{R} = \lambda_{0}\mathrm{Id}_{\mathrm{L}_{m+n}} \oplus \lambda_{1}\mathrm{Id}_{\mathrm{L}_{m+n-2}}\oplus \dots \oplus \lambda_{\min\{m,n\}}\mathrm{Id}_{\mathrm{L}_{|m-n|}}.
\end{equation}

\begin{prop}\label{prop:Rmat_lambdas}
    The operator $\check{R}: [\alpha_1, \bt_1][\al_2, \bt_2] \longrightarrow [\al_2, \bt_2][\alpha_1, \bt_1]$,
    defined by \eqref{eq:interwinner_decomp} is an $\Uqa$-homomorphism if and only if
    \begin{equation}\label{eq:Rmat_lambdas}
        \lambda_{k} = c  \prod_{l = 0}^{k-1}\left(1 - z q^{m+n-2l}\right) \prod_{l = k}^{\min\{m,n\}-1}\left(z - q^{m+n-2l}\right).
    \end{equation}
    Here $c\in \mathbb{C}$, $z = q^{a - b}$, where $a = \alpha_1 + m - 1,\;\; b = \alpha_2 + n - 1$ are the evaluation parameters.
\end{prop}

This statement is known  (see for example \cite{chari1995guide}). For completeness we provide a proof in Appendix \ref{app:etc}.

Due to \eqref{eqn:Rmat_comult}, Proposition \ref{prop:Rmat_lambdas} allows us to compute $R$-matrix acting on a tensor product of two arbitrary tensor products of evaluation modules.

\subsection{The Yangian $\Ygn$.} We define the Yangian $\Ygn$ and discuss sets of generators.

We start with the $\Ygngl$.
\begin{defi}
    The $\Ygngl$ is an associative unital algebra with generators $t_{ij}^{(r)}$,  $i,j = 1,2$, $r\in \mbZ_{>0}$, collected in series $t_{ij}(u)$ in a formal variable $u$, 
    \begin{equation*}
        t_{ij}(u) = \delta_{ij} + \sum_{r=1}^{\infty}t_{ij}^{(r)}u^{-r},
    \end{equation*}
    and defining relations 
    \begin{equation}
    (u - v)[t_{ij}(u), t_{kl}(v)] = t_{kj}(u)t_{il}(v) - t_{kj}(v)t_{il}(u).
    \end{equation}
\end{defi}
    In modes the defining relations take the form
    \begin{equation}
        [t_{ij}^{(r+1)}, t_{kl}^{(s)}] - [t_{ij}^{(r)}, t_{kl}^{(s+1)}] = t_{ij}^{(r)}t_{kl}^{(s)} - t_{ij}^{(s)}t_{kl}^{(r)}.
    \end{equation}
    The defining relations can also be written in a matrix form
    \begin{equation*}
      R(u-v)T_{1}(u) T_{2}(v) = T_{2}(v) T_{1}(u) R(u-v), \qquad   T(u) = \begin{pmatrix}
            t_{11}(u) & t_{12}(u)\\
            t_{21}(u) & t_{22}(u) 
        \end{pmatrix},
    \end{equation*} 
     where $T_1(u)=T(u)\otimes 1$, $T_2(u)=1\otimes T(u)$, $R(u) = 1 - Pu^{-1}$. 
     
We use the Hopf algebra structure on $\Ygn$ with comultiplication $\Delta$, antipode $S$, and counit $\epsilon$ given by
\begin{equation}
    \begin{aligned}[c]
    \Delta: \Ygngl &\longrightarrow \Ygngl \otimes \Ygngl,\qquad & S: \Ygngl &\longrightarrow \Ygngl, \qquad & \varepsilon:\Ygngl &\longrightarrow \mbC,\\
     t_{ij}(u) &\longmapsto \sum_{k = 1,2} t_{ik}(u) \otimes t_{kj}(u),&  T(u) &\longmapsto T^{-1}(u), & T(u) &\longmapsto 1.
\end{aligned}
\end{equation}

There is an embedding $\iota: U\sL \longrightarrow \Ygngl$ given by
\begin{equation}
    \begin{aligned}[c]
    h &\longmapsto t_{11}^{(1)} - t_{22}^{(1)},\\ 
    e &\longmapsto t_{12}^{(1)},\\
    f &\longmapsto t_{21}^{(1)}.
\end{aligned}\qquad
\end{equation}
Here $e,h,f$ are the standard generators of $U\sL$.
We often write simply $U\sL \subset \Ygngl$ for the image of this embedding.

There is a homomorphism of associative algebras depending on a parameter $a \in \mbC$ defined by
\begin{equation}
    \begin{aligned}[c]
    \ev_a : \Ygngl &\longrightarrow U\mathfrak{gl}_{2},\\
    t_{ij}(u) &\longmapsto \left(\delta_{ij}1 + \frac{E_{ij}}{u + a}\right).
\end{aligned}
\end{equation}
Here $E_{ij}$, $i,j=1,2$, are the standard generators of $U\mathfrak{gl}_{2}$.
\medskip

Let $$\mathrm{qdet}(u) = t_{11}(u)t_{22}(u{-}1) - t_{21}(u)t_{12}(u{-}1)$$ be the quantum determinant.  Coefficients of $\mathrm{qdet}(u)$ generate the center $\mathcal{Z}\Ygngl$ of $\Ygngl$ (see \cite{chari1995guide}).

It can be verified that
\begin{equation*}
    \Delta(\mathrm{qdet}(u)) = \mathrm{qdet}(u)\otimes \mathrm{qdet}(u), \qquad  S(\mathrm{qdet}(u)) = \mathrm{qdet}(u)^{-1}.
\end{equation*}

The algebra $\Ygn$ is the quotient Hopf algebra 
\begin{equation}
    \Ygn = \Ygngl / \left(\mathrm{qdet}(u) - 1\right).
\end{equation}
We denote $\bar{t}^{(r)}_{ij}\in \Ygn$ and $\bar{t}^{(r)}_{ij}(u)\in \Ygn[[u^{-1}]]$ the images of generators of $\Ygngl$  and corresponding formal series, respectively.

Composition of the inclusion $\iota$  with the quotient map   gives an injective homomorphism of Hopf algebras $\bar \iota:\ U\sL \longmapsto \Ygn$.

Let  $\varphi(u) \in\mathcal{Z}\Ygngl[[u^{-1}]]$ be such that
\begin{equation*}
\varphi(u)\varphi(u-1)\mathrm{qdet}(u) = 1, \qquad \varphi(\infty) = 1. 
\end{equation*}
We have $\Delta(\varphi(u))=\varphi(u)\otimes \varphi(u)$ and there is an injective homomorphism of Hopf algebras.
\begin{equation*}
\begin{aligned}
&\mathfrak{i}: \Ygn \longrightarrow \Ygngl,\\
&\hspace{8pt}\bar{t}_{ij}(u) \longmapsto \varphi(u)t_{ij}(u).
\end{aligned}\end{equation*}

Composition of an evaluation homomorphism $\ev_a$ with $\mathfrak{i}$ gives a homomorphisms of associative algebras $\Ygn \longrightarrow \mathrm{U}\mathfrak{gl}_{2}$ which we  denote by $\overline{\ev}_a$.

{There is a family of shift automorpisms $\tau_a$ of $\Ygngl$ parameterized by a complex number $a$,
\begin{equation}\label{eq:Ygn_shift}
    \begin{aligned}        &\hspace{0pt}\tau_{a}:\Ygngl  \longrightarrow \Ygngl,\\
    &\hspace{25pt}t_{ij}(u) \longmapsto t_{ij}(u+a).\\
    \end{aligned}\end{equation} 
Note that the ideal $(\mathrm{qdet}(u)-1)$ is preserved by $\tau_{a}$, therefore this automorphism descends to an automorphism of $\Ygn$ which we  denote by $\bar{\tau}_{a}$.

 \medskip

To answer questions about submodule structure of a representation of $\Ygn$ computationally, it is convenient to use finite sets of generators.
\begin{prop}
For any $i,j\in\{1,2\}$, the set $\{\bar{t}_{11}^{(1)}-\bar{t}_{22}^{(1)},\bar{t}_{12}^{(1)}, \bar{t}_{21}^{(1)}, \bar{t}_{ij}^{(2)}\}$ generates $\Ygn$.
\end{prop}
\begin{proof} Denote $Y_{ij}$ the subalgebra of $\Ygn$ generated by $\{\bar{t}_{11}^{(1)}-\bar{t}_{22}^{(1)},\bar{t}_{12}^{(1)}, \bar{t}_{21}^{(1)}, \bar{t}_{ij}^{(2)}\}$.

Write $\mathrm{qdet}(u)=\sum_{s=1}^\infty \mathrm{qdet}_s\  u^{-s}$. Then
\begin{equation}\label{qdet mode}
\mathrm{qdet}_s =\bar{t}_{11}^{(m)} + \bar{t}_{22}^{(m)} + \textit{lower terms}.
\end{equation}

    Let $i = j$.  From \eqref{qdet mode} with $m=2$, we conclude that $\bar{t}_{11}^{(2)}\in Y_{ij}$ and $\bar{t}_{22}^{(2)}\in Y_{ij}$. Observe that
    \begin{subequations}
    \begin{equation*}
         [\bar{t}_{12}^{(1)}, \bar{t}_{22}^{(2)}] = [\bar{t}_{12}^{(1)}, \bar{t}_{22}^{(2)}] - [\bar{t}_{12}^{(0)}, \bar{t}_{22}^{(3)}] = \bar{t}_{22}^{(0)}\bar{t}_{12}^{(2)} - \bar{t}_{22}^{(2)}\bar{t}_{12}^{(0)} = \bar{t}_{12}^{(2)} - \bar{t}_{22}^{(2)},
    \end{equation*}
    \begin{equation*}
         [\bar{t}_{21}^{(1)}, \bar{t}_{22}^{(2)}] = [\bar{t}_{21}^{(1)}, \bar{t}_{22}^{(2)}] - [\bar{t}_{21}^{(0)}, \bar{t}_{22}^{(3)}] = \bar{t}_{21}^{(0)}\bar{t}_{22}^{(2)} - \bar{t}_{21}^{(2)}\bar{t}_{22}^{(0)} = \bar{t}_{22}^{(2)} - \bar{t}_{21}^{(2)}.
    \end{equation*}
    \end{subequations}
    Hence $\bar{t}_{12}^{(2)}\in Y_{ij}$ and $\bar{t}_{21}^{(2)}\in Y_{ij}$.

    Let $i\neq j$. Without loss of generality, we assume $i = 1, j = 2$. Then 
    \begin{equation*}
        [\bar{t}_{21}^{(1)}, \bar{t}_{12}^{(2)}] = [\bar{t}_{21}^{(1)}, \bar{t}_{12}^{(2)}] - [\bar{t}_{21}^{(0)}, \bar{t}_{12}^{(3)}] = \bar{t}_{11}^{(0)}\bar{t}_{22}^{(2)} - \bar{t}_{22}^{(0)}\bar{t}_{11}^{(2)} = \bar{t}_{11}^{(2)} - \bar{t}_{22}^{(2)}.
    \end{equation*}
   From \eqref{qdet mode} with $m=2$, we conclude that $\bar{t}_{11}^{(2)}\in Y_{ij}$ and $\bar{t}_{22}^{(2)}\in Y_{ij}$. It follows that $\bar{t}_{21}^{(2)}\in Y_{ij}$.

Thus, in all cases we have $\bar{t}_{kl}^{(2)}\in Y_{ij}$ for all $k,l,i,j$.
    Suppose $\bar{t}_{kl}^{(s)}\in Y_{ij}$ for all $s \leq N$. Then 
    \begin{multline*}
        [\bar{t}_{12}^{(2)}, \bar{t}_{21}^{(N)}] = [\bar{t}_{12}^{(1)}, \bar{t}_{21}^{(N+1)}] + \bar{t}_{22}^{(1)}\bar{t}_{11}^{(N)} - \bar{t}_{22}^{(N)}\bar{t}_{11}^{(1)} = \left([\bar{t}_{12}^{(0)}, \bar{t}_{21}^{(N+2)}] + \bar{t}_{22}^{(0)}\bar{t}_{11}^{(N+1)} - \bar{t}_{22}^{(N+1)}\bar{t}_{11}^{(0)}\right) +\\+ \bar{t}_{22}^{(1)}\bar{t}_{11}^{(N)} - \bar{t}_{22}^{(N)}\bar{t}_{11}^{(1)} 
        = \bar{t}_{11}^{(N+1)} - \bar{t}_{22}^{(N+1)} + \textit{lower terms}.
    \end{multline*}
    
    Using \eqref{qdet mode} with $m=N+1$, $\bar{t}_{11}^{(N+1)}, \bar{t}_{22}^{(N+1)}\in Y_{ij}$. Then
    \begin{subequations}
    \begin{equation*}
        [\bar{t}_{12}^{(1)}, \bar{t}_{22}^{(N+1)}] = [\bar{t}_{12}^{(1)}, \bar{t}_{22}^{(N+1)}] - [\bar{t}_{12}^{(0)}, \bar{t}_{22}^{(N+2)}] = \bar{t}_{22}^{(0)}\bar{t}_{12}^{(N+1)} - \bar{t}_{22}^{(N+1)}\bar{t}_{12}^{(0)} = \bar{t}_{12}^{(N+1)},
    \end{equation*}
    \begin{equation*}
        [\bar{t}_{21}^{(1)}, \bar{t}_{22}^{(N+1)}] = [\bar{t}_{21}^{(1)}, \bar{t}_{22}^{(N+1)}] - [\bar{t}_{21}^{(0)}, \bar{t}_{22}^{(N+2)}] = \bar{t}_{21}^{(0)}\bar{t}_{22}^{(N+1)} - \bar{t}_{21}^{(N+1)}\bar{t}_{22}^{(0)} = -\bar{t}_{21}^{(N+1)}.
    \end{equation*}
    \end{subequations}
    Therefore, $\bar{t}_{kl}^{(N+1)}\in Y_{ij}$. 
    \end{proof}

Categories of finite-dimensional representations of $\Ygn$ and of $\Uqa$ are equivalent, see \cite{gautam2016yangians}. In the paper we mainly use $\Uqa$. We believe that the corresponding consideration for the Yangian are essentially equivalent. However, for brute force computer computations, it is often advantageous to use $\Ygn$.

\section{Known results on representations of $\Uqa$.}\label{sec:basic_repth}
We collect some useful facts about $\Uqa$-modules and supply some proofs.
\subsection{Properties of modules over a Hopf algebra.}\label{subsec:duals}

    We use properties of $\Uqa$-modules induced by the Hopf algebra structure.
    
    Let $\rH$ be a Hopf algebra over $\mbC$ with $\Delta, \eta, \epsilon, S$ being comultiplication, unit, counit and antipode respectively. All modules of $\rH$ we consider are assumed to be finite-dimensional.

    The trivial module of $\rH$ is vector space $\mbC$ with $a\in \rH$ acting on $\mbC$ by $\epsilon(a)$. Let $\rV$ be an $\rH$-module.   Since $(\epsilon \otimes \mathrm{id}) \circ \Delta = (\mathrm{id} \otimes \epsilon) \circ \Delta = \mathrm{id}$,  for an $\rH$-module $\rV$, we have $\rV \otimes \mbC \cong \mbC \otimes \rV \cong \rV$ as $\rH$-modules.

    Denote by $\rho_{\rV}$ the corresponding homomorphism $\rH\to\mathrm{End}_{\mbC}\left(\rV\right)$. The left dual module $\rV^{*}$ is the dual vector space on which $a\in \rH$ acts by $\rho_{\rV}(S(a))^*$. The right dual module ${}^{*}\rV$ is the dual vector space on which $a\in \rH$ acts by $\rho_{\rV}(S^{-1}(a))^*$. 

    We have $({}^*\rV)^*\cong {}^*(\rV^*)\cong \rV$. We often call left dual module simply dual module.
    
    The following properties are standard exercises in the theory of Hopf algberas. 

    \begin{lemma}\label{lemma:hom_move}
    Let $\rV$, $\rW$ and $\rU$ be $\rH$-modules. Then there are canonical isomorphisms of $\rH$-modules
    \begin{align}
         \left( \rV \otimes \rW \right)^{*} &\cong \rW^* \otimes \rV^{*},\\  {}^*\hspace{-2pt}\left(\rV \otimes \rW \right) &\cong {}^*\rW \otimes {}^{*}\rV.\hspace{15pt}{}
    \end{align}
In addition there are canonical isomorphisms vector spaces
    \begin{align}
        \mathrm{Hom}_{\rH}\left(\rV\otimes \rW, \rU\right) &\cong \mathrm{Hom}_{\rH}\left(\rV, \rU\otimes \rW^{*}\right),\label{eq:dual_product}\\
       \mathrm{Hom}_{\rH}\left( \rW \otimes \rV, \rU \right) &\cong \mathrm{Hom}_{\rH} \left( \rV, {}^{*}\rW\otimes \rU \right).\\       
    \end{align} 
\end{lemma}
We supply the proof of Lemma \ref{lemma:hom_move} in the Appendix \ref{app:etc}. 

The following corollary is immediate.
\begin{cor}\label{cor:hom_switch}
    Let $\rV$ and $\rW$ be $\rH$-modules. Then there are canonical isomorphisms of vector spaces
    \begin{equation}\label{eq:triv_mod_eq}
        \mathrm{Hom}_{\rH}\left(\rV, \rW\right) \cong  \mathrm{Hom}_{\rH}\left(\rW^{*}, \rV^{*}\right) \cong \mathrm{Hom}_{\rH}\left(\mbC, \rW\otimes \rV^{*}\right).
    \end{equation}
    \qed
\end{cor}

We also use the following standard fact,
\begin{lemma}\label{lemma:exact_functors}
    Let $\rM$ be an $\rH$-module.
    \begin{itemize}
    \item Functor $\Hom(\rM, \cdot)$ is exact from the left.
    \item Functors $\rM\otimes \cdot$ and $\cdot\otimes \rM$ are exact 
    \end{itemize}
    In other words, given a short exact sequence of $\rH$-modules
    $$
     \rU \hookrightarrow \rV \twoheadrightarrow \rW,
    $$
    there are exact sequences
    \begin{align*}
          \Hom(\rM, \rU) &\hookrightarrow \Hom(\rM, \rV) \rightarrow \Hom(\rM, \rW),\\
         \rM\otimes \rU &\hookrightarrow \rM\otimes \rV \twoheadrightarrow \rM\otimes \rW,\\
         \rU \otimes \rM &\hookrightarrow  \rV\otimes \rM \twoheadrightarrow \rW \otimes \rM.
    \end{align*}
    \qed
\end{lemma}

Given an algebra automorphism $\varphi:\rH \longrightarrow \rH$ and an $\rH$ module $\rV$, let $\rV_{\varphi}$ be the module twisted by $\varphi$. In other words, $\rV_{\varphi} = \rV$ as a vector space and any $a\in \rH$ acts on $\rV_{\varphi}$  by $\rho_{\rV}(\varphi(a))$.

The following lemma is obvious.

\begin{lemma}\label{lemma:twist} Let $\rV, \rW$ be  $\rH$-modules and $\varphi$ a Hopf algebra automorphism of $\rH$. Then 
\begin{align*}
    (\rV \otimes \rW)_{\varphi} &\cong \rV_{\varphi} \otimes \rW_{\varphi}, \\
    (\rV^*)_{\varphi} &\cong (\rV_{\varphi})^*, \\
    \Hom_{\rH}\lb \rV_{\varphi}, \rW_{\varphi}\rb &\cong \Hom_{\rH}\lb \rV, \rW\rb.
\end{align*}
\qed
\end{lemma}

We end this section with another simple fact.
\begin{lemma}\label{prop:aaut_hom_tp}
    Let $\hat{\omega}$ be an isomorphism of Hopf algebras $(\rH, \Delta)$ and $(\rH, \Delta^{\textit{op}})$. 
    Let $\rV,\rW$ be $\rH$-modules.
    Then
    \begin{align}
    (\rV\otimes \rW)_{\hat{\omega}} &\cong \rW_{\hat{\omega}} \otimes \rV_{\hat{\omega}},\label{eq:aaut_hom_tp_1}\\
        \Hom_{\rH}(\rV_{\hat{\omega}}, \rW_{\hat{\omega}}) &\cong \Hom_{\rH}(\rV, \rW)\label{eq:aaut_hom_tp_2}.
    \end{align}
\end{lemma}}
\begin{proof}
Equations \eqref{eq:aaut_hom_tp_1}, \eqref{eq:aaut_hom_tp_2} are straightforward, cf. Lemma \ref{lemma:twist}.
\end{proof}

\subsection{Duals and shifts.} The dual of an evaluation module is an evaluation module with shifted evaluation parameter.
\begin{lemma}\label{lemma:dual}
    We have
\begin{equation}\label{eq:eval_dual}
    [\alpha, \beta]^* \cong [\alpha+2, \beta+2].
\end{equation}More generally,
\begin{equation*} 
([\alpha_1,\beta_1][\al_2,\bt_2]\dots [\al_n, \bt_n])^*\cong [\alpha_n+2,\beta_n+2]\dots [\al_2+2,\bt_2+2] [\al_1+2, \bt_1+2].
\end{equation*}
\end{lemma}
\begin{proof}
    Denote $\rV = [\alpha, \alpha+2,\beta]$ and recall evaluation parameter $a = \frac{\alpha+\beta}{2}$. Note that $S(K) = K^{-1}$ which implies that the set of $\Uq$-weights for the module $\rV^*$ is the same as for the module $\rV$. This means that $\rV=\rV^*$ as $\Uq$-modules. It remains to note that
    \begin{multline*}
        \rho_{\rV^*}(f_0)= \rho_{\rV}(S(f_0))^* = -\rho_{\rV}(K^{-1}f_0)^* = -\rho_{\rV, \Uq}(K^{-1}q^{-a}e)^{*} = \\ = q^{-a-2}\rho_{\rV,\Uq}(-e K^{-1})^{*} = \rho_{\rV}(S(e_1))^{*} = \rho_{\rV^*}(q^{-a-2}e_1).
    \end{multline*}
    Similarly,
    \begin{equation*}
        \rho_{\rV^*}(e_0) = \rho_{\rV^*}(q^{a+2}f_1).
    \end{equation*}
    
    The general case follows from Lemma \ref{lemma:hom_move}.
\end{proof} 

Recall the shift automorphisms $\tau_a$ of Hopf algebra $\Uqa$, see Section \ref{subsec:uqamod_and_q-char}. 
\begin{lemma}\label{lemma:shift}
    The twist of an evaluation module by a shift automorphism $\tau_a$ is the evaluation module with shifted parameter, 
    \begin{equation*}    [\alpha, \beta]_{\tau_{a}} \cong [\alpha+a,\beta+a].
    \end{equation*}
    More generally,
    \begin{equation*} 
    ([\alpha_1,\beta_1][\al_2,\bt_2]\dots [\al_n, \bt_n])_{\tau_{a}}\cong [\al_1+a, \bt_1+a][\al_2+a,\bt_2+a]\dots  [\alpha_n+a,\beta_n+a].
\end{equation*}
In particular,
\begin{equation*}   
\Hom_{\rH}(\mbC, [\alpha_1,\beta_1]\dots [\al_n,\bt_n])= \Hom_{\rH}(\mbC, [\alpha_1+a,\beta_1+a]\dots [\al_n+a,\bt_n+a]).
\end{equation*}

\end{lemma}
\begin{proof}
The first statement follows from the identity $\ev_{a} \circ \tau_{b} = \ev_{a+b}$.  The rest follows from Lemma \ref{lemma:twist} and $\mbC_{\tau_a}\cong \mbC$.
\end{proof}
\subsection{Lattice restrictions.}\label{subsec:lattice_restrictions}
Let  $\mathcal{L} = 2\mbZ + \frac{2\pi  \ri}{\log(q)}\mbZ$ be the lattice. 

If the evaluation parameters of two evaluation modules are not equal modulo $\mcl{L}$, then these modules are in general position. In particular, their tensor products in different orders are irreducible and isomorphic. Therefore, in any tensor product of evaluation modules we can order tensor factors by the class of evaluation parameter. In other words any product of evaluation modules $\rV$ can be written in the form

\begin{equation}\label{eq:eval_prod_abstr}
    \rV \cong \rV_{(a_1)} \otimes \rV_{(a_2)} \otimes \dots \otimes \rV_{(a_k)},
\end{equation}
where $\rV_{(a_j)}$ is product of all evaluation modules with parameters equal $a_j$ modulo $\mathcal{L}$ and $a_i$ are distinct modulo $\mathcal L$. Here $\rV_{(a_i)}$ can be a trivial module.
It is clear that the product is commutative $\rV_{(a_j)}\otimes \rV_{(a_{j+1})} \cong \rV_{(a_{j+1})}\otimes \rV_{(a_j)}$. 

To study spaces of homomorphisms of $\Uqa$-modules, it is sufficient to consider the case $k=1$, that is it is sufficient to consider tensor products of evaluation modules with evaluation parameters equal modulo $\mathcal L$. Moreover, by Lemma \ref{lemma:shift} all choices of $\mcl L$-coset are equivalent, so, one can assume $a_1=0$. 

Such a reduction is a known idea, see e.g. \cite{etingof2003elliptic}. We formulate the precise statement and give a proof which illustrates our methods.

\begin{prop}\label{prop:hom_factorize_lattice}
    Let $\rV$, $\rW$ be products of evaluation $\Uqa$-modules
    \begin{equation*}
        \rV = \rV_{(a_1)} \otimes  \dots \otimes \rV_{(a_k)},\;\; \rW = \rW_{(a_1)}  \otimes \dots \otimes \rW_{(a_k)},
    \end{equation*}
    as in \eqref{eq:eval_prod_abstr}. Then there is an isomorphism of vector spaces
    \begin{equation}\label{eq:hom_factorize_lattice}
        \Hom_{\Uqa}\left(\rV, \rW\right) \cong \bigotimes_{l=1}^{k}\Hom_{\Uqa}\left(\rV_{(a_l)},\rW_{(a_l)}\right).
    \end{equation}\end{prop}
\begin{proof}
    We prove \eqref{eq:hom_factorize_lattice}  by induction on $k$. Base $k=1$ is tautological. 

   Assume the statement for tensor products \eqref{eq:hom_factorize_lattice} with $k-1$ factors.  Denote $\bar \rV= \rV_{(a_2)} \otimes \dots \otimes \rV_{(a_k)}$, $\bar \rW =  \rW_{(a_2)} \otimes \dots \otimes \rW_{(a_k)}$. By Lemma \ref{lemma:hom_move},
    \begin{equation*}
    \Hom_{\Uqa}\left(\rV, \rW\right) \cong \Hom_{\Uqa}\lb \rW_{(a_1)}^{*} \otimes \rV_{(a_1)}, \bar{\rW}\otimes\bar{\rV} ^{*}\rb.
    \end{equation*}
    Note that the evaluation parameters of factors of $\rV_{(a_i)}^{*}$, $\rW_{(a_i)}^{*}$ equal to $a_i$ modulo $\mathcal L$.

    Let $\varphi \in \Hom_{\Uqa}\big( \rW_{(a_1)}^{*}\otimes \rV_{(a_1)} , \bar{\rW}\otimes\bar{\rV}^{*} \big)$. The image of $\varphi$ is a submodule of $\rW_{(a_1)}^{*}\otimes \rV_{(a_1)}$ isomorphic to a quotient module of $\bar{\rW}\otimes\bar{\rV}^{*}$. The only possible common monomial of $q$-characters,  $\chi_q(\rW_{(a_1)}^{*}\otimes \rV_{(a_1)})$ and $\chi_q(\bar{\rW}\otimes\bar{\rV}^{*})$ is $1$. 
    It follows that the $q$-character of the image of $\varphi$ is $1$ with some multiplicity.

    Let $\rU$ be an $\Uqa$-module such that $\chi_q(\rU)=m\cdot 1$ for $m\in\mbZ_{\geq 0}$. Then $U$ is a direct sum of trivial modules, $\rU=\oplus_{i=0}^m \mbC$. Indeed, $e_i,f_i$ act by zero in $\rU$ and $K$ is diagonalizable in any finite-dimensional $\Uq$-module.

     This implies that $\varphi$ belongs to the image of embedding
    \begin{equation}\label{eq:hom_factorize_lattice_pf1}
        \Hom_{\Uqa}\lb \rW_{(a_1)}^{*}\otimes \rV_{(a_1)},\mbC \rb \otimes \Hom_{\Uqa}\lb \mbC,\bar{\rW}\otimes \bar{\rV}^{*}\rb \longrightarrow \Hom_{\Uqa}\lb \rW_{(a_1)}^{*} \otimes \rV_{(a_1)}, \bar{\rW}\otimes\bar{\rV} ^{*}\rb,
    \end{equation}    given by composition. Hence, \eqref{eq:hom_factorize_lattice_pf1} is actually an isomorphism.
    
   Finally, by Lemma \ref{lemma:hom_move},
    \begin{equation*}
    \begin{aligned}        \Hom_{\Uqa}\lb \rW_{(a_1)}^{*}\otimes \rV_{(a_1)},\mbC \rb &\cong \Hom_{\Uqa}\lb \rV_{(a_1)}, \rW_{(a_1)} \rb,\\
        \Hom_{\Uqa}\lb \mbC,\bar{\rW}\otimes \bar{\rV}^{*}\rb&\cong \Hom_{\Uqa}\lb \bar{\rV},\bar{\rW}\rb.
    \end{aligned}    \end{equation*}
    Applying the induction hypothesis we obtain \eqref{eq:hom_factorize_lattice}. 
\end{proof}
In particular, Proposition \ref{prop:hom_factorize_lattice} asserts that 
\begin{equation*}\Hom_{\Uqa}\left(\mbC,\mathop{\otimes}_{l=1}^k \rW_{(a_l)} \right) \cong \bigotimes_{l=1}^{k}\Hom_{\Uqa}\left(\mbC,\rW_{(a_l)}\right).
\end{equation*}

We have formulated Proposition \ref{prop:hom_factorize_lattice} for tensor products of evaluation modules to serve the needs of this text. Clearly, one can formulate and prove in the same way a similar statement for arbitrary finite-dimensional $\Uqa$ modules.

Proposition \ref{prop:hom_factorize_lattice} can be strengthen. We discuss such possibilities in Sections \ref{subsec:submod_factor} and \ref{subsec:fact_by_cont}.

 From now on we always assume (unless directly said otherwise) that all our evaluation parameters are even integers. In other words, we always assume that the evaluation parameters are zero modulo $\mathcal L$. 

\subsection{Tensor products of two evaluation modules.}
As we know, tensor products of evaluation modules in general position are irreducible. In this section we give a description of all tensor products of two evaluation modules in non-general position obtained in \cite{chari1991quantum}. We give a proof using $q$-characters and an explicit form of $R$-matrix.

\begin{prop}\label{prop:two_strings_prod}\hfill
    Let $\al_1,\bt_1$  $\al_2 ,\bt_2$ be four even integers such that $\al_1 < \al_2\leq \bt_1+2\leq \bt_2$. Then there are two non-split short exact sequences
    \bse\label{eq:two_strings_prod}
       \begin{equation}\label{eq:two_strings_prod1}
            [\alpha_1, \alpha_2 - 4]  [\beta_{1} +4, \beta_2] \hookrightarrow [\alpha_1, \beta_{1}]  [\alpha_{2},\beta_{2}] \twoheadrightarrow   [\alpha_{1},\beta_{2}][\alpha_{2}, \beta_{1}],
       \end{equation}
       \begin{equation}\label{eq:two_strings_prod2}
          [\alpha_{1},\beta_{2}][\alpha_{2}, \beta_{1}] \hookrightarrow [\alpha_{2},\beta_{2}][\alpha_1, \beta_{1}]   \twoheadrightarrow [\alpha_1, \alpha_2 - 4]  [\beta_{1} +4, \beta_2] .
       \end{equation}
       \ese
\end{prop}
\begin{proof}
First we use $q$-characters to find the composition factors of $[\alpha_1, \beta_{1}][\alpha_{2},\beta_{2}]$. 

Note that products $[\al_1,\bt_2][\al_2,\bt_1]$ and $[\al_1,\al_2-4][\bt_1+4,\bt_2]$ are in general position and therefore they are irreducible. The highest monomials of the $q$-characters of modules $[\alpha_1, \beta_{1}][\alpha_{2},\beta_{2}]$ and $[\alpha_{2},\beta_{1}][\alpha_{1}, \beta_{2}]$ coincide,
    \begin{equation*}        1_{\alpha_1}1_{\alpha_1+2}\dots1_{\beta_1}1_{\alpha_2}\dots1_{\beta_2} = \lb 1_{\alpha_1}1_{\alpha_1+2}\dots1_{\beta_1}1_{\beta_1+2}\dots1_{\beta_2}\rb \lb 1_{\alpha_2}\dots1_{\beta_1} \rb.
    \end{equation*}
    Moreover, the difference
   
    \begin{equation*}
        \chi_q([\alpha_1, \beta_{1}])\chi_q([\alpha_{2},\beta_{2}]) - \chi_q( [\alpha_{1},\beta_{2}][\alpha_{2}, \beta_{1}])
    \end{equation*}
    contains exactly one dominant monomial,
    \begin{equation*}
        \lb 1_{\alpha_1}1_{\alpha_1+2}\dots1_{\alpha_2-4} 1_{\alpha_2}^{-1}\dots1_{\beta_{1}+2}^{-1}\rb  \lb 1_{\alpha_2}1_{\alpha_2+2}\dots1_{\beta_2}\rb  = \lb 1_{\alpha_1}1_{\alpha_1+2}\dots1_{\alpha_2-4}\rb \lb 1_{\beta_1+4}\dots1_{\beta_2} \rb,
    \end{equation*}
 see Example \ref{ex:q_char_mult_1}. Therefore, this difference coincides with $\chi_q([\al_1,\al_2-4][\bt_1+4,\bt_2])$.
    
Hence, the products $[\al_1,\bt_2][\al_2,\bt_1]$ and $[\al_1,\al_2-4][\bt_1+4,\bt_2]$ form the composition factors of $[\al_1,\bt_1][\al_2\dots\bt_2]$.

The module $[\al_1,\al_2-4][\bt_1+4,\bt_2]$ is a submodule of 
$[\al_1,\bt_1][\al_2\dots\bt_2]$. Indeed, by 
Proposition \ref{prop:Rmat_lambdas} there exists a $\Uqa$-homomoprhism $[\al_1,\bt_1][\al_2\dots\bt_2]\longrightarrow [\al_2,\bt_2][\al_1\dots\bt_1]$ with non-trivial kernel which does not contain a vector of the highest $\sL$-weight. Therefore, the kernel of this map is a submodule isomorphic to $[\al_1,\al_2-4][\bt_1+4,\bt_2]$.

Similarly, the module $[\al_1,\bt_2][\al_2,\bt_1]$ is a submodule of 
$[\al_2,\bt_2][\al_1\dots\bt_1]$.

It remains to check that the sequences \eqref{eq:two_strings_prod} do not split. One sequence is a shifted dual of the other, see Section \ref{subsec:duals}. Therefore, if one of sequences splits, so does the other. If both modules  $[\alpha_1, \beta_{1}]  [\alpha_{2},\beta_{2}]$ and $[\alpha_{2},\beta_{2}] [\alpha_1, \beta_{1}]$ were direct sums of two simple modules, then the space of $\Uqa$-homomorphisms between them would be two-dimensional in contradiction to Proposition \ref{prop:Rmat_lambdas}.
\end{proof}

We illustrate the short exact sequence
\begin{equation*}
  [0][14,16] \hookrightarrow [0,10][4,16] \twoheadrightarrow [4,10][0,16],
\end{equation*}
which is a particular case of \eqref{eq:two_strings_prod1}, by the following picture. 
\medskip
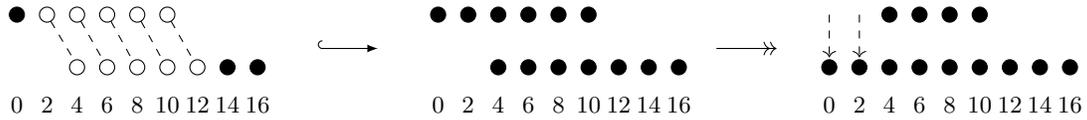
\begin{center}
\begin{figure}[H]
\tikzset{decorate sep/.style 2 args=
{decorate,decoration={shape backgrounds,shape=circle,shape size=#1,shape sep=#2}}}
\begin{tikzpicture}
\draw[decorate sep={2mm}{4mm},fill] (3.8,0) -- (4.21,0);
\draw[decorate sep={2mm}{4mm},fill] (1, 0.7) -- (1.01, 0.7);

\draw[decorate sep={2mm}{4mm}] (1.8,0) -- (3.6,0);
\draw[decorate sep={2mm}{4mm}] (1.4,0.7) -- (3.0,0.7);
\draw[-, dashed] (1.45, 0.6125) -- (1.765, 0.0875);
\draw[-, dashed] (1.85, 0.6125) -- (2.15,0.0875);
\draw[-, dashed] (2.25, 0.6125) -- (2.54,0.0875);
\draw[-, dashed] (2.65, 0.6125) -- (2.94,0.0875);
\draw[-,dashed] (3.05, 0.6125) -- (3.35,0.0875);

\draw[decorate sep={2mm}{4mm},fill] (7.4,0) -- (10.0,0);
\draw[decorate sep={2mm}{4mm},fill] (6.6,0.7) -- (8.6,0.7);

\draw[decorate sep={2mm}{4mm},fill] (12.6,0.7) -- (14.2,0.7);
\draw[decorate sep={2mm}{4mm},fill] (11.8,0) -- (15.2,0);
\draw[->, dashed] (11.8,0.7) -- (11.8,0.13);
\draw[->, dashed] (12.2,0.7) -- (12.2,0.13);


\draw[->,right hook-latex] (5.0,0.25) -- (5.8,0.25);

\draw[->>] (10.3,0.25) -- (11.1,0.25);



\node at (1, -0.5) {{\scriptsize $0$}};
\node at (1.4, -0.5) {\scriptsize $2$};
\node at (1.8, -0.5) {\scriptsize $4$};
\node at (2.2, -0.5) {\scriptsize $6$};
\node at (2.6, -0.5) {\scriptsize $8$};
\node at (3.0, -0.5) {\scriptsize $10$};
\node at (3.4, -0.5) {\scriptsize  $12$};
\node at (3.8, -0.5) {\scriptsize $14$};
\node at (4.2, -0.5) {\scriptsize $16$};

\node at (6.6, -0.5) {{\scriptsize $0$}};
\node at (7, -0.5) {\scriptsize $2$};
\node at (7.4, -0.5) {\scriptsize $4$};
\node at (7.8, -0.5) {\scriptsize $6$};
\node at (8.2, -0.5) {\scriptsize $8$};
\node at (8.6, -0.5) {\scriptsize $10$};
\node at (9.0, -0.5) {\scriptsize  $12$};
\node at (9.4, -0.5) {\scriptsize $14$};
\node at (9.8, -0.5) {\scriptsize $16$};

\node at (11.8, -0.5) {{\scriptsize $0$}};
\node at (12.2, -0.5) {\scriptsize $2$};
\node at (12.6, -0.5) {\scriptsize $4$};
\node at (13.0, -0.5) {\scriptsize $6$};
\node at (13.4, -0.5) {\scriptsize $8$};
\node at (13.8, -0.5) {\scriptsize $10$};
\node at (14.2, -0.5) {\scriptsize  $12$};
\node at (14.6, -0.5) {\scriptsize $14$};
\node at (15.0, -0.5) {\scriptsize $16$};
\end{tikzpicture}
\end{figure}
\end{center}
On this picture each circle corresponds to a variable $1_a$ where the horizontal coordinate is given by $a$. The variables corresponding to the filled circles aligned horizontally form the dominant monomial of one of the factors. The factors are ordered from top to bottom.  We show by the dashed lines and non-filled circles the variables which are cancelled during $q$-character multiplication.  The quotient module on the right is a result of vertical reordering of variables $1_a$ which produces a product of strings in general position.

\section{Basics of trivial submodules.}\label{sec:triv_submod}
In this paper we call arbitrary tensor products of 2-dimensional evaluation modules with evaluation parameters in $2\mbZ$ by words, see Section \ref{subsec:word_arcs} below. A major goal of this paper is to study the spaces $\Hom_{\Uqa}(w_1, w_2)$ for arbitrary words $w_1,w_2$. By Corollary \ref{cor:hom_switch} it is equivalent to the study of the spaces  $\Hom_{\Uqa}(\mbC, w)$. 

Denote 
$$
H(w) = \Hom_{\Uqa}(\mbC, w), \qquad  h(w) = \dim(H(w)).
$$
The spaces $H(w)$ and the numbers $h(w)$ are one of the main objects of study in this text.

The  space $H(w)$ is canonically identified with the space of $\ell$-singular vectors of $\ell$-weight $1$ in $w$. We do not make distinction between these two spaces.

\subsection{Words and arcs.}\label{subsec:word_arcs} 
We develop the language related to products of two-dimensional evaluation modules. 

We start from combinatorial objects related to study of tensor powers of $\rL_1$. Fix some $n\in \mbZ_{>0}$. Let $i, j \in \{1,\dots, 2n\}$, we call an ordered pair $(i,j)$ with $i<j$ an uncolored arc. For an uncolored arc $(i,j)$, we call $i$ (respectively $j$) left (respectively right) end of $(i,j)$. We say that the arc $(i,j)$ connects to $i$ and to $j$. For a set of uncolored arcs $C = \{(i_{k},j_{k})\}$, denote
    $\len (C) = \{i_1,\dots, i_n\},\;\; \ren (C) = \{j_1,\dots, j_n\}$ the sets of left and right ends of arcs in $C$, respectively. We call a pair of uncolored arcs $(i_1, j_1), (i_2, j_2)$ intersecting if $i_1 <  i_2 < j_1 < j_2$ or $i_2 < i_1 < j_2 < j_1$.

We call a set of uncolored arcs $\{(i_1, j_1),\dots , (i_n, j_n)\}$ an uncolored arc configuration if
    \begin{equation*}
        \{1,\dots, 2n\} = \mathop{\bigsqcup}_{k = 1}^{n}\{i_k, j_k\}.
    \end{equation*}
Given an uncolored arc configuration $C =\{(i_1, j_1),\dots, (i_n,j_n)\}$ we say that letters $i_k$ and $j_k$ of $w$ are connected by an arc in $C$. 

We call an uncolored arc configuration connected if for any two arcs $\alpha, \beta$ there exists a sequence $(\alpha = \alpha_0, \alpha_1,\dots, \alpha_m=\beta)$, such that for all $i\in \{0,\dots, m-1\}$ arcs $\alpha_i$ and $\alpha_{i+1}$ intersect. 

We call an uncolored arc configuration $C = \{(i_1, j_1),\dots , (i_n, j_n)\}$ an uncolored Catalan arc configuration if any two pairs of arcs in that configuration do not intersect. There is a convenient pictorial description for uncolored Catalan arc configurations. For instance, there are five uncolored Catalan arc configurations for $n=3$ given by the following picture.

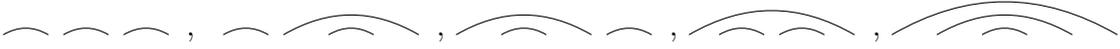
\begin{figure}[H]
\begin{tikzpicture}
    \draw[-] (0,0) to [out=30,in=150] (0.6,0);
    \draw[-] (0.8,0) to [out=30,in=150] (1.4,0);
    \draw[-] (1.6,0) to [out=30,in=150] (2.2,0);
\end{tikzpicture}\;\;,\;\;
\begin{tikzpicture}
    \draw[-] (0.4,0) to [out=30,in=150] (1.0,0);
    \draw[-] (1.2,0) to [out=30,in=150] (3.0,0);
    \draw[-] (1.8,0) to [out=30,in=150] (2.4,0);
\end{tikzpicture}\;\;,
\begin{tikzpicture}
    \draw[-] (1.2,0) to [out=30,in=150] (3.0,0);
    \draw[-] (1.8,0) to [out=30,in=150] (2.4,0);
    \draw[-] (3.2,0) to [out=30,in=150] (3.8,0);
\end{tikzpicture}\;\;,
\begin{tikzpicture}
    \draw[-] (0.2,0) to [out=30,in=150] (2.4,0);
    \draw[-] (0.6,0) to [out=30,in=150] (1.2,0);
    \draw[-] (1.4,0) to [out=30,in=150] (2.0,0);
\end{tikzpicture}\;\;,
\begin{tikzpicture}
    \draw[-] (0,0) to [out=30,in=150] (3.0,0);
    \draw[-] (0.6,0) to [out=30,in=150] (2.4,0);
    \draw[-] (1.2,0) to [out=30,in=150] (1.8,0);
\end{tikzpicture}\,\,
\end{figure}
Note that an uncolored Catalan arc configuration $C$ is uniquely defined by the set of the left ends $\len(C)$ (or by the set of the right ends $\ren(C)$).

We denote the set of uncolored arcs, uncolored arc configurations and uncolored Catalan arc configurations as $\uarc(2n),\; \uconf(2n),\;\ucconf(2n)$.

For a given $C = \{(i_k, j_k)\}_{k=1}^n\in \uconf(2n)$ and a set of $m$ arcs $S \subset C$ define a new arc configuration $\mathring{C}_S\in \uconf(2(n-m))$ obtained by removing arcs in $S$ from $C$. Formally, let $C\backslash S = \{ (i_{k}^{\prime}, j_{k}^{\prime})\}_{k \in \{ 1, \dots, n-m\}}$. Denote
\begin{equation*}
\begin{aligned}    i_{k}^{\prime \prime} = i_{k}^{\prime} - |((\len(S)\cup \ren(S)) \cap \{1,\dots, i_{k}^{\prime}-1\})|,\\
    j_{k}^{\prime \prime} = j_{k}^{\prime} - |((\len(S)\cup \ren(S)) \cap \{1,\dots, j_{k}^{\prime}-1\})|,
\end{aligned}\end{equation*}
and set $\mathring{C}_S = \{ (i_{k}^{\prime\prime}, j_{k}^{\prime\prime})\}_{k \in \{ 1, \dots, n-m\}}\in \uconf(2(n-m))$.

Note that if $C \in \ucconf(2n)$, then $\mathring{C}_{S} \in \ucconf(2(n-m))$.

\medskip 

Uncolored arc configurations are suited to study singular $\Uq$ singular vectors in $\rL_1^{\otimes n}$.
Now we introduce combinatorial objects related to $\Uqa$ singular vectors in  products of two-dimensional evaluation modules. We denote the product of two-dimensional evaluation modules $[a_1,a_1]\dots [a_k,a_k]$ by $(a_1,\dots,a_k)$. 

We call the sequence of even integers $w=(a_1,\dots,a_k)$ a word $w$. Often we do not distinguish between a word $w = (a_1,..., a_k)$ and $\Uqa$-module $[a_1,a_1]\dots [a_k,a_k]$.

For a word $w = (a_1,\dots,a_k)$ we write simply $a_1a_2\dots a_k$. In particular, in all examples in the present work all $a_i$ are one-digit even numbers, so there is no confusion.

For a word $w = (a_1,..., a_k)$, we call $k = l(w)$ length of $w$. We use the notation $W_k = \lb2\mbZ\rb^{\times k}$ for the set of all words of the length $k$. The permutation group $S_{k}$ acts on the set $W_{k}$ by permutations of letters.

Due to Lemma \ref{lemma:shift}, $h(w)$ does not change if we shift all letters of $w$ by the same number. This allows to consider words modulo overall shift of all letters by $2\mbZ$. It is clear that if $l(w)$ is odd, then $h(w) = 0$, so we restrict our attention to the cases when $l(w)$ is even.

Given a word $w = (a_1,\dots, a_{k})$ and a letter $b \in 2\mbZ$, we denote $I_{w}(b) = \{j\in \{1,\dots, k\}\ |\ a_{j} = b\}$. We call the set $\mathrm{supp}(w) = \{b\in 2\mbZ\ |\ I_{w}(b)\neq \varnothing\}$ the support of the word $w$. We call a word $w$ connected if $\supp(w)$ is a segment in $2\mbZ$, that is if $\supp(w)=\{a,a+2,\dots,a+b\}$ for some $a,b\in 2\mbZ$. For a word $w = (a_1,\dots, a_k)$, we call the multiset $\{a_1, \dots, a_k\}$ the content of $w$ and denote it $\cont(w)$.

For a word $w = (a_1,..., a_k)$ we call $w^* = (a_{k}+2, a_{k-1}+2,..., a_{1}+2)$ the dual word,

For a pair of words $w_1 = (a_1,\dots, a_k)\in W_k,\;\; w_2 = (b_1,\dots, b_l)\in W_l$ we write $w_1 w_2$ for the concatenation $(a_1,\dots, a_k, b_1, \dots, b_l)$.

Let $w = (a_1,\dots,a_k)$ and $\tilde{w} = (b_1,\dots, b_l)$ be words. We call $\tilde{w}$ is a subword of $w$ if $\tilde w$ is obtained from $w$ by deleting some letters, in other words if $b_j = a_{i_j}$ for some increasing sequence $1 \leq i_1 < i_2 < \dots < i_l \leq k$. In the case of $w=w_1\tilde w w_2$, that is if $i_{j+1} = i_{j}+1$ for all $1\leq j\leq l-1$, we call $\tilde{w}$ a factor of $w$.

We use notation $a^n$ for the word $(a,\dots, a) = a\dots a$, where $a$ is repeated $n$ times.

Given two words $w_1, w_2$ we call a word $w_3$ a shuffle of $w_1$ and $w_2$ if $w_3$ contains $w_1$ as a subword and removing this subword from $w_3$ gives $w_2$.

\begin{defi}\label{def:arc}
    For a word $w = (a_1,...,a_{2n})$ we call an uncolored arc $(i,j)\in \uarc(2n)$ an arc of color $b$ if $a_{j}-1 = a_{i} + 1=b$.
\end{defi}
We denote the set of arcs of a word $w$ of a color $b$ by $\arc_b(w)$ and we denote 
$$\arc(w) = \mathop{\bigsqcup}_{b\in 2\mbZ+1}\arc_b(w).$$
We call elements of $\arc(w)$ arcs in the word $w$. 

A motivation for Definition \ref{def:arc} is given by the following lemma. 

\begin{lemma}
    Module $w = (a_1, ... , a_k)$ is irreducible if and only if $\arc(w) = \arc(w^*) = \varnothing$.
\end{lemma}
\begin{proof} The lemma follows directly from the description of irreducible modules, see Section \ref{subsec:uqamod_and_q-char}.
\end{proof}
The concept of arcs is useful for other reasons as well and we will extensively use it.

\begin{defi}\label{def:arc_conf}
    For a  word $w$ of length $2n$, we call a set of arcs $\{(i_1, j_1),\dots , (i_n, j_n)\} \subseteq \arc(w)$ an arc configuration if
    \begin{equation*}
        \{1,\dots, 2n\} = \mathop{\bigsqcup}_{k = 1}^{n}\{i_k, j_k\}.
    \end{equation*}
\end{defi}
 We denote by $\conf(w)$ the set of all arc configurations of a word $w$.

\begin{lemma}
Let $w$ be a word such that the set $\conf(w)$ is non-empty. Then the number $m_{b}$ of arcs of the color $b$ in $C\in \conf(w)$  does not depend on $C$ and equals 
\begin{equation}\label{eq:arc_number}
m_b=\sum_{j=0}^{\infty}(-1)^{j}|I_{w}(b-1-2j)|=\sum_{j=0}^{\infty}(-1)^{j}|I_{w}(b+1+2j)|. 
\end{equation}
\end{lemma}
\begin{proof} 
If $b=\min (\supp(w))+1$, then $m_b=|I_w(b-1)|$.  
If $b=\min (\supp(w))+3$, then clearly $m_b=|I_w(b-1)|-|I_w(b-3)|$.  
The first equation of the lemma follows by induction on $b$ since $m_b=|I_w(b-1)|-m_{b-2}$. The second equation of the lemma is similar.
\end{proof}
Clearly, $\arc(w) \subset \uarc(l(w)),\;\;\conf(w) \subset \uconf(l(w))$, so the language developed above for uncolored arcs (uncolored arc configurations) is applicable to arcs (arc configurations). 

Note also that if $C \in \conf(w)$, then $\mathring{C}_{S} \in \conf(\mathring{w}_S)$, where $\mathring{w}_S$ is the word $w$ with letters in positions $\len(S)\cup \ren(S)$ removed.

We denote $\cconf(w) = \ucconf(l(w))\cap \conf(w)$.

\subsection{Vectors corresponding to Catalan arc configurations.}
We denote by $H_{2n}$ the space $\Hom_{\Uq} \big( \mbC, \mathrm{L}_1^{\otimes{2n}}\big)$. The space $H_{2n}$ is naturally identified with the subspace of $\Uq$ singular vectors of weight $0$ in $(\rL_1)^{\otimes 2n}$.

For a word $w\in W_{2n}$, clearly 
\begin{equation*} H(w) \subseteq \Hom_{\Uq}\lb \mbC, w\rb = H_{2n}.\end{equation*}

It is well known that $\dim\lb H_{2n} \rb = \frac{1}{n+1}\binom{2n}{n} = C_n$ is the $n$th Catalan number. There is the following combinatorial description for a basis of $H_{2n}$. 

Let $\{(+),(-)\}$ be a basis of $\Uq$-module $\mathrm{L}_1$ such that $f(+) = (-)$, $e(-) = (+)$, $K(+) = q(+)$, $K(-) = q^{- 1}(-)$. Let $\{(\epsilon_1,\dots, \epsilon_k) \ |\ \epsilon_{i} \in \{+,-\}\}$ be the basis for $\Uq$-module  $\mathrm{L}_1^{\otimes k}$,  where $(\epsilon_1,\dots, \epsilon_k) =(\epsilon_1) \otimes \dots \otimes (\epsilon_k)$. We call this basis standard.

We order the basis of  $\mathrm{L}_1$ by $(+)>(-)$ and extend this order to a basis of $\mathrm{L}_1^{\otimes n}$ lexicographically. For a non-zero vector $v \in \mathrm{L}_1^{\otimes{n}}$, write 
\begin{equation}\label{eq:lower_terms}
v = a(\epsilon_1,\dots, \epsilon_n) + \textit{lower terms}, \qquad a\in \mbC^{\times}.
\end{equation}
Then we call $(\epsilon_1,\dots, \epsilon_n)$ the leading term of $v$.

For an arc $(i,j)\in \arc(2n)$, define an operator $A_{(i,j)}$ acting in $\mathrm{L}_1^{\otimes 2n}$ by $A_{(i,j)} = 1 - q^{-1}P_{ij},$ where $P_{ij}$ is the operator of permutation of the $i$th and $j$th tensor factors.

Let $C = \{(i_k, j_{k})\}_{k=1}^{n} \in \ucconf(2n)$.
Define a vector
\begin{equation}\label{eq:Catalan_vec}
    v_{C} = A_{(i_1,j_1)}\dots A_{(i_n,j_n)}\epsilon_C.
\end{equation}
Here $\epsilon_C=(\epsilon_1(C),\dots ,\epsilon_{2n}(C))$, $\epsilon_{i_{k}}(C) = (+)$ and $\epsilon_{j_{k}}(C) =(-)$ for all $k$.

Note that if two arcs $(i_1,j_1), (i_2,j_2)$ do not have common ends, then the operators $A_{(i_1,j_1)}, A_{(i_2,j_2)}$ commute. Moreover, $A_{(1,2)}(+,-)$ is the unique up to proportionality singular vector of weight zero in $\rL_1^{\otimes 2}$. Since $C$ is a Catalan arc configuration, it follows that the vector $v_C$ is singular. Moreover, the leading term of $v_C$ is $\epsilon_C$. Therefore, $\{v_C\}_{C\in \ucconf(2n)}$ is a basis of $H_{2n}$. We call this basis the Catalan basis.

\begin{lemma}\label{lemma:cat_sing}
    Let $w$ be a word $w \in W_{2n}$ and $C \in \cconf(w)$, then 
    \begin{equation*}
        v_{C} \in H(w).
    \end{equation*}
\end{lemma}
\begin{proof}
    We proceed by induction on $n$. For the case $n = 1$ we have just to check that $(+-) - q^{-1}(-+)$ is annihilated by $f_0$. Indeed,
    \begin{equation*}
        (f_0 \otimes 1 + K \otimes f_{0})((+-) - q^{-1}(-+)) = q^{1 - a - 2}(++) - q^{-1}q^{-a}(++) = 0.
    \end{equation*}

    Now, assume that the statement is known for the case $n - 1$. There exists an arc in $C$ of the form $(i,i+1)$.

     Let $\iota: \mbC \longrightarrow (a_i, a_i+2)$ be an embedding of $\Uqa$-modules  sending $1\mapsto A_{12}(+,-)$. Then 
     $$\iota_i = \Id_{(a_1,\dots, a_{i-1})}\otimes \iota\otimes \Id_{(a_{i+2},\dots, a_{2n})} :(a_1, \dots , \hat{a}_i, \hat{a}_{i+1}, \dots, a_{2n}) \longrightarrow (a_1, \dots, a_{2n})$$
is also an embedding of $\Uqa$-modules. 

     Recall that $\mathring{C}_{\{(i, i+1)\}}$ is $C$ with the arc $(i,i+1)$ removed. Then $v_{C} = \iota_{i}(v_{\mathring{C}_{\{(i, i+1)\}}})$. By the induction hypothesis $v_{\mathring{C}_{(i, i+1)}}$ is an $\ell$-singular vector. Therefore, $v_C$ is also an $\ell$-singular vector. 
     
     \end{proof}
    Since the set of vectors $\{v_C\}_{C\in \cconf(2n)}$ is linearly independent, we obtain a lower bound for the dimension $h(w)$ of $\Hom_{\Uqa}(\mbC,w)$.
    \begin{cor} For any word $w\in W_{2n}$,
        \begin{equation}\label{eq:catalan_lower_bound}
            h(w) \geq |\cconf(w)|.
        \end{equation}\qed
    \end{cor}
\subsection{Action of $\check{R}$-matrix.}
The $R$-matrix in standard basis of $ab=\mathrm{L}_1 \otimes \mathrm{L}_1$ is explicitly given by
\begin{equation}\label{eq:Rmat2}
    \check{R}(a,b) = \begin{pmatrix}
    q^{a-b}-q^2 & 0 & 0 & 0 \\
 0 & q^{a-b}-q^{a-b+2} & q^{a-b+1}-q & 0 \\
 0 & q^{a-b+1}-q & 1-q^2 & 0 \\
 0 & 0 & 0 & q^{a-b}-q^2
    \end{pmatrix}.
\end{equation}
Let $\check{R}_{i,2n}(a,b) = \Id_{\rL_{1}^{\otimes (i-1)}} \otimes \check{R}(a,b) \otimes \Id_{\rL_{1}^{\otimes (2n-i-1)}}$ be the $R$-matrix acting in the $i$th and $(i+1)$st factors of  $\rL_1^{\otimes 2n}$.

For all $a,b$, the operator $\check{R}(a,b)$ is a homomorphism of $\Uq$-modules. In particular, the $R$-matrix $\check{R}_{i,2n}(a,b)$ acts on multiplicity space $H_{2n}=\Hom_{\Uq}\lb \mbC, \rL_1^{\otimes 2n}\rb$. We describe this action in the basis of uncolored Catalan arc configurations. 

The action of $\check{R}_{i,2n}(a,b)$ affects only arcs which have an end in $\{i, i+1\}$. There are $4$ cases which we describe in Lemma \ref{lemma:Rmat_Cat}. In pictures, we do not show arcs which are not affected by the action. We write parameters $a$ and $b$ respectively on $i$th and $(i+1)$st positions on the left and on $(i+1)$st and $i$th places on the right, to indicate the evaluation parameters of the $\Uqa$ modules.  
\begin{lemma}\label{lemma:Rmat_Cat}
    The action of an $\check{R}$-matrix $\check{R}_{i,2n}(a,b)$  on the basis $\{v_C\}_{C\in \ucconf(2n)}$ is given by
\begin{equation*}\begin{aligned}        \vcenter{\hbox{\begin{tikzpicture}
        \draw[-] (0,0.25) to [out=30,in=150] (1.2,0.25);
        \draw[-] (1.8,0.25) to [out=30,in=150] (3.0,0.25);
        \node at (0.6,0.25) {$\dots$};
        \node at (2.4,0.25) {$\dots$};
        \node at (1.2, -0.05) {$a$};
        \node at (1.8, 0) {$b$};
    \end{tikzpicture}}} \;\; &\longmapsto \;\;
    (q^{a-b}-q^2)\vcenter{\hbox{\begin{tikzpicture}
        \draw[-] (0,0.25) to [out=30,in=150] (1.2,0.25);
        \draw[-] (1.8,0.25) to [out=30,in=150] (3.0,0.25);
        \node at (0.6,0.25) {$\dots$};
        \node at (2.4,0.25) {$\dots$};
        \node at (1.2, 0) {$b$};
        \node at (1.8, -0.05) {$a$};
        \end{tikzpicture}}}\;\; + \;\;(q^{a-b + 1} - q)\vcenter{\hbox{\begin{tikzpicture}
        \draw[-] (0,0.25) to [out=30,in=150] (3.0,0.25);
        \draw[-] (1.2,0.25) to [out=30,in=150] (1.8,0.25);
        \node at (0.6,0.25) {$\dots$};
        \node at (2.4,0.25) {$\dots$};
        \node at (1.2, 0) {$b$};
        \node at (1.8, -0.05) {$a$};
    \end{tikzpicture}}} ,\\
        \vcenter{\hbox{\begin{tikzpicture}
        \draw[-] (0,0.25) to [out=30,in=150] (3.0,0.25);
        \draw[-] (0.6,0.25) to [out=30,in=150] (1.8,0.25);
        \node at (1.2,0.25) {$\dots$};
        \node at (2.4,0.25) {$\dots$};
        \node at (0, -0.05) {$a$};
        \node at (0.6, 0) {$b$};
    \end{tikzpicture}}} \;\; &\longmapsto \;\;
    (q^{a-b} - q^2)\vcenter{\hbox{\begin{tikzpicture}
        \draw[-] (0,0.25) to [out=30,in=150] (3.0,0.25);
        \draw[-] (0.6,0.25) to [out=30,in=150] (1.8,0.25);
        \node at (1.2,0.25) {$\dots$};
        \node at (2.4,0.25) {$\dots$};
        \node at (0, 0) {$b$};
        \node at (0.6, -0.05) {$a$};
    \end{tikzpicture}}}\;\; + \;\; (q^{a-b+1}-q)\vcenter{\hbox{\begin{tikzpicture}
        \draw[-] (0,0.25) to [out=30,in=150] (0.6,0.25);
        \draw[-] (1.8,0.25) to [out=30,in=150] (3.0,0.25);
        \node at (1.2,0.25) {$\dots$};
        \node at (2.4,0.25) {$\dots$};
        \node at (0, 0) {$b$};
        \node at (0.6, -0.05) {$a$};
        \end{tikzpicture}}},\\
    \vcenter{\hbox{\begin{tikzpicture}
        \draw[-] (0, 0.25) to [out=30,in=150] (3.0,0.25);
        \draw[-] (1.2, 0.25) to [out=30,in=150] (2.4,0.25);
        \node at (0.6,0.25) {$\dots$};
        \node at (1.8,0.25) {$\dots$};
        \node at (2.4, -0.05) {$a$};
        \node at (3.0, 0) {$b$};
    \end{tikzpicture}}} \;\; &\longmapsto \;\;
    (q^{a-b} - q^2)\vcenter{\hbox{\begin{tikzpicture}
        \draw[-] (0, 0.25) to [out=30,in=150] (3.0,0.25);
        \draw[-] (1.2, 0.25) to [out=30,in=150] (2.4,0.25);
        \node at (0.6,0.25) {$\dots$};
        \node at (1.8,0.25) {$\dots$};
        \node at (2.4, 0) {$b$};
        \node at (3.0, -0.05) {$a$};
    \end{tikzpicture}}}\;\; + \;\; (q^{a-b+1}-q)\vcenter{\hbox{\begin{tikzpicture}
        \draw[-] (0, 0.25) to [out=30,in=150] (1.2,0.25);
        \draw[-] (2.4, 0.25) to [out=30,in=150] (3.0,0.25);
        \node at (0.6,0.25) {$\dots$};
        \node at (1.8,0.25) {$\dots$};
        \node at (2.4, 0) {$b$};
        \node at (3.0, -0.05) {$a$};
        \end{tikzpicture}}},\\
    \vcenter{\hbox{\begin{tikzpicture}
        \draw[-] (0.6 ,0.25) to [out=30,in=150] (1.2,0.25);
        \node at (0,0.25) {$\dots$};
        \node at (1.8,0.25) {$\dots$};
        \node at (0.6, -0.05) {$a$};
        \node at (1.2, 0) {$b$};
    \end{tikzpicture}}} \;\; &\longmapsto \;\;
    (1-q^{a+2-b})\vcenter{\hbox{\begin{tikzpicture}
        \draw[-] (0.6 ,0.25) to [out=30,in=150] (1.2,0.25);
        \node at (0,0.25) {$\dots$};
        \node at (1.8,0.25) {$\dots$};
        \node at (0.6, 0) {$b$};
        \node at (1.2, -0.05) {$a$};
        \end{tikzpicture}}}.
\end{aligned}\end{equation*}
\end{lemma}
\begin{proof}
The lemma is proved by a direct computation.
    
\end{proof}

\subsection{Slide equivalence.}

\begin{defi}                                
    Define a map
    \begin{equation*}    \begin{aligned}        s: W_k &\longrightarrow W_k,\\
           (a_1,\dots, a_k) &\longmapsto (a_2, a_3, \dots, a_{k}, a_1 + 4).
    \end{aligned}    \end{equation*}\end{defi}
    We call the map $s$ the slide. We call words related by a sequence of slides, inverses of slides and shift automorphisms slide equivalent. 
\begin{lemma}\label{lemma:slide_isom}
    Let $w \in W_{k}$ be a word. There is a canonical isomorphism,
    \begin{equation}\label{eq:slide_isom}
        H(w) \cong H(s(w)).
    \end{equation}\end{lemma}
\begin{proof}
  By Lemma \ref{lemma:hom_move},
    \begin{equation*}        H(w) = \Hom_{\Uqa}(\mbC, a_1\dots a_k) \cong \Hom_{\Uqa}(a_1^*, a_2\dots a_k) \cong  \Hom_{\Uqa}(\mbC, a_2\dots a_k(a_1^{*})^{*}).
    \end{equation*}   Since $a^* \cong a+2$, the lemma follows.
\end{proof}

For any word $w\in W_{2n}$, the map $s$ induces a bijection $s_{\conf(w)}:\conf(w) \rightarrow \conf(s(w))$ on arc configurations given by $$s_{\conf(w)}: \{(1,j_1), (i_2, j_2),\dots, (i_n,j_n) \} \mapsto \{(j_1-1, 2n),(i_2 -1, j_2 -1), \dots, (i_n-1, j_n-1)\}.$$

This bijection preserves the Catalan arc configurations: if $C \in \cconf(w)$, then $s_{\conf(w)}(C) \in \cconf(s(w))$. Moreover, for any $C \in \cconf(w)$, isomorphism \eqref{eq:slide_isom} identifies $v_{C}$ with $\alpha v_{s_{\conf{w}}(C)}$, where $\alpha \in \mbC^{\times}$ depends on a choice of isomorphism $[a,a]^* \cong [a+2,a+2]$.

\medskip

We describe another map on $W_{2n}$ which preserves $h(w)$.

Let $\hat{\omega}:\Uqa \longrightarrow \Uqa$ be the map defined on the Jimbo-Drinfeld generators by
    \begin{equation}
    \begin{aligned}[c]
    \hat{\omega}: \Uqa &\longrightarrow \Uqa ,\\
    K^{\pm1} &\longmapsto K^{\mp 1},\\ 
    e_i &\longmapsto f_{i},\;\; i\in\{0,1\},\\
    f_i &\longmapsto e_{i},\;\; i\in\{0,1\}.
    \end{aligned}
    \end{equation}

\begin{lemma}
    The map $\hat{\omega}$ is an isomorphism of Hopf algebras and that $[\alpha, \beta]_{\hat{\omega}} \cong [-\beta, -\alpha]$.
\end{lemma}
\begin{proof}
    The lemma is checked by a direct computation.
\end{proof}

\begin{cor}\label{cor:reverse_word}

    Define an anti-involution $\omega$ on $W_{n}$ by $\omega:(a_{1},\dots, a_{n}) \longmapsto (-a_{n},\dots, -a_{1})$. Then 
    \begin{equation}\label{eq:rev_isom}
        H(w) \cong H(\omega(w)).
    \end{equation}\end{cor}
\begin{proof}
    For a two-dimensional evaluation module $(a)$ we have $(a)_{\hat{\omega}} = (-a)$. By \eqref{eq:aaut_hom_tp_1}, we obtain $((a_1,\dots, a_{n}))_{\hat{\omega}} \cong (-a_{n},\dots, -a_1)$. For the counit $\epsilon$ we have $\epsilon \circ \hat{\omega} = \epsilon$. Therefore, $\mbC_{\hat{\omega}} \cong \mbC$. By \eqref{eq:aaut_hom_tp_2}, this implies \eqref{eq:rev_isom}.
\end{proof}

\subsection{The tensor product structure of $H(w)$ for non-connected words $w$.}\label{subsec:submod_factor} In Section
\ref{subsec:lattice_restrictions}, we discussed the tensor product structure of spaces of homomorphisms when evaluation parameters belong to different lattices. Here we give a tensor product decomposition for the case of one lattice.

A word $w$ is not connected if and only if $w$ is a shuffle of two non-empty words of the form $w_1 = (a_1,\dots, a_k)\in W_{k}, w_2 = (b_1,\dots, b_l) \in W_{l}$ such that $b_i\geq a_j+4$ for all $i,j$.

\begin{prop}\label{prop:no_gaps}
Let word $w$ be a shuffle of two non-empty words $w_1 = (a_1,\dots, a_k)\in W_{k}, w_2 = (b_1,\dots, b_l) \in W_{l}$ such that $b_i\geq a_j+4$ for all $i,j$. Then
\begin{equation*}        H(w) \cong H(w_1) \otimes H(w_2),\qquad h(w) = h(w_1)h(w_2).
    \end{equation*}\end{prop}
\begin{proof}
    Recall that $a b \cong b a$ unless $|a-b|=2$, see Section \ref{subsec:uqamod_and_q-char}. Therefore, letters of $w_1$ commute with letters of $w_2$. Thus, $w\cong w_2\otimes w_1$.

    By Lemmas \ref{lemma:hom_move} and \ref{lemma:dual},
    \begin{equation*}
        H(w) \cong \Hom_{\Uqa}(\mbC, w_2 \otimes w_1) \cong \Hom_{\Uqa}((a_k - 2,\dots, a_{1}-2), w_2).
    \end{equation*}

    Similarly to the proof of Proposition \ref{prop:hom_factorize_lattice} the $q$-character of $(a_k - 2,\dots, a_{1}-2)$ is a sum of monomials of variables $1_a$ with $a < d$ and $q$-character of $w_2$ is a sum of monomials of variables $1_b$ with $b > d$, where $d=\max_i \{a_i\}+2$. This implies that the map
    \begin{equation*}
        \Hom_{\Uqa}((a_k - 2,\dots, a_{1}-2), \mbC) \otimes \Hom_{\Uqa}(\mbC, w_2) \longrightarrow \Hom_{\Uqa}((a_k - 2,\dots, a_{1}-2), w_2),
    \end{equation*}
    induced by composition is an isomorphism.

    Applying Lemma \ref{lemma:hom_move} once again, we get $$\Hom_{\Uqa}((a_k - 2,\dots, a_{1}-2), \mbC) \cong \Hom_{\Uqa}( \mbC, (a_1,\dots, a_k)) = H(w_1).$$ 
    The proposition follows.
\end{proof}
\begin{remark}
Similarly to Proposition \ref{prop:no_gaps}, one can show a more general statement. Namely, let $\rV_1$, $\rV_2$, $\rW_1$ $\rW_2$ be $\Uqa$-modules. Assume that if a positive power  of $1_a$ appears in $\chi_q(\rW_1)$ or  $\chi_q(\rV_1)$ and  a positive power of $1_b$ appears in $\chi_q(\rW_2)$ or $\chi_q(\rV_2)$, then $b\geq a+4$. 
Then
$$
\Hom_{\Uqa} (\rV_1\otimes \rV_2, \rW_1 \otimes \rW_2)=\Hom_{\Uqa} (\rV_1, \rW_1)\otimes \Hom_{\Uqa} (\rV_2, \rW_2).
$$
 \end{remark}

Proposition \ref{prop:no_gaps} reduces the computation of $H(w)$ to the case of connected words. 

We expect that one can further strengthen the definition of connected words preserving Proposition \ref{prop:no_gaps}. We discuss this issue in Section \ref{subsec:fact_by_cont}.

\section{Bounds for $h(w)$.}\label{sec:bounds}
\subsection{Lower bound.}\label{subsec:lower_bound}
In this section we give a lower estimate for $h(w)=\dim(\Hom_{\Uqa}(\mbC, w))$ for a word $w \in W_{2n}$ in terms of $\conf(w)$ which improves the bound by the number of Catalan arc configurations \eqref{eq:catalan_lower_bound}.

Let $C \in \conf(w)$ and let $(i_1, j_1), (i_2, j_2)$ be two intersecting arcs in $C$.
\begin{defi}\label{def:irr_arc_conf}
    Assume $i_1 < i_2 < j_1 < j_2$. We call intersection of arcs $(i_1, j_1), (i_2, j_2)$ irreducible if $a_{i_2} \notin \{a_{i_1}, a_{j_1}\}$ and reducible otherwise. We call $C\in\conf(w)$ an irreducible arc configuration if all intersections of arcs in $C$ are irreducible. For a word $w\in W_{2n}$ denote the subset of irreducible arc configurations $\iconf(w) \subseteq \conf(w)$.
\end{defi}

Clearly, $\cconf(w)\subset \iconf(w)$.

A configuration obtained by removing arcs from an irreducible configuration is irreducible. 

Irreducible arc configurations are compatible with slides.

\begin{prop}\label{prop:slides_irr_pr}
 Let  $w \in W_{2n}$ be a word and $C\in\conf(w)$ an arc configuration. Then $C$ irreducible if and only if $s_{\conf(w)}(C)$ is irreducible. 

\end{prop}
\begin{proof}
    The proof is illustrated by the following picture.
\begin{equation*}\begin{aligned}    s: \vcenter{\hbox{\begin{tikzpicture}
        \draw[-] (0,0.25) to [out=30,in=150] (2.4,0.25);
        \node at (0.6,0.25) {$\dots$};
        \draw[-] (1.2,0.25) to [out=30,in=150] (3.6,0.25);
        \node at (1.8,0.25) {$\dots$};
        \node at (3.0,0.25) {$\dots$};
        \node at (4.2,0.25) {$\dots$};
        \node at (0, -0.075) {$a$};
        \node at (1.2, -0.05) {$b$};
        \node at (2.4, -0.05) {$a{+}2$};
        \node at (3.6, -0.05) {$b{+}2$};
    \end{tikzpicture}}} \;\; &\longmapsto \;\;
    \vcenter{\hbox{\begin{tikzpicture}
        \draw[-] (1.2,0.25) to [out=30,in=150] (3.6,0.25);
        \node at (0.6,0.25) {$\dots$};
        \draw[-] (2.4,0.25) to [out=30,in=150] (4.8,0.25);
        \node at (1.8,0.25) {$\dots$};
        \node at (3.0,0.25) {$\dots$};
        \node at (4.2,0.25) {$\dots$};
        \node at (4.8, -0.075) {$a{+}4$};
        \node at (1.2, -0.05) {$b$};
        \node at (2.4, -0.05) {$a{+}2$};
        \node at (3.6, -0.05) {$b{+}2$};
    \end{tikzpicture}}}\;\;\;
\end{aligned}\end{equation*} The slide $s$ preserves the set of irreducible arc configurations since $$b\notin \{a, a+2\} \Leftrightarrow b+2\notin \{a{+}2, a{+}4\} \Leftrightarrow a{+}2\notin \{b, b+2\}.$$

\end{proof}
\begin{remark}\label{rem:irr_int_slides_preserved}
    Proposition \ref{prop:slides_irr_pr} can be extended to all arc configurations as follows. Let $w$ be a word and $C\in\conf(w)$. Then the number of intersections of arcs and the number of irreducible intersections of arcs in $C$ and $s_{\conf(w)}(C)$ are the same.   
\end{remark}

Next we show that irreducible arc configurations are defined by the set of left ends. We prove it for a slightly larger set of arc configurations.

Denote by $\nconf(w)\subset\conf(w)$ the subset of arc configurations of $w$ such that there are no intersecting arcs of the same color. In particular,  $\iconf(w) \subset \nconf(w)$.

\begin{lemma}\label{lemma:nconf_to_ends}
    For a word $w\in W_{2n}$, an arc configuration $C\in \nconf(w)$ is uniquely defined by the set of left ends $\len(C)$.
    Namely, for $C_1,C_2\in\nconf(w)$, if $\len(C_1)=\len(C_2)$ then $C_1=C_2$.
    \end{lemma}
    \begin{proof}
       Let $w\in W_{2n}$ and $C \in \nconf(w)$. Let $I=\len(C) \subset \{1,\dots, 2n\}$ be the set of the left ends of $C$.  We prove that $C$ is a unique arc configuration in $\nconf(w)$ with the set of left ends equal to $I$.

       We use induction on $n$. Let $a$ be the smallest letter occurring in $w$, $a = \min(\mathrm{supp}(w)).$ Let $i$ be the maximal position where $a$ is occurring, $i = \mathrm{max}(I_{w}(a))$. Then in $C$, $i$ is connected with an element of the set $J = I_w(a+2) \cap \ren(C) \cap \{i + 1,\dots, 
       2n\}$. In particular, $J$ is nonempty. Denote $j = \min(J)$. If $(i,j) \notin C$, the arc ending at $j$ starts to the left of $i$ and the arc starting at $i$ ends to the right of $j$ which gives an intersection of arcs of the same color contradicting to $C\in \nconf(w)$. Therefore, $i$ must be connected to $j$, $(i,j)\in C$.  Let $\mathring{C}_{\{(i,j)\}}$ and $\mathring{w}_{\{(i,j)\}}$ be the configuration and the word obtained by removing the arc $(i,j)$ from $C$. Then $\mathring{C}_{\{(i,j)\}}\in \nconf(\mathring{w}_{\{(i,j)\}})$. Therefore, the arcs in  $\mathring{C}_{\{(i,j)\}}$ are uniquely determined by the set of left ends by the induction hypothesis.
    \end{proof}

\begin{cor}\label{cor:nconf_to_ends}
    For a word $w\in W_{2n}$, an irreducible arc configuration $C \in \iconf(w)$ is uniquely defined by $\len(C)$.\qed
\end{cor}

Now we are ready to prove the main result of this subsection.
\begin{theorem}\label{thm:lower_bound_irr}
    Let $w\in W_{2n}$ be a word. Then
    \begin{equation*}        h(w) \geq |\iconf(w)|.
    \end{equation*}
\end{theorem}
\begin{proof}
 Let $w= (a_1,\dots, a_{2n})$. To prove the theorem, it is sufficient to find a map
    \begin{equation*}        \Sigma_{w} : \iconf(w) \longrightarrow H(w),
    \end{equation*}     such that the image of $\Sigma_{w}$ is a linearly independent set of vectors. We construct a map $\Sigma_{w}$ such that the leading term of $\Sigma_{w}(C)$ has $+$'s at the positions $\len(C)$, see \eqref{eq:lower_terms}. Then the image of $\Sigma_{w}$ is independent by Corollary \ref{cor:nconf_to_ends}.       
    \medskip

 We construct the map $\Sigma_{w}$ by induction on pairs of non-negative integers $(n, m)$, where $2n= l(w)$ and $m = |\{\textit{arc intersections in }C\}|$, ordered lexicographically. 

    For arc configurations without arc intersections  $C\in\cconf(w)$, we set $\Sigma_w(C) = v_C$, see Lemma \ref{lemma:cat_sing}.
    
    Otherwise, for  $C\in \iconf(w)$ let $(i_1, j_1), (i_2, j_2)$ be a pair of intersecting arcs in $C$ such that $i_1 < i_2 < j_1 < j_2$ with the property that $i_2$ is minimal possible and among pairs with such minimal $i_2$, the number $i_1$ is maximal possible.
    
    \paragraph*{Case 1. $i_2 > i_1 + 1$.}
    
    \medskip
    
    For any arc in $C$ with an end in the interval $\{i_1+1,\dots, i_2-1\}$ the other end of this arc also should be contained in this interval. Indeed, let $(i,j)$ be an arc in $C$. First, let  $i\in \{i_1+1,\dots, i_2-1\}$. If $j > j_2$, then  we have intersecting arcs $(i_1, j_1), (i, j)$ which contradicts minimality of $i_2$. And if $j_2 > j > i_2$, we have intersecting arcs $(i, j), (i_2, j_2)$ which contradicts maximality of $i_1$. Second, let $j\in \{i_1+1,\dots, i_2-1\}$ and $i < i_1$, then we have intersecting arcs $(i, j), (i_1, j_1)$ which contradicts minimality of $i_2$.
    
    This implies that the set of arcs of $C$ is a disjoint union of two non-intersecting sets of arcs $C = C_1 \sqcup C_2$, where $C_1$ is the set of all arcs in $C$ with both ends contained in the interval $\{i_1+1,\dots, i_2-1\}$.
    Denote $\tilde{w}_1 = (a_{i_1 + 1}, \dots, a_{i_2 - 1})$, $\tilde{w}_2 = (a_{1}, \dots, a_{i_1 - 1}, a_{i_2}, \dots, a_{2n})$.  Note that $w_1, w_2$ are both non-empty.
    
    Denote $\tilde{C}_1 = \mathring{C}_{C_2} \in \iconf(\tilde{w}_1)$ and  $\tilde{C}_2 = \mathring{C}_{C_1} \in \iconf(\tilde{w}_2)$. By the induction hypothesis, the vectors $\Sigma_{\tilde{w}_1}(\tilde{C}_1),\; \Sigma_{\tilde{w}_2}(\tilde{C}_2)$ are already constructed. The vector $\Sigma_{\tilde{w}_1}(\tilde{C}_1)$ defines an embedding $\iota_{\tilde{C}_1}: \mbC \longrightarrow \tilde{w}_1$ which gives $\iota = \Id_{(a_1,\dots, a_{i_1 -1})}\otimes \iota_{\tilde{C}_1} \otimes \Id_{(a_{i_2 + 1},\dots, a_{2n})}: (a_{1}, \dots, a_{i_1 - 1}, a_{i_2}, \dots, a_{2n}) \longrightarrow w$. Then we define $\Sigma_{w}(C) = \iota(\Sigma_{\tilde{w}_2}(\tilde{C}_2))$. Clearly, the leading term of $ \iota(\Sigma(\tilde{C}_2))$ has $+$'s at the positions $\len(C_1)\sqcup \len(C_2) = \len(C)$.

    \paragraph*{Case 2. $i_2 = i_1 + 1$.}
    
    \medskip
    
    Consider the arc configuration obtained by swapping positions $i_1$ and $i_1+1$. Namely, let $\tilde{C} = \{(i_1, j_2), (i_1 + 1, j_1)\} \sqcup (C\backslash \{(i_1, j_1), (i_1 + 1, j_2)\})\in\conf(\tilde{w})$, where $\tilde{w} = (a_1, \dots, a_{i_1 -1}, a_{i_1 + 1}, a_{i_1}, a_{i_1 +2}, \dots, a_{2n})$. The arc configuration $\tilde{C}$ has one less intersection than $C$, so by induction hypothesis, the vector $\Sigma_{\tilde{w}}(\tilde{C})$ is already constructed. Set
    
    $$\Sigma_{w}(C) =  \check{R}_{i_1, 2n}(a_{i_1+1}, a_{i_1}) \lb\Sigma_{\tilde{w}}(\tilde{C}) \rb.$$

    Since $C$ is irreducible, $a_{i_1+1} \neq a_{i_1} + 2$ so the matrix element of $\check{R}(a_{i_1+1}, a_{i_1})$ between monomials $++$ and $++$ is non-zero, see \eqref{eq:Rmat2}. This implies that constructed vector $\Sigma_{w}(C)$ is non-zero and has the same leading term as $\Sigma_{\tilde{w}}(\tilde{C})$. In addition, $\len(\tilde{C}) = \len(C)$. Since $\check{R}_{i_1, 2n}(a_{i_1+1}, a_{i_1})$ is a homomorphism from $\tilde{w}$ to $w$ and $\Sigma_{\tilde{w}}(\tilde{C}) \in H(\tilde{w})$, we have $\Sigma_{{w}}({C}) \in H(w)$.
    
    The map $\Sigma_{w}$ is constructed.

\end{proof}

\begin{example}\label{ex:lower_bd_exact}
    Consider word $w = 0220420422$. The set $\arc(w)$ contains $12$ arc configurations. There are two irreducible arc configurations given by
    
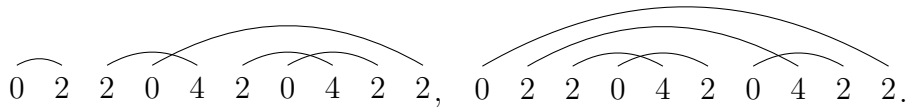
\begin{figure}[H]
\begin{tikzpicture}
    \node at (0, 0) {$0$};
    \node at (0.6,0) {$2$};
    \node at (1.2,0) {$2$};
    \node at (1.8,0) {$0$};
    \node at (2.4,0) {$4$};
    \node at (3,0)   {$2$};
    \node at (3.6,0) {$0$};
    \node at (4.2,0) {$4$};
    \node at (4.8,0) {$2$};
    \node at (5.4,0) {$2$};
    \draw[-] (0 ,0.3) to [out=30,in=150] (0.6 ,0.3);
    \draw[-] (1.2,0.3) to [out=30,in=150] (2.4,0.3);
    \draw[-] (1.8 ,0.3) to [out=30,in=150] (5.4,0.3);
    \draw[-] (3 ,0.3) to [out=30,in=150] (4.2,0.3);
    \draw[-] (3.6,0.3) to [out=30,in=150] (4.8,0.3);

    \node at (5.6,-0.20) {$,$};
\end{tikzpicture}
\begin{tikzpicture}
    \node at (0, 0) {$0$};
    \node at (0.6,0) {$2$};
    \node at (1.2,0) {$2$};
    \node at (1.8,0) {$0$};
    \node at (2.4,0) {$4$};
    \node at (3,0)   {$2$};
    \node at (3.6,0) {$0$};
    \node at (4.2,0) {$4$};
    \node at (4.8,0) {$2$};
    \node at (5.4,0) {$2$};
    \draw[-] (0 ,0.3) to [out=30,in=150] (5.4,0.3);
    \draw[-] (0.6 ,0.3) to [out=30,in=150] (4.2,0.3);
    \draw[-] (1.2,0.3) to [out=30,in=150] (2.4,0.3);
    \draw[-] (1.8 ,0.3) to [out=30,in=150] (3 ,0.3);
    \draw[-] (3.6,0.3) to [out=30,in=150] (4.8,0.3);
    \node at (5.6,-0.25) {$ $};
    \node at (5.6,-0.15) {$.$};
\end{tikzpicture}
\end{figure}
We will also show that  $h(w) \leq 2$, see Example \ref{ex:upper_bd_exact}. Therefore, the lower bound given by Theorem \ref{thm:lower_bound_irr} is exact for this word.
\end{example}

\begin{example}\label{ex:lower_bd_non_exact}
    Consider word $w = 0202462424$. The only irreducible arc configuration is a Catalan arc configuration 
    \begin{figure}[H]
\begin{tikzpicture}
    \node at (0, 0) {$0$};
    \node at (0.6,0) {$2$};
    \node at (1.2,0) {$0$};
    \node at (1.8,0) {$2$};
    \node at (2.4,0) {$4$};
    \node at (3,0)   {$6$};
    \node at (3.6,0) {$2$};
    \node at (4.2,0) {$4$};
    \node at (4.8,0) {$2$};
    \node at (5.4,0) {$4$};
    \draw[-] (0 ,0.3) to [out=30,in=150] (0.6,0.3);
    \draw[-] (1.2 ,0.3) to [out=30,in=150] (1.8,0.3);
    \draw[-] (2.4 ,0.3) to [out=30,in=150] (3,0.3);
    \draw[-] (3.6,0.3) to [out=30,in=150] (4.2,0.3);
    \draw[-] (4.8 ,0.3) to [out=30,in=150] (5.4,0.3);
    \node at (5.6,-0.15) {$.$};
\end{tikzpicture}
\end{figure}
Computations show that $h(w) = 2$. Therefore, the lower bound given by Theorem \ref{thm:lower_bound_irr} is not exact for this word.
\end{example}
Note that the proof of Theorem \ref{thm:lower_bound_irr} gives an algorithm of an explicit construction of a linearly independent set of vectors in $H(w)$.

\subsection{An upper bound from $q$-characters.}\label{subsec:up_bound_chars} Let $w\in W_{2n}$. We have a trivial upper bound for $h(w)$ by the Catalan number, $h(w)\leq C_n$. There is another easy but much better bound for $h(w)$ by the number of trivial modules $\mbC$ among composition factors of a Jordan-Holder series of $w$. This bound can be easily computed by $q$-characters.

Namely, write the $q$-character of $w$ as a sum of $q$-characters of irreducible modules, $\chi_q(w) = \sum_{i = 1}^{n}\chi_q(\rV_i)$. Let 
\begin{equation}\label{eq:char_comp}
    h_{\textit{char}}(w) = |\{i: \chi_q(\rV_i)=1\}|.
\end{equation}

The following lemma is trivial.
\begin{lemma} We have $h(w)\leq h_{\textit{char}}(w)$. \qed
\end{lemma}

To compute the number $h_{\textit{char}}(w)$ we start with another lemma.
\begin{lemma}\label{lemma:upper_bound_qchar}
    Let $w \in W_{2n}$ be a word. The number $h_{\textit{char}}(w)$ equals the coefficient of $1$ in the $q$-character $\chi_q(w) \in \mbZ[1_a^{\pm 1}]_{a\in 2\mbZ}$.
\end{lemma}
\begin{proof}
    The lemma claims that $1$ does not appear in a character of an irreducible non-trivial $\Uqa$-module. Indeed, let $\rV = [\alpha_1,\beta_1]\dots[\alpha_k, \beta_k]$ be a non-trivial product of evaluation modules corresponding to strings in general position. Without loss of generality, let $[\alpha_1, \beta_1]$ be a longest string. Expand the product $\chi_q([\alpha_1,\beta_1])\dots \chi_q([\alpha_k, \beta_k])$.
    The monomials which contain the dominant monomial from the factor $\chi_{q}([\alpha_1, \beta_1])$ contain $1_{\al_1}$ which cannot be cancelled. All other monomials contain $1_{\beta_1 + 2}^{-1}$ which cannot be cancelled. 
    
\end{proof}
Now we give an expression for the number $h_{\textit{char}}(w)$ in terms of content of $w$.

Let $w \in W_{2n}$ be a word such that $\supp(w)\subset\{0, 2,\dots,2N\}$. For $k=1,\dots, N$, let $n_{2k} = |I_w(2k)|$ be the number of letters $2k$ in $w$ and  $m_{2k-1} = \sum_{j=0}^{\infty}(-1)^{j}|I_{w}(2k-2-2j)|$, cf. \eqref{eq:arc_number}.
 
\begin{prop}\label{prop:char_bound_computed}
   Let $w \in W_{2n}$ be a word such that $\supp(w)\subset\{0, 2,\dots,2N\}$. Then 
    \begin{equation}\label{eq:char_bound}
        h_{\textit{char}}(w) = \prod_{k = 1}^{N+1}\binom{n_{2k}}{m_{2k-1}}.
    \end{equation}   Here, by convention $\binom{0}{a} = \delta_{a,0}$ and $\binom{a}{b}=0$ if $b<0$. 
\end{prop}
\begin{proof}
    Let $[1]\{f\}$ denote number of monomials equal to $1$ in the expression $f$. Then we have
    
    \begin{equation}\label{eq:hom_qchar_bound}
        h_{\textit{char}}(w) = [1]\{(1_0 + 1_2^{-1})^{n_0}(1_2 + 1_4^{-1})^{n_2}...(1_{2N} + 1_{2N+2}^{-1})^{n_{2N}}\}.
    \end{equation}
    
    Then
    \begin{multline*}
        \hspace{-7pt}[1]\{(1_0 {+} 1_2^{{-}1})^{n_0}(1_2 {+} 1_4^{{-}1})^{n_2}...(1_{2N} {+} 1_{2N{+}2}^{{-}1})^{n_{2N}}\} {=} [1]\{\binom{n_2}{n_0}(1_{4}^{{-}1})^{(n_2 {-} n_0)}(1_{4}{+}1_{6}^{{-}1})^{n_4}...(1_{2N} {+} 1_{2N{+}2}^{{-}1})^{n_{2N}}\} {=}\\{=} \binom{n_2}{m_1}[1]\{(1_{4}^{{-}1})^{m_3}(1_{4}{+}1_{6}^{{-}1})^{n_4}...(1_{2N} {+} 1_{2N{+}2}^{{-}1})^{n_{2N}}\} {=} \binom{n_2}{m_1}\binom{n_4}{m_3}[1]\{(1_{6}^{{-}1})^{n_4{-}m_3}...(1_{2N} {+} 1_{2N{+}2}^{{-}1})^{n_{2N}}\} {=}\\ {=} \dots {=} \binom{n_2}{n_0}\dots \binom{n_{2N{-}2}}{m_{2N{-}3}}[1]\{(1_{2N{-}2}^{{-}1})^{m_{2N{-}1}}(1_{2N{-}2} {+} 1_{2N}^{{-}1})^{n_{2N}}\} {=} \binom{n_2}{n_0}\dots \binom{n_{2N}}{m_{2N{-}1}}\delta_{n_{2N},m_{2N{-}1}},
    \end{multline*}
    where we used $m_{2k -1} = n_{2k-2}-m_{2k-3}$. It remains to note that $\delta_{n_{2N},m_{2N-1}}= \binom{0}{n_{2N}-m_{2N-1}}=\binom{n_{2N+2}}{m_{2N+1}}$. 
\end{proof}

    Clearly, the number $h_{\textit{char}}(w)$ does not depend on the order of letters of $w$.

\begin{remark}
    Note that unlike $h(w)$, the quantity $h_{\textit{char}}(w)$ is not preserved under slides. 
\end{remark}

\begin{example}\label{ex:qchar_slide_improve}
    Let $w = 0022$. Then,
    $$h_{\textit{char}}(0022)=h(0022) = 1=h (s(0022))=h(0224)< h_{\textit{char}}(0224) = 2.$$
\end{example}

 This next example shows that the upper bound $h_{\textit{char}}(w)$ is not the best among bounds independent of the order of letters in the word.
\begin{example}
Let $w=00222244$. 
We have 
$$
n_0=2, \quad n_2=4, \quad n_4=2, \quad m_1=2, \quad m_2=2, \quad m_3=0.
$$

Then for any permutation $\si\in S_8$, we have $h_{\textit{char}}(\si w)=h_{\textit{char}}(w)=6$.
However, using the upper bound given below by Theorem \ref{thm:upper_bound},  we obtain  $h(\si w)\leq 3$ for all $\sigma \in S_8$. 
\end{example}

\subsection{Steady arc configurations upper bound.}\label{subsec:up_bound_steady}
We extend the definition of arc configuration to arbitrary tensor products of evaluation modules. In this subsection only, for an evaluation module $[\al, \bt]$, we also use the notation $[\al, \dots, \bt]$, $[\al,\al+2,\dots,\beta]$, etc.

\begin{defi} \label{def:thick_arc_conf} Let $\rV=u_1\dots u_k$ where $u_i=[\al_i,\dots, \beta_i]$. Let 
$$w=(\underbrace{\al_1, \al_1 {+}2,\dots, \bt_1{-}2,\beta_1}_{u_1},\dots,\underbrace{\al_k,\dots,\beta_k}_{u_k}),$$
be the word consisting of all letters in $u_1,\dots,u_k$ written in the same order. An arc configuration of $\rV$ is an arc configuration of $w$ such that no arc connects two letters from the same $u_i$.

\end{defi}

Note that if  $\al_k=\bt_k$ for all $k$, then  Definition \ref{def:thick_arc_conf} coincides with Definition \ref{def:arc_conf}.

Now we describe an algorithm which gives an upper bound for $h(w)$  with $w\in W_{2n},\; n \in \mbZ_{>0}$.

For an arbitrary tensor product of evaluation modules we use similar notation $H(\rV)=\Hom(\mbC, \rV)$, $h(\rV)=\dim(\Hom(\mbC, \rV))$.

We start by observing that if the right factor of a tensor product of evaluation modules contains the smallest letter then $h(w)=0$.
\begin{lemma}\label{lemma:right0_nohom}
    Let $\rV = u_1\dots u_k[0,\dots, 2m]$, where $m\geq 0$, $u_i=[\al_i,\dots,\bt_i]$, be such that all letters of $u_i$ are non-negative, $\al_i\geq 0$. Then $H(\rV) = 0$.

     Similarly, let $\rV = [0,\dots, 2m] u_1\dots u_k$, where $u_i=[\al_i,\dots,\bt_i]$, be such that all letters of $u_i$ are at most $2m$, $\bt_i\leq 2m$. Then $H(\rV) = 0$.
\end{lemma}
\begin{proof}
 By Lemma \ref{lemma:hom_move} and Lemma \ref{lemma:dual} we get
    \begin{equation*}
        H(u_1\dots u_k[0,\dots, 2m]) = H([-2,\dots, 2m-2] u_1\dots u_k).
    \end{equation*}
    The only dominant monomial of $\chi_q([-2,\dots,2m-2])$ contains $1_{-2}$ and monomials of $\chi_q(u_1\dots u_k)$ contain only $1_a$ with $a\geq 0$. Hence $[-2,\dots, 2m-2]$ does not appear among composition factors of $u_1\dots u_k$.
    
    Proof of the second case is similar.
\end{proof}
Let $\rV = w_1[0,\dots, 2m]w_2$ where $w_1,w_2$ are words such that all letters of $w_2$ are positive and all letters of $w_1$ are non-negative. In other words, the $0$ in $[0,\dots,2m]$ is the smallest letter, and the rightmost zero.  
We give a recursive upper bound for $h(\rV)$ for such modules by moving the module $[0,\dots,2m]$ to the right. In particular, this provides an upper bound for $h(w)$ for any word $w$.

Define a partial order on the sequences of the form $w_1[0, \dots,2m]w_2$ where $w_1,w_2$ are words as follows. We say
\begin{equation}\label{eq:partial_order}
w_1'[0, \dots,2m']w_2'> w_1[0, \dots,2m]w_2,\textit{ if }\;\  \begin{matrix} \hspace{-75pt}m'+l(w_2')>m+l(w_2),\textit{ or }\\
m'+l(w_2')=m+l(w_2)\textit{ and }l(w_2')>l(w_2).\end{matrix}
\end{equation}
We will describe a recursion which provides an upper bound for $\rV$ in terms of upper bounds for modules smaller than $\rV$ with respect to this partial order.

The answer is given by the number of arc configurations which we call steady.

\begin{defi}
An arc configuration $C$ of  $\rV$ is called a relevant arc configuration if all arcs connecting letters in $[0,\dots,2m]$ have left ends in $[0,\dots,2m]$ and right ends in $w_2$. Let $\rconf(\rV)$ be the set of all relevant arc configurations for $\rV$.
\end{defi}
Note that if $m=0$ then all arc configurations are relevant.

Note that if $w_2=\varnothing$ then there are no arc configurations and, in particular, $\rconf(\rV)=\varnothing$.

\begin{defi}\label{def:steady_arc_conf}
Define the set of steady arc configurations $\mathrm{\sconf(\rV)} \subset \rconf(\rV)$ inductively as follows. If $w_2 = \varnothing$, we set $\sconf(\rV) = \varnothing$. Let $l(w_2) \geq 1$. Denote  the first letter of $w_2$ by $a$, $w_2 = a\tilde{w}_2$. We consider several cases.

\begin{enumerate}
        \item If $a \neq 2m+2$, set
    \begin{equation*}
        \sconf(\rV)= \iota_0(\sconf(w_1a[0, \dots,2m]\tilde{w}_2)),
    \end{equation*}
    where $\iota_0$ is an embedding
    \begin{equation*}
        \iota_0: \rconf(w_1a[0,\dots, 2m]\tilde{w}_2) \longrightarrow \rconf(\rV),
    \end{equation*}
    given by moving the letter $a$ to the right through $[0,\dots, 2m]$ together with the end of the arc connecting to $a$. Note that the other end of this arc is not inside $[0, \dots,2m]$ since we apply $\iota_{0}$ only to relevant arc configurations and hence the result is a relevant arc configuration.
        \item If $a = 2m + 2$,
        set
    \begin{equation*}
        \sconf(\rV)= \iota_1(\sconf(w_1[0, \dots,2m+2]\tilde{w}_2))\sqcup \iota_2(\sconf(w_1[0,\dots, 2m-2]\tilde{w}_2)).
    \end{equation*}
        Here, $\iota_1$ is an embedding 
        \begin{equation*}
            \iota_1: \rconf(w_1[0,\dots, 2m+2]\tilde{w}_2) \longrightarrow \rconf(\rV),
        \end{equation*}
        mapping $\{(i_l, j_l)\} \longmapsto \{(i_l, j_l)\}$. 

        The embedding $\iota_2$  
        \begin{equation*}
            \iota_2: \rconf(w_1[0,\dots, 2m-2]\tilde{w}_2) \longrightarrow \rconf(\rV),
        \end{equation*}
        is given by $\iota_2(\{(i_{l},j_{l})\}) = \{(i^{\prime}_l, j^{\prime}_l)\} \cup \{(l(w_1)+m+1, l(w_1) + m + 2)\}$ with
        \begin{equation*}
            i^{\prime}_l = \begin{cases}
            i_l,\; &i_l\leq l(w_1)+m,\\
            i_l+2,\; &i_l > l(w_1)+m,
            \end{cases}\;\;\;\;
            j^{\prime}_l = \begin{cases}
            j_l,\; &j_l\leq l(w_1)+m,\\
            j_l+2,\; &j_l > l(w_1)+m.             \end{cases}
        \end{equation*}
        
        In other words, embedding $\iota_2$ adds an arc which connects adjacent letters $2m$ and $2m+2$ right after the segment $[0,\dots,2m-2]$ and then adds the letter $2m$ to this segment. 
\end{enumerate}

We call the elements of  set $\sconf(\rV)$ steady arc configurations.
\end{defi}
We illustrate Definition \ref{def:steady_arc_conf} by the following picture.
\begin{figure}[H]
\hspace{57.5pt}
\begin{tikzpicture}
    \node at (-0.1,-0.12) {$\dots$};
    \node at (0.25, 0) {$a$};
    \node at (1.4,0)  {$[0,\dots,2m]$};
    \draw[-] (0.25 ,0.2) to [out=30,in=195] (0.85,0.45);
    \draw[dashed] (0.85,0.45) to [out=15,in=185] (1.45, 0.55);
    
    \node at (2.6,-0.12) {$\dots$};

    \node at (3.4, 0.24) {$\iota_0$};
    \node at (3.4, 0.) {$\longmapsto$};

    \node at (4.2,-0.12) {$\dots$};
    \node at (6.5, 0) {$a$};
    \node at (5.4,0)  {$[0,\dots,2m]$};
    \draw[-] (6.5 ,0.2) to [out=30,in=195] (7.1,0.45);
    \draw[dashed] (7.1,0.45) to [out=15,in=185] (7.7, 0.55);
    
    \node at (7.0,-0.12) {$\dots$};

    \node at (9.5,0) {$,\;\;  a\neq 2m+2$,};
\end{tikzpicture}

\medskip
\begin{tikzpicture}
    \node at (0,-0.38) {$\dots$};

    \node at (1.5,-0.26)  {$[0,\dots,2m{-}2]$};
    \node at (3,-0.38) {$\dots$};

    \node[rotate=15] at (3.8, -0.07) {$\longmapsto$};
    \node at (3.8, 0.3) {$\iota_2$};

    \node at (4.6,-0.02) {$\dots$};
    \node at (6.45,0.1)  {$[0,\dots,2m{-}2,2m]$};
    \node at (8.65, 0.1) {$2m{+}2$};
    \node at (9.6,-0.02) {$\dots$};
    \draw[-] (7.8, 0.35) to [out=30,in=150] (8.4,0.35);

    \node at (4.6, 0.8) {$\dots$};
    \node at (5.8,0.92)  {$[0,\dots,2m]$};
    \node at (7.3, 0.92) {$2m{+}2$};
    \node at (8.175,0.8) {$\dots$};
    \draw[-] (6.4,1.12) to [out=30,in=195] (7, 1.37);
    \draw[dashed] (7, 1.37) to [out=15,in=185] (7.6, 1.47);
    \draw[-] (7.4,1.12) to [out=30,in=195] (8, 1.37);
    \draw[dashed] (8, 1.37) to [out=15,in=185] (8.6, 1.47);

    \node at (-0.65, 1.4) {$\dots$};
    \node at (1.2, 1.52)  {$[0,\dots,2m{+}2]$};
    \node at (3.15,1.4) {$\dots$};
    
    \draw[-] (1.15,1.72) to [out=30,in=195] (1.75, 1.97);
    \draw[dashed] (1.75, 1.97) to [out=15,in=185] (2.35, 2.07);
    \draw[-] (2.25,1.72) to [out=30,in=195] (2.85, 1.97);
    \draw[dashed] (2.85, 1.97) to [out=15,in=185] (3.45, 2.07);

    \node[rotate=-15] at (3.8, 1.1) {$\longmapsto$};
    \node at (3.8, 1.47) {$\iota_1$};
\end{tikzpicture}
\end{figure}

The steady arc configurations for a module are obtained as the image of $\iota_0$ or the disjoint union of the images of $\iota_1$ and $\iota_2$ from the steady arc configurations of smaller modules.

\begin{example}\label{ex:sconf_def}
    Consider the word $w = 22402464$.

    We build a graph of $\Uqa$-modules where edges correspond to maps $\iota_1, \iota_2$ or $\iota_0$ applied to the corresponding sets of relevant arc configurations. Note that $\iota_0$ does not change the length of the segment, $\iota_1$ decreases that length and $\iota_2$ increases. The edges corresponding to $\iota_0$ are vertical, the edges corresponding to $\iota_1$ go from right to left and the edges corresponding to $\iota_2$ go from left to right. We terminate at vertices where the word is empty or does not have relevant arc configurations. 
    
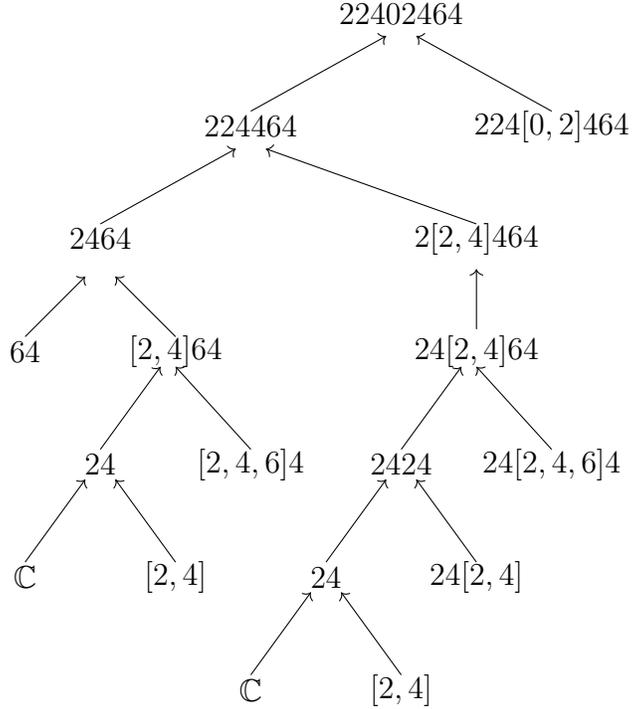
\begin{center}

    \begin{tikzpicture}
    \node at (2, 1.5) {$22402464$};
    
    {
    \draw[<-] (2.2, 1.2) to (4.0, 0.2);
    \node at (4.0, 0.0) {$224[0,2]464$};

    \node at (0.0, 0.0) {$224464$}; 
    {
    \draw[<-] (1.8, 1.2) to (0.0, 0.2);
    \node at (3.0, -1.5) {$2[2,4]464$};
    {
    \draw[<-] (3.0, -1.9) to (3.0, -2.7);
    \node at (3.0, -3) {$24[2,4]64$};
    {
    \draw[<-] (3.0, -3.2) to (4.0, -4.3);
    \node at (4.0, -4.5) {$24[2,4,6]4$};
    
    }
    \draw[<-] (2.8, -3.2) to (2.0, -4.3);
    \node at (2.0, -4.5) {$2424$};
    {
    \draw[<-] (2.2, -4.7) to (3.0, -5.8);
    \node at (3.0, -6.0) {$24[2,4]$};
    
    \draw[<-] (1.8, -4.7) to (1.0, -5.8);
    \node at (1.0, -6.0) {$24$};
    {
    \draw[<-] (0.8, -6.2) to (0.0, -7.3);
    \node at (0.0, -7.5) {$\mbC$};
    }
    {
    \draw[<-] (1.2, -6.2) to (2.0, -7.3);
    \node at (2.0, -7.5) {$[2,4]$};
    }
    }
    }
       
    \node at (-2.0, -1.5) {$2464$};
    {
    \draw[<-] (-0.2,-0.3) to (-2.0 ,-1.3);
    \draw[<-] (0.2,-0.3) to (3.0, -1.3);
    \node at (-3.0, -3) {$64$};    
    \draw[<-] (-2.2,-2) to (-3.0, -2.8);
    \draw[<-] (-1.8,-2) to (-1.0, -2.8);
    \node at (-1.0, -3) {$[2,4]64$};
    \node at (-2.0, -4.5) {$24$};
    \node at (0.0, -4.5) {$[2,4,6]4$};
    \node at (-3.0, -6.0) {$\mbC$};
    \node at (-1.0, -6.0) {$[2,4]$};
    \draw[<-] (-1.2, -3.2) to (-2.0, -4.3);
    \draw[<-] (-1.0, -3.2) to (-0.0, -4.3);
    \draw[<-] (-2.2, -4.7) to (-3.0, -5.8);
    \draw[<-] (-1.8, -4.7) to (-1.0, -5.8);
    }
    }
    }
    \end{tikzpicture}
    \end{center}
    
    In this setting the set $\sconf(w)$ is in bijection with the leaves of the graph corresponding to empty words (to $\mbC$ as a module). 

    We show a way to recover steady arc configurations from the graph. We extract two paths from the empty word to $22402464$ corresponding to two leaves with the empty word and draw consequent images of the maps $i_0, i_1$ or $i_2$. 
    \begin{center}
    \begin{tikzpicture}[scale=0.8]
        \node at (0,0) {$\mbC$};
        \node at (1.5,0) 
        {   \begin{tikzpicture}[scale=0.75]
                \node at (0, 0) {$2$};
                \node at (0.6,0) {$4$};
                \draw[-,red] (0 ,0.3) to [out=30,in=150] (0.6,0.3);
            \end{tikzpicture}
        };
        \node at (4,0.05) 
        {   \begin{tikzpicture}[scale=0.75]
                \node at (0, -0.05) {$[2,$};
                \node at (0.6,-0.05) {$4]$};
                \node at (1.2,0) {$6$};
                \node at (1.8,0) {$4$};
                \draw[-,red] (0 ,0.3) to [out=30,in=150] (1.8,0.3);
                \draw[-,blue] (0.6 ,0.3) to [out=30,in=150] (1.2,0.3);
            \end{tikzpicture}
        };
        \node at (7,0.05) 
        {   \begin{tikzpicture}[scale=0.75]
                \node at (0, 0) {$2$};
                \node at (0.6,0) {$4$};
                \node at (1.2,0) {$6$};
                \node at (1.8,0) {$4$};
                \draw[-,red] (0 ,0.3) to [out=30,in=150] (1.8,0.3);
                \draw[-,blue] (0.6 ,0.3) to [out=30,in=150] (1.2,0.3);
            \end{tikzpicture}
        };
        \node at (10.5,0.15) 
        {   \begin{tikzpicture}[scale=0.75]
                \node at (0, 0) {$2$};
                \node at (0.6,0) {$2$};
                \node at (1.2,0) {$4$};
                \node at (1.8,0) {$4$};
                \node at (2.4,0) {$6$};
                \node at (3.0,0) {$4$};
                \draw[-,red] (0 ,0.3) to [out=30,in=150] (3.0,0.3);
                \draw[-,blue] (1.8 ,0.3) to [out=30,in=150] (2.4,0.3);
                \draw[-,ForestGreen, thick] (0.6 ,0.3) to [out=30,in=150] (1.2,0.3);
            \end{tikzpicture}
        };
        \node at (15.5,0.235) 
        {   \begin{tikzpicture}[scale=0.75]
                \node at (0, 0) {$2$};
                \node at (0.6,0) {$2$};
                \node at (1.2,0) {$4$};
                \node at (1.8,0) {$0$};
                \node at (2.4,0) {$2$};
                \node at (3.0,0) {$4$};
                \node at (3.6,0) {$6$};
                \node at (4.2,0) {$4$};
                \draw[-, red] (0 ,0.3) to [out=30,in=150] (4.2,0.3);
                \draw[-, ForestGreen, thick] (0.6 ,0.3) to [out=30,in=150] (1.2,0.3);
                \draw[-,blue] (3.0 ,0.3) to [out=30,in=150] (3.6,0.3);
                \draw[-,violet, thick] (1.8 ,0.3) to [out=30,in=150] (2.4,0.3);
            \end{tikzpicture}
        };
        \draw[->] (0.3, 0) to (0.9,0);
        \node at (0.6, 0.25) {$\iota_2$};
        \draw[->] (2.1, 0) to (2.7,0);
        \node at (2.4, 0.25) {$\iota_2$};
        \draw[->] (5.2, -0.05) to (5.8, -0.05);
        \node at (5.5, 0.2) {$\iota_1$};
        \draw[->] (8.15, -0.1) to (8.7, -0.1);
        \node at (8.4, 0.2) {$\iota_2$};
        \draw[->] (12.3, -0.1) to (13.1, -0.1);
        \node at (12.7, 0.2) {$\iota_2$};
    \end{tikzpicture}

    \begin{tikzpicture}[scale=0.8]
        \node at (0,0) {$\mbC$};
        \node at (1.5,0) 
        {   \begin{tikzpicture}[scale=0.75]
                \node at (0, 0) {$2$};
                \node at (0.6,0) {$4$};
                \draw[-,red] (0 ,0.3) to [out=30,in=150] (0.6,0.3);
            \end{tikzpicture}
        };
        \node at (4,0.05) 
        {   \begin{tikzpicture}[scale=0.75]
                \node at (0, 0) {$2$};
                \node at (0.6,0) {$4$};
                \node at (1.2,0) {$2$};
                \node at (1.8,0) {$4$};
                \draw[-,red] (0 ,0.3) to [out=30,in=150] (0.6,0.3);
                \draw[-,blue] (1.2 ,0.3) to [out=30,in=150] (1.8,0.3);
            \end{tikzpicture}
        };
        \node at (7.5,0.05) 
        {   \begin{tikzpicture}[scale=0.75]
                \node at (0, 0) {$2$};
                \node at (0.6,0) {$4$};
                \node at (1.2,-0.05) {$[2,$};
                \node at (1.8,-0.05) {$4]$};
                \node at (2.4,0) {$6$};
                \node at (3.0,0) {$4$};
                \draw[-,red] (0 ,0.3) to [out=30,in=150] (0.6,0.3);
                \draw[-,blue] (1.2 ,0.3) to [out=30,in=150] (3.0,0.3);
                \draw[-,ForestGreen, thick] (1.8 ,0.3) to [out=30,in=150] (2.4,0.3);
            \end{tikzpicture}
        };
        \node at (11.5,0.15) 
        {   \begin{tikzpicture}[scale=0.75]
                \node at (0, 0) {$2$};
                \node at (0.6,-0.05) {$[2,$};
                \node at (1.2,-0.05) {$4]$};
                \node at (1.8,0) {$4$};
                \node at (2.4,0) {$6$};
                \node at (3.0,0) {$4$};
                \draw[-,red] (0 ,0.3) to [out=30,in=150] (1.8,0.3);
                \draw[-,blue] (0.6 ,0.3) to [out=30,in=150] (3.0,0.3);
                \draw[-,ForestGreen, thick] (1.2 ,0.3) to [out=30,in=150] (2.4,0.3);
            \end{tikzpicture}
        };
        
        \node at (15.5,0.15) 
        {   \begin{tikzpicture}[scale=0.75]
                \node at (0, 0) {$2$};
                \node at (0.6,0) {$2$};
                \node at (1.2,0) {$4$};
                \node at (1.8,0) {$4$};
                \node at (2.4,0) {$6$};
                \node at (3.0,0) {$4$};
                \draw[-,red] (0 ,0.3) to [out=30,in=150] (1.8,0.3);
                \draw[-,blue] (0.6 ,0.3) to [out=30,in=150] (3.0,0.3);
                \draw[-,ForestGreen, thick] (1.2 ,0.3) to [out=30,in=150] (2.4,0.3);
            \end{tikzpicture}
        };
        
        \node at (20.1,0.235) 
        {   \begin{tikzpicture}[scale=0.75]
                \node at (0, 0) {$2$};
                \node at (0.6,0) {$2$};
                \node at (1.2,0) {$4$};
                \node at (1.8,0) {$0$};
                \node at (2.4,0) {$2$};
                \node at (3.0,0) {$4$};
                \node at (3.6,0) {$6$};
                \node at (4.2,0) {$4$};
                \draw[-, red] (0 ,0.3) to [out=30,in=150] (3,0.3);
                \draw[-, blue] (0.6 ,0.3) to [out=30,in=150] (4.2,0.3);
                \draw[-,ForestGreen, thick] (1.2 ,0.3) to [out=30,in=150] (3.6,0.3);
                \draw[-,violet, thick] (1.8 ,0.3) to [out=30,in=150] (2.4,0.3);
            \end{tikzpicture}
        };
        \draw[->] (0.3, 0) to (0.9,0);
        \node at (0.6, 0.25) {$\iota_2$};
        \draw[->] (2.1, 0) to (2.7,0);
        \node at (2.4, 0.25) {$\iota_2$};
        \draw[->] (5.2, -0.05) to (5.8, -0.05);
        \node at (5.5, 0.2) {$\iota_2$};
        \draw[->] (9.25, -0.1) to (9.7, -0.1);
        \node at (9.5, 0.2) {$\iota_0$};
        \draw[->] (13.3, -0.1) to (13.8, -0.1);
        \node at (13.5, 0.2) {$\iota_1$};
        \draw[->] (17.2, -0.1) to (17.8, -0.1);
        \node at (17.45, 0.2) {$\iota_2$};
    \end{tikzpicture}
    \end{center}
\end{example}

\medskip

We give another combinatorial description of steady arc configurations.

We call a word $(a_1,\dots, a_{2k})$ a chain word if $a_{2k}=a_1+2$, $a_i>a_1$, and $a_{i+1}-a_i=\pm 2$ for $i=1,\dots,2k-1$.

A chain word has a distinguished arc configuration defined recursively as follows.
Let $i$ be the smallest number such that $a_i>a_{i-1}$ and $a_i>a_{i+1}$.
Then connect $a_{i-1}$ to $a_i$ with an arc and remove both letters. The new word is again a chain word and we can repeat the construction. The distinguished arc configuration is a steady Catalan arc configuration. 

A chain subword $(a_1,\dots,a_{2k})$ in a word $w$ is called admissible if for $i=1,\dots, 2k-1$  there are no letters in $w$ between $a_i$ and $a_{i+1}$ which are equal to $a_{i+1}$. Let  $w$ have two non-intersecting chain subwords $w_1=(a_1,\dots, a_{2k})$, $w_2=(b_1,\dots, b_{2l})$. We say $w_1<w_2$ if $a_1<b_1$ or if $a_1=b_1$ and $a_1$ is positioned in $w$ to the right of $b_1$. 

Suppose a word $w$ is a shuffle of chain words $w_1,\dots,w_s$. Assume that $w_i<w_{i+1}$ for all $i$. Then there is a unique arc configuration of $w$ which is composed of the distinguished arc configurations of chain words $w_1,\dots, w_s$.
One can show  that such an arc configuration is steady if for $i=1,\dots, 2s-1$ the chain word $w_i$ is admissible in the subword of $w$ obtained by the shuffle of chain words $w_i,\dots,w_s$. Moreover, all steady arc configurations are obtained in this way.

In Example \ref{ex:sconf_def}  for the first steady arc configuration we have $w_1=02$ (shown in violet color), $w_2=24$ (green) and $w_3=2464$ (blue and red).  For the  second steady arc configuration we have $w_1=02$ (violet), $w_2=2464$ (green and blue), $w_3=24$ (red).

\medskip
 
The main result of this section is the following theorem.
\begin{theorem}\label{thm:upper_bound}
    Let $w\in W_{2n}$. Then
    \begin{equation}\label{steady upper inequality}
        h(w) \leq |\sconf(w)|.
    \end{equation}\end{theorem}
\begin{proof}
    
    Write $w = w_1 (0) w_2$ with all letters of $w_2$ larger than $0$ and all letters of $w_1$ not less than $0$.
    
We prove more general statement, $h(\rV) \leq |\sconf(\rV)|$ for a representation $\rV = w_1[0, \dots, 2m]w_2$ with an assumption that all letters of $w_2$ are strictly larger than $0$ and all letters of $w_1$ are not less than $0$. We use induction with respect to the partial order \eqref{eq:partial_order}. If $l(w_2) = 0$ then both sides of the inequality \eqref{steady upper inequality} are $0$.

   Let $l(w_2) > 0$ and denote the first letter by $a_1$,  $w_2 = a_1\tilde{w}_2$.
        
         \paragraph*{Case 1. Assume that $a_1\neq 2m+2$.}\hfill
         
         Then $a_1 \neq -2$. It follows that the strings $[0,\dots,2m]$ and $[a_1]$ are in general position. Therefore, $[0, \dots,2m]a_1\cong a_1[0, \dots,2m]$, so $h(\rV) = h(w_1a_1[0, \dots,2m]\tilde{w}_2)$. By definition, $\sconf(\rV) = \iota_0(\sconf(w_1a_1[0, \dots,2m]\tilde{w}_2))$, hence $|\sconf(\rV)| = |\sconf(w_1a_1[0, \dots,2m]\tilde{w}_2)|$ and the result follows from the induction hypothesis.
         
        \paragraph*{Case 2. Assume that $a_1 = 2m + 2$.}\hfill
        
        By Proposition \ref{prop:two_strings_prod} there is a short exact sequence of $\Uqa$-modules
        $$
        0 \longrightarrow [0, \dots,2m-2] \longrightarrow [0, \dots,2m]2m+2 \longrightarrow [0, \dots,2m+2] \longrightarrow 0,
        $$
        where in case $m=0$, by convention $[0, \dots,2m-2] = \mbC$. By Lemma \ref{lemma:exact_functors} we obtain an exact sequence
        $$
        0 \longrightarrow H(w_1[0, \dots,2m-2]\tilde{w}_2) \longrightarrow H(\rV) \longrightarrow H(w_1[0,\dots, 2m+2]\tilde{w}_2),
        $$
    which gives $$h(\rV) \leq h( w_1[0,\dots, 2m-2]\tilde{w}_2) + h( w_1[0,\dots, 2m+2]\tilde{w}_2).$$

    Since $l(\tilde{w}_2)<l(w_2)$ we can apply the induction hypothesis to $H(w_1[0, \dots,2m+2]\tilde{w}_2)$. For $H(w_1[0, \dots,2m-2]\tilde{w}_2)$ we can apply the induction hypothesis since we have decreased overall length. 

Therefore,
\begin{multline*}
 h(w_1[0,\dots, 2m-2]\tilde{w}_2) + h(w_1[0,\dots, 2m+2]\tilde{w}_2)\leq\\ \leq |\sconf(w_1[0, \dots,2m-2]\tilde{w}_2)|+|\sconf(w_1[0,\dots, 2m+2]\tilde{w}_2)|.
\end{multline*}
Finally, directly from Definition \ref{def:steady_arc_conf} $$ |\sconf(w_1[0, \dots,2m-2]\tilde{w}_2)| + |\sconf(w_1[0,\dots, 2m+2]\tilde{w}_2)| = |\sconf(\rV)|.$$ 
Combining the inequalities we obtain the theorem.
\end{proof}
\begin{example}\label{ex:upper_bd_nonexact}
    Consider a word $w=22402464$ (cf. Example \ref{ex:sconf_def}). 
There are two steady arc configurations for this word. One of them is a Catalan arc configuration, hence by Lemma \ref{lemma:cat_sing}, $h(w)\geq 1$. By Lemma \ref{lemma:slide_isom} we get $$h(w) = h(s^2(w))=h(40246466) \leq h_{\textit{char}}(40246466) = 1.$$

Therefore, $h(w) = 1 < |\sconf(w)|$ and the upper bound  given by Theorem \ref{thm:upper_bound} is not exact in that case.

However, note that after applying slide twice, the word $s^2(w) = 40246466$ has only one steady arc configuration. Therefore, after two slides the bound given by Theorem \ref{thm:upper_bound} becomes exact. 

\end{example}

\begin{example}\label{ex:upper_bd_exact}
    Consider a word $w = 0220420422$ from Example \ref{ex:lower_bd_exact}. Steady arc configurations of that word are
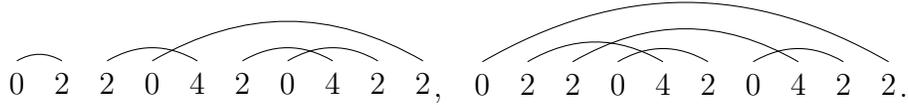
\begin{figure}[H]
\begin{tikzpicture}
    \node at (0, 0) {$0$};
    \node at (0.6,0) {$2$};
    \node at (1.2,0) {$2$};
    \node at (1.8,0) {$0$};
    \node at (2.4,0) {$4$};
    \node at (3,0)   {$2$};
    \node at (3.6,0) {$0$};
    \node at (4.2,0) {$4$};
    \node at (4.8,0) {$2$};
    \node at (5.4,0) {$2$};
    \draw[-] (0 ,0.3) to [out=30,in=150] (0.6 ,0.3);
    \draw[-] (1.2,0.3) to [out=30,in=150] (2.4,0.3);
    \draw[-] (1.8 ,0.3) to [out=30,in=150] (5.4,0.3);
    \draw[-] (3 ,0.3) to [out=30,in=150] (4.2,0.3);
    \draw[-] (3.6,0.3) to [out=30,in=150] (4.8,0.3);

    \node at (5.6,-0.20) {$,$};
\end{tikzpicture}
\begin{tikzpicture}
    \node at (0, 0) {$0$};
    \node at (0.6,0) {$2$};
    \node at (1.2,0) {$2$};
    \node at (1.8,0) {$0$};
    \node at (2.4,0) {$4$};
    \node at (3,0)   {$2$};
    \node at (3.6,0) {$0$};
    \node at (4.2,0) {$4$};
    \node at (4.8,0) {$2$};
    \node at (5.4,0) {$2$};
    \draw[-] (0 ,0.3) to [out=30,in=150] (5.4,0.3);
    \draw[-] (0.6 ,0.3) to [out=30,in=150] (2.4,0.3);
    \draw[-] (1.2,0.3) to [out=30,in=150] (4.2,0.3);
    \draw[-] (1.8 ,0.3) to [out=30,in=150] (3,0.3);
    \draw[-] (3.6    ,0.3) to [out=30,in=150] (4.8,0.3);

    \node at (5.6,-0.25) {$ $};
    \node at (5.6,-0.15) {$.$};
\end{tikzpicture}
\end{figure}

Example \ref{ex:lower_bd_exact} shows that $h(w) \geq 2$, henceforth we conclude that $h(w)=2$ and the upper bound given by Theorem \ref{thm:upper_bound} is exact in that case.
\end{example}

There is an analogue of Lemma \ref{lemma:nconf_to_ends} valid for steady arc configurations. 
\begin{lemma}\label{lemma:sconf_to_ends}
    For a word $w\in W_{2n}$, an arc configuration $C\in \sconf(w)$ is uniquely defined by the set of left ends $\len(C)$.
    Namely, for $C_1,C_2\in\sconf(w)$, if $\len(C_1)=\len(C_2)$ then $C_1=C_2$.
\end{lemma}
\begin{proof}
    We prove the statement of the lemma for a more general class of sequences of the form $\rV = w_1[0,\dots, 2m]w_2$ such that all letters of $w_2$ are strictly larger than $0$ and all letters of $w_1$ are non-negative. We apply induction with respect to partial order \eqref{eq:partial_order}. In the case of $l(w_2) = 0$ we have $\sconf(\rV) = \varnothing$ so there is nothing to prove.

    Let $2n = l(w_1) + m+1 + l(w_2)$. Define the map $\mu_{\rV}: \sconf(\rV) \rightarrow \{+,-\}^{\times 2n}$, which maps an arc configuration $\{(i_1,j_1),\dots,(i_n,j_n)\}$ to the monomial with $"+"$'s exactly at the positions $\{i_1,\dots, i_n\}$ of the left ends of arcs.

    In case $l(w_2) > 0$, write $\rV = w_1[0,\dots, 2m]a\tilde{w}_2$. 

    Assume first, that $a \neq 2m+2$, then we have a commutative diagram
    \begin{equation*}
    \begin{tikzcd}
        \sconf(w_1 a[0, \dots,2m]\tilde{w}_2) \arrow[r, "\iota_0"' , "\sim"] \arrow[d, "\mu_{w_1 a[0, \dots,2m]\tilde{w}_2}"]
        & \sconf(\rV) \arrow[d, "\mu_{\rV}"] \\
        \{+,-\}^{\times 2n} \arrow[r, "\sigma"', "\sim"]
        & |[]| \{+,-\}^{\times 2n}.
    \end{tikzcd}
    \end{equation*}
    Here $\sigma$ is the bijection induced from the permutation of factors $a[0,\dots,2m] \mapsto [0,\dots, 2m]a$ to the sets of left ends of the corresponding arc configurations. By the induction hypothesis the map $\mu_{w_1 a[0, \dots,2m]\tilde{w}_2}$ is injective, which implies injectivity of $\mu_{\rV}$.

    Assume now that $a = 2m+2$, then we have a commutative diagram
    \begin{equation*}
    \begin{tikzcd}
        \sconf(w_1 [0,\dots, 2m-2]\tilde{w}_2) \sqcup \sconf(w_1 [0,\dots, 2m+2]\tilde{w}_2) \arrow[r, "\iota_1 \sqcup \iota_2"' , "\sim"] \arrow[d, "{j_{l(w_1) + m}\circ \mu_{w_1 [0, \dots,2m-2]\tilde{w}_2}\sqcup \mu_{w_1 [0,\dots, 2m+2]\tilde{w}_2}}"]
        & \sconf(\rV) \arrow[d, "\mu_{\rV}"] \\
        \{+,-\}^{\times 2n} \arrow[r, "\Id"', "\sim"]
        & |[]| \{+,-\}^{\times 2n}.
    \end{tikzcd}
    \end{equation*}
    Here the injective map $j_{l(w_1) + m}:\{+,-\}^{\times 2(n-1)} \rightarrow \{+,-\}^{\times 2n}$ adds a pair of letters $(+,-)$ after the $l(w_1) + m$'th letter. 
    
    By the induction hypothesis, the maps $\mu_{w_1 [0,\dots, 2m- 2]\tilde{w}_2}$ and $\mu_{w_1 [0, \dots,2m+2]\tilde{w}_2}$ are injective. Moreover, any monomial $(\epsilon_1,\dots, \epsilon_{2n})$ from the image of $\mu_{w_1 [0,\dots, 2m+2]\tilde{w}_2}$ has $\epsilon_{l(w_1)+m+2} = +$, hence by the construction of $j_{l(w_1) + m}$ images of $\mu_{w_1 [0,\dots, 2m+2]\tilde{w}_2}$ and $j_{l(w_1) + m}\circ \mu_{w_1 [0,\dots, 2m-2]\tilde{w}_2}$ are disjoint and therefore the map $j_{l(w_1) + m}\circ \mu_{w_1 [0, \dots,2m-2]\tilde{w}_2}\sqcup \mu_{w_1 [0,\dots, 2m+2]\tilde{w}_2}$ is injective, which implies injectivity of $\mu_{\rV}$.
\end{proof}

\subsection{The words $w$ with non-zero $h(w)$.}\label{subsec:support}

In this section we apply results of Section \ref{sec:bounds} to compute the set $\{w \in W_{2n} | h(w)\neq 0\}$ for each $n$ and give a simple combinatorial criterion whether a word has non-zero $h(w)$. We call this set support of $h$. Due to factorizations of $h(w)$, see Sections \ref{subsec:lattice_restrictions} and \ref{subsec:submod_factor}, it is sufficient to consider only words $w$ with letters in $2\mbZ_{\geq 0}$ whose
 $\mathrm{supp}(w)$ is a segment in $2\mbZ$.

We begin with a combinatorial statement.

\begin{lemma}\label{lemma:confex_iconfex}
         $\conf(w) \neq \varnothing$ if and only if  $\iconf(w) \neq \varnothing$.
    \end{lemma}
    \begin{proof}
        
        We show that given an arc configuration with a reducible intersection, we can produce an arc configuration with smaller number of reducible intersections.

        There are only two types of reducible intersections given by
            \begin{figure}[H]
            
\begin{tikzpicture}
    \node at (0, -0.08)   {$\dots$};
    \node at (0.6,0) {$a$}; 
    \node at (1.2,-0.08) {$\dots$};
    \node at (1.8,0) {$a$};
    \node at (2.4,-0.08) {$\dots$};
    \node at (3.2,0.02)   {$a{+}2$};
    \node at (4,-0.08) {$\dots$};
    \node at (4.8,0.02) {$a{+}2$};
    \node at (5.6,-0.08) {$\dots$};
    
    \draw[-] (0.6 ,0.3) to [out=30,in=150] (3.2,0.3);
    \draw[-] (1.8 ,0.3) to [out=30,in=150] (4.8,0.3);
    
    \node at (6,-0.20) {$,$};
\end{tikzpicture}
\begin{tikzpicture}
    \node at (0, -0.08)   {$\dots$};
    \node at (0.6,0) {$a$}; 
    \node at (1.2,-0.08) {$\dots$};
    \node at (2,0.02) {$a{+}2$};
    \node at (2.8,-0.08) {$\dots$};
    \node at (3.6,0.02)   {$a{+}2$};
    \node at (4.4,-0.08) {$\dots$};
    \node at (5.2,0.02) {$a{+}4$};
    \node at (6,-0.08) {$\dots$};
    
    \draw[-] (0.6 ,0.3) to [out=30,in=150] (3.6,0.3);
    \draw[-] (2 ,0.3) to [out=30,in=150] (5.2,0.3);

    \node at (6.4,-0.25) {$ $};
    \node at (6.4,-0.15) {$.$};
\end{tikzpicture}
\end{figure}

    The way we change an arc configuration to decrease the number of reducible intersections is given by
    \begin{figure}[H]
    \hspace{2pt}
\begin{tikzpicture}
    \node at (0, -0.08)   {$\dots$};
    \node at (0.6,0) {$a$}; 
    \node at (1.2,-0.08) {$\dots$};
    \node at (1.8,0) {$a$};
    \node at (2.4,-0.08) {$\dots$};
    \node at (3.2,0.02)   {$a{+}2$};
    \node at (4,-0.08) {$\dots$};
    \node at (4.8,0.02) {$a{+}2$};
    \node at (5.6,-0.08) {$\dots$};
    
    \draw[-] (0.6 ,0.3) to [out=30,in=150] (3.2,0.3);
    \draw[-] (1.8 ,0.3) to [out=30,in=150] (4.8,0.3);

    \node at (6.4, 0) {$\longmapsto$};

    \draw[-,dashed] (2.1, 0.6) to [out=0,in=180] (2.8,0.6);

    \node at (7.4, -0.08)   {$\dots$};
    \node at (8,0) {$a$}; 
    \node at (8.6,-0.08) {$\dots$};
    \node at (9.2,0) {$a$};
    \node at (9.8,-0.08) {$\dots$};
    \node at (10.6,0.02)   {$a{+}2$};
    \node at (11.4,-0.08) {$\dots$};
    \node at (12.2,0.02) {$a{+}2$};
    \node at (13,-0.08) {$\dots$};
    
    \draw[-] (8 ,0.3) to [out=30,in=150] (12.2,0.3);
    \draw[-] (9.2 ,0.3) to [out=30,in=150] (10.6,0.3);

    \node at (13.4,-0.20) {$,$};
\end{tikzpicture}

\medskip

\begin{tikzpicture}
    \node at (0, -0.08)   {$\dots$};
    \node at (0.6,0) {$a$}; 
    \node at (1.2,-0.08) {$\dots$};
    \node at (2,0.02) {$a{+}2$};
    \node at (2.8,-0.08) {$\dots$};
    \node at (3.6,0.02)   {$a{+}2$};
    \node at (4.4,-0.08) {$\dots$};
    \node at (5.2,0.02) {$a{+}4$};
    \node at (6,-0.08) {$\dots$};
    
    \draw[-] (0.6 ,0.3) to [out=30,in=150] (3.6,0.3);
    \draw[-] (2 ,0.3) to [out=30,in=150] (5.2,0.3);

    \draw[-,dashed] (2.785, 0.4) to [out=90,in=-90] (2.785,1);

    \node at (6.8, 0) {$\longmapsto$};

    \node at (7.6, -0.08)   {$\dots$};
    \node at (8.2,0) {$a$}; 
    \node at (8.8,-0.08) {$\dots$};
    \node at (9.6,0.02) {$a{+}2$};
    \node at (10.4,-0.08) {$\dots$};
    \node at (11.2,0.02)   {$a{+}2$};
    \node at (12,-0.08) {$\dots$};
    \node at (12.8,0.02) {$a{+}4$};
    \node at (13.6,-0.08) {$\dots$};
    
    \draw[-] (8.2, 0.3) to [out=30,in=150] (9.6,0.3);
    \draw[-] (11.2, 0.3) to [out=30,in=150] (12.8,0.3);

    \node at (14,-0.25) {$ $};
    \node at (14,-0.15) {$.$};
\end{tikzpicture}

\end{figure}
        More formally, let $C\in \conf(w)$ be an arc configuration with the smallest number $N$  of reducible intersections. Assume $N > 0$, then there exists a pair of intersecting arcs $((i_1,j_1), (i_2, j_2))$, $i_1 < i_2 < j_1 < j_2$, such that the intersection is reducible. Write $C = \{(i_1, j_1), (i_2, j_2)\} \sqcup \tilde{C}$, let $w = (a_1,\dots, a_{2n})$, and set 
        $$\hat{C} = \begin{cases}
            \{(i_1, j_2), (i_2, j_1)\} \sqcup \tilde{C}, \;\;\textit{if }a_{i_2} = a_{i_1},\\
            \{(i_1, i_2), (j_1, j_2)\} \sqcup \tilde{C}, \;\;\textit{if }a_{i_2} = a_{i_1} + 2.\\
        \end{cases}$$
        Then $\hat{C}\in \conf(w)$ and we claim that the number of reducible intersections in $\hat{C}$ is less than in $C$. 
        We compare intersections of an arc $(i,j)\in \tilde{C}$ with arcs $(i_1,j_1),\;(i_2,j_2)$ in $C$ and with arcs $(i_1,i_2), (j_1, j_2)$ in $\hat{C}$.
        
        Let $a_{i_2} = a_{i_1}$. If an arc $(i,j) \in \tilde{C}$ intersects both $(i_1, j_2)$ and $(i_2, j_1)$, then 
        either $i< i_1 < i_2 < j < j_1 < j_2$ or $i_1 < i_2 < i < j_1 < j_2 < j$. In both cases $(i,j)$ intersects both $(i_1, j_1)$ in and $(i_2, j_2)$ in $C$, which are of the same color. It remains to check that if an arc $(i,j)\in \tilde{C}$ intersects exactly one of arcs $(i_1,j_2), (i_2, j_1)$ then it intersects at least one of arcs $(i_1, j_1),(i_2, j_2)$. Assume an arc $(i,j)\in\tilde{C}$ intersects $(i_1, j_2)$ but not $(i_1,j_1)$. Then either $i < i_1 < i_2 < j_1 < j < j_2$ or $ i_1 < i_2 < j_1 < i < j_2 < j$. In both cases $(i,j)$ and $(i_2, j_2)$ intersect. Assume an arc $(i,j)\in\tilde{C}$ intersects $(i_2, j_1)$ but not $(i_1,j_1)$. Then $i_1 < i < i_2 < j < j_1 < j_2$ and $(i,j)$ and $(i_2, j_2)$ intersect. Note that in all described cases intersection of the arc $(i,j)$ with arcs $(i_1,j_2), (i_2,j_1)$ is reducible then intersections with arcs $(i_1,j_1), (i_2,j_2)$ are reducible.

        The check for the case $a_{i_2} = a_{i_1} + 2$ goes similarly.
    
    Clearly, $(i_1,i_2), (j_1, j_2)$ do not intersect, hence the number of reducible intersections in $\hat{C}$ is less than in $C$ which contradicts minimality of $N$.
    \end{proof}

Then the criterion for non-triviality of $H(w)$ is given by the following theorem.
\begin{theorem}\label{thm:supp}
    Let $w\in W_{2n}$ be a word. Then $h(w) = 0$ if and only if $\conf(w) = \varnothing$. 
\end{theorem}

\begin{proof}
    If $\conf(w) = \varnothing$, then  $\sconf(w) = \varnothing$, so $h(w)=0$ by Theorem \ref{thm:upper_bound}. The opposite implication follows from Lemma \ref{lemma:confex_iconfex}. 
\end{proof}

Now we proceed to a simpler combinatorial description of words $w\in W_{2n}$ such that $h(w)\neq 0$. We need one more definition.

\begin{defi}\label{def:std_conf}
    Let $C\in \conf(w)$ for some $w\in W_{2n}$. We call $C$ a standard arc configuration if either $n = 1$ or $\alpha = (\max(I_{w}(\min(\supp(w)))), \max(I_{w}(\min(\supp(w)){+}2))) \in C$ and $\mathring{C}_{\alpha}$ is a standard arc configuration.
\end{defi}
More informally, the standard configuration is constructed recursively on each step connecting the rightmost smallest letter, say called $a$, with the rightmost letter $a+2$. Note that if at some point doing so we do not obtain an arc, that is if the rightmost $a+2$ is to the left of rightmost $a$, then the standard configuration does not exist.

Clearly, if a standard configuration exists it is unique.

\begin{lemma}\label{lemma:std_exists}
Let $w\in W_{2n}$ be a word. Then $\conf(w) \neq \varnothing$ if and only if there exists the standard arc configuration of $w$.
\end{lemma}
\begin{proof}

    The if direction is clear.

    To prove the other direction we first show that if there is any configuration of the word $w$ then there is a configuration of the word $w$ such that the right most smallest letter, say called $a$ is connected to the right most letter $a+2$.  Then we remove the arc connecting these right most $a$ and $a+2$ and use the induction on $n$.

    Let $w = (a_1,\dots, a_{2n})$. Assume that $\conf(w) \neq \varnothing$. Denote $a = \min(\supp(w))$, $i = \max(I_{w}(a))$, $j = \max(I_{w}(a+2))$. We claim that there is an arc configuration $C\in \conf(w)$ such that $\{i,j\}\in C$. Indeed, since in any arc configuration $i$ is the left end of an arc, $i < j$. Let $\tilde{C}\in \conf(w)$, then if $(i,j) \notin \tilde{C}$ we have $(i, j^{\prime}) \in \tilde{C}$ and either $(j, k)\in \tilde{C}$ or $(i^{\prime}, j)\in \tilde{C}$ where $j^{\prime} < j$ in the first case and $i^{\prime} < i$ in the second case. Set $C = \tilde{C}\backslash \{(i, j^{\prime}),  (j, k)\} \sqcup \{(i,j), (j^{\prime}, k)\}$ in the first case and $C = \tilde{C}\backslash \{(i, j^{\prime}),  (i^{\prime}, j)\} \sqcup \{(i,j), (i^{\prime}, j^{\prime})\}$ in the second case. Then $C\in \conf(w)$ and contains $(i,j)$. Hence $\mathring{C}_{\{(i,j)\}}$ is an arc configuration of $\mathring{w} = (a_1,\dots, a_{i-1}, a_{i+1},\dots, a_{j-1}, a_{j+1},\dots, a_{2n})$. By the induction hypothesis there exists the standard configuration of $\mathring{w}$, denote this configuration by $C_{0}$. Adding arc $(i,j)$ to $\mathring{C}$ we obtain the standard arc configuration of $w$.

\end{proof}

\begin{cor}
    Let $w\in W_{2n}$ be a word. Then $h(w) = 0$ if and only if $w$ has no standard arc configuration. \qed
\end{cor}
Standard arc configurations are compatible with slides.
\begin{lemma}\label{lemma:std_conf_slide}
    Let $w\in W_{2n}$ be a word such that $\conf(w)\neq \varnothing$. Let $C_0$ be the standard arc configuration of $w$. Then the arc configuration $s_{\conf(w)}(C_0)$ is the standard arc configuration of the word $s(w)$.
\end{lemma}
\begin{proof}
    For a word $\tilde{w}$ denote the standard arc configuration of $\tilde{w}$ by $C_{0,\tilde{w}}$.

    Let the first letter of $w$ be not the rightmost minimal letter. Then the first arc in the recursive construction of $C_{0,w}$ is mapped to the first arc in the recursive construction of $C_{0,s(w)}$. Removing these arcs we obtain a pair of words related by $s$. By induction on $n$ standard arc configurations for this pair are related by a slide as well. Adding removed arcs back we obtain standard arc configurations of $w$ and $s(w)$. By construction these arc configurations are related by a slide. 
    
    Let the first letter of $w$ be the rightmost minimal letter in $w$. Then the minimal letter in $w$ is unique. Let $j$ be such that $(1,j)$ is in the standard arc configuration of $w$. Then after a slide the $j$'th letter becomes the rightmost minimal one. This implies that the first arc in the recursive construction of the standard arc configuration of $s(w)$ is $(j-1,2n)$. Note that $\mathring{w}_{(1,j)} = \mathring{s(w)}_{(j-1,2n)}$. Therefore, these words have the same standard arc configuration. Adding removed arcs back we obtain standard arc configurations of $w$ and $s(w)$. By construction these arc configurations are related by a slide. 
\end{proof}

\section{Degenerations and limits.}\label{sec:degs_limits}
If a word $w$ of length $2n$ contains $n$ subwords $(a_i,a_i+2)$ with generic $a_i$, then there is a unique trivial submodule which can be described. In this section we study the limits of these submodules when $a_i$ go to non-generic points.

In this section we use both multiplicative and additive notations for evaluation parameters. To denote additive evaluation parameters we use letters $a, b$. To denote multiplicative evaluation parameters we use letters $u, v$ where we assume $u = q^a, v = q^b$.

In this section by a word $w$ we mean an arbitrary tuple $w\in \mbC^{2n}$ (in additive notation). We use the language we developed for words with integer letters.

\subsection{Generic singular vectors.}\label{subsec:gen_sing_vec}

We start by studying words with a single arc configuration.

\begin{defi}
Let $C\in\uconf(2n)$ be an uncolored arc configuration. We call $w\in \mbC^{2n}$ a generic word corresponding to $C$ if $\conf(w)=\{C\}$.
\end{defi}
\begin{lemma}
    Let $w\in \mbC^{2n}$ be a generic word corresponding to an uncolored arc configuration $C$. Then for the word $w$ the arc configuration $C$ is standard, steady, and irreducible. Moreover $h(w)=1$.
\end{lemma}
\begin{proof}
Since $\conf(w)\neq \varnothing$, the standard arc configuration for $w$ exists by Lemma \ref{lemma:std_exists}, moreover, $\iconf(w)\neq \varnothing$ by Lemma \ref{lemma:confex_iconfex}. Thus, $C$ is standard and irreducible. By Theorem \ref{thm:supp}, $h(w)>0$, therefore by Theorem \ref{thm:upper_bound} the set of steady arc configurations is non-empty. Hence $C$ is steady. 

For the last part, by Lemma \ref{prop:hom_factorize_lattice} it is sufficient to consider $w\in 2\mbZ$.
Therefore, $h(w)=1$ follows from and Theorem \ref{thm:supp} and the upper bound in Theorem \ref{thm:upper_bound}.   
\end{proof}

Let \ $a=(a_1, \dots, a_{n}) \in \mbC^{ n}$ and let  
$C\in\uconf(2n)$ be an uncolored arc configuration. 
Write
$C = \{(i_k, j_k)\}_{k=1}^n$  where $1 = i_1 < i_2 < \dots < i_n$. Define a tuple $w_C(a) = (\alpha_1, \dots, \alpha_{2n}) \in \mbC^{2n}$ of length $2n$ by setting $\alpha_{i_k} = a_k$ and $\alpha_{j_k}=a_k+2$.

Clearly, if $a$ is generic (for example, if $a_i-a_j\not\in 2\mbZ$ for all $i\neq j$), then $w_C(a)$ is generic.
 We denote an $\ell$-singular vector of weight zero in module $w_C(a)$ by $v_C(u)$. Here, as always, $u=q^a$. The vector $v_C(u)$ is defined up to proportionality. We call $v_C(u)\in\rL_1^{\otimes 2n}$ a generic $\ell$-singular vector.

In case $C = \{(1,2),(3,4),\dots, (2n-1,2n)\}$ we have 
\begin{equation}\label{eq:gen_sing_simp}
v_{C}(u) = ((+-)-q^{-1}(-+))^{\otimes n}.
\end{equation}For any other $C$, one can construct $v_C(u)$ by applying ${R}$-matrices \eqref{eq:Rmat2} to vector \eqref{eq:gen_sing_simp}. 
Since $C$ is irreducible, one can use the procedure described in the proof of Theorem \ref{thm:lower_bound_irr}.

Note that for all $C$, the vector $v_C(u)$ is a rational function of variables $u_i = q^{a_i}$.

The map $v_C: u \mapsto v_{C}(u)$, where $u=(u_1,\dots, u_n)$, is a rational map from a torus $(\mbC^{\times})^{\times n}$ to $\mathbb{P}(H_{2n})$. Moreover, the $R$-matrix is homogeneous in $u_i$, therefore $v_C(u) \in \mathbb{P}(H_{2n})$ does not change if all $a_i$ are changed to $a_i+b$ or, equivalently if all $u_i$ are changed to $u_iq^b$.

Therefore, we have a rational map $\hat{v}_C:(\mbC^{\times})^{\times n}/\mbC^{\times} \rightarrow \mathbb{P}(H_{2n})$.

One application of our study in this section is an alternative proof of the only if part of Theorem \ref{thm:supp}. Namely, we use $v_C(u)$ to show that if  a word $w$ has arc configuration, then $h(w)\neq 0$.

\begin{proof} [Another proof of only if part of Theorem \ref{thm:supp}]
Let $C$ be an arc configuration of a word $w$. Let $w=w_{C}(a^0)$.
We deform evaluation parameters $a^0$ to a generic tuple $a$ so that the deformed word $w_C(a)$ is a generic word corresponding to $C$. Then we have a unique up to a scalar vector $v_{C}(u) \in H(w(a))$ by Proposition \ref{prop:hom_factorize_lattice}.

Consider a sequence of generic points $(u(m))_{m\in \mbZ_{>0}}$ converging to $u^0= q^{a^0}$. Since the space $\mathbb{P}(H_{2n})$ is compact, there is a limit point of the sequence $(v_C(u(m)))_{m\in \mbZ_{>0}}$. The module $w_C(a)$ and the vector $v_C(u)$ depend on the parameters $a$ (or $u$) continuously.
Therefore, the limit point is in $H(w_C(a^0)) = H(w)$.
\end{proof}

\begin{conj}\label{conj:deg_onto} For any word $w$ the subspace $H(w)$ of $\ell$-singular vectors of weight $0$ is spanned by limits of generic singular vectors $v_C(u)$, $C\in\conf(w)$.
\end{conj}

Let $w$ be a word and $C\in\conf(w)$ an arc configuration.
One obtains a vector in $H(w)$ from a generic singular vector $v_C(u)$ by taking consecutive limits as follows. Since the map $v_C$ is a rational map to projective space, its singular locus has codimension at least $2$. Therefore, we can restrict the map $v_C$ to a hyperplane of the form $a_i - a_j = m$ for some constant $m$ (or, equivalently, $u_i=q^m u_j$). After that we can restrict to another hyperplane of the same form. Repeating such restrictions, we arrive to a one-dimensional intersection of hyperplanes which correspond to shifts of the word $w$. 
The corresponding vectors $v_C$  are proportional and belong to $H(w)$.

 In general, restrictions performed in different orders produce different vectors in $H(w)$.
\begin{example}\label{ex:degeneration_different_ways}
     Let $w = 020242$ and $C = \{(1,4), (2,5), (3,6)\}
     = \begin{tikzpicture}
    \draw[-] (0,0) to [out=30,in=150] (1.8,0);
    \draw[-] (0.6,0) to [out=30,in=150] (2.4,0);
    \draw[-] (1.2,0) to [out=30,in=150] (3.0,0);
\end{tikzpicture}$.

We compute $v_C(u)$ in homogeneous coordinates induced by the basis of Catalan vectors in $H_{6}$ ordered as in the picture in Section \ref{subsec:word_arcs},
    \begin{multline*}
        v_C(u_1, u_2,u_3) = \left(q (u_{1}{-}u_{3}) (q^2 u_{1}{-}u_{2}) (q^2 u_{2}{-}u_{3}):q^2 (u_{1}{-}u_{2}) (u_{1}{-}u_{3}) (q^2 u_{2}{-}u_{3}): \right.\\ \left. :q^2 (u_{1}{-}u_{3}) (u_{2}{-}u_{3}) (q^2 u_{1}{-}u_{2}):q
   (u_{1}{-}u_{3}) (q^2 u_{1}{-}u_{2}) (q^2 u_{2}{-}u_{3}):(q^2 u_{1}{-}u_{2}) (q^2 u_{1}{-}u_{3}) (q^2 u_{2}{-}u_{3})\right).
    \end{multline*}
Here the word $w$ corresponds to $u^0 = (1, q^2, 1)$.

    Let us first restrict to $u_2 = q^2 u_1$, then $$v_C(u_1,q^2u_1,u_3) = \left(0: 1 : 0 : 0: 0\right).$$
    
    The latter restrictions $u_3 = u_1$ and $u_1 = 1$ give $\left(0: 1 : 0 : 0: 0\right)$ which is a line  in $H(w)$.

    Alternatively, let us first restrict to $u_3 = u_1$, then $$v_C(u_1,q^2u_1,u_1) = \left(0: 0: 0 : 0: 1\right).$$
    
    The latter restrictions $u_2 = q^2 u_1$ and $u_1 = 1$ give $\left(0: 0 : 0 : 0: 1\right)$ which is a different line  in $H(w)$.

    The sets $\iconf(w)$ and $\sconf(w)$ coincide and consist of two Catalan arc configurations, therefore, we've reconstructed the whole space of $\ell$-singular vectors of weight $0$ as different limit points of a single generic singular vector $v_C(u)$.
\end{example}

In Example \ref{ex:degeneration_different_ways}, arc configuration $C$ has intersecting arcs and the $R$-matrices corresponding to these intersections are non-invertible at the limit point $u^0$ which, in this case, leads to different limits of $v_C(u)$.

If all $R$-matrices used to construct $v_C(u)$ are invertible at $u=u^0$, then $v_C(u)$ has a unique limit point. We can improve this statement as follows.
\begin{lemma}
     If $C$ is an irreducible configuration of the word $w = w(a^0)$ then the limit point of $v_C(u)$ as $u\rightarrow u^0$ is unique.
\end{lemma}
\begin{proof} 
Let $a^0$ be such that $C$ is an irreducible configuration of the word $w_C(a^0)$. We construct $v_C(u)$ as in Theorem \ref{thm:lower_bound_irr}. Then $v_C(u)$ is a polynomial vector in $u$. Since $C$ is irreducible, the coefficient of the highest monomial does not vanish at $u^0$ (see the proof of Theorem \ref{thm:lower_bound_irr} for the details). Therefore, the limit of $v_C(u)$ as $u\to u^0$ is a well-defined unique non-zero vector.
\end{proof}

We expect a stronger version of Conjecture \ref{conj:deg_onto}.
\begin{conj}\label{conj:deg_std_onto}
    For any word $w$ the subspace $H(w)$ of $\ell$-singular vectors of weight $0$ is spanned by limits of generic singular vectors $v_C(u)$, where $C$ is the standard configuration of $w$.
\end{conj}

\begin{example}\label{ex:degs_long}
    Consider a word $w = (a_1,\dots, a_{10})=2020224244$. The standard arc configuration of $w$ is $C = \{(1, 7), (2, 6), (3, 9), (4, 8), (5, 10)\}$. Then the corresponding generic word $w_C(a)$ in multiplicative notation is given by $(u_1, u_2, u_3, u_4, u_5, q^2u_2, q^2u_1, q^2u_4, q^2u_3, q^2u_5)$. 
    Let $l_{ij}$ be the hyperplane given by the equation $u_i = q^{a_i-a_j}u_j$. A computation shows that there are three linearly independent vectors in $H(w)$ obtained by restrictions of $v_C(u)$ to the one-dimensional intersection of hyperplanes $l_{12}$ , $l_{13}$, $l_{14}$, $l_{15}$ in different orders,
    $$
    (v_{C}(u))|_{l_{12}}|_{l_{13}}|_{l_{14}}|_{l_{15}}, \quad (v_{C}(u))|_{l_{12}}|_{l_{14}}|_{l_{15}}|_{l_{13}}, \quad (v_{C}(u))|_{l_{14}}|_{l_{15}}|_{l_{23}}|_{l_{12}}.
    $$

    By computing steady arc configuration we get $|\sconf(w)| = 3$. Therefore,  $h(w) =3$. Note that there is only one irreducible arc configuration of $w$.
\end{example}
We have checked Conjecture \ref{conj:deg_onto} for all words up to $10$ letters and Conjecture \ref{conj:deg_std_onto} for all words up to $8$ letters computationally.

\medskip 

It is interesting to study the (closure of the) image of the map $v_C$. For example, one can obtain a quadric and cubic surfaces.

\begin{example}\label{ex:im_close1}
    Let $C = \{(1,3), (2,5), (4,6)\}$. In homogeneous coordinates of Catalan vectors as in Example \ref{ex:degeneration_different_ways} the map $v_C$ is given by 
    \begin{multline*}
        (u_1, u_2, u_3)  \mapsto (q (u_{1}{-}u_{2}) (u_{2}{-}u_{3}){:}(u_{1}{-}u_{2}) (q^2 u_{2}{-}u_{3}){:}(u_{2}{-}u_{3}) (q^2 u_{1}{-}u_{2}){:}\frac{(q^2 u_{1}{-}u_{2}) (q^2 u_{2}{-}u_{3})}{q}{:}0).
    \end{multline*}
    The closure of the image of the map $v_C$ is described by equations $x_5 = 0, x_2 x_3 - x_1 x_4 = 0$.
\end{example}

\begin{example}\label{ex:im_close2}
    Consider the map $v_{C}$ from Example \ref{ex:degeneration_different_ways}. Then the closure of image  of the map $v_C$ is described by the equations
    \begin{equation*}
        x_1 - x_4 = 0,\;\; x_1^3 - x_1 (x_3 x_5 + x_2 x_3 + x_2 x_5) + [2]_q x_2 x_3 x_5 = 0.
    \end{equation*}
    For generic $q$, these equations define a singular cubic surface of type $3A_1$ (see \cite{BruceWallOnClassificationOfCubic}) in $\mbC\mathbb{P}^3$. Such a cubic contains not $27$ but $12$ lines.

    Three of these lines are projectivizations of two-dimensional spaces $H(w)$ for the words $w_1 = 020242, w_2 = 202424, w_3 = 024246$. These lines form a triangle. The vertex which is the intersection of sides corresponding to words $w_i, w_j$ is the line generated by the Catalan vector in the space $H(w_i)\cap H(w_j)$. 
    
    As well the vertices of this triangle are the three singular points of type $A_1$ on the cubic.
\end{example}
    In Examples \ref{ex:im_close1} and \ref{ex:im_close2} the map $v_C$ is injective and there exists a rational inverse map defined on the image of $v_C$. It is interesting to understand whether this holds for any connected arc configuration $C$.

\subsection{Degeneration graphs.}\label{subsec:deg_graphs}
Let $\mbC^{2n}$ be a space with coordinates $(\al_1,\dots,\al_{2n})$. Let $C$ be an uncolored arc configuration. Define an affine subspace   $l_C \subset \mbC^{2n}$ of dimension $n$ given by solutions of $n$ equations $\al_{j}=\al_i+2$ labelled by arcs $(i,j)\in C$. Clearly, for each $a\in \mbC^{n}$,  we have  $w_C(a)\in l_C$.

For a positive integer $n$ define a directed graph $DG(2n)$ as follows. The vertices are non-empty intersections $\mathop{\cap}\limits_{C\in S}l_{S}$, where $S$ runs through all subsets of $\uconf(2n)$. Here, if two intersections coincide, $\mathop{\cap}\limits_{C\in S_1}l_{S_1}=\mathop{\cap}\limits_{C\in S_2}l_{S_2}$ then these two intersections are the same vertex.

Each vertex of the degeneracy graph is an affine space which we simply call a plane. 

We connect vertices $A_1$ and $A_2$ by a directed edge $A_1\to A_2$ if and only if the planes $A_1,A_2$ satisfy $A_2 \subset A_1$ and $\dim(A_1) - \dim(A_2) = 1$. 

We call the resulting graph the degeneracy graph and denote it by $DG(2n)$.

\begin{remark}
 The arrangement of hyperplanes $\al_j=\al_i+2$  for al $i<j$ is called a linial hyperplane arrangement, \cite{MR1796891}. The arrangement generated by codimension $n$ planes $l_C$ is a "subarrangement" of the linial arrangement. 
\end{remark}

Motivated by slides, see Lemma \ref{lemma:slide_isom}, we consider an action of group $\mbZ$ on $\mbC^{2n}$ given by 
\begin{equation*}
(\al_1, \dots, \al_{2n}) \mapsto (\al_{2n}-4, \al_1, \dots, \al_{2n-1}).
\end{equation*}
This action descends to an action of the cyclic group $\mbZ_{2n}$ on the quotient $\mbC^{2n}/\mbC$ of $\mbC^{2n}$ by $\mbC$ included diagonally (cf. Section \ref{subsec:gen_sing_vec}).

Then the group $\mbZ_{2n}$ permutes the planes $l_C$ and, therefore, acts on the degeneracy graph $DG(2n)$. 

One could include the action of anti-involution  $\omega$, see Corollary \ref{cor:reverse_word}, and obtain an action of the dihedral group $D_{2n}$ on $DG(2n)$. We do not use the action of $D_{2n}$ in our examples.

We explicitly give the degeneracy graphs $DG(2n)$ for $n=1,2,3$.

\begin{example}
For $n = 1$ the degeneracy graph is a single vertex $l_{\{(1,2)\}}$. 
\end{example}
\begin{example}\label{ex:n2degeneracy_graph}
For $n = 2$,
\begin{equation*}
    \uconf(4) = \{\begin{tikzpicture}
        \draw[-] (0.4,0) to [out=30,in=150] (1,0);
        \draw[-] (1.2,0) to [out=30,in=150] (1.8,0);
    \end{tikzpicture}\ ,\;\ \begin{tikzpicture}
        \draw[-] (0,0) to [out=30,in=150] (1.2,0);
        \draw[-] (0.6,0) to [out=30,in=150] (1.8,0);
    \end{tikzpicture}\ ,\;\  \begin{tikzpicture}
        \draw[-] (0,0) to [out=30,in=150] (1.8,0);
        \draw[-] (0.6,0) to [out=30,in=150] (1.2,0);
    \end{tikzpicture}\}.
\end{equation*}

The corresponding planes are 
\begin{align*}l_{\{(1,2),(3,4)\}} = \{(\alpha_1, \alpha_1 + 2, \alpha_3, \alpha_3 + 2)\},\\ l_{\{(1,3),(2,4)\}} = \{(\alpha_1, \alpha_2, \alpha_1 + 2,  \alpha_2 + 2)\},\\ l_{\{(1,4),(2,3)\}} = \{(\al_1, \al_2, \al_2 + 2, \al_1 + 2)\}.\end{align*}

We have 
\begin{align*}l_{\{(1,2),(3,4)\}} \cap l_{\{(1,3),(2,4)\}}&=\{(\al,\al+2,\al+2,\al+4)\},\\ l_{\{(1,3),(2,4)\}}\cap l_{\{(1,4),(2,3)\}}& = \{(\al, \al, \al+2, \al+2)\},\\ l_{\{(1,2),(3,4)\}} \cap l_{\{(1,4),(2,3)\}}&=\varnothing.
\end{align*}

Therefore, $DG(4)$ is given by
    
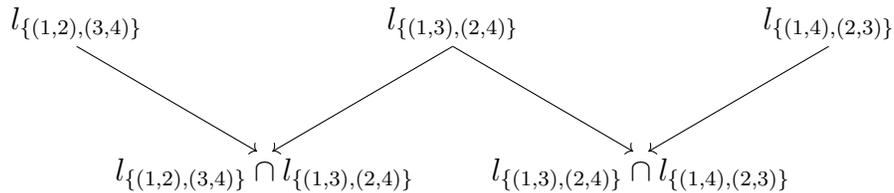
\begin{center}
\begin{tikzpicture}
    \node at (-5,0) {$l_{\{(1,2),(3,4)\}}$};
    \node at (-0,0) {$l_{\{(1,3),(2,4)\}}$};
    \node at (5,0) {$l_{\{(1,4),(2,3)\}}$};
    \node at (-2.5,-2) {$l_{\{(1,2),(3,4)\}} \cap l_{\{(1,3),(2,4)\}}$};
    \node at (2.5,-2) {$l_{\{(1,3),(2,4)\}} \cap l_{\{(1,4),(2,3)\}}$};
    \draw[->] (-5,-0.3) to (-2.6,-1.7);
    \draw[->] (0,-0.3) to (-2.4,-1.7);
    \draw[->] (0,-0.3) to (2.4,-1.7);
    \draw[->] (5,-0.3) to (2.6,-1.7);
\end{tikzpicture}.
\end{center}

The action of the generator of $\mbZ_4$ swaps $l_{\{(1,2),(3,4)\}}\leftrightarrow l_{\{(1,4),(2,3)\}}$ and preserves $l_{\{(1,3), (2,4)\}}$.
We have $h(w)=1$ for all $w$ in all vertices of $DG(4)$.
\end{example}

\begin{example}\label{ex:n3degenracy_graph}
   In case of $n = 3$, the number of vertices is $57$. We identify vertices in the orbit of action of $\mbZ_6$ and show the resulting graph.
   
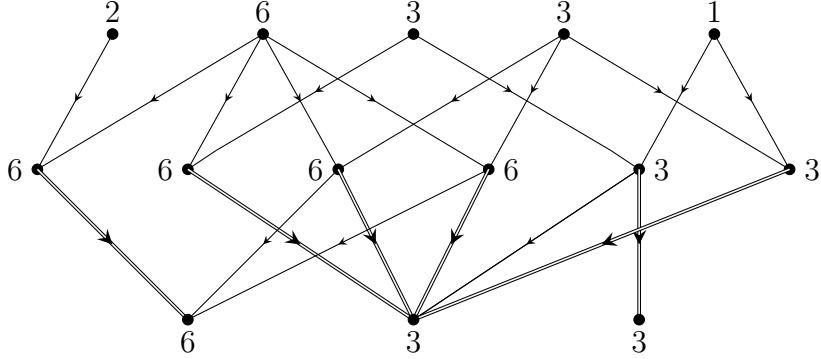
\begin{center}
\begin{tikzpicture}[decoration={markings, 
	mark= at position 0.5 with {\arrow{stealth}}}] 
    
    \node at (-4,0.1) {$2$};
    \filldraw [black] (-4,-0.2) circle (2pt);
    \draw[postaction={decorate}] (-4,-0.2) to (-5,-2);
    
    \node at (-2,0.1) {$6$};
    \filldraw [black] (-2,-0.2) circle (2pt);
    \draw[postaction={decorate}] (-2,-0.2) to  (-5,-2);
    \draw[postaction={decorate}] (-2,-0.2) to  (-3,-2);
    \draw[postaction={decorate}] (-2,-0.2) to  (-1,-2);
    \draw[postaction={decorate}] (-2,-0.2) to  (1,-2);
    
    \node at (0,0.1) {$3$};
    \filldraw [black] (0,-0.2) circle (2pt);
    \draw[postaction={decorate}] (0,-0.2) to  (-2.625,-1.75);
    \draw (-2.625,-1.75) to  (-3,-2);
    \draw[postaction={decorate}] (0,-0.2) to  (2.625,-1.75);
    \draw (2.625,-1.75) to  (3,-2);

    \node at (2,0.1) {$3$};
    \filldraw [black] (2,-0.2) circle (2pt);
    \draw[postaction={decorate}] (2,-0.2) to  (-1,-2);
    \draw[postaction={decorate}] (2,-0.2) to  (1.125,-1.775);
    \draw (1.125,-1.775) to (1,-2);
    \draw[postaction={decorate}] (2,-0.2) to  (4.625,-1.775);
    \draw (4.625,-1.775) to  (5,-2);
    
    \node at (4,0.1) {$1$};
    \filldraw [black] (4,-0.2) circle (2pt);
    \draw[postaction={decorate}] (4,-0.2) to  (3.125,-1.775);
    \draw (3.125, -1.775) to (3, -2);
    \draw[postaction={decorate}] (4,-0.2) to  (5,-2);
    
    \node at (-5.3,-2) {$6$};
    \filldraw [black] (-5,-2) circle (2pt);
    \draw[postaction={decorate}, double] (-5,-2) to  (-3,-4);
    
    \node at (-3.3,-2) {$6$};
    \filldraw [black] (-3,-2) circle (2pt);
    \draw[postaction={decorate}, double] (-3,-2) to  (0,-4);
    
    \node at (-1.3,-2) {$6$};
    \filldraw [black] (-1,-2) circle (2pt);
    \draw[postaction={decorate}] (-1,-2) to  (-3,-4);
    \draw[postaction={decorate},double] (-1,-2) to  (0,-4);
    
    \node at (1.3,-2) {$6$};
    \filldraw [black] (1,-2) circle (2pt);
    \draw[postaction={decorate}] (1,-2) to  (-3,-4);
    \draw[postaction={decorate},double] (1,-2) to  (0,-4);
    
    \node at (3.3,-2) {$3$};
    \filldraw [black] (3,-2) circle (2pt);
    \draw[postaction={decorate}] (3,-2) to  (0,-4);
    \draw (3,-2) to  (0,-4);
    \draw[postaction={decorate},double] (3,-2) to  (3,-4);
    
    \node at (5.3,-2) {$3$};
    \filldraw [black] (5,-2) circle (2pt);
    \draw[postaction={decorate},double] (5,-2) to  (0,-4);
    
    \node at (-3,-4.3) {$6$};
    \filldraw [black] (-3,-4) circle (2pt);
    \node at (0,-4.3) {$3$};
    \filldraw [black] (0,-4) circle (2pt);
    \node at (3,-4.3) {$3$};
    \filldraw [black] (3,-4) circle (2pt);

\end{tikzpicture}
\end{center}
 We write next to each vertex the cardinality of the  corresponding $\mbZ_{6}$ orbit. 
We draw an edge between two vertices $A$ and $B$ with multiplicity $m$ if each element in the orbit $B$  in the degeneracy graph $DG(6)$ is connected to $m$ vertices in the orbit $A$. For $n=3$ we have only edges of multiplicities $1$ and $2$.

Note that the vertices in the top row are planes of dimension $3$, in the middle row of dimension $2$ and the bottom row of dimension $1$.

We have $h(w)=1$ for generic points in all vertices except for the rightmost one in the bottom row. This is the orbit which contains the one-dimensional vertex $$l_{\{(1,2),(3,6),(4,5)\}}\cap l_{\{(1,4),(2,5),(3,6)\}}\cap l_{\{(1,6),(2,5),(3,4)\}} =(\al, 2 + \al, \al, 2 + \al, 4 + \al, 2 + \al),$$cf. Example \ref{ex:im_close2}. For any $w$ of such a form, $h(w)=2$. In particular, $w$ is a word of length $6$ with $h(w)=2$ if and only if $w$ is obtained from the word $020242$ by slides and shifts. 

\end{example}

Note that all hyperplanes $l_C$ are invariant with respect to the overall shift $(\al_1,\dots, \al_{2n}) \mapsto (\al_1 +b, \al_2+b,\dots, \al_{2n}+b)$, therefore $DG(2n)$ does not have vertices of dimension $0$. Moreover, all vertices of dimension $1$ have the form $w + (\al,\dots, \al)$, where $w$ is some particular word and $\al\in\mbC$. The vertices of dimension 1 have no outgoing arrows. 

We have $(2n-1)!!$ vertices of dimension $n$. Obviously, such vertices do not have incoming arrows. We have $2{2n \choose 4} (2n-5)!!$ vertices of dimension $n-1$.

We expect the following properties of the degeneracy graphs.

\begin{conj}\label{conj:deg_graph_conn} 
Let $n\in\mbZ_{>0}$.  
\begin{enumerate}
\item\label{conj:deg_graph_conn_1} The degeneracy graph $DG(2n)$ is connected (as a non-direct graph). 
\item\label{conj:deg_graph_conn_2} If a vertex has no incoming arrows then it has dimension $n$ and corresponds to $l_C$ for some $C\in\uconf(2n)$.
\end{enumerate}
\end{conj}
In addition, we will prove that if a vertex has no outgoing arrows then it has dimension $1$. 

\medskip 

\begin{lemma} Part \eqref{conj:deg_graph_conn_2} of Conjecture \ref{conj:deg_graph_conn} implies part \eqref{conj:deg_graph_conn_1} of Conjecture \ref{conj:deg_graph_conn}.
\end{lemma}
\begin{proof}
    By part \eqref{conj:deg_graph_conn_2} every vertex is connected (ignoring the direction of vertices) to a vertex of dimension $n$. 

    We show that any two vertices of dimension $n$ are connected even if all vertices of dimensions less than $n-1$ are deleted. Indeed, the statement is true for $DG(4)$, see Example \ref{ex:n2degeneracy_graph}. Therefore, any two $n$-dimensional vertices $l_C$ and $l_C'$ which differ only by two arcs are connected. Then by induction on $n$, every vertex is connected to $l_{C_0}$ with $C_0=\{(1,2),(3,4),\dots,(2n-1,2n)\}$. Namely, $C$ either contains an arc $(1,2)$ or arcs $(1,a)$, $(2,b)$. In the latter case, $C$ is connected to a new configuration which contains $(1,2)$ and $(\min (a,b),\max(a,b))$ and the other arcs are the same as in $C$. 
\end{proof}

We describe words $w$ which appear in vertices of the degeneracy graph of the dimension $k$.

Given a word $w\in\mbC^{2n}$ and $i, j \in \{1,\dots, 2n\}$ we write $i\bar{\sim}_w j$ if $i=j$ or there exists an arc configuration $C\in \conf(w)$ such that $(i,j)\in C$ or $(j,i)\in C$. For any word $w$, clearly, the relation $\bar{\sim}_w$ is symmetric and reflective. We denote by $\sim_w$ the transitive closure of  $\bar{\sim}_w$. Then $\sim_w$ is an equivalence relation. 

\begin{defi}\label{def:conf_connected}
We call the equivalence classes with respect to this relation by conf-connected components of $w$. 
We call a word $w$ conf-connected if $w$ has only one conf-connected component.
\end{defi}

\begin{lemma}\label{lemma:dg_dims}
    The dimension of a vertex of $DG(2n)$ equals the number of conf-connected components of a generic word in this vertex.
\end{lemma}
\begin{proof}
    Each vertex is 
    \begin{equation}\label{eq:dg_dims}
    \bigcap_{C\in S} l_C = \bigcap_{C\in S}\bigcap_{(i,j)\in C} \{\alpha_i - \alpha_j = 2\}.
    \end{equation}    For each pair $i\sim_{w}j$ there is a chain of equivalences $i = i_0\bar{\sim}_w i_1\bar{\sim}_w \dots \bar{\sim}_w i_k = j$ and for each $i_l\bar{\sim}_w i_{l+1}$ there is a relation $\alpha_{i_l}-\alpha_{i_{l+1}} = \pm 2$. Summing them we get $\alpha_i - \alpha_j = \textit{const}$, hence the dimension does not exceed the number of equivalence classes. 

    Conversely, let $\{1,\dots,2n\}=I_1\sqcup\dots \sqcup I_k$, where $I_j$ are conf-connected components of the word $w$. Then for any $\nu_1,\dots,\nu_k$, the shift $\alpha_{i}\mapsto \nu_j+\alpha_i$ whenever $i\in I_j$ preserves the set \eqref{eq:dg_dims}.
\end{proof}

In particular, Lemma \ref{lemma:dg_dims} implies that the vertices of $DG(2n)$ of dimension $1$ are precisely conf-connected words of length $2n$ considered modulo shifts.

Consider a vertex $A$ of degeneracy graph $DG(2n)$ of dimension $k$.

Part \eqref{conj:deg_graph_conn_2} of Conjecture \ref{conj:deg_graph_conn} implies that $A$ is given by intersection of exactly $n-k$ planes of the form $l_C$. 

Note that $A$ may be given by intersection of less than $n-k$ planes. 
For example, we have a $1$-dimensional vertex of $DG(6)$ written as $$l_{\{(1,2),(3,6),(4,5)\}}\cap l_{\{(1,4),(2,5),(3,6)\}}\cap l_{\{(1,6),(2,5),(3,4)\}} = l_{\{(1,2),(3,6),(4,5)\}}\cap l_{\{(1,6),(2,5),(3,4)\}}.$$

Let $S(A)$ be the maximal subset of $\uconf(2n)$ such that  $A=\underset{C\in S}{\cap}l_C$. In other words, $S(A) = \{C\in \uconf(2n)\ |\ l_C \supset A\}$.  If part \eqref{conj:deg_graph_conn_2} Conjecture \ref{conj:deg_graph_conn} is true, then we have $|S(A)|\geq n-k$. 

A word $w$ belongs to $A$ if and only if $S(A) \subset \conf(w)$. We call a word $w$  $A$-generic if $\conf(w)=S(A)$. In other words, a word $w$ is $A$-generic if $w \in A$ and for any $C \in \conf(w)$, we have $l_C \supset A$.

Every $A$-generic word $w$ is in $A$, and $A$-generic words are dense in $A$. By Lemma \ref{lemma:dg_dims}, an $A$-generic word has exactly $k$ conf-connected components. 

Let $w$ be $A$-generic. Let $I_1\sqcup \dots  \sqcup I_k=\{1,\dots,2n\}$ be the conf-connected components of $w$. Note that the partition $I_1\sqcup \dots  \sqcup I_k=\{1,\dots,2n\}$ depends on $\conf(w) = S(A)$ only and therefore does not depend on a choice of $A$-generic word. Note that cardinalities of $I_i$ are even.

A word $\bar w\in A$ if and only if there exist $(\nu_1,\dots,\nu_k)\in\mbC^k$ such that $\bar w_i-w_i=\nu_j$ whenever $i\in I_j$.

Let $w^{(1)},\dots, w^{(k)}$ be the subwords of $w$ related to the conf-connected components. The subword  $w^{(i)}$ consists of letters $w_j$, $j\in I_i$. Then each $w^{(i)}$ is a conf-connected word. We have a natural identification
$\conf(w^{(1)})\times \dots\times \conf(w^{(k)})= \conf(w)$.

Every word $w$ is $A$-generic for a unique vertex $A$ and dimension of $A$ equals the number of conf-connected components of $w$. Namely, such vertex $A$ is given by $A = \underset{C\in \conf(w)}{\cap}l_C$.

\begin{cor}\label{cor:conf_conn_determ}
    For conf-connected words $w_1, w_2$ if $\conf(w_1) = \conf(w_2)$, then $w_1$ and  $w_2$ differ by a shift. \qed
\end{cor}

\begin{lemma}
    If a vertex of degeneracy graph $DG(2n)$ has no outgoing arrows then the vertex has dimension $1$.
\end{lemma}
\begin{proof}
    Let a vertex $A=\mathop{\cap}\limits_{C\in S}l_C$ be of dimension more than $1$. Then $A$ has at least two conf-connected components $I\sqcup J\subset\{1,\dots,2n \}$. Let $C$ be any configuration in $S$. Let $(i_1,i_2),(j_1,j_2)\in C$ such that $i_1\in I$, $j_1\in J$. Then $i_2\in I$, and $j_2\in J$. Let $C'\in\uconf(2n)$ be the configuration which has the same arcs as $C$ except that the arcs $(i_1,i_2),(j_1,j_2)$ are replaced with the arcs $(\min(i_1,j_1), \max(i_1,j_1)),(\min(i_2,j_2), \max(i_2,j_2))$. 

    Then the degeneracy graph has an arrow from the vertex $A$ to the vertex $B=l_{C'}\cap\big(\underset{C\in S}{\cap}l_C\big)$. 
\end{proof}

The following combinatorial statement seems to be a way to prove Conjecture \ref{conj:deg_graph_conn}.

\begin{conj}\label{conj:vertex_over}
 Let $w= (\al_1,\dots,\al_{2n})$ be a conf-connected word. Let $a$ be the minimal letter of $w$. Let $i = \max(I_w(a)),\; j= \max(I_w(a+2))$ be the rightmost positions of $a$ and $a+2$ respectively. Then the word obtained from the word $w$ by removing $\al_i$ and $\al_j$, 
 $$\mathring{w}_{\{(i,j)\}} = (\al_1,\dots, \al_{i-1}, \al_{i+1},\dots, \al_{j-1}, \al_{j+1},\dots, \al_{2n}),$$
 is conf-connected.
\end{conj}

We checked Conjecture \ref{conj:vertex_over} numerically for all conf-connected words up to the length $10$.

\begin{lemma}
    Conjecture \ref{conj:vertex_over} implies part \eqref{conj:deg_graph_conn_2} of Conjecture \ref{conj:deg_graph_conn}.
\end{lemma}
\begin{proof}
    Let $A$ be a vertex of $DG(2n)$ of dimension $k < n$. We need to prove that there is a vertex $B \supset A$ of dimension $k+1$. Let $w = (a_1,\dots, a_{2n})$ be an $A$-generic word and  let $I_1 \sqcup I_2 \sqcup \dots \sqcup I_k=\{1,\dots,2n\}$ be conf-connected components of $w$. Since $\dim A<n$, there exists $i$ such that $|I_i|>2$. By Conjecture \ref{conj:vertex_over}, there exists $j_1,j_2\in I_i$ such that $a_{j_2}=a_{j_1}+2$, $j_1<j_2$ and the word consisting of letters $a_s$ with $s\in I_i\setminus\{j_1,j_2\}$ is conf-connected. 

Let $v$ be obtained from $w$ by replacing $a_{j_1}$ and $a_{j_2}$ by  $a_{j_1}+\nu$ and $a_{j_2}+\nu$ where $\nu$ is a generic complex number. Then $\conf(v)=\{C\in \conf(w), \ (j_1,j_2) \in C\}\subset\conf(w)$. Set 
    $$B=\mathop{\bigcap}_{C\in \conf(v)}l_C.$$  
    Clearly, $A\subset B$. The word $v$ is $B$-generic and has $k+1$ connected components.
\end{proof}

\subsection{Degeneracy graphs and $h(w)$.}

An $A$-generic word $w = (a_1,\dots, a_n)$ is called admissible if $a_i-a_j\not\in \{-2,0,2\}$ whenever $i,j$ belong to different conf-connected components of $w$. Admissible $A$-generic words are still dense in $A$. 

For example, let $A=l_{\{(1,4), (2,3)\}}$. Then $w=0242$ is $A$-generic but not admissible $A$-generic.

\begin{prop}\label{vertex-dim prop}
Let $v,w$ be two $A$-generic admissible words. Then $h(v)=h(w)$. 
\end{prop} 
\begin{proof}
The conf-connected components of $v$ and $w$ are the same. Let $\{v^{(i)}\}$ and $\{w^{(i)}\}$ be the subwords of $v$ and $w$ corresponding to conf-connected components. Then, after a suitable reordering, the subwords $v^{(i)}$ and $w^{(i)}$ are shifts of each other. In particular, $h(v^{(i)})=h(w^{(i)})$ for all $i$. 

Repeating the proof of Proposition \ref{prop:no_gaps}, we have $h(v)=\mathop{\prod}\limits_{i=1}^k h(v^{(i)})$ and $h(w)=\mathop{\prod}\limits_{i=1}^k h(w^{(i)})$. 
\end{proof}

In Section \ref{subsec:fact_by_cont} we continue the discussion of relations of $h(w)$ to $\mathop{\prod}\limits_{i=1}^k h(w^{(i)})$.

We expect that the condition of being admissible in Proposition \ref{vertex-dim prop} can be dropped. 

\begin{conj} \label{conj:drop_adm}
Let $v,w$ be two $A$-generic words. Then $h(v)=h(w)$. In other words, if $\conf(v)=\conf(w)$ then $h(v)=h(w)$. 
\end{conj} 
If Conjecture \ref{conj:drop_adm} is true, then it is sufficient to compute $h(w)$ for conf-connected words. 

\medskip 

We assign to each vertex $A$ of the degeneracy graph a natural number $h(A)$ equal to  $h(w)$, where  $w$ is any $A$-generic admissible word. By Proposition \ref{vertex-dim prop} this number does not depend on the choice of $A$-generic admissible word $w$.

In particular, $h(l_C)=1$ for all vertices  of dimension $n$. The main question we study in this paper is the values $h(A)$ for one-dimensional vertices $A$. 

By Lemma \ref{lemma:slide_isom} the numbers $h(A)$ are preserved by the action of $\mbZ_{2n}$. 

For $n=2$ all numbers $h(A)$ are $1$. For $n=3$, the numbers $h(A)$ are described in Example \ref{ex:n3degenracy_graph}.
For $n=4$, there are $20$ orbits of one-dimensional vertices of $DG(8)$  and possible values of $h(w)$ are from $1$ to $3$. For $n=5$, a computer computation shows that the number of orbits of one-dimensional vertices of $DG(10)$ is $260$ and possible values of $h(w)$ are from $1$ to $6$. We list the corresponding connected words (considered up to certain symmetries) and the numbers $h(A)$ in Appendix \ref{app:tables}.

The numbers $h(A)$ are compatible with the degeneracy graph structure as follows.
\begin{prop}
    Let $A_1, A_2$ be two vertices of $DG(2n)$. If $A_1$ is connected to $A_2$, then $h(A_2) \geq h(A_1)$.
\end{prop}
\begin{proof}
Let $h(A_1)=d$. Recall that $A_1$-generic admissible points $w$ are dense in $A_1$. 
We have a rational map from $A_1$ to $\mathrm{Gr}(d, \mbC^{2n})$, which maps a $A_1$-generic admissible point $w$ on the plane to the space $H(w)$ of $\ell$-singular vectors in $w$ of weight zero. 

Let $A_2$ be a vertex such that there is an edge from $A_1$ to $A_2$. Since $A_2\subset A_1$, then  every $A_2$ generic admissible point $v$ is a limit point of $A_1$-generic admissible points $w$. 
Let $w(t)\in A_1$, $t\in[0,1)$ be a continuous curve that $w(t)$ is $A_1$-generic admissible for $t\in (0,1)$ and $w(0)=v$. 
Since the Grassmannian of $d$-planes is compact, the curve $H(w(t))$ has a limit point $H_0$ as $t\to 0$. Hence $w(0)=v$ has at least $d$-dimensional space $H_0$ of $\ell$-singular vectors of weight $0$. 
\end{proof}

\section{Classification of words.}\label{sec:classification}
In general, it is not easy to say whether two words are isomorphic as $\Uqa$ modules or not. In this section we give several results related to this problem.
\subsection{Separating invariant.}\label{subsec:sep_inv}  
In Section \ref{subsec:support} we gave a description of the set of words $w$ for which $h(w) \neq 0$. In this section we apply this result to show that some tensor products of evaluation modules are non-isomorphic.

We start with a simple observation.
\begin{lemma}
Let two words $w_1$ and $w_2$ be isomorphic. Then $w_2$ is obtained from $w_1$ by a permutation of letters. In other words, $\cont(w_1)=\cont(w_2)$.
\end{lemma}
\begin{proof}
The lemma follows from the comparison of $q$-characters of $w_1$ and $w_2$.    
\end{proof}
We give a combinatorial way to show that some words are non-isomorphic.
\begin{prop}\label{prop:sep_inv}
Let $w_1,w_2\in W_n$ be two words such that $\cont(w_1)=\cont(w_2)$. Let $\tilde w$ be a word such that $\conf(\tilde ww_1)\neq \varnothing$ and  $\conf(\tilde ww_2)= \varnothing$. Then $w_1$ and $w_2$ are non-isomorphic $\Uqa$ modules.
\end{prop} 
\begin{proof}

If $w_1 \cong w_2$, then $\tilde{w}w_1\cong \tilde{w}w_2$, which implies $h(\tilde{w}w_1) = h(\tilde{w}w_2)$. Therefore, the statement follows from Theorem \ref{thm:supp}.
\end{proof}
To use Proposition \ref{prop:sep_inv} we have the following convenient lemma. 

\begin{lemma}Let $w_1,w_2\in W_n$ be two words such that $\cont(w_1)=\cont(w_2)$. Then the following three statements are equivalent.
\begin{itemize}
    \item There exists a word  $\tilde w$ such that $\conf(\tilde ww_1)\neq \varnothing$ and  $\conf(\tilde ww_2)= \varnothing$.
    \item There exists a word  $\tilde w$ such that $\conf(w_1\tilde w)\neq \varnothing$ and  $\conf(w_2\tilde w)= \varnothing$.
    \item  There exists a non-empty subword $\bar{w}$ of $w_1$  such that $\conf(\bar w)\neq \varnothing$ and such that for any subword $\hat{w}$ of $w_2$ satisfying $\cont(\hat{w}) = \cont(\bar{w})$, we have $\conf(\hat{w}) = \varnothing$.
\end{itemize}
\end{lemma} 
\begin{proof}

The equivalence of the first and the second statements simply follows by applying slides.

Assume that we have $\tilde w$ as in the first statement. Let $C\in \conf(\tilde ww_1)$. We reduce $\tilde w$ by removing all arcs in $C$ which start and end in $\tilde w$. Then all arcs in $C$ which start in $\tilde w$ end in $w_1$. Let $\bar w$ be the subword of $w_1$ obtained by removing all ends of the arcs in $C$ which start at $\tilde w$. Then $\bar w$ satisfies the third statement. Indeed, clearly $\conf(\bar w)\neq \varnothing$ as an arc configuration is obtained from $C$ by removing all arcs starting in $\tilde w$. Suppose the word $w_2$ has a subword $\hat w$ such that $\cont(\hat{w}) = \cont(\bar{w})$ and  $\conf(\hat w)\neq \varnothing$. Then $\conf(\tilde ww_2)\neq \varnothing$ since an arc configuration for the word $\tilde ww_2$ can be constructed by connecting the letters of $\tilde  w$ to the letters in $w_2$ which are not in $\hat w$ as in $C$. This is a contradiction.

Assume that we have $\bar w$ as in the third statement.
 Let $\tilde{w}$ be the word obtained by reverse sorting of the complement $\cont(w_1)\backslash \cont(\bar{w})$ and common shift of all letters by $-2$. Then $\tilde w$ satisfies the first statement. Indeed, we obtain an arc configuration of $\tilde{w}w_1$ by connecting letters of $\tilde {w}$ with letters in $w_1$ which are not in $\bar{w}$ and connecting letters in $\bar{w}$ as in an arc configuration of  $\bar{w}$. Moreover, suppose there is an arc configuration $C$ of $\tilde w w_2$. By construction of the word $\tilde{w}$ all arcs in $C$ which  start in $\tilde w$ must end in $w_2$. Let $\hat w$ be obtained from $w_2$ by removing all ends of arcs in $C$ which start in $\tilde w$. Then $\cont(\hat w)=\cont(\bar w)$ and $\cont(\hat w)\neq \varnothing$ . This is a contradiction.
\end{proof}

We give an example of non-isomorphic words which are not distinguished by Proposition \ref{prop:sep_inv}.
\begin{example}
    Consider words $w_1 = 020$ and $w_2 = 002$. As we will see in Section \ref{subsec:comm}, $020 \cong 0 \oplus 0[0,2]$. However, it is easy to see $\Hom_{\Uqa}(0[0,2], 002) \cong \Hom_{\Uqa}(00[0,2], 00) = 0$. Therefore, the
    words $020$ and $002$ are non-isomorphic. However, the only subword of $w_1$ containing a trivial submodule is $02$, $w_2$ contains such a subword. As well, the only subword of $w_2$ containing a trivial submodule is  of the form $02$ and $w_2$ contains such a subword.
\end{example}

The following example illustrates the difficulty of comparing words as $\Uqa$-modules.

For a word $w$ and a letter $a$, we call the maximal subword of $w$ consisting of letters $a$ and $a+2$ the $a$-skeleton of $w$.

Recall that we use the notation $a^n$ for the word $(a,\dots, a)$, where $a$ is repeated $n$ times.
\begin{example}
    Consider words $w_1 = 20422$ and $w_2 = 22042$. The $0$-skeletons of $w_1$ and $w_2$ are $2022$ and $2202$ respectively. In Section \ref{subsec:comm} 
    it is shown that $2022 \cong 2202 \cong 2^2[0,2] \oplus 2^2$. Similarly, the $2$ skeletons $2422$ and $2242$ are isomorphic. Thus, all $a$-skeletons of words $w_1$ and $w_2$ are isomorphic. However, words $w_1$ and $w_2$ are non-isomorphic. Indeed, $|\iconf(020w_1)| = 2$ and $|\sconf(020w_2)| = 1$, therefore $h(020w_1) \geq 2 > 1 = h(020w_2)$, which implies that $020w_1\ncong020w_2$, henceforth $w_1 \ncong w_2$.
\end{example}

\subsection{Class counting.}\label{subsec:class_count}
In this section we conjecture a classification of isomorphism classes for tensor products of two-dimensional evaluation modules with evaluation parameters in $\{0,2\}$.

For $I \subset \mbZ$ denote $W_{n}^{I} = \{w\in W_n\ | \ \mathrm{supp}(w) \subset I\}$. In this section we focus on $W_n^{\{0,2\}}$.

In Section \ref{subsec:comm} we conjecture that 

\begin{equation}\label{eq:020isoms}
    0^{k+1}2^k0^k \cong 0^k2^{k}0^{k+1}, \qquad 2^k0^k2^{k+1} \cong 2^{k+1}0^k2^k\end{equation} 
    as $\Uqa$ modules for all $k\in \mbZ_{\geq 0}$.
\begin{conj}\label{conj:02isoms}
    Two words in $W_{n}^{\{0,2\}}$ are isomorphic as $\Uqa$-modules if and only if they can be obtained from each other by a sequence of isomorphisms \eqref{eq:020isoms}.
\end{conj}
 We checked this conjecture by computer for $n \leq 5$. 

We compute expected number of classes of isomorphism of modules in $W_n^{\{0,2\}}$ assuming Conjecture \ref{conj:02isoms}.

Define an algebra
\begin{equation}
    \walg = \mathbb{Q}\left< x, y \right> /(x^ky^kx^{k+1} - x^{k+1}y^kx^k,\;y^kx^ky^{k+1} - y^{k+1}x^ky^{k}, \;  k\geq 0).
\end{equation}
The algebra $\walg$ has a grading such that $\deg(x) = \deg(y) = 1$. Let $(\walg)_n\subset \walg$ be the subspace of elements of degree $n$.

Recall that an overpartition is a partition where the first occurence of each integer may be overlined. Equivalently, an overpartition is a pair of a partition and a strict partition. The degree of an overparition is the total number of boxes of the underlying partitions.

\begin{prop}\label{prop:02isom_class}
The set 
        \begin{equation}\label{eq:02reps_class}
            \{x_1^{\alpha_1}x_2^{\alpha_2}\dots x_r^{\alpha_r}\ |\ \alpha_1 \leq \alpha_2 \leq \dots \leq \alpha_m > \alpha_{m+1} > \dots > \alpha_r,\;\;\al_i\in\mbZ_{>0},\ \ x_{i} \in \{x, y\},\; x_{i}\neq x_{i+1}\},
        \end{equation}
    is a basis of $\walg$.
    
There exists a degree preserving bijection of the set \eqref{eq:02reps_class} to the set of all overpartitions. In particular, the graded dimension of  $\walg$ is given by
    \begin{equation}\label{eq:02class_graded_dim}        
   \sum_{n=0}^\infty \dim (\walg)_n\, q^n= \prod_{i=1}^{\infty}\frac{1+q^i}{1-q^i}.
    \end{equation}    
\end{prop}
\begin{proof}
    Clearly, $\walg = \mathbb{Q} \bigoplus x\walg \bigoplus y\walg$.

We use the diamond lemma, see \cite{BERGMAN1978178}, to compute a basis of subalgebras $x\walg$ and $y\walg$. To apply the diamond lemma, one defines a suitable order on monomials, finds the largest monomial in each relation, and checks that if a monomial can be made smaller by using a relation in two different ways, the results can be matched by further applying the relations in a way so that on each step the resulting monomial is made smaller.

The monomials which can not be made smaller by applying a relation are called reduced monomials. The diamond lemma asserts that the reduced monomials form a basis.
    
We use a degree lexicographic order on monomials in $\mathbb{Q}\langle x,y\rangle$ as follows. Let $m_1,m_2$, where $m_i = x^{l_{i,1}}y^{l_{i,2}}\dots x^{l_{i,k_i}}$, $i=1,2$,  be two distinct monomials. We set $m_1>m_2$ if either $\deg(m_1)> \deg(m_2)$ or $\deg(m_1) = \deg(m_2)$ and $(l_{1,1},\dots , l_{1,k_1}) > (l_{2,1},\dots, l_{k_2})$ in lexicographic order. Here if $k_1 \neq k_2$ we add to sequences $l_{i,j}$ a sufficient number of zeroes.
    
    For any monomial $u$, if $m_1 > m_2$, then $um_1 > um_2$ and $m_1u>m_2u$. For any monomial $m$ the set of monomials smaller than $m$ is finite. These are the conditions required for the diamond lemma.

 First we consider the subalgebra $\walg^{x}=\mathbb{Q} \bigoplus x\walg$.

    Defining relations of $\walg^{x}$ are
    \begin{equation}\label{eq:rels_02x}
      x^{n+1}y^nx^{n} =x^{n}y^nx^{n+1},\qquad  xy^{n+k}x^ny^{n} =xy^{n+k-1}x^{n}y^{n+1},  \qquad n, k\in\mbZ_{\geq 0}.
    \end{equation}   
Here the monomial on the left is larger than the monomial on the right. 

Then overlaps of the larger monomials in the relations happen in the monomials
    $$x^{n+1}y^nx^{n+m}y^{k}x^{k},\qquad  x y^{n+l}x^{n}y^{n+m}x^{k}y^{k},\qquad x^{n+1}y^nx^ny^{m+l+1}x^{m+1}y^{m+1},$$
    where $l>0,\;n>0,\;m\geq 0,\;n+m>k>0.$
These ambiguities are resolved as follows. 

    For the first case we have
    
    \begin{tabular}{c c}
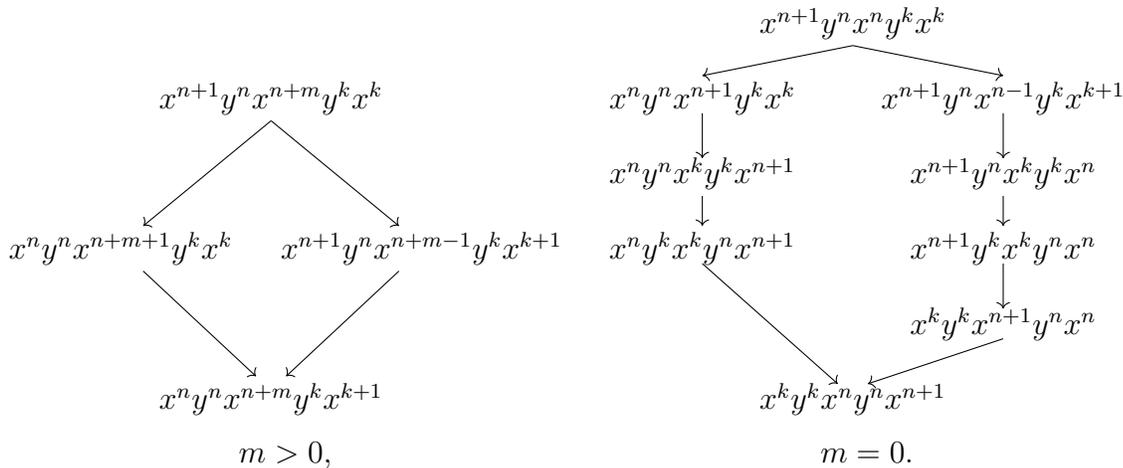
       \begin{tikzpicture}
    \node at (0,0) {$x^{n+1}y^nx^{n+m}y^{k}x^{k}$};
    \node at (-2,-2) {$x^{n}y^nx^{n+m+1}y^{k}x^{k}$};
    \node at (2,-2) {$x^{n+1}y^nx^{n+m-1}y^{k}x^{k+1}$};
    \node at (0,-4) {$x^{n}y^nx^{n+m}y^{k}x^{k+1}$};
    \draw[->] (0,-0.3) to (-1.7,-1.7);
    \draw[->] (0,-0.3) to (1.7,-1.7);
    \draw[->] (-1.7,-2.3) to (-0.2,-3.7);
    \draw[->] (1.7,-2.3) to (0.2,-3.7);
\end{tikzpicture}
         &
\begin{tikzpicture}
    \node at (0,0) {$x^{n+1}y^nx^{n}y^{k}x^{k}$};
    \node at (2,-1) {$x^{n+1}y^nx^{n-1}y^{k}x^{k+1}$};
    \node at (2,-2) {$x^{n+1}y^nx^{k}y^{k}x^{n}$};
    \node at (2,-3) {$x^{n+1}y^{k}x^{k}y^{n}x^{n}$};
    \node at (2,-4) {$x^{k}y^{k}x^{n+1}y^{n}x^{n}$};
    \node at (-2,-1) {$x^{n}y^nx^{n+1}y^{k}x^{k}$};
    \node at (-2,-2) {$x^{n}y^nx^{k}y^{k}x^{n+1}$};
    \node at (-2,-3) {$x^{n}y^{k}x^{k}y^{n}x^{n+1}$};
    \node at (0,-5) {$x^{k}y^{k}x^{n}y^{n}x^{n+1}$};

    \draw[->] (0,-0.3) to (-2,-0.7);
    \draw[->] (0,-0.3) to (2,-0.7);
    \draw[->] (-2,-1.2) to (-2, -1.8);
    \draw[->] (2,-1.2) to (2, -1.8);
    \draw[->] (-2,-2.3) to (-2,-2.7);
    \draw[->] (2,-2.3) to (2,-2.7);
    \draw[->] (-2, -3.2) to (-0.2,-4.8);
    \draw[->] (2, -3.2) to (2,-3.8);
    \draw[->] (2, -4.2) to (0.2,-4.8);
\end{tikzpicture}
         \\
    $m >0$, & $m = 0$.\\
    \end{tabular}

    The second case is similar. For the third case we have 

    
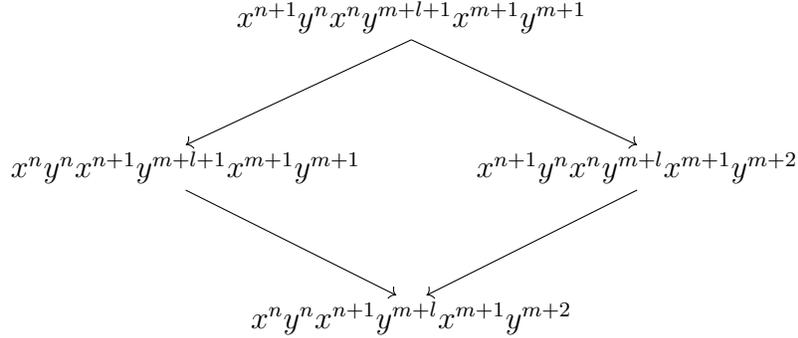
\begin{center}
       \begin{tikzpicture}
    \node at (0,0) {$x^{n+1}y^nx^ny^{m+l+1}x^{m+1}y^{m+1}$};
    \node at (-3,-2) {$x^{n}y^nx^{n+1}y^{m+l+1}x^{m+1}y^{m+1}$};
    \node at (3,-2) {$x^{n+1}y^nx^ny^{m+l}x^{m+1}y^{m+2}$};
    \node at (0,-4) {$x^{n}y^nx^{n+1}y^{m+l}x^{m+1}y^{m+2}$};
    \draw[->] (0,-0.3) to (-3,-1.7);
    \draw[->] (0,-0.3) to (3,-1.7);
    \draw[->] (-3,-2.3) to (-0.2,-3.7);
    \draw[->] (3,-2.3) to (0.2,-3.7);
        \end{tikzpicture}.
    \end{center}
    Therefore, by the diamond lemma the reduced monomials form a basis of $\walg^x$. In our case a monomial is reduced if it does not contain a factor equal to a larger monomial of relations \eqref{eq:rels_02x}.

    Consider a monomial $x_1^{\alpha_1}\dots x_k^{\alpha_k}\in\walg^x$ where $x_{2i+1} = x,\; x_{2i} = y,$ and $\al_i\in\mbZ_{>0}$. We claim that this monomial is reduced if and only if there exists an $m$ such that 
 $\alpha_1 \leq \dots \leq \alpha_m > \alpha_{m+1} > \dots > \alpha_{k}$. Indeed, a sequence $(\alpha_1,\dots, \alpha_k)$ is not of this form if and only if there is a subsequence $\alpha_{i-1} > \alpha_{i} \leq \alpha_{i+1}$. Equivalently, there is either a factor of the form $x^{\al_{i}+1}y^{\alpha_i}x^{\al_{i}}$ or a factor of the form $x y^{\al_{i}+k}y^{\alpha_i}x^{\al_{i}}$ with $k>0$. 
 These are exactly the conditions on a monomial to be reducible.

    Therefore, the set 
    \begin{equation*}
        \{x_1^{\alpha_1}x_2^{\alpha_2}\dots x_r^{\alpha_r}\ |\ \alpha_1 \leq \alpha_2 \leq \dots \leq \alpha_m > \alpha_{m+1} > \dots > \alpha_r,\ \ \al_i\in\mbZ_{>0}, \ \ x_{2i+1} = x,\; x_{2i} = y\}
    \end{equation*}
    is a basis of $\walg^x$. 
    
    Similarly, the set 
    \begin{equation*}
        \{x_1^{\alpha_1}x_2^{\alpha_2}\dots x_r^{\alpha_r} \ |\ \alpha_1 \leq \alpha_2 \leq \dots \leq \alpha_m > \alpha_{m+1} > \dots > \alpha_r,\ \ \al_i\in\mbZ_{>0},\ \ x_{2i+1} = y,\; x_{2i} = x\}
    \end{equation*}
    is a basis of $\walg^y=\mathbb{Q}\oplus y\walg$. 
    
    Therefore, $\walg$ has a basis labelled by \eqref{eq:02reps_class}.

We construct a bijection between set \eqref{eq:02reps_class} and the set of overpartition as follows.
    Given an overpartition $\bar{\lambda}$, let $\mu$ be the strict partition containing the overlined parts of $\bar{\lambda}$ and let $\lambda$ be the partition containing the parts of $\bar{\lambda}$ which are not overlined. Let the largest part of $\bar{\lambda}$ be overlined. Then we identify $\bar{\lambda}$ with an element of the set \eqref{eq:02reps_class} given by $\alpha_{m-i} = \lambda_{i}$, $\alpha_{m+j-1} = \mu_{j}$, and $x_m = x$. Here, $m-1$ is the number of parts of $\la$.
    Let the largest part of $\bar{\lambda}$ be not overlined. Then we we identify $\bar{\lambda}$ with an element of the set \eqref{eq:02reps_class} given by 
 $\alpha_{m-i+1} = \lambda_{i}$ ,  $\alpha_{m+j} = \mu_{j}$, and $x_m = y$. Here, $m$ is the number of parts of $\la$. This identification is clearly a bijection.

 The counting function for partitions is $\mathop{\prod}\limits_{i=1}^\infty \frac{1}{1-q^i}$. The counting function for strict partitions is $\mathop{\prod}\limits_{i=1}^\infty (1+q^i)$. Therefore, we have formula \eqref{eq:02class_graded_dim}.
\end{proof}

If Conjecture \ref{conj:02isoms} is true then by Proposition \ref{prop:02isom_class} isomorphism classes of $\Uqa$-modules in $ W_{n}^{\{0,2\}}$ are labelled by the set \eqref{eq:02reps_class}, where $x=0$, $y=2$ and $\mathop{\sum}\limits_{i=1}^l\al_i=n$. We expect that these words are linearly independent over $\mathbb{Q}$. By that we mean that two direct sums of words are isomorphic as $\Uqa$-modules if and only if after a suitable permutation the summands are isomorphic. That is we expect that the algebra of $\Uqa$-modules generated by $0$ and $2$ is isomorphic to $\walg$.

Similarly, one can consider the algebra of $\Uqa$-modules generated by modules $0$, $2$, and $4$. We have expected relations \eqref{eq:020isoms} and similar relations between $2$ and $4$. In addition, we have an obvious relation $04\cong40$. It would be interesting to find other relations or to prove that they do not exist. 

\section{Examples.}\label{sec:examples}

\subsection{Explicit computations of $h(w)$.}\label{subsec:exact_comps}
Although we cannot compute $h(w)$ in general, in many examples our methods do give precise answers. Here we give a few such calculations. We use the notation $H(\rV) = \Hom_{\Uqa}(\mbC, \rV)$ for any $\Uqa$-module.

Let $w\in \mbZ^{2n}$ be a word.

First, we compute $h(w)$ in the case of $\supp(w) = \{0,2\}$.
\begin{lemma}\label{lemma:supp02_hom}
Let $w\in W_{2n}^{\{0,2\}}$, then $h(w) = 0$  if $\conf(w) = \varnothing$ and $h(w) = 1$ otherwise.
\end{lemma}
\begin{proof}
The $q$-characters give the upper bound $h(w)\leq 1$. Thus, the lemma follows from Theorem \ref{thm:supp}.
\end{proof}

The next case is  $\supp(w) = \{0,2,4\}$. This case is already difficult.

In the case $\supp(w) = \{0,2,4\}$ there is a simple but useful lemma.
\begin{lemma}\label{lemma:024_simp}
    Let $\rV= \rW\otimes 024 \otimes \rU$, where $\rU, \rW$ are tensor products of modules from the set $\{0,2,4, [0,2], [2,4], [0,4]\}$. Then
    \begin{equation*}
        H( \rV) \cong H( \rW\otimes 0 \otimes \rU)\oplus H( \rW\otimes 4 \otimes \rU).
    \end{equation*}
\end{lemma}
\begin{proof}

    There is a short exact sequence
    \begin{equation}\label{eq:024_smp_pf_1}
        0\oplus 4 \hookrightarrow 024 \twoheadrightarrow [0,4].
    \end{equation}
    Indeed, we have submodules $\mbC \subset 02$ and $\mbC \subset 24$. Therefore, we have submodules $0\subset 024$ and $4\subset 024$. Since $0$ and $4$ are distinct and irreducible, we have $0\oplus 4 \subset 024$. The quotient $024/(0\oplus 4)\cong [0,4]$ by comparing the $q$-characters.
    
    Short exact sequence \eqref{eq:024_smp_pf_1} implies exact sequence
    \begin{equation*}
        H( \rW\otimes 0 \otimes \rU)\oplus H( \rW\otimes 4 \otimes \rU) \hookrightarrow H( \rV) \rightarrow H( \rW \otimes [0,4] \otimes \rU).
    \end{equation*}

    Since $[0,4] \otimes \rU \cong \rU \otimes [0,4]$ we have by Lemma \ref{lemma:hom_move} 
    $$
        H( \rW \otimes [0,4] \otimes \rU) \cong \Hom_{\Uqa}([-2,2], \rW \otimes  \rU).
    $$
   The last space is zero because $1_{-2}1_01_2$ is not an $\ell$-weight of $\rW\otimes \rV$.

Therefore, the lemma follows.
\end{proof}

In particular, we have an algorithm for the computation of $h(w)$ for $w\in W^{\{0,2,4\}}_{2n}$ when only one letter of $w$ is $4$.

Let $w\in \mbZ^{2n}$ be a word such that $\supp(w) = \{0,2,4\}$ and $\cont(w)$ contains only one letter $4$. 

We perform the following steps.

\begin{algorithm}\label{alg:one_four}\hfill
    \begin{enumerate}
        \item In this step we remove $0$s on the right of $4$. If there are no $0$ on the right of $4$, proceed to the next step.
        \begin{enumerate}
        \item If the last letter of $w$ is $0$, $h(w) = 0$.
        \item Otherwise, remove the rightmost $0$ together with the next letter (which is necessarily $2$). Start again.
        \end{enumerate}
        \item  Assume, there are no letters $0$ to the right of $4$.
        \begin{enumerate}
            \item If $4$ is the last letter, apply slide (see Lemma \ref{lemma:slide_isom}): $h(\bar w4)=h(0\bar w)$. Then $0\bar w\in W^{\{0,2\}}_{2n}$ contains only $0$s and $2$s and $h(w)$ is computed by Lemma \ref{lemma:supp02_hom}.
            \item If $4$ is the first letter, $h(w) = 0$.
            \item If the letter before $4$ is $0$, permute this $0$ with $4$. Start again.
            \item If $4$ is the second letter, remove it together with the preceding $2$. Then the resulting word contains only $0$s and $2$s and $h(w)$ is computed by Lemma \ref{lemma:supp02_hom}.
            \item Assume $4$ is neither the first, the second, nor the last letter,  the letter immediately before $4$ is $2$, there are no $0$ on the right of $4$.
            \begin{enumerate}
                    \item If the second letter to the left of $4$ is $2$ apply an isomorphism $2242 \cong 2422$ to move this letter to the right of $4$. Start again.
                    \item If the second letter to the left of $4$ is $0$, then we have $w=\bar w 024 \tilde w$. Then
                    $h(w) = h(\bar w0\tilde w) +  h(\bar w4\tilde w)$. Apply the algorithm to words $\bar w0\tilde w$ and $\bar w4\tilde w$.
            \end{enumerate}
        \end{enumerate}
    \end{enumerate}
\end{algorithm}
\begin{prop}
    Let $w\in \mbZ^{2n}$ be a word such that $\supp(w) = \{0,2,4\}$ and $\cont(w)$ contains only one letter $4$. Then $h(w)$ is computed by Algorithm \ref{alg:one_four}. 
\end{prop}
\begin{proof}
Step 1a is justified by Lemma \ref{lemma:right0_nohom}. For  step 1b, we use the short exact sequence $\mbC \hookrightarrow 02 \twoheadrightarrow [0,2]$. Thus, if $w=\bar w 02\tilde w$ then $h(\bar w\tilde w)\leq h(w)\leq h(\bar w\tilde w)+h(\bar w[0,2]\tilde w)$. By Lemma \ref{lemma:right0_nohom} $h(\bar w[0,2]\tilde w)=0$.

Steps 2b and 2d are justified in a similar way. For step 2e(ii) we use Lemma \ref{lemma:024_simp}. 

The algorithm stops after a finite number of steps since each step either decreases the length of the word or preserves the length of the word and increases the word in lexicographic order.
\end{proof}

\begin{example}\label{ex:one_four_ach} Applying the algorithm, we obtain
\begin{equation*}h((02)^{n-1-k}42(02)^k)=n-1-k, \qquad k=0,\dots,n-1.
\end{equation*}
It is not difficult to see that in this example $\iconf(w)=\sconf(w)=\cconf(w)$ and $|\iconf(w)| = n-1-k$.

\end{example}
Example \ref{ex:one_four_ach} provides a word with the largest $h(w)$ among words of length $2n$ with $\supp(w)=\{0,2,4\}$ and with one $4$. 
\begin{prop}
    For a word $w \in W^{\{0,2,4\}}_{2n}$ for $n >1$ containing at most one letter $4$ we have
    \begin{equation*}        h(w) \leq n-1.
    \end{equation*}    Moreover, this bound is exact for each $n$.
\end{prop}
\begin{proof}
    The proof is by induction in degree lexicographic order.
    The base $n=2$ is proved by computing $h(w)$ of all possible words. 
    
    Now we show the induction step. Given a word $w$ applying Algorithm \ref{alg:one_four} we obtain either $h(w) = 0$ or $h(w) = 1$, or $h(w) = h(\widehat{w})$, where $\widehat{w}$ is a word of smaller length, or $h(w) = h(w_1) + h(w_2)$, where $w_1 \in W^{\{0,2\}}_{2n-2}$ and $w_2$ has length $2n-2$. By Lemma \ref{lemma:supp02_hom}, $h(w_1) \leq 1$ and, by induction hypothesis,  $h(w_2) \leq n-2$. 
    
    The bound is exact due to Example \ref{ex:one_four_ach}. 
\end{proof}

We reproduce the result of the work \cite{BritoChari} in the case of $\Uqa$. Namely, Proposition \ref{prop:weyl_mod_homs} computes the dimensions of spaces of homomorphisms between two $\Uqa$ Weyl modules.

A word $w\in W_{n}$ is called a Weyl module if all letters in $w$ are non-decreasing. 

The highest weight vector of a Weyl module $w$ is cyclic, see  \cite{chari2002braid}. 

We say that  $w_1$ is compatible with  $w_2$ , if there is a word $\tilde{w}_1$ such that the word $w_2$ is a shuffle of words $w_1$ and $\tilde{w}_1$ and $\conf(\tilde{w}_1)\neq \varnothing$.

If a word $w_1$ is compatible with a word $w_2$ then for any permutation $\sigma$ of the letters of the word $w_1$ $\conf(^*\sigma(w_1)w_2) \neq \varnothing$. In particular, by Lemma \ref{lemma:hom_move}, $\Hom_{\Uqa}(\sigma(w_1),w_2) \neq 0$.
\begin{prop}\label{prop:weyl_mod_homs}
    Let $w_1 = a_1\dots a_n \in W_{n}, w_2 = b_1\dots b_m\in W_{m}$ be a pair of words such that $a_i\leq a_j$ and $b_i\leq b_j$. Then 
    $$
    \dim(\Hom_{\Uqa}(w_1, w_2))=\begin{cases} 1, & w_1\ \textrm{is\ compatible\  with}\  w_2,\\
    0, & \textrm{otherwise}.
    \end{cases}
    $$ 
    In particular, $h(w_2)\leq 1$.
\end{prop}
\begin{proof}

    By Lemma \ref{lemma:hom_move} we need to compute $h(^*w_{1}w_2)$. 
    Let $w =\hspace{0pt}^*w_{1}w_2 = (c_1, \dots, c_{2k})$ where $2k=m+n$. Then  $c_1\geq\dots  \geq c_{r-1}\geq  c_{r}\leq c_{r+1}\leq \dots\leq c_{2k}$, where $r=n$ or $r=n+1$.

Applying a shift we can assume that the smallest letter $c_r$ of $w$ is $0$. 

    Any arc which starts in $^*w_{1}$ must end at $w_2$. It follows that $\conf(w)\neq\varnothing$ if and only if $w_1$ is compatible with $w_2$. By Theorem \ref{thm:supp} it follows that $h(w)\neq 0$ if and only if $w_1$ is compatible with $w_2$.

It remains to prove that $h(w)\leq 1$.
    By Theorem \ref{thm:upper_bound} it is sufficient to show that $|\sconf(w)|\leq 1$. 

   We use induction on $k$.  The word $w$ has the form $w = \tilde{w}_102^l\tilde{w}_2$. Then $\tilde{w}_2$ is a word where the letters are non-decreasing and at least $4$. The word $\tilde{w}_1$ is a word where letters are non-increasing and non-negative. In case $l = 0$, $\sconf(w) = \conf(w) = \varnothing$. Otherwise, by the definition of steady arc configurations, 
    \begin{align*}
        |\sconf(w)| &= |\sconf(\tilde{w}_12^{l-1}\tilde{w}_2)| + |\sconf(\tilde{w}_1[0,2]2^{l-1}\tilde{w}_2)| = \\ &=|\sconf(\tilde{w}_12^{l-1}\tilde{w}_2)| + |\sconf(\tilde{w}_12^{l-1}[0,2]\tilde{w}_2)|.
    \end{align*}
    We have $\sconf(\tilde{w}_12^{l{-}1}[0{,}2]\tilde{w}_2) \subset\rconf(\tilde{w}_12^{l{-}1}[0{,}2]\tilde{w}_2) {=} \varnothing$.  Therefore, we conclude that $|\sconf(w)| {=}$ $|\sconf(\tilde{w}_12^{l-1}\tilde{w}_2)|$ and the latter is less than or equal to one by the induction hypothesis. 
\end{proof}
Let $$b_n :=\max_{w\in W_{2n}}\{h(w)\}$$ be the maximal dimension of homomorphisms from the trivial module to a module corresponding to a word of length $2n$. 

The first terms of the sequence $(b_n)_{n\geq 1}$ are $1,1,2,3,6$. 

 We study the growth of the sequence $(b_n)$.
 
\begin{prop}\label{prop:max_bound} The sequence $b_n$ is non-decreasing. For all $n \in \mbZ_{>0}$ we have 
\begin{equation*}\begin{aligned}
{2n \choose n} \leq\ &b_{2n+1}\ \leq C_{2n+1}= \frac{1}{2n+2}{4n+2 \choose 2n+1},\\
{2n \choose n-1} \leq\ &b_{2n} \quad\,  \leq C_{2n}\quad= \frac{1}{2n+1}{4n \choose 2n}.
\end{aligned}
\end{equation*}
In particular, for any $0<a<2$,  we have
\begin{equation*}
b_n \geq a^n
\end{equation*}
for sufficiently large $n$.
\end{prop}
\begin{proof}
Clearly $h(w)\leq 02w$ as $w$ is a submodule of $02w$. Therefore, $b_n$ is not decreasing.

The number $b_n$ is clearly bounded from above by the number of $\Uq$-singular vectors which is given by the
$n$-th Catalan number $C_n$.

The lower bounds follow from Example \ref{ex:bn_bound} below.
\end{proof}

\begin{example}\label{ex:bn_bound} We have a completely reducible module $020=0\oplus [0,2]0$, see Proposition \ref{prop:020_compl}. Inductively, this example is generalized to 
\begin{equation*}(02)^n0=\bigoplus_{k=0}^n([0,2]^{k}0)^{\oplus \binom{n}{k}}\hspace{0pt}.
\end{equation*}
It follows that the dimension of the homomorphisms between these modules are given by:
    \begin{multline*}
        \dim(\mathrm{Hom}_{\Uqa}((02)^l0, (02)^n0)) = \dim\left(\bigoplus_{k=0}^n\bigoplus_{m=0}^l \mathrm{Hom}_{\Uqa}\left(([0,2]^m 0)^{\oplus \binom{l}{m}},([0,2]^k 0)^{\oplus \binom{n}{k}}\right)\right) = 
        \\ = \dim\left(\bigoplus_{k=0}^{\min(n,l)} \mathrm{Mat}_{\binom{n}{k}\times \binom{l}{k}}(\mbC)\right) = \sum_{k=0}^{\min(n,l)}\binom{n}{k}\binom{l}{k} = \binom{n+l}{n}.
    \end{multline*}

On the other hand, 
\begin{equation*}\Hom_{\Uqa}((02)^n0, (02)^l0) \cong H( (02)^l0(24)^n2) = h(w_{l,n}),
\end{equation*}where $w_{l,n} = (02)^l0(24)^n2$.
Therefore, 
\begin{equation*}h((02)^l0(24)^n2)={n+l \choose n}.
\end{equation*}

In particular, the largest answer for $w_{l,s}$ with fixed length $4n+2=2(l+s+1)$ occurs for $l=s=n$ and for $w_{l,s}$ with fixed length $4n=2(s+l+1)$ for $l=n, s=n-1$ or for $l=n-1, s=n$. The answers are ${2n \choose n}$ and ${2n-1 \choose n}$.

Note that the first terms of $b_k$ have the form ${0\choose 0}, {1\choose 0}, {2\choose 1},{3\choose 2}, {4 \choose 2}$ and coincide with $h(w_{n,n})$ for $k=2n+2$ and with $h(w_{n-1,n})$ for $k=2n$.

It is not difficult to see that in this example 
\begin{equation*}
\iconf(w_{l,n})=\sconf((w_{l,n})=\cconf(w_{l,n}),\qquad |\iconf(w_{l,n})|= {n+l \choose n}.
\end{equation*}

We note that there is a similar family of examples: $h((20)^l2(42)^n4)={n+l\choose n}$. The two families of examples are related by anti-involution $\omega$, 
see Corollary \ref{cor:reverse_word}, combined with shift by $4$.
\end{example}

We generalize Proposition \ref{prop:max_bound}. Let $\rV$ be an irreducible $\Uqa$-module. Let $\lambda$ be the highest $\Uq$-weight of $\rV$. Let
 $$b_n^\rV :=\max_{w\in W_{2n+\lambda}}\{\dim(\Hom_{\Uqa}(\rV,w))\}=\max_{w\in W_{2n+\lambda}}\{\dim(\Hom_{\Uqa}(w,\rV^*))\}$$ be the maximal dimension of homomorphisms from the module $\rV$ to a module corresponding to a word of length $2n+\lambda$. 
 \begin{cor}
The sequence $b_n^{\rV}$ is non-decreasing. 
For any $0<a<2$, we have
\begin{equation*}a^n \leq b_n^{\rV}\leq  C_{\lceil n+\frac{\lambda}2\rceil}
\end{equation*}
for sufficiently large $n$.
\end{cor}
\begin{proof}
Clearly $\Hom_{\Uqa}({\rV},w)\subset  \Hom_{\Uqa}({\rV},02w)$ as $w$ is a submodule of $02w$. Therefore, $b_n^{\rV}$ is not decreasing. 

   Clearly $b_n^{\rV}\leq C_{\lceil n+\frac{\lambda}{2}\rceil}$ as the dimension of $\ell$-singular vectors of a given weight in $w$ is bounded by the dimension of singular vectors of that weight which in turn is bounded by the Catalan number. This gives an upper bound.

   The module $\rV$ is a tensor product of evaluation modules $[a_k,a_k+2m_k]$. The module $[a_k,a_k+2m_k]$ is a submodule of the word $(a_k,a_k+2,\dots,a_k+2m_k)$. Therefore, there exists a word $w$ of length $\lambda$ such that $\Hom_{\Uqa}({\rV},w)\neq 0$. Then  $$\dim(\Hom_{\Uqa}(\rV,(02)^n0(24)^n2w))\geq \dim(\Hom_{\Uqa}(\mbC,(02)^n0(24)^n2)) \geq {2n \choose n}.$$ 
   
   Therefore, $b^{\rV}_{2n+1} \geq {2n \choose n}.$ The lower bound follows. 
\end{proof}

\medskip 

Thus, the multiplicities of a simple module in a socle of a word can be almost of order $2^n=\sqrt{\dim(w)}$. The multiplicities of a simple module in a Jordan-Holder series of a word can be even larger.

Denote $b_{n,\textit{char}} =\max\limits_{w\in W_{2n}}\{h_{\textit{char}}(w)\}$  the maximal number of composition factors equal to $\mbC$ in a Jordan-Holder series of a word of length $2n$.

\begin{lemma}
    The sequence $b_{n, \textit{char}}$ is non-decreasing. For any $0 < a < 4$ for $n$ large enough we have
    $$ b_{n, \textit{char}} > a^n.$$
\end{lemma}
\begin{proof}
For a word $w$ assume that $a$ is the maximal letter of $w$, then $h_{\textit{char}}(w(a+4)(a+6)) = h_{\textit{char}}(w)$. This implies that the sequence $b_{n,\textit{char}}$ is non-decreasing.

For $n,k\in \mbZ_{>2}$ consider a word $v_{n,k} = 0^n2^{2n}\dots(2k-2)^{2n}(2k)^{n}$. Then the length of $v_{n,k}$ is $l(v_{n,k}) = 2nk$.
By Lemma \ref{lemma:upper_bound_qchar} we have $$h_{\textit{char}}(v_{n,k}) = \binom{2n}{n}^{k-1} = \frac{4^{n(k-1)}}{(\pi n)^{\frac{k-1}{2}}}\Big(1 + \bar{O}\Big(\frac{1}{n}\Big)\Big).$$
Then it is sufficient to show that
$$n(k-1) - \frac{(k-1)}{2}\log_4(\pi n)  + \bar{O}\Big(\frac{1}{n}\Big) > nk\log_4(a),$$
for any $0<a<4$ and $n$ large enough. Indeed, since $\log_4(a) < 1$, there exists a large enough $k$ such that $\frac{k}{k-1}\log_{4}(a) < 1$ which implies the inequality for $n$ large enough.
\end{proof}

\subsection{Commutativity.}\label{subsec:comm}
We often use the following decomposition. 
\begin{prop}\label{prop:020_compl}
    There is an isomorphism of $\Uqa$-modules
    \begin{equation}
        020 \cong 0\oplus[0,2]0 ,\;\;202 \cong 2\oplus[0,2]2. 
    \end{equation}
\end{prop}
\begin{proof}
By Proposition \ref{prop:two_strings_prod}, there are submodules $\mbC \subset 02$ and $[0,2] \subset 20$. Taking the tensor product of the first one with $0$ from the right and of the second one with $0$ from the left, we get $0 \subset 020$ and $[0,2] 0 \subset 020$. Since modules $0$ and $[0,2] 0$ are distinct and irreducible, they do not intersect. Therefore, $0\oplus [0,2] 0 \subset 020$. The proposition follows by comparing dimensions. The proof for $202$ is analogous.
\end{proof}
\begin{cor}\label{cor:020_commute} The modules $0$ and $020$ commute,
    $0020 \cong 0200$. Similarly, $2202 \cong 2022$.
\end{cor}
\begin{proof}
    The corollary follows from $0  [0,2] \cong [0,2] 0$ since the strings $0$ and $[0,2]$ are in general position.
\end{proof}

\begin{prop} The modules $0$ and $0^22^20^2$ commute,
 $00^22^20^2\cong 0^22^20^2  0$. Similarly, $22^20^22^2\cong 2^20^22^22$.
\end{prop}
\begin{proof}
    We check the proposition by a brute force computation.
\end{proof}

We expect that such a commutativity holds in general.

\begin{conj}\label{conj:commute}
    The modules $0$ and $0^n2^n0^n$ commute,
 $00^n2^n0^n\cong 0^n2^n0^n  0$. Similarly, $22^n0^n2^n\cong 2^n0^n2^n2$.
\end{conj}

We call isomorphisms $a^{n+1}(a+2)^na^n \cong a^n(a+2)^na^{n+1}$ and $a^{n+1}(a-2)^na^n \cong a^n(a-2)^na^{n+1}$ $n$-exchanges.

The modules  $0^n2^n0^n$ have an interesting structure. Although we cannot describe it completely, we give head and socle.
\begin{prop}\label{prop:soc_cosoc_struct} We have
    \begin{align*}
        \mathrm{soc}(0^n2^n0^n) \cong \mathrm{cosoc}(0^n2^n0^n)\cong  \bigoplus_{k=0}^{n}[0,2]^{k}0^n, \\ 
        \mathrm{soc}(2^n0^n2^n) \cong \mathrm{cosoc}(2^n0^n2^n)\cong  \bigoplus_{k=0}^{n}[0,2]^{k}2^n.
    \end{align*}
\end{prop}
\begin{proof}
    We start by computing the $q$-characters:
    \begin{equation}\label{the equation}
        \chi_q(0^n2^n0^n) = \chi_q(02)^n\chi_q(0^n) = (1 + \chi_q([0,2]))^n\chi_q(0^n) = \sum_{k=0}^{n}\binom{n}{k}\chi_q([0,2]^k0^n).
    \end{equation}
    Therefore, $\mathrm{soc}(0^n2^n0^n) = \bigoplus_{k=0}^{n}(0^n[0,2]^{k})^{\oplus m_k}$, for some $m_k\geq 0$. The multiplicity  $m_k$ is the dimension of the space
    $$\Hom_{\Uqa}([0,2]^k0^n, 0^n2^n0^n) \cong H( 0^n2^n0^n[2,4]^k2^n).$$
    Using short exact sequence $4 \hookrightarrow 0[2,4] \twoheadrightarrow [0,4]$ we obtain an exact sequence
    $$H( 0^n2^n0^{n-1}[2,4]^{k-1}42^n) \hookrightarrow H( 0^n2^n0^n[2,4]^k2^n) \rightarrow H( 0^n2^n0^{n-1}[2,4]^{k-1}[0,4]2^n).$$
    
    In addition, similarly to Lemma \ref{lemma:right0_nohom}, we have 
    \begin{multline*}
    H( 0^n2^n0^{n-1}[2,4]^{k-1}[0,4]2^n) \cong H( 0^n2^n0^{n-1}[2,4]^{k-1}2^n[0,4]) \cong \\ \cong \Hom_{\Uqa}([-2,2], 0^n2^n0^{n-1}[2,4]^{k-1}2^n) = 0,    
    \end{multline*}
   since the monomial $1_{-2}1_01_2$ is clearly not contained in the $q$-character of $0^n2^n0^{n-1}[2,4]^k2^n)$. 
   
    Therefore
    $$H( 0^n2^n0^n[2,4]^k2^n)\cong H( 0^n2^n0^{n-1}[2,4]^{k-1}42^n).$$
    Applying the same argument repeatingly we get
    $$H( 0^n2^n0^n[2,4]^k2^n) \cong H( 0^n2^n0^{n-k}4^k2^n) \cong H( 0^n2^n4^k0^{n-k}2^n).$$ Using short exact sequence $\mbC \hookrightarrow 02 \twoheadrightarrow [0,2]$ in the same fashion, we obtain
    $$
    H( 0^n2^n4^k0^{n-k}2^n) \cong H( 0^n2^n4^k2^k).$$

    Applying slides, see Lemma \ref{lemma:slide_isom} we move $0^n$ to the right and then  shift parameters, see Lemma \ref{lemma:shift}, 
    $$
    H( 0^n2^n4^k2^k) \cong 
    H( 2^n4^k2^k4^n) \cong H( 0^n2^k0^k2^n).$$ 

    The last space is one-dimensional by Lemma \ref{lemma:supp02_hom}. Thus, $m_k=1$.

    The statement for $\mathrm{cosoc}$ follows from Corollary \ref{cor:hom_switch} $$\Hom_{\Uqa}(0^n2^n0^n, [0,2]^k0^n) \cong \Hom_{\Uqa}(([0,2]^k0^n)^*, (0^n2^n0^n)^*) = \Hom_{\Uqa}([2,4]^k2^n, 2^n4^n2^n).$$

    The computation for $2^n0^n2^n$ is similar.
\end{proof}

For small $n$ it is possible to observe the structure of  $0^n2^n0^n$ explicitly. In particular, we expect that $0^n2^n0^n=\oplus_{k=0}^n \rV_k$ is a direct sum of indecomposable cyclic modules $\rV_k$ with $\chi_q(\rV_k)={n \choose k} \chi_q([0,2]^k0^n)$. 
\begin{example}\label{ex:002200} For $n=2$ we have
    \begin{equation*}        0^22^20^2 \cong 0^2 \oplus \rW \oplus 0^2[0,2]^2,
    \end{equation*}
    where $\rW$ is a non-trivial extension of $0^2[0,2]$ by itself.
\begin{proof}
    By \eqref{the equation} the composition factors of $0^22^20^2$ are given by $\{0^2, 0^2[0,2], 0^2[0,2], 0^2[0,2]^2\}$. By Proposition \ref{prop:soc_cosoc_struct}, the socle and cosocle of $0^22^20^2$ contain  $0^2\oplus 0^2[0,2]^2$. Therefore, $0^22^20^2 \cong 0^2 \oplus \rW \oplus 0^2[0,2]^2$ where $\rW$ has composition factors ${0^2[0,2], 0^2[0,2]}$.  Since $0^2[0,2]$ appears in $\mathrm{soc}(0^22^20^2)$ with multiplicity $1$, the module $\rW$ does not split.
\end{proof}
\end{example}

\subsection{Implicit bounds for $h(w)$.}\label{subsec:implicit_bounds}
In Section \ref{sec:bounds} we gave three bounds for $h(w)$. In this subsection we show on examples that these bounds can be improved by using slides (see Lemma \ref{lemma:slide_isom}), anti-involution $\omega$ (see Corollary \ref{cor:reverse_word}), and $n$-exchanges (see Corollary \ref{cor:020_commute}).  We further refer to slides, anti-involution $\omega$, and $n$-exchanges as symmetries.

We recall that we do not have a proof of $n$-exchanges for $n>2$. In our examples, we use only $1$-exchanges.

We call a bound obtained by a combination of bounds of Theorems \ref{thm:lower_bound_irr}, \ref{thm:upper_bound}, and Lemma \ref{lemma:upper_bound_qchar}  with symmetries, an implicit bound.
 Although the implicit bounds are much better, still, there are examples when the best implicit bounds are not exact. We give these examples below (except for the implicit steady arc configuration bound). 

\medskip 

We start with the lower bound given by Theorem \ref{thm:lower_bound_irr}. This bound is the same for the words related by a slide or anti-involution $\omega$. However, it can be different for words which are isomorphic as $\Uqa$-modules. 
We now study this phenomenon.

We start with a combinatorial observation. 

\begin{lemma}\label{lemma:lower_bound_impr_comb}
    Let $w_1, w_2$ be a pair of words, $a \in 2\mbZ$. Then 
    \begin{align*}
    |\iconf(w_1a(a+2)aw_2)| \geq |\iconf(w_1a^2(a+2)w_2)|,\\
    |\iconf(w_1a(a-2)aw_2)| \geq |\iconf(w_1(a-2)a^2w_2)|.
    \end{align*}
\end{lemma}
\begin{proof} We prove the first inequality. The proof of the second is similar.

We construct an embedding $$\iota: \iconf(w_1a^2(a+2)w_2)\longrightarrow \iconf(w_1a(a+2)aw_2).$$ 

For $C\in \iconf(w_1a^2 (a+2)w_2)$ we define $\iota(C)$  according to the following picture.

\medskip
\begin{figure}[H]
\begin{tikzpicture}
    \node at (1.25,-0.27) {$\dots$};
    \node at (3.35,-0.27) {$\dots$};
    
    \node at (1.7,-0.2) {$a$};
    \node at (2,-0.2) {$a$};
    \node at (2.6,-0.175) {$a{+}2$};

    \draw[-] (1.7 ,0) to [out=30,in=195] (2.3,0.25);
    \draw[dashed] (2.3,0.25) to [out=15,in=185] (2.9 ,0.35);
    
    \draw[-] (2,0) to [out=30,in=150] (2.6,0);

    \node at (4.25,-0.1) {$\longmapsto$};
    \node at (4.2, 0.14) {$\iota$};
    
    \node at (5.2,-0.27) {$\dots$};
    \node at (7.35,-0.27) {$\dots$};
    
    \node at (5.65,-0.2) {$a$};
    \node at (6.25,-0.175) {$a{+}2$};
    \node at (6.85,-0.2) {$a$};

    \draw[-] (5.65,0) to [out=30,in=150] (6.25,0);

    \draw[-] (6.85 ,0) to [out=30,in=195] (7.45,0.25);
    \draw[dashed] (7.45,0.25) to [out=15,in=185] (8.05, 0.35);

\end{tikzpicture},\\
\hspace{-14.5pt}
\begin{tikzpicture}

\node at (1.25,-0.27) {$\dots$};
    \node at (3.35,-0.27) {$\dots$};
    
    \node at (1.7,-0.2) {$ a$};
    \node at (2,-0.2) {$ a$};
    \node at (2.6,-0.175) {$ a{+}2$};

    \draw[-] (2,0) to [out=150,in=-15] (1.4,0.25);
    \draw[dashed] (1.4,0.25) to [out=165,in=-5] (0.8,0.35);
    \draw[-] (2.6 ,0) to [out=30,in=195] (3.2,0.25);
    \draw[dashed] (3.2,0.25) to [out=15,in=185] (3.8, 0.35);

    \node at (4.25,-0.1) {$\longmapsto$};
    \node at (4.2, 0.14) {$\iota$};
    
    \node at (5.2,-0.27) {$\dots$};
    \node at (7.35,-0.27) {$\dots$};
    
    \node at (5.65,-0.2) {$ a$};
    \node at (6.25,-0.175) {$ a{+}2$};
    \node at (6.85,-0.2) {$ a$};

    \draw[-] (6.85,0) to [out=150,in=-15] (6.25,0.25);
    \draw[dashed] (6.25,0.25) to [out=165,in=-5] (5.65,0.35);
    \draw[-] (6.25, 0) to [out=30,in=195] (6.85,0.25);
    \draw[dashed] (6.85,0.25) to [out=15,in=185] (7.45, 0.35);

    \draw[black,fill=black] (6.55,0.15) circle (.2ex);
\end{tikzpicture}.
\end{figure}

More precisely, $\iota$ is defined by the following rules. Let $w_1\in W_k,\; w_2\in W_l$.
    \begin{enumerate}
    \item Let $(k+2, k+3)\in C$. Write $C = \{(k+2, k+3), (k+1, j)\}\sqcup \tilde{C}$ or $C = \{(k+2, k+3), (j, k+1)\}\sqcup \tilde{C}$. Then set $\iota(C) = \{(k+1, k+2), (k+3, j)\}\sqcup \tilde{C}$  or $\iota(C) = \{(k+1, k+2), (j, k+3)\}\sqcup \tilde{C}$ respectively. Clearly, the colors of intersecting arcs for $\iota(C)$ are the same as in $C$, hence $\iota(C)$ is irreducible.

    \item Let $(k+2, k+3)\notin C$. Write $C = \{(i_m,j_m)\}_{m = 1,\dots,(k+l+3)/2}$. Let $s_{k+1}\in S_{k+l+3}$ be the elementary transposition switching $k+1$ with $k+2$. Then we set $\iota(C) = \{(s_k(i_m), s_{k}(j_m))\}$. The only new intersection of arcs which might appear for $\iota(C)$ is marked on the picture by a black circle. These two arcs can have colors $a-1$ and $a+3$ or  $a+3$ and $a+1$ or  $a+1$ and $a-1$. In both cases when colors differ by two, the arc on the left has larger color.
    Therefore, $\iota(C)$ is irreducible.  
    
    \end{enumerate}
\end{proof}

This combinatorial statement leads to the following corollary.

\begin{cor}\label{cor:lower_bd_improve}
    Let a word $w \in W_{2n}$ be of the form $\tilde{w}_1a^2(a+2)a\tilde{w}_{2}$ or of the form $\tilde{w}_1(a+2)a( a+2)^2\tilde{w}_{2}$. Set  $\tilde{w}= \tilde{w}_1a(a+2)a^2\tilde{w}_{2}$ in the first case and  $\tilde{w}= \tilde{w}_1(a+2)^2a(a+2)\tilde{w}_{2}$ in the second case. Then $h(w)=h(\tilde w)$ and   $|\iconf(w)|\leq |\iconf(\tilde w)|$.
\end{cor}

\begin{proof}
    In Section \ref{subsec:comm} we prove that $a(a+2)a^2\cong a^2(a+2)a$ and $(a+2)a(a+2)^2\cong(a+2)^2a(a+2)$ as $\Uqa$-modules. Therefore, $h(w)=h(\tilde w)$. The corollary follows from Lemma \ref{lemma:lower_bound_impr_comb}.
\end{proof}
Note that in Corollary \ref{cor:lower_bd_improve}, $\tilde w>w$ in the lexicographical order.

\begin{example}\label{ex:lower_bd_improve}
    Let $w = 4 2 0 22 4 6 2 44$.  The only irreducible arc configuration is
    
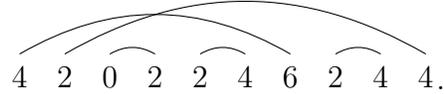
\begin{figure}[H]
\begin{tikzpicture}
    \node at (0, 0) {$4$};
    \node at (0.6,0) {$2$};
    \node at (1.2,0) {$0$};
    \node at (1.8,0) {$2$};
    \node at (2.4,0) {$2$};
    \node at (3,0)   {$4$};
    \node at (3.6,0) {$6$};
    \node at (4.2,0) {$2$};
    \node at (4.8,0) {$4$};
    \node at (5.4,0) {$4$};
    \draw[-] (0 ,0.3) to [out=30,in=150] (3.6,0.3);
    \draw[-] (0.6 ,0.3) to [out=30,in=150] (5.4,0.3);
    \draw[-] (1.2 ,0.3) to [out=30,in=150] (1.8,0.3);
    \draw[-] (2.4,0.3) to [out=30,in=150] (3,0.3);
    \draw[-] (4.2 ,0.3) to [out=30,in=150] (4.8,0.3);
    \node at (5.6,-0.15) {$.$};
\end{tikzpicture}
\end{figure}
However, replacing the factor $2022$ in $w$ with $2202$  we obtain the word $\tilde{w} = 4 2 2 0 2 4 6 2 4 4$. For the word $\tilde{w}$ irreducible arc configurations are
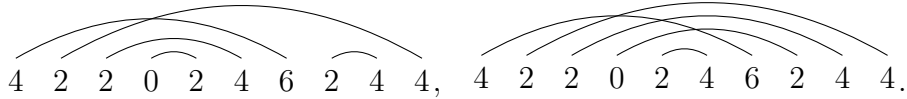
\begin{figure}[H]
\begin{tikzpicture}
    \node at (0, 0) {$4$};
    \node at (0.6,0) {$2$};
    \node at (1.2,0) {$2$};
    \node at (1.8,0) {$0$};
    \node at (2.4,0) {$2$};
    \node at (3,0)   {$4$};
    \node at (3.6,0) {$6$};
    \node at (4.2,0) {$2$};
    \node at (4.8,0) {$4$};
    \node at (5.4,0) {$4$};
    \draw[-] (0 ,0.3) to [out=30,in=150] (3.6,0.3);
    \draw[-] (0.6 ,0.3) to [out=30,in=150] (5.4,0.3);
    \draw[-] (1.2 ,0.3) to [out=30,in=150] (3,0.3);
    \draw[-] (1.8,0.3) to [out=30,in=150] (2.4,0.3);
    \draw[-] (4.2 ,0.3) to [out=30,in=150] (4.8,0.3);
    \node at (5.6,-0.15) {$,$};
\end{tikzpicture}
\begin{tikzpicture}
    \node at (0, 0) {$4$};
    \node at (0.6,0) {$2$};
    \node at (1.2,0) {$2$};
    \node at (1.8,0) {$0$};
    \node at (2.4,0) {$2$};
    \node at (3,0)   {$4$};
    \node at (3.6,0) {$6$};
    \node at (4.2,0) {$2$};
    \node at (4.8,0) {$4$};
    \node at (5.4,0) {$4$};
    \draw[-] (0 ,0.3) to [out=30,in=150] (3.6,0.3);
    \draw[-] (0.6 ,0.3) to [out=30,in=150] (5.4,0.3);
    \draw[-] (1.2 ,0.3) to [out=30,in=150] (4.8,0.3);
    \draw[-] (1.8,0.3) to [out=30,in=150] (4.2,0.3);
    \draw[-] (2.4 ,0.3) to [out=30,in=150] (3,0.3);
    \node at (5.6,-0.15) {$.$};
    \node at (5.6,-0.25) {$ $};
\end{tikzpicture}
\end{figure}
Since $w\cong \tilde w$ as $\Uqa$-modules, we obtain $h(\tilde{w}) = h(w) \geq 2$. One could use Theorem \ref{thm:upper_bound} or a computation similar to Section \ref{subsec:exact_comps} to show that $h(w)=2$. Thus, the lower bound given by Theorem \ref{thm:lower_bound_irr} is exact for $\tilde{w}$ but not exact for $w$.
\end{example}

The next example shows that the best implicit lower bound obtained from Theorem \ref{thm:lower_bound_irr} may be not exact.

\begin{example}
Consider the word $w = 0^22^20^22^24^22^2$. Then by the explicit check one obtains that all slides of $w$ do not contain a factor of the form $a^n(a+2)^na^{n+1}$ or of the form $a^{n+1}(a-2)^na^{n}$. Then one obtains $|\iconf(w)| = 3$. On the other hand
$$
h(w) = \dim(\mathrm{End}_{\Uqa}(0^22^20^2)) = 4,
$$
which follows from the structure of the module $0^22^20^2$ described in Example \ref{ex:002200}. Namely, 
\begin{align*}
\dim(\mathrm{End}_{\Uqa}(0^22^20^2)) &= \dim(\mathrm{End}_{\Uqa}(0^2)) + \dim(\mathrm{End}_{\Uqa}(\rW)) + \dim(\mathrm{End}_{\Uqa}(0^2[0,2]^2)) = \\
&=1 + 2 +1 = 4,
\end{align*}
where we used that $\rW$ is a non-trivial self-extension of an irreducible module.
\end{example}

The $q$-character upper bound, see Lemma \ref{lemma:upper_bound_qchar}, is invariant with respect to anti-involution $\omega$ and $n$-exchanges, however, it can be improved by slides. We have used it on many occasions. Here we recall the simplest example which illustrates the situation.
\begin{example} Let $w=20$. Then $h_{\textit{char}}(w)=1$. After a slide, we have $s(w)=06$ and $h_{\textit{char}}(s(w))=0$. It follows that $h(w)=h(s(w))=0$.
\end{example}

The next example shows that the best implicit upper bound obtained from the $q$-characters, see Lemma \ref{lemma:upper_bound_qchar}, may be not not exact.
\begin{example}
    Consider the word $w_1 = 020242$. The slides of $w_1$ are $w_2 = 202424, w_3 = 024246$. They contain no factors of the form $a^n(a+2)^na^{n+1}$ or of the form $a^{n+1}(a-2)^na^{n}$. The values of $h_{\textit{char}}$ on the words $w_1, w_2, w_3$ are $3,3,4$ respectively. However, $|\iconf(w)| = |\sconf(w)| = 2$, which implies $h(w) = 2$ by Theorems \ref{thm:lower_bound_irr}, \ref{thm:upper_bound}.
\end{example}

The following examples show that the upper bound given by Theorem \ref{thm:upper_bound} can be improved by using symmetries. 

The first one uses slides.
\begin{example}
Consider the word $w = 00224022$. There are two steady arc configurations of $w$ given by 
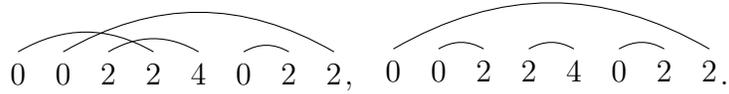
\begin{figure}[H]
\begin{tikzpicture}
    \node at (0, 0) {$0$};
    \node at (0.6,0) {$0$};
    \node at (1.2,0) {$2$};
    \node at (1.8,0) {$2$};
    \node at (2.4,0) {$4$};
    \node at (3,0)   {$0$};
    \node at (3.6,0) {$2$};
    \node at (4.2,0) {$2$};
    \draw[-] (0 ,0.3) to [out=30,in=150] (1.8,0.3);
    \draw[-] (0.6 ,0.3) to [out=30,in=150] (4.2,0.3);
    \draw[-] (1.2 ,0.3) to [out=30,in=150] (2.4,0.3);
    \draw[-] (3,0.3) to [out=30,in=150] (3.6,0.3);
    \node at (4.4,-0.15) {$,$};
\end{tikzpicture}
\begin{tikzpicture}
    \node at (0, 0) {$0$};
    \node at (0.6,0) {$0$};
    \node at (1.2,0) {$2$};
    \node at (1.8,0) {$2$};
    \node at (2.4,0) {$4$};
    \node at (3,0)   {$0$};
    \node at (3.6,0) {$2$};
    \node at (4.2,0) {$2$};
    \draw[-] (0 ,0.3) to [out=30,in=150] (4.2,0.3);
    \draw[-] (0.6 ,0.3) to [out=30,in=150] (1.2,0.3);
    \draw[-] (1.8 ,0.3) to [out=30,in=150] (2.4,0.3);
    \draw[-] (3.0,0.3) to [out=30,in=150] (3.6,0.3);
    \node at (4.4,-0.15) {$.$};
    \node at (4.4,-0.25) {$ $};
\end{tikzpicture}
\end{figure}
After two slides we obtain the word $\tilde{w} = s^2(w) = 22402244$. The word $\tilde{w}$ has one steady arc configuration given by
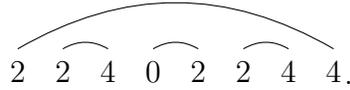
\begin{figure}[H]
\begin{tikzpicture}
    \node at (0, 0) {$2$};
    \node at (0.6,0) {$2$};
    \node at (1.2,0) {$4$};
    \node at (1.8,0) {$0$};
    \node at (2.4,0) {$2$};
    \node at (3,0)   {$2$};
    \node at (3.6,0) {$4$};
    \node at (4.2,0) {$4$};
    \draw[-] (0 ,0.3) to [out=30,in=150] (4.2,0.3);
    \draw[-] (0.6 ,0.3) to [out=30,in=150] (1.2,0.3);
    \draw[-] (1.8 ,0.3) to [out=30,in=150] (2.4,0.3);
    \draw[-] (3,0.3) to [out=30,in=150] (3.6,0.3);
    \node at (4.4,-0.15) {$.$};
    \node at (4.4,-0.25) {$ $};
\end{tikzpicture}
\end{figure}
Therefore, $ h(w)=h(\tilde{w}) = 1$.
\end{example}

In the second example, slides are not sufficient, but we are helped by $1$-exchanges, see Corollary \ref{cor:020_commute}. 
\begin{example}
    Consider the word $w = 0024022462$. There are two irreducible arc configurations given by
    
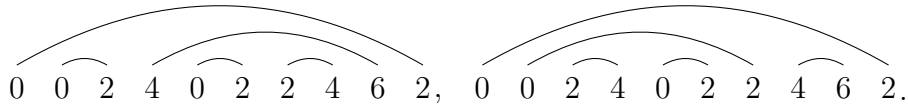
\begin{figure}[H]
\begin{tikzpicture}
    \node at (0, 0) {$0$};
    \node at (0.6,0) {$0$};
    \node at (1.2,0) {$2$};
    \node at (1.8,0) {$4$};
    \node at (2.4,0) {$0$};
    \node at (3,0)   {$2$};
    \node at (3.6,0) {$2$};
    \node at (4.2,0) {$4$};
    \node at (4.8,0) {$6$};
    \node at (5.4,0) {$2$};
    \draw[-] (0 ,0.3) to [out=30,in=150] (5.4,0.3);
    \draw[-] (0.6 ,0.3) to [out=30,in=150] (1.2,0.3);
    \draw[-] (1.8 ,0.3) to [out=30,in=150] (4.8,0.3);
    \draw[-] (2.4,0.3) to [out=30,in=150] (3.0,0.3);
    \draw[-] (3.6,0.3) to [out=30,in=150] (4.2,0.3);
    \node at (5.6,-0.15) {$,$};
    \node at (5.6,-0.25) {$ $};
\end{tikzpicture}
\begin{tikzpicture}
    \node at (0, 0) {$0$};
    \node at (0.6,0) {$0$};
    \node at (1.2,0) {$2$};
    \node at (1.8,0) {$4$};
    \node at (2.4,0) {$0$};
    \node at (3,0)   {$2$};
    \node at (3.6,0) {$2$};
    \node at (4.2,0) {$4$};
    \node at (4.8,0) {$6$};
    \node at (5.4,0) {$2$};
    \draw[-] (0 ,0.3) to [out=30,in=150] (5.4,0.3);
    \draw[-] (0.6 ,0.3) to [out=30,in=150] (3.6,0.3);
    \draw[-] (1.2 ,0.3) to [out=30,in=150] (1.8,0.3);
    \draw[-] (2.4,0.3) to [out=30,in=150] (3.0,0.3);
    \draw[-] (4.2,0.3) to [out=30,in=150] (4.8,0.3);
    \node at (5.6,-0.15) {$.$};
    \node at (5.6,-0.25) {$ $};
\end{tikzpicture}
\end{figure}
Both of these arc configurations are steady. There are three steady arc configurations of the word $w$. The third one is given by 
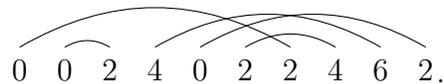
\begin{figure}[H]
\begin{tikzpicture}
    \node at (0, 0) {$0$};
    \node at (0.6,0) {$0$};
    \node at (1.2,0) {$2$};
    \node at (1.8,0) {$4$};
    \node at (2.4,0) {$0$};
    \node at (3,0)   {$2$};
    \node at (3.6,0) {$2$};
    \node at (4.2,0) {$4$};
    \node at (4.8,0) {$6$};
    \node at (5.4,0) {$2$};
    \draw[-] (0 ,0.3) to [out=30,in=150] (3.6,0.3);
    \draw[-] (0.6 ,0.3) to [out=30,in=150] (1.2,0.3);
    \draw[-] (1.8 ,0.3) to [out=30,in=150] (4.8,0.3);
    \draw[-] (2.4,0.3) to [out=30,in=150] (5.4,0.3);
    \draw[-] (3.0,0.3) to [out=30,in=150] (4.2,0.3);
    \node at (5.6,-0.15) {$.$};
    \node at (5.6,-0.25) {$ $};
\end{tikzpicture}
\end{figure}
Explicit check shows that for any $0\leq k \leq 9$ one gets $|\sconf(s^k(w))| = 3$. However, by Corollary \ref{cor:020_commute}, we have $w \cong 0020422426 \cong 0200424226$. The word $0200424226$ has two steady arc configurations given by 
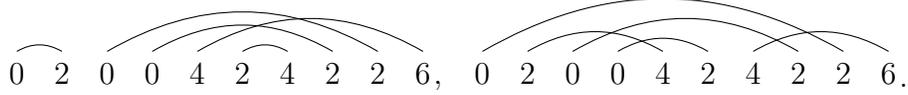
\begin{figure}[H]
\begin{tikzpicture}
    \node at (0, 0) {$0$};
    \node at (0.6,0) {$2$};
    \node at (1.2,0) {$0$};
    \node at (1.8,0) {$0$};
    \node at (2.4,0) {$4$};
    \node at (3,0)   {$2$};
    \node at (3.6,0) {$4$};
    \node at (4.2,0) {$2$};
    \node at (4.8,0) {$2$};
    \node at (5.4,0) {$6$};
    \draw[-] (0 ,0.3) to [out=30,in=150] (0.6,0.3);
    \draw[-] (1.2 ,0.3) to [out=30,in=150] (4.8,0.3);
    \draw[-] (1.8 ,0.3) to [out=30,in=150] (4.2,0.3);
    \draw[-] (2.4,0.3) to [out=30,in=150] (5.4,0.3);
    \draw[-] (3.0,0.3) to [out=30,in=150] (3.6,0.3);
    \node at (5.6,-0.15) {$,$};
    \node at (5.6,-0.25) {$ $};
\end{tikzpicture}
\begin{tikzpicture}
    \node at (0, 0) {$0$};
    \node at (0.6,0) {$2$};
    \node at (1.2,0) {$0$};
    \node at (1.8,0) {$0$};
    \node at (2.4,0) {$4$};
    \node at (3,0)   {$2$};
    \node at (3.6,0) {$4$};
    \node at (4.2,0) {$2$};
    \node at (4.8,0) {$2$};
    \node at (5.4,0) {$6$};
    \draw[-] (0 ,0.3) to [out=30,in=150] (4.8,0.3);
    \draw[-] (0.6 ,0.3) to [out=30,in=150] (2.4,0.3);
    \draw[-] (1.2 ,0.3) to [out=30,in=150] (4.2,0.3);
    \draw[-] (1.8,0.3) to [out=30,in=150] (3.0,0.3);
    \draw[-] (3.6,0.3) to [out=30,in=150] (5.4,0.3);
    \node at (5.6,-0.15) {$.$};
    \node at (5.6,-0.25) {$ $};
\end{tikzpicture}
\end{figure}
Therefore, $h(w) = 2$.
\end{example}

In the third example we use anti-involution $\omega$ to improve the upper bound given by the steady arc configurations. 

\begin{example}
    Consider a word $w = 002242$. There is one irreducible arc configuration for the word $w$ given by
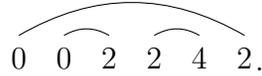
\begin{figure}[H]
    \begin{tikzpicture}
    \node at (0, 0) {$0$};
    \node at (0.6,0) {$0$};
    \node at (1.2,0) {$2$};
    \node at (1.8,0) {$2$};
    \node at (2.4,0) {$4$};
    \node at (3.0,0) {$2$};
    
    \draw[-] (0 ,0.3) to [out=30,in=150] (3.0,0.3);
    \draw[-] (0.6 ,0.3) to [out=30,in=150] (1.2,0.3);
    \draw[-] (1.8 ,0.3) to [out=30,in=150] (2.4,0.3);
    
    \node at (3.2,-0.15) {$.$};
    \node at (3.2,-0.25) {$ $};
\end{tikzpicture}
\end{figure}

This arc configuration is steady. There are two steady arc configurations of the word $w$. The second steady arc configuration one is given by 
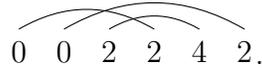
\begin{figure}[H]
    \begin{tikzpicture}
    \node at (0, 0) {$0$};
    \node at (0.6,0) {$0$};
    \node at (1.2,0) {$2$};
    \node at (1.8,0) {$2$};
    \node at (2.4,0) {$4$};
    \node at (3.0,0) {$2$};
    
    \draw[-] (0 ,0.3) to [out=30,in=150] (1.8,0.3);
    \draw[-] (0.6 ,0.3) to [out=30,in=150] (3.0,0.3);
    \draw[-] (1.2 ,0.3) to [out=30,in=150] (2.4,0.3);
    
    \node at (3.2,-0.15) {$.$};
    \node at (3.2,-0.25) {$ $};
\end{tikzpicture}
\end{figure}

However, for the word $\tilde{w} = 202244$ obtained from the word $\omega(w)$ by shift by $4$ there is only one steady arc configuration given by
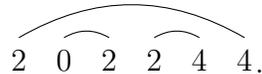
\begin{figure}[H]
    \begin{tikzpicture}
    \node at (0, 0) {$2$};
    \node at (0.6,0) {$0$};
    \node at (1.2,0) {$2$};
    \node at (1.8,0) {$2$};
    \node at (2.4,0) {$4$};
    \node at (3.0,0) {$4$};
    
    \draw[-] (0 ,0.3) to [out=30,in=150] (3.0,0.3);
    \draw[-] (0.6 ,0.3) to [out=30,in=150] (1.2,0.3);
    \draw[-] (1.8 ,0.3) to [out=30,in=150] (2.4,0.3);
    
    \node at (3.2,-0.15) {$.$};
    \node at (3.2,-0.25) {$ $};
\end{tikzpicture}
\end{figure}
Therefore, $h(w) = h(\tilde{w}) = 1$.
\end{example}

\section{Extensions.}\label{sec:extensions}
Let $\rU$, $\rW$ be two $\Uqa$-modules. We call a module $\rV$ an extension of $\rW$ by $\rU$ if $\rU$ is a submodule of $\rV$ and the quotient module $\rV/\rU$ is isomorphic to $\rW$. We call an extension $\rV$ of $\rW$ by $\rU$ trivial if $\rV \cong \rU\oplus\rW$.

The abundance of non-trivial extensions is the main difficulty and attraction of the representation theory of $\Uqa$.
In this section we describe all extensions in the case when $\rU$ and $\rW$ are irreducible evaluation $\Uqa$-modules.

\subsection{Classifications of extensions in special cases.}
We start with a general property of extensions of modules over a Hopf algebra.
\begin{lemma}\label{lemma:ext_dual_aut}
Let $\rH$ be a Hopf algebra. Let $\varphi$ be an automorphism of $\rH$. Let $\rV, \rW, \rU$ be $\rH$-modules.
    \begin{itemize}
        \item If $\rV$ is an extension of $\rW$ by $\rU$, then $\rV^*$ is an extension of $\rU^*$ by $\rW^*$. The extension $\rV$ is trivial if and only if the extension $\rV^*$ is trivial.
        \item If $\rV$ is an extension of $\rW$ by $\rU$, then $\rV_{\varphi}$ is an extension of $\rU_{\varphi}$ by $\rW_{\varphi}$. The extension $\rV$ is trivial if and only if the extension $\rV_{\varphi}$ is trivial.
    \end{itemize}
\qed\end{lemma}

Recall the extension \eqref{eq:two_strings_prod2} given in Proposition \ref{prop:two_strings_prod}. 
\begin{prop}\label{prop:ext_two_strings_prod}
   Let $\al_1,\bt_1$  $\al_2 ,\bt_2$ be four even integers such that $\al_1 < \al_2\leq \bt_1+2\leq \bt_2$. Then
any non-trivial extension of $[\alpha_1,\alpha_2 - 4]  [\beta_{1} +4,  \beta_2]$ by $[\alpha_{1}, \beta_{2}][\alpha_{2}, \beta_{1}]$ is isomorphic to $[\al_2,\beta_2][\alpha_1,\beta_1]$.

\end{prop}
\begin{proof}
   Let $\rV$ be such an extension. Then by Proposition \ref{prop:two_strings_prod} there exists a short exact sequence of $\Uqa$-modules 
    \begin{equation}\label{eq:ext_un_pf1}
        [\alpha_{1},\beta_{2}][\alpha_{2}, \beta_{1}] \hookrightarrow \rV \twoheadrightarrow [\alpha_1, \alpha_2 - 4]  [\beta_{1} +4, \beta_2],
    \end{equation}    which splits over $\Uq$.

    Denote by $\rU$ the $\Uqa$-submodule $ [\alpha_{1}, \beta_{2}][\alpha_{2},  \beta_{1}] \subset \rV$.  Denote by $\rW$ the  unique $\Uq$-module which is a complement  to $\rU \subset \rV$ \. Then  $\rW$ isomorphic to $[\alpha_1, \alpha_2 - 4]  [\beta_{1} +4, \beta_2]$ as an $\Uq$-module.
    Let $\pi: \rV\to \rV$ be the projection onto $\rU$ along $\rW$. 
    
    Denote $\mu_1 = \frac{\bt_2-\al_1}{2}+1, \mu_2 = \frac{\bt_1 - \al_2}{2} + 1, \nu_1 = \frac{\al_2 -\al_1}{2} - 1, \nu_2 = \frac{\bt_2 - \bt_1}{2} - 1$ the highest weights of modules $[\alpha_{1},\beta_{2}],[\alpha_{2}, \beta_{1}], [\alpha_1, \alpha_2 - 4],  [\beta_{1} +4, \beta_2]$ respectively.
    
    We have decompositions of $Uq$-modules
 \begin{equation*}
    \begin{aligned}    \rW &\underset{\Uq}{\cong} \mathrm{L}_{\nu_1 + \nu_2} \oplus \dots \oplus \mathrm{L}_{|\nu_1-\nu_2|},\\
    \rU &\underset{\Uq}{\cong} \mathrm{L}_{\mu_1 + \mu_2} \oplus \dots \oplus \mathrm{L}_{\mu_1 -\mu_2}.
    \end{aligned}    \end{equation*}
    Note that $\nu_1 + \nu_2 = \mu_1 - \mu_2 - 2$.
    
    The $\Uqa$-module $\rV$ is uniquely determined by linear maps $\pi e_0|_{\rW}, \pi  f_0|_{\rW}$.

    Since $f_0$ commutes with $e_1$, action of $\pi  f_0|_{\rW}$ is uniquely defined by its restriction to $\mathrm{Ker}(f_1)\cap \rW$.
    
    Let $\bar{w} \in \mathrm{Ker}(f_1)\cap \rW$  be of a weight $-l$. Then $e_1^{l}f_0\bar{w} \in \mathrm{Ker}(e_1)$. Therefore, $e_1^l\pi f_0 \bar{w}\in \mathrm{Ker}(e_1)$. 
    
    The vector $e_1^{l}\pi f_0 \bar{w}$ has weight $l + 2$ and belongs to $ \mathrm{Ker}(e_1)\cap \rU$. 
     
    For the case $l < \nu_1 + \nu_2$ this implies $e_1^{l}\pi f_0 \bar{w} = 0$ which in turn implies $\pi f_0 \bar{w} = 0$. 
    
    For the case $l = \nu_1 + \nu_2$ this implies that $\pi f_0 \bar{w} \in \mathrm{L}_{\nu_1+\nu_2}$ and has weight $-l+2$.

    This defines the action of $\pi \circ f_0|_{\rW}$ uniquely up to multiplication by a non-zero constant. Analogously, the action of $\pi \circ e_0|_{\rW}$ is defined uniquely up to multiplication by a non-zero constant.

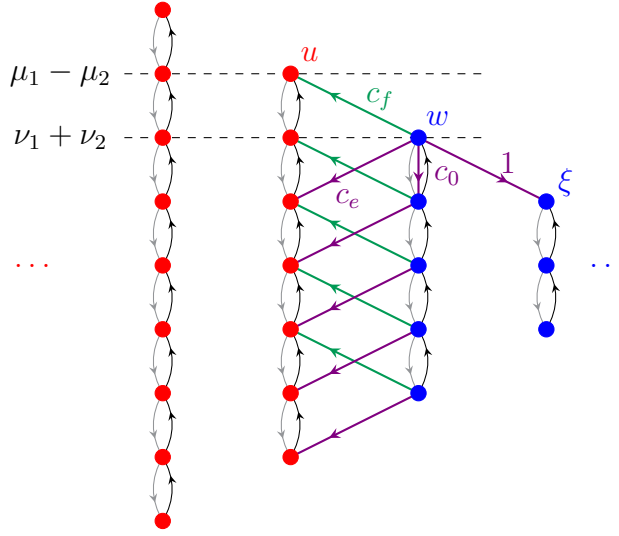
\begin{figure}[H]
\begin{tikzpicture}[scale=1.7,decoration={markings, mark= at position 0.7 with {\arrow{stealth}}}]
\node[ForestGreen] at (-0.8, 1.3) {$c_f$};

\draw[-,postaction={decorate}, ForestGreen, thick] (-0.5, 1) to  (-1.5, 1.5);
\draw[-,postaction={decorate}, ForestGreen, thick] (-0.5, 0.5) to  (-1.5, 1);
\draw[-,postaction={decorate}, ForestGreen, thick] (-0.5, 0) to  (-1.5, 0.5);
\draw[-,postaction={decorate}, ForestGreen, thick] (-0.5, -0.5) to  (-1.5, 0);
\draw[-,postaction={decorate}, ForestGreen, thick] (-0.5, -1) to  (-1.5, -0.5);

\node[violet] at (0.2, 0.8) {$1$};
\node[violet] at (-0.275, 0.7) {$c_0$};
\node[violet] at (-1.05, 0.525) {$c_e$};

\draw[-,postaction={decorate}, violet, thick] (-0.5, 1) to  (-1.5, 0.5);
\draw[-,postaction={decorate}, violet, thick] (-0.5, 0.5) to  (-1.5, 0);
\draw[-,postaction={decorate}, violet, thick] (-0.5, 0) to  (-1.5, -0.5);
\draw[-,postaction={decorate}, violet, thick] (-0.5, -0.5) to  (-1.5, -1.0);
\draw[-,postaction={decorate}, violet, thick] (-0.5, -1) to  (-1.5, -1.5);

\draw[-,postaction={decorate}, violet, thick] (-0.5, 1) to  (-0.5, 0.5);
\draw[-,postaction={decorate}, violet, thick] (-0.5, 1) to  (0.5, 0.5);

\draw[-,postaction={decorate},Gray] (-2.5,2.) to [out=-120,in=120] (-2.5,1.5);
\draw[-,postaction={decorate},Gray] (-2.5,1.5) to [out=-120,in=120] (-2.5,1.);
\draw[-,postaction={decorate},Gray] (-2.5,1.) to [out=-120,in=120] (-2.5,0.5);
\draw[-,postaction={decorate},Gray] (-2.5,0.5) to [out=-120,in=120] (-2.5,0.);
\draw[-,postaction={decorate},Gray] (-2.5,0.) to [out=-120,in=120] (-2.5,-0.5);
\draw[-,postaction={decorate},Gray] (-2.5,-0.5) to [out=-120,in=120] (-2.5,-1.);
\draw[-,postaction={decorate},Gray] (-2.5,-1.) to [out=-120,in=120] (-2.5,-1.5);
\draw[-,postaction={decorate},Gray] (-2.5,-1.5) to [out=-120,in=120] (-2.5,-2.);

\draw[-,postaction={decorate},Gray] (-1.5,1.5) to  [out=-120,in=120] (-1.5,1.);
\draw[-,postaction={decorate},Gray] (-1.5,1.) to  [out=-120,in=120] (-1.5,0.5);
\draw[-,postaction={decorate},Gray] (-1.5,0.5) to  [out=-120,in=120] (-1.5,0.);
\draw[-,postaction={decorate},Gray] (-1.5,0.) to  [out=-120,in=120] (-1.5,-0.5);
\draw[-,postaction={decorate},Gray] (-1.5,-0.5) to  [out=-120,in=120] (-1.5,-1.);
\draw[-,postaction={decorate},Gray] (-1.5,-1.) to  [out=-120,in=120] (-1.5,-1.5);

\draw[-,postaction={decorate},Gray] (-0.5,1.) to  [out=-120,in=120] (-0.5,0.5);
\draw[-,postaction={decorate},Gray] (-0.5,0.5) to  [out=-120,in=120] (-0.5,0.);
\draw[-,postaction={decorate},Gray] (-0.5,0.) to  [out=-120,in=120] (-0.5,-0.5);
\draw[-,postaction={decorate},Gray] (-0.5,-0.5) to [out=-120,in=120] (-0.5,-1.);

\draw[-,postaction={decorate},Gray] (0.5,0.5) to  [out=-120,in=120] (0.5,0.0);
\draw[-,postaction={decorate},Gray] (0.5,0.0) to  [out=-120,in=120] (0.5,-0.5);

\draw[-,postaction={decorate}] (-2.5,1.5) to  [out=60,in=-60] (-2.5,2.);
\draw[-,postaction={decorate}] (-2.5,1.) to  [out=60,in=-60] (-2.5,1.5);
\draw[-,postaction={decorate}] (-2.5,0.5) to  [out=60,in=-60] (-2.5,1.);
\draw[-,postaction={decorate}] (-2.5,0.) to  [out=60,in=-60] (-2.5,0.5);
\draw[-,postaction={decorate}] (-2.5,-0.5) to  [out=60,in=-60] (-2.5,0.);
\draw[-,postaction={decorate}] (-2.5,-1.) to  [out=60,in=-60] (-2.5,-0.5);
\draw[-,postaction={decorate}] (-2.5,-1.5) to  [out=60,in=-60] (-2.5,-1.);
\draw[-,postaction={decorate}] (-2.5,-2.) to  [out=60,in=-60] (-2.5,-1.5);

\draw[-,postaction={decorate}] (-1.5,1.) to  [out=60,in=-60] (-1.5,1.5);
\draw[-,postaction={decorate}] (-1.5,0.5) to  [out=60,in=-60] (-1.5,1.);
\draw[-,postaction={decorate}] (-1.5,0.) to  [out=60,in=-60] (-1.5,0.5);
\draw[-,postaction={decorate}] (-1.5,-0.5) to  [out=60,in=-60] (-1.5,0.);
\draw[-,postaction={decorate}] (-1.5,-1.) to  [out=60,in=-60] (-1.5,-0.5);
\draw[-,postaction={decorate}] (-1.5,-1.5) to  [out=60,in=-60] (-1.5,-1.);

\draw[-,postaction={decorate}] (-0.5,0.5) to  [out=60,in=-60] (-0.5,1.);
\draw[-,postaction={decorate}] (-0.5,0.) to  [out=60,in=-60] (-0.5,0.5);
\draw[-,postaction={decorate}] (-0.5,-0.5) to  [out=60,in=-60] (-0.5,0.);
\draw[-,postaction={decorate}] (-0.5,-1.) to  [out=60,in=-60] (-0.5,-0.5);

\draw[-,postaction={decorate}] (0.5,0.) to  [out=60,in=-60] (0.5,0.5);
\draw[-,postaction={decorate}] (0.5,-0.5) to  [out=60,in=-60] (0.5,0.);

\draw[-, dashed] (-2.8, 1.5) to (-0, 1.5);
\draw[-, dashed] (-2.8, 1) to (-0, 1);
\node at (-3.3, 1.5) {$\mu_1 -\mu_2$};
\node at (-3.3, 1) {$\nu_1 +\nu_2$};

\node[blue] at (-0.35, 1.15) {$w$};
\node[blue] at (0.65, 0.65) {$\xi$};
\node[red] at (-1.35, 1.65) {$u$};

\node [blue] at (1.0, 0) {$\dots$};
\filldraw [blue] (0.5,-0.5) circle (1.7pt);
\filldraw [blue] (0.5,0.) circle (1.7pt);
\filldraw [blue] (0.5,0.5) circle (1.7pt);

\filldraw [blue] (-0.5,-1.) circle (1.7pt);
\filldraw [blue] (-0.5,-0.5) circle (1.7pt);
\filldraw [blue] (-0.5,0.) circle (1.7pt);
\filldraw [blue] (-0.5,0.5) circle (1.7pt);
\filldraw [blue] (-0.5,1.) circle (1.7pt);

\filldraw [red] (-1.5,-1.5) circle (1.7pt);
\filldraw [red] (-1.5,-1.) circle (1.7pt);
\filldraw [red] (-1.5,-0.5) circle (1.7pt);
\filldraw [red] (-1.5,0.) circle (1.7pt);
\filldraw [red] (-1.5,0.5) circle (1.7pt);
\filldraw [red] (-1.5,1.) circle (1.7pt);
\filldraw [red] (-1.5,1.5) circle (1.7pt);

\filldraw [red] (-2.5,-2.) circle (1.7pt);
\filldraw [red] (-2.5,-1.5) circle (1.7pt);
\filldraw [red] (-2.5,-1.) circle (1.7pt);
\filldraw [red] (-2.5,-0.5) circle (1.7pt);
\filldraw [red] (-2.5,0.) circle (1.7pt);
\filldraw [red] (-2.5,0.5) circle (1.7pt);
\filldraw [red] (-2.5,1.) circle (1.7pt);
\filldraw [red] (-2.5,1.5) circle (1.7pt);
\filldraw [red] (-2.5,2.) circle (1.7pt);
\node [red] at (-3.5, 0) {$\dots$};
\end{tikzpicture}
\end{figure}

 We illustrate the notation in the proof with a picture. The bullets are basis vectors of the module $\rV$, the red ones correspond to $\rU$ and blue ones to $\rW$. The bullets are placed at the height equal to the weight of the corresponding vector. The operators $e_1$ and $f_1$ act vertically as indicated by black and gray arrows. By green arrows we indicated the matrix elements of $f_0$ from $\rW$ to $\rU$. By violet arrows we indicated the matrix elements of $e_0$ from $\rW$ to $\rU$. We added two more violet arrows to  describe the action of $e_0$ on $w$ completely. We also show the vectors $u,w,\xi$ and matrix elements $c_e,c_f,c_0$ participating in the proof.

    Fix non-zero vectors $w\in \Ker(e_1)\cap \rL_{\nu_1 +\nu_2}$ and $u \in \Ker(e_1)\cap \rL_{\mu_1-\mu_2}$.  We have
    $$
    f_0 w = c_{f}u,\;\; e_0 w = c_ef_1^2u +c_0 f_1w + \xi, 
    $$
    where $\xi \in \Ker(e_1)\cap \rL_{\nu_1 + \nu_2 - 2}$, $c_0, c_f, c_e\in\mbC$ constants. The constant $c_0$ and vector $\xi$ are uniquely determined by $\Uqa$-module structure of $[\al_1,\al_2-4][\bt_1+4,\bt_2]$.
    The constants $c_f, c_e$ determine the extension $\rV$.

    Then we obtain 
    $$
    e_0f_0w = c_fe_0u,\;\;f_0e_0w = c_e f_0f_1^2u + c_0 f_0f_1w +  f_0\xi.
    $$
    Note that $f_0f_1^2u \in \rU$. Additionally, since weights of all vectors in $\Ker (e_1)\cap \rU$ are greater or equal than $2$, we have $f_0f_1^2u \neq 0$.

    Applying $\pi$ to the equality $(f_0e_0 - e_0f_0)w = \frac{(K-K^{-1})}{q-q^{-1}}w$ we obtain
    \begin{equation}\label{eq:ext_tmp}
    c_0\pi f_0f_1w + c_e f_0f_1^2u - c_fe_0u = 0.
    \end{equation}
    Note that $e_1^2\pi f_0f_1 w= 0$, which implies $c_0\pi f_0f_1w = \alpha f_1u$ for some $\alpha\in\mbC$. Applying $e_1$ to both parts of this equation, we find $\alpha = \frac{[\nu_1+\nu_2]_q}{[\mu_2-\mu_1]_q}c_0c_f$. In the end we get
    \begin{equation}\label{eq:ext_tmp2}
    c_e f_0f_1^2 u = c_f\Big(e_0 u + \frac{[\nu_1+\nu_2]_q}{[\mu_2-\mu_1]_q}c_0f_1 u\Big),
    \end{equation}
    Note that $f_0f_1^2u \neq 0$, therefore we have a linear relation between $c_e$ and $c_f$. Also this implies that if $c_e \neq 0$, then $c_f \neq 0$. One can symmetrically consider the action of $[e_0, f_0]$ on a non-zero vector from $\Ker(f_1)\cap \rL_{\nu_1+\nu_2-2}$ and conclude that if $c_e \neq 0$, then $c_f\neq 0$. 

    Alternatively, by an explicit computation in the submodule $\rU$ one can show that vectors $e_0u$ and $f_1u$ are linearly independent, therefore from \eqref{eq:ext_tmp2} we obtain $c_e = \lambda c_{f}$ for some non-zero $\lambda$.
    
    If $c_e = 0$, $\rV$ the sequence splits.  If $c_e\neq 0$, rescaling $w$, we are reduced to $c_e=1$, which uniquely fixes $\rV$.

\end{proof}

We give an example of a non-trivial extension which is not a tensor product.

\begin{example}\label{ex:ev_self_ext}
    Let $\rW = [\al, \bt]$ where $\bt \geq \al $. There exists an extension  $\rV$ of $\rW$ by $\rW$ given as follows.

    Let $\rV=\rW\oplus\rW$ as an  $\Uq$-module. Choose a basis $\{v_i\}_{i=0}^{m}$ of $\rW$. We use the basis of $\rV$ given by $\{(v_i,0),(0,v_i)\}_{i=0}^{m}$. As an $\Uq$-module, $\rW=\rL_m$ with $m = (\bt-\al)/2+1$. Denote $a = \frac{\al + \bt}{2}$.

    Let $E = \rho_{\rL_m}(e)$ and $F = \rho_{\rL_m}(f)$
    be matrices corresponding to images of generators $e$ and $f$ of $\Uq$ in $\mathrm{End}(\rL_m)$. Define a structure of $\Uqa$ module on  $\Uq$-module $\rV$ by  
\begin{equation}
e_0 \longmapsto q^a\begin{pmatrix}
F & 
F\\
0 & F
\end{pmatrix},\;\;\;
f_0 \longmapsto q^{-a}\begin{pmatrix}
E & -E\\
0 & E
\end{pmatrix}.
\end{equation}
We check that these formulas satisfy the relations of $\Uqa$. The only non-instant check is the Serre relations which are reduced to
\begin{equation*}
    E^3F - [3]_qE^2FE + [3]_qEFE^2 - FE^3 = 0,\;\;F^3E - [3]_qF^2EF + [3]_qFEF^2 - EF^3=0.
\end{equation*}
This identity holds in $\Uq$ since the evaluation homomorphism is well-defined. Alternatively, one deduces this identity in $\rL_m$ from the
$q$-number identity
\begin{equation}
   [i{+}3]_q [j]_q - [3]_q[i{+}2]_q[j{+}1]_q + [3]_q[i{+}1]_q[j{+}2]_q -[i]_q[j{+}3]_q = 0,
\end{equation}
where  $i,j\in \mbZ$.

The module $\rV$ is not isomorphic to $\rW\oplus\rW$ as $\Uqa$-module since the action of $e_0$ in $\rV$ is not proportional to that of $f_1$.

The module $\rV$ is clearly not a tensor product of two non-trivial $\Uq$-modules.
\end{example}
Such an extension is also unique up to isomorphism, cf. Proposition \ref{prop:ext_two_strings_prod}.
\begin{prop}\label{prop:ev_self_ext_un} Any non-trivial extension $\rV$ of $[\al,\bt]$ by  $[\al,\bt]$ is isomorphic to the extension described in Example \ref{ex:ev_self_ext}.
\end{prop}
\begin{proof}
    Let $\rV$ be a non-trivial extension of $[\al, \bt]$ by itself. Denote $m = (\bt-\al)/2+1, a = (\bt+\al)/2$.  Choose a basis $v_0, w_0$  of the space of $\Uq$-singular vectors in $\rV$ such that $v_0$ belongs to the $\Uqa$-submodule $[\al,\bt]$. Acting by $f_1$ on $v_0$ and $w_0$, we obtain a basis $v_0,\dots, v_m, w_0,\dots, w_m$ of $\rV$ given by  $v_i = \frac{f_1^iv_0}{[i]_q!}, w_i = \frac{f_1^iw_0}{[i]_q!}$ for all $i$. By a weight consideration $e_0w_0 = q^aw_1 + \lambda v_1$ for some $\lambda \in \mbC$. Applying $\frac{f_1^i}{[i]_q!}$ to both sides, we obtain $e_0w_i = [i+1]_q(q^aw_{i+1} + \lambda v_{i+1})$ for each $i$. Similarly, there exists $\mu \in \mbC$ such that $f_0w_i = [m-i+1]_q(q^{-a}w_{i-1} + \mu v_{i-1})$ for each $i$. Then we have 
    $$[m]_qw_0 =\frac{(K-K^{-1})}{q-q^{-1}}w_0 = (f_0e_0-e_0f_0)w_0 = f_0e_0w_0 = [m]_{q}(w_0 + (q^{-a}\lambda+ q^{a}\mu)v_0).$$
    This gives $\mu = -q^{-2a}\lambda$. If $\lambda = 0$, the extension splits. Otherwise, by replacing $w_0$ by $q^{a} w_0/\la$ we are reduced to $\mu=-q^{-a}, \lambda = q^{a}$ as in Example \ref{ex:ev_self_ext}.
\end{proof}

We also give an example of a parametric family of extensions.
\begin{prop}\label{prop:ext_family}
    Let $\rW = \rU = ab$ be an irreducible product of two 
    two-dimensional evaluation modules, i.e. $a, b\in \mbC$ and $a-b \notin\{-2, 2\}$. Then the set of non-trivial extensions of $\rW$ by $\rU$ is identified with $\mbC\mathbb{P}^1$.
\end{prop}    
\begin{proof} 
Let $\rV$ be a non-trivial extension of $\rW$ by $\rU$. Fix a decomposition of $\Uq$-modules, $\rV = \rU \oplus \tilde{\rW}$, where $\rU \cong \tilde{\rW} \cong \rL_2 \oplus \mbC$. 

Let $v_0, w_0$ be highest weight vectors of $\rU$ and $\tilde{\rW}$ correspondingly. Let $v \in \rU$ be the unique non-zero vector such that $v = e_0 v_0 - \alpha f_1v_0$ and $e_1v = 0$. Let $w\in\rV$ be the unique non-zero vector such that $w = e_0 w_0 - \alpha f_1w_0 - \beta f_1v_0$, where $\bt\in \mbC$ is such that $e_1w = 0$. The vectors $v,w$ exist because $a-b \neq \pm 2$. Note that $\alpha$ in the definitions of $v$ and $w$ is the same. 
Note also that $v$ and $w$ are linearly independent. Then $B = \{v_0, f_1v_0, f_1^{(2)}v_0, v, w_0, f_1w_0, f_1^{(2)}w_0, w\}$ is a basis of $\rV$.  The basis $B$ is unique up to a choice of the highest weight vectors  $v_0 \mapsto \lambda v_0, w_0\mapsto \mu w_0+\nu v_0$, here $\lambda,\mu \in \mbC^{\times}, \nu\in\mbC$.

    The structure of extension is uniquely defined by the action of $e_0$ and $f_0$ in $B$. By the construction, $e_0w_0 = w + \alpha f_1w_0 + \beta f_1v_0$.  Furthermore, $e_0w = \xi f_1^{(2)}w_0 + \gamma f_1^{(2)}v_0$. Since $[e_0,f_1]=0$ that determines the action of $e_0$ on other basis vectors. Then the action of $f_0$ is uniquely recovered from  the commutator $[f_0, e_0] = \frac{K-K^{-1}}{q-q^{-1}}$.
    
    Constants $\al, \xi$ are uniquely fixed by the structure of quotient module $\rW$. After computing them one obtains the action of $f_0$ and $e_0$ in the basis $B$.
    \begin{subequations}
    \begin{align*}
        &e_0(f_1^{(i)}v_0) = \frac{(q^a+q^b)[i{+}1]}{[2]}f_{1}^{(i{+}1)}v_0 + \delta_{0,i}v, &e_0v = \nu_{a,b}f^{(2)}v_0,\hspace{54pt} \\
        &e_0(f_1^{(i)}w_0) = \frac{(q^a+q^b)[i+1]}{[2]}f_{1}^{(i+1)}w_0 + \delta_{0,i}w + \beta[i{+}1] f_{1}^{(i+1)}v_0,&e_0w = \nu_{a,b}f_{1}^{(2)}w_0 + \gamma f_1^{(2)}v_0,\\
    &f_0(f_1^{(i+1)}v_0) = \frac{(q^{-a} + q^{-b})}{[i+1]}f_1^{(i)}v_0 + q^{-a-b}\delta_{i,1}v,\; &f_0v= q^{-a-b}\nu_{a,b}v_0,\hspace{46pt}
    \end{align*}
    \begin{align*}
        &f_0(f_1^{(i+1)}w_0) = \frac{1}{[i+1][2]}\Big(\beta (q^{-a-b}[2]^2 -2(q^{-2a} + q^{-2b})) - \gamma q^{-2a-2b}(q^a + q^b)\Big)f_1^{(i)}v_0 + \\ 
        &\hspace{130pt}+ \frac{(q^{-a} + q^{-b})}{[i+1]}f_{1}^{(i)}w_0 +\delta_{i,1}\Big(q^{-a-b}w - \frac{q^{-2a-2b}(2(q^a+q^b)\beta + \gamma )}{[2]}v \Big),\\
       &f_0(w) = q^{{-}a{-}b}\nu_{a,b}w_0 +  \frac{(q^{-a} + q^{-b})}{[2]}\Big(-2q^{-a-b}\beta \nu_{a,b} + \gamma   \frac{(q^{{-}a}+q^{{-}b})}{[2]}\Big)v_0.
    \end{align*}
    \end{subequations}
    Here $\nu_{a,b} = \frac{1}{[2]}(q^{a+1}-q^{b-1})(q^{b+1}-q^{a-1})$.
    
   A direct computation shows that the relation $[e_1, f_0] = 0$ is satisfied for any $\beta, \gamma$. In this representation Serre relations follow from the rest of relations of $\Uqa$. Therefore, for all values $\beta, \gamma$ we obtained an extension of $\rW$ by $\rU$. Clearly, the value $\beta = \gamma = 0$ gives the direct sum $\rW \oplus\rU$.
    
    The constants $\beta, \gamma$ are simultaneously rescaled as we change the choice of basis to $\{v_0, f_1v_0, f_1^{(2)}v_0, v,$ $ \lambda w_0, \la f_1w_0, \la f_1^{(2)}w_0, \la w\}$ for any $\la\in\mbC^{\times}$. Therefore, all extensions with the same ratio $(\beta : \gamma)$ are isomorphic.
    
    For different values of ratio $(\beta:\gamma)$ are non-isomorphic. Indeed, the only non-canonical choice in the construction is a choice of the vector $w_0$. Any other choice of split gives a basis $\{v_0, f_1v_0, f_1^{(2)}v_0, v, w_0 + \mu v_0, f_1(w_0+\mu v_0), f_1^{(2)}(w_0 + \mu v_0), w + \mu v\}$, which does not affect $\beta$ and $\gamma$. Henceforth, the isomorphism class of an extension is uniquely determined by the ratio $(\beta: \gamma)$.

\end{proof}

\begin{remark} There are two natural extensions of $ab$ by itself. For any $c\in\mbC$ let $c\tilde{\oplus}c$ be the unique non-trivial extension of two-dimensional module $c$ by itself. Then $(a\tilde{\oplus}a) \otimes b$ and $a \otimes (b\tilde{\oplus}b)$ are two extensions of $ab$ by itself. It is a direct check that they are non-trivial and non-isomorphic. Therefore, all extensions of $ab$ by itself described by Proposition \ref{prop:ext_family} are linearly combinations of extensions $(a\tilde{\oplus}a) \otimes b$ and $a \otimes (b\tilde{\oplus}b)$.
\end{remark}

\subsection{Necessary conditions for existence of non-trivial extensions.}
There is a sufficient condition on the absence of extensions between two modules given in terms of $\ell$-weights. Recall the affine root $A_{a} = 1_{a-1}1_{a+1}$.
\begin{lemma}\label{lemma:noext_root}
    Let $\rW$ and $\rU$ be two $\Uqa$-modules. If for any two monomomials $m_1, m_2$ in $q$-characters of $\rW$ and $\rU$ respectively for any $a$ 
    \begin{equation}\label{eq:weight_root_sep}
        \frac{m_1}{m_2} \notin \{A_a, A_{a}^{-1},1\},
    \end{equation} 
    then there is no non-trivial extension of $\rW$ by $\rU$.
\end{lemma}
\begin{proof}
    Let $\rV$ be an extension of $\rW$ by $\rU$. Let $M_{\rU}, M_{\rW}$ be sets of $\ell$-weights appearing in $q$-characters of $\rU$ and $\rW$ respectively. By \eqref{eq:weight_root_sep} we have $M_{\rU}\cap M_{\rW} = \varnothing$. Let $\tilde{\rW} = \mathop{\oplus}\limits_{\mu \in M_{\rW}}\rV[\mu]$, where $\rV[\mu]$ is a generalized $\ell$-weight space of the weight $\mu$ in $\rV$. 
    By Proposition \ref{prop:ellroot_action}, for any mode $x_{r}^{\pm}$, for any $\ell$-weight $\mu$ we have $x_{r}^{\pm}(\rV[\mu]) \subset \mathop{\oplus}\limits_{\epsilon\in \{1,-1\},\, a}\rV[A_a^{\epsilon}\mu]$. Since for any $\mu \in M_{\rW}$ and for any $a\in\mbC$ we have $A_{a}^{\pm1}\mu \in M_{\rW}$, we have  $x_{r}^{\pm}(\tilde{\rW})\subset \tilde{\rW}$. Henceforth, $\tilde{\rW}$ is a submodule of $\rV$. Since $\rU = \mathop{\oplus}\limits_{\substack{\mu \in M_{\rU}}}\rV[\mu]$ the submodule $\tilde{\rW}$ completes $\rU$ to $\rV$.
\end{proof}

Now we are ready to classify the extensions of irreducible evaluation modules.
\begin{theorem}\label{thm:eval_ext} Any non-trivial extension $\rV$ of $\rV_1 = [\alpha_1, \beta_1]$ by $\rV_2 = [\alpha_2, \beta_2]$ is isomorphic to one of the following
    \begin{enumerate}
        \item $\rV =[\al_1, \bt_1-2]\bt_1$, then $\alpha_1 = \alpha_2, \beta_2 = \beta_1 - 4$,
        \item $\rV = (\bt_1+4)[\al_1, \bt_1+2]$, then $\alpha_1 = \alpha_2, \beta_2 = \beta_1 + 4$,
        \item $\rV = \al_1[\al_1+2, \bt_1]$, then $\alpha_2 = \alpha_1+4, \beta_2 = \beta_1$,
        \item $\rV = [\al_1-2, \bt_1](\al_1-4)$, then $\alpha_2 = \alpha_1-4, \beta_2 = \beta_1$,
        \item $\rV$ as in Example \ref{ex:ev_self_ext}, $\alpha_1 = \alpha_2, \beta_1 = \beta_2$, and $\rV_1 = \rV_2 \neq \mbC$.
    \end{enumerate}
\end{theorem}
\begin{proof}
In the cases listed in the proposition the non-trivial extensions are unique by Proposition \ref{prop:ext_two_strings_prod} (first four items) and Proposition \ref{prop:ev_self_ext_un} (the last item). 

    By Lemma \ref{lemma:ext_dual_aut}, the non-trivial extensions of 
$\rV_1 = [\alpha_1, \beta_1]$ by $\rV_2 = [\alpha_2, \beta_2]$ correspond to 
    non-trivial extensions of $\rV_2 = [\alpha_2, \beta_2]$ by $\rV_1 = [\alpha_1, \beta_1]$.

We use Lemma \ref{lemma:noext_root} to prove that there are no other extensions as follows. If an extension of $\rV_2$ by $\rV_1$ is non-trivial then there exists a pair of monomials $m_1, m_2$ in $q$-characters of $\rV_1, \rV_2$ respectively and $a \in \mbC$ such that $\frac{m_1}{m_2}\in \{A_a, A_a^{-1},1\}$. Any monomial $m$ in an irreducible evaluation module has the form $m=1_a1_{a+2}\dots,1_b1_{b+4}^{-1}1_{b+6}^{-1}\dots 1_c^{-1}$. Then $mA_{d}^{-1}$ has the same form if and only if $d\in\{a+1, b+1,c+3\}$. 

The case $d=b+1$ corresponds to the last item of the proposition, the case $d=a+1$ to the third and fourth items,   the case $d=c+3$ to the first and second items.

    Finally, $\chi_q(\rV_1)$ and $\chi_q(\rV_2)$ share a monomial if and only if $\rV_1\cong \rV_2$ which corresponds to the last item.

\end{proof}

\section{The graphs of modules.}\label{sec:mod_graphs}
In this section we study the structure of modules over $\Uqa$. We have complete information about composition factors from the $q$-characters. The composition factors are organized by the socle filtration. Assuming that each socle is multiplicity free, we enhance this information with a graph structure.

This graph structure is our way to visualize how the composition factors are glued together. 
One question we aim at is to describe $q$-characters of submodules of a given module.

\subsection{Socles.}
Recall that the socle of a module $\rV$ is the maximal semi-simple submodule. We denote the socle of a module $\rV$ by $\soc(\rV)$. 
For $i\in\mathbb{Z}_{\geq 0}$ we define the $i$-th socle of $\rV$  by $\soc_0(\rV)=0$,  $\soc_1(\rV)=\soc(\rV)$, and for $i>1$,
\begin{equation}
    \soc_{i+1}(\rV) = \soc_{i}\left(\rV/\soc(\rV)\right),
\end{equation}

The sequence $(\soc_{i}(\rV))_{i=1}^{\infty}$ is called the socle filtration of $\rV$. The socle filtration  is an invariant of a module. That is if the socle filtrations of $\rV$ and of $\rW$ are different, then $\rV$ and $\rW$ are not isomorphic. The notions of socles and socle filtration are standard, see \cite{barbasch1983filtrations}, \cite{irving1988socle}, \cite{humphreys2021representations}.
Some of the statements in this subsection should be known to experts. 

We call maximal $i$ such that $\soc_i(\rV)\neq 0$ the height of $\rV$. We denote the height of a module $\rV$ by $ht(\rV)$.

 There is an alternative convenient way to describe socles. We call a chain $\mathcal{F} = (0\subset \mclf_1 \subset \dots \subset \mclf_n \subset \dots)$, a filtration of a module $\rV$ if $\mclf_i$ are submodules, and $\mclf_n = \rV$ for $n \gg 1$. We call a filtration $\mathcal{F}$ \  semi-simple if all consequent quotients $\mclf_{i+1}/\mclf_i$ are semi-simple.

We say a filtration $\mathcal{F} = (0\subset \mclf_1 \subset \dots \subset \mclf_n \subset \dots)$ is an extension  of  a filtration $\mathcal{F'} = (0\subset \mclf_1' \subset \dots \subset \mclf_n' \subset \dots)$ if for each $i$, the submodule $\mathcal F_i'$ coincides with $\mathcal F_j$ for some $j$. 

The following lemma is straightforward.
\begin{lemma}  Every module $\rV$ has a semi-simple filtration. Moreover, any filtration of $\rV$ can be extended to a semi-simple filtration.  \qed
\end{lemma}
 The semi-simple filtrations can be added.
 \begin{lemma}\label{lemma:ss_filtr_addition}
 Assume that $\mclf, \tilde{\mclf}$ are two semi-simple filtrations of $\rV$. Then $\mclf + \tilde{\mclf}$ defined by $(\mclf + \tilde{\mclf})_i= \mclf_i + \tilde{\mclf}_i$ is a semi-simple filtration of $\rV$.
 \end{lemma}
\begin{proof}
    There is a natural surjection
    \begin{equation}
        (\mclf_{i+1}/\mclf_i) \oplus (\tilde{\mclf}_{i+1}/\tilde{\mclf}_i) \twoheadrightarrow  (\mclf_{i+1} + \tilde{\mclf}_{i+1})/(\mclf_i + \tilde{\mclf}_i),
    \end{equation}
    therefore, $(\mclf_{i+1} + \tilde{\mclf}_{i+1})/(\mclf_i + \tilde{\mclf}_i)$ is semi-simple as a quotient of semi-simple module $\mclf_{i+1}/\mclf_i \oplus \tilde{\mclf}_{i+1}/\tilde{\mclf}_i$.
\end{proof}

A filtration $\mathcal{F}(\rV)$ is called maximal if it is semi-simple and if for any semi-simple filtration $\mclf'$ of $\rV$, $\mclf_i' \subset \mathcal{F}_i(\rV)$.

\begin{cor}\label{cor}
    For any module $\rV$ there exists a unique maximal filtration. 
\end{cor}
\begin{proof}
   The existence and uniqueness of maximal filtration follows from Lemma \ref{lemma:ss_filtr_addition}. 
\end{proof}
We denote the maximal filtration of a module $\rV$ by $\mcls(\rV)$.
\begin{lemma}\label{lemma:filtration_quot} We have
\begin{align}
    \mcls_{i}\left(\rV/\mcls_{j}(\rV)\right) &= \mcls_{i+j}(\rV)/\mcls_{j}(\rV),\label{eq:filtr_quot1}\\
    \soc_i(\rV)&\cong\mcls_i(\rV)/\mcls_{i-1}(\rV),\label{eq:soc_alt}\\
    \soc_{i}(\rV/\mcls_j(\rV)) &\cong \soc_{i+j}(\rV).\label{eq:soc_quot}
\end{align}
\end{lemma}
\begin{proof}
The filtration given by the right hand side of Equation \eqref{eq:filtr_quot1} is semi-simple. Indeed,
\begin{equation*}\left(\mcls_{i+j+1}(\rV)/\mcls_{j}(\rV)\right)/\left(\mcls_{i+j}(\rV)/\mcls_{j}(\rV)\right) \cong \mcls_{i+j+1}(\rV)/\mcls_{i+j}(\rV).
\end{equation*}Therefore, $\mcls_{i+j}(\rV)/\mcls_{j}(\rV) \subset \mcls_{i}(\rV/\mcls_j(\rV))$.

Let $\pi_j$ be the canonical projection $\pi_j: \rV \twoheadrightarrow \rV/\mcls_j(\rV)$. Then 
\begin{equation*}0 \subset \mcls_1(\rV) \subset \dots \subset \mcls_j(\rV) \subset \pi_{j}^{-1}(\mcls_{1}(\rV/\mcls_j(\rV))) \subset \pi_{j}^{-1}(\mcls_{2}(\rV/\mcls_j(\rV))) \subset \dots,
\end{equation*}is a semi-simple filtration for $\rV$, which implies $\pi_{j}^{-1}(\mcls_{i}(\rV/\mcls_j(\rV))) \subset \mcls_{i+j}(\rV)$, therefore $\mcls_{i}(\rV/\mcls_j(\rV)) \subset \mcls_{i+j}(\rV)/\mcls_{j}(\rV)$. Therefore\eqref{eq:filtr_quot1} follows.

We prove \eqref{eq:soc_alt} by induction on $i$. 
\begin{multline}\soc_i(\rV) = \soc_{i-1}(\rV/\mcls_1(\rV)) = \mcls_{i-1}(\rV/\mcls_1(\rV))/\mcls_{i-2}(\rV/\mcls_1(\rV)) =\\= (\mcls_{i}(\rV)/\mcls_1(\rV))/(\mcls_{i-1}(\rV)/\mcls_1(\rV)) \cong \mcls_i(\rV)/\mcls_{i-1}(\rV).\end{multline}
Then \eqref{eq:soc_quot} follows since
\begin{multline}
  \soc_{i}(\rV/\mcls_j(\rV)) \cong \mcls_{i}(\rV/\mcls_j(\rV))/\mcls_{i-1}(\rV/\mcls_j(\rV)) =\\= (\mcls_{i+j}(\rV)/\mcls_j(\rV))/(\mcls_{i+j-1}(\rV)/\mcls_j(\rV)) \cong  \mcls_{i+j}(\rV)/\mcls_{i+j-1}(\rV) = \soc_{i+j}(\rV).
\end{multline}

\end{proof}

 The maximal filtration is compatible with submodules.

\begin{lemma}\label{lemma:max_submod_filtration}
    Let $\rW \subset \rV$ be a submodule of $\rV$. Then the maximal filtration of $\rW$ is given by $\mcls_i(\rW) = \mcls_i(\rV) \cap \rW$.
\end{lemma}
\begin{proof}
We have $(\mcls_{i+1}(\rV) \cap \rW) / (\mcls_i(\rV) \cap \rW) \subset \mcls_{i+1}(\rV)/\mcls_{i}(\rV)$, therefore $(\mcls_{i+1}(\rV) \cap \rW) / (\mcls_i(\rV) \cap \rW)$ is a semi-simple module as a submodule of a semi-simple module. Hence $0\subset \mcls_{1}(\rV) \cap \rW \subset \dots \subset \mcls_{i}(\rV) \cap \rW \subset \dots$ is a semi-simple filtration of $\rW$ and $\mcls_{i}(\rV) \cap \rW \subset \mcls_{i}(\rW)$.

Let $n$ be a number such that $\mcls_n(\rW) = \rW$. There exists a semi-simple filtration $\mclf$ of $\rV$ such that $\mclf_i = \mcls_i(\rW)$ for $i \leq n$. Therefore,  $\mcls_i(\rW) \subset \mcls_i(\rV)$, which implies $\mcls_i(\rW) \subset \mcls_{i}(\rV) \cap \rW$.
\end{proof}

\begin{cor}\label{cor:subsoc}
    If $\rW \subset \rV$, then there is an embedding $\soc_{i}(\rW) \subset \soc_{i}(\rV)$ for all $i$.
\end{cor}

In particular, if $\rW_1,\rW_2\subset \rV$ are submodules, then $\soc_i(\rW_1)+\soc_i (\rW_2)\subset \soc_i(\rW_1+\rW_2)$ as submodules of $\soc_i(\rV)$. However, this inclusion does not have to be equality.
\begin{example}\label{ex:soc_not_additive}
    Let $\rV = 20 \oplus \mbC$. Fix embeddings $\iota_{20}:20 \hookrightarrow \rV$ and $\iota_{\mbC}:\mbC \hookrightarrow \rV$.
    Let $\pi:\, 20 \twoheadrightarrow \mbC$
    
    be the canonical projection. Let $\rW_1 = \iota_{20}(20)$, $\rW_2 = (\iota_{20} + c\, \iota_{\mbC}\circ \pi)(20)$,  where $c$ is a non-zero complex number. We have $\rW_1\cong\rW_2\cong 20$. Clearly, $\rW_1 + \rW_2 = \rV$, hence $\soc(\rW_1 + \rW_2) = [0,2] \oplus \mbC$, but  $\soc(\rW_1) =\soc(\rW_2) = [0,2]\subset \rV$.
\end{example}

We have $\chi_q(\rV)=\sum_{i=1}^{ht(\rV)}\chi_q(\soc_i\rV)$. 
Define the graded $q$-character of $\rV$ by 
$$
\chi_{q,s}(\rV)=\sum_{i=1}^{ht(\rV)}s^i\chi_q(\soc_i(\rV)). 
$$
More generally, for a filtration $\mathcal{F}$ of $\rV$, define the corresponding filtered $q$-character by 

$$
\chi^{\mathcal{F}}_{q,s}(\rV)=\sum_{i=1}^{\infty}s^i\chi_q(\mathcal{F}_{i}/\mathcal{F}_{i-1}). 
$$
The graded $q$-character is the filtered $q$-character for the maximal filtration $\mcls(\rV)$ of $\rV$.

Let $\rV$ be a module. Let $\rR_{1},\dots,\rR_k$ be irreducible composition factors of $\rV$. 
Then the graded character of $\rV$ has the form $$\chi_{q,s}(\rV) = \sum\limits_{j=1}^{k}s^{h_j}\chi_{q}(\rR_j).$$
For any semi-simple filtration of a module, the filtered character is "larger" than the graded $q$-character in the following sense.

\begin{lemma}\label{lemma:TP_factors_drop} Let $\mathcal F$ be a semi-simple filtration of module $\rV$. The filtered $q$-character $\chi^{\mathcal{F}}_{q,s}(\rV)$ can be written in the form
\begin{equation}\label{filtered character}
\chi^{\mathcal{F}}_{q,s}(\rV) = \sum\limits_{j=1}^{k}s^{l_j}\chi_{q}(\rR_j),
\end{equation}where $l_j \geq h_j$ for $j=1, \dots, k$.
\end{lemma}
\begin{proof}
We have
    \begin{equation*}
        \chi_{q,s}^{\mclf}(\rV) = (1-s)\sum_{j=1}^{\infty}s^{j}\chi_q(\mclf_j).
    \end{equation*}
  Therefore,
    \begin{equation}\label{computation}
        \chi_{q,s}^{\mclf}(\rV) - \chi_{q,s}(\rV) = (1-s)\sum_{j=1}^{\infty}s^{j}(\chi_q(\mclf_j)-\chi_q(\mcls_j(\rV))) = (s-1)\sum_{j=1}^{\infty}s^{j}\chi_q(\mcls_j(\rV)/\mclf_j).
    \end{equation}
Let $\rM_1,\dots,\rM_t$ be the distinct irreducible modules in the sequence $\rR_1,\dots,\rR_k$. 
We have $\chi_{q,s}^{\mathcal F}(\rV)=\sum_{i=1}^t b_i(s)\chi_q(\rM_s)$ and  $\chi_{q,s}(\rV)=\sum_{i=1}^t c_i(s)\chi_q(\rM_s)$, where  
$b_i(s), c_i(s)$ are polynomials in $s$ with non-negative integer coefficients.
    
Since $b(1)=c(1)$, the polynomial $b_i(s)-c_i(s)$ is divisible by $s-1$. By \eqref{computation}  the quotient has non-negative integer coefficients. Let $b_i(s)=s^{\al_1}+\dots +s^{\al_r}$ with $\al_1\geq\dots \geq \al_r$ and  $c_i(s)=s^{\bar \al_1}+\dots +s^{\bar \al_r}$ with $\bar \al_1\geq\dots \geq \bar \al_r$. Then
$$
\frac{b_i(s)-c_i(s)}{s-1}=\sum_{j=1}^r \frac{s^{\al_j}-1}{s-1}- \sum_{j=1}^r \frac{ s^{\bar \al_j}-1}{s-1}=
 \sum_{j=1}^r\sum_{l=0}^{\al_j-1} s^l -\sum_{j=1}^r\sum_{l=0}^{\bar \al_j-1} s^l \in \mathbb{Z}_{\geq 0}[s].
$$
It follows that $\al_j\geq \bar \al_j$ for $j=1,\dots, r$.
\end{proof}

Graded $q$-characters are not compatible with taking sums or tensor products. However, graded characters are compatible with direct sums, and graded character of a submodule is a subcharacter of graded character of a module. We also have $\chi_{q,s}(\rV/S_i(\rV))=s^{-i}(\chi_{q,s}(\rV)-\chi_{q,s}(S_i(\rV)))$. Finally, $\deg_s(\chi_{q,s}(\rV))=ht(\rV).$

Instead of socles, one can similarly consider heads. The two approaches are related by taking duals. We note that socles are not compatible with taking duals. 
\begin{lemma} If $\rV$ is a module of height $h$ then $(\soc_h(\rV))^*\subset \soc (\rV^*)$. 
\end{lemma} 
\begin{proof}
    We have a surjective map $\rV\to \soc_h(\rV)=\rV/\mathcal{S}_{h-1}(\rV) $. Therefore, we have an injective map  $(\soc_h(\rV))^*\to \rV^*$. The socle $\soc_h(\rV)$ is semi-simple, thus the dual module $(\soc_h(\rV))^*$ is semi-simple, and therefore the image of the injective map is in $\soc(\rV^*)$.
\end{proof}

But the inclusion in the lemma can be proper, see Example \ref{ex:weight_graph} below.

\subsection{Definition and first properties of graphs of modules.}\label{subsec:graphs_of_mods}
A module $\rV$ is called multiplicity free if all composition factors of $\rV$ are non-isomorphic.

\begin{defi}
    We call a module $\rV$ socle-multiplicity free if for all $i$, $\soc_i(\rV)$ is multiplicity free.
\end{defi}

By Corollary \ref{cor:subsoc}, if $\rW \subset \rV$ is a submodule and $\rV$ is socle-multiplicity free, then $\rW$ is also socle-multiplicity free.  Also, by Lemma \ref{lemma:filtration_quot}, for any $i$, $\rV/\mcls_i(\rV)$ is socle-multiplicity free.  

A quotient of a socle-multiplicity free module does not have to be socle-multiplicity free. For example, the quotient of socle-multiplicity free module $20\oplus \mbC$ by $[0,2]$ is $\mbC\oplus \mbC$, cf. Example \ref{ex:soc_not_additive}.

A module dual to a socle-multiplicity free module does not have to be socle-multiplicity free. For example, the dual of socle-multiplicity free module $2024$ is not socle-multiplicity free,  cf. Example \ref{ex:mod_graphs_4let}.

Let $\rV$ be a socle-multiplicity free $\Uqa$-module. Let $\rR_1,\dots, \rR_k$ be composition factors of $\rV$. We fix an identification of the multiset  $\{\rR_1,\dots, \rR_k\}$ with the union of sets of composition factors of all socles $\soc_i(\rV)$. Then we say that $\rR_i$ has degree $j$, $h_i=\deg(\rR_i)=j$, if $\rR_i$ is identified with a composition factor of $\soc_j(\rV)$. 
There exists a composition factor $\rR_i$ of degree $j$ if and only if there exists a submodule $\rW\subset \rV$ of height $j$ such that $\soc_j (\rW)=\rR_i$.

The module $\rV$ is socle-multiplicity free, therefore if $\rR_i\cong \rR_j$ then degrees are different, $h_i\neq h_j$.

Let $\rW \subset \rV$ be a submodule and let $(\tilde{\rR}_i, \tilde{h}_{i})$ be the composition factors of $\rW$ and their heights in $\rW$. By Corollary \ref{cor:subsoc}, the set of pairs $\{(\tilde{\rR}_i, \tilde{h}_{i})\}$ is naturally identified with a subset of the set of pairs $\{(\rR_i, h_i)\}$.

Let $G$ be a directed graph with vertices labelled by pairs $(\rR_i,h_i)$ of composition factors  of $\rV$ paired with their degrees and with a set of edges such that $G$ has no (non-oriented) cycles and if there is an edge from  
$(\rR_i,h_i)$ to $(\rR_j,h_j)$ then $h_i>h_j$.

A subgraph $H\subset G$ is a directed graph obtained from $G$ by removing several vertices and those edges which were attached to the removed vertices. A subgraph $H\subset G$ is called closed if for any vertex $(\rR_i,h_i)\in H$ and any edge in $G$ from  $(\rR_i,h_i)$ to $(\rR_j,h_j)$ we have $(\rR_j,h_j)\in H$. 

The graded $q$-character of a subgraph of $H\subset G$ is defined by
$$
\chi_{q,s}(H)=\sum_{(\rR_i,h_i)\in H} s^{h_i}\chi_q(\rR_i).
$$

Note that since $\rV$ is socle-multiplicity free, distinct subgraphs of $G$ have distinct graded $q$-characters: 
$\chi_{q,s}(H_1)=\chi_{q,s}(H_2)$ if and only if $H_1=H_2$.

\begin{defi}
An edge in $G$ connecting $(\rR_i,h_i)$ and $(\rR_j,h_j)$ is called admissible, if
any submodule $\rW\subset \rV$ with a composition factor $\rR_i$ of degree $h_i$ also has a composition factor $\rR_j$ of degree $h_j$.

  The graph $G$ is called a graph of the module $\rV$, if all edges of $G$ are admissible.
\end{defi}

We list a few trivial properties of the graphs of $\rV$.

\begin{lemma}\label{trivial lem}  Let $G$ be a graph of a module $\rV$ .  Then
\begin{enumerate}
    \item A graph obtained by removing any edge of $G$ is a graph of module $\rV$. In particular, a graph with no edges is a graph of $\rV$ which we call trivial.

    \item  If there is an edge from $(\rR_i,h_i)$ to $(\rR_j,h_j)$ in $G$ then $h_i-h_j>0$. We call $h_i-h_j$ the length of the edge.

\item If there is a directed path in $G$ connecting $(\rR_i,h_i)$ to  $(\rR_j,h_j)$, then the edge from $(\rR_i,h_i)$ to  $(\rR_j,h_j)$ is admissible.
\item\label{item:submod_subgraph} If $\rW\subset \rV$ is a submodule, then there exists a unique subgraph $H\subset G$ such that $\chi_{q,s}(H)=\chi_{q,s}(\rW)$. Moreover, this subgraph $H$ is closed and $H$ is a graph of $\rW$. 
\end{enumerate}
  \qed
\end{lemma}

Let $\tilde G(\rV)$ be the graph which contains all admissible edges and no other edges. 
We reduce it to a graph $G(\rV)$ of $\rV$ by removing edges as follows. We keep all edges of length one. We remove any edge from a vertex $(\rR_i,h)$ to a vertex $(\rR_j,h-2)$  of length two if there is a path from $(\rR_i,h)$ to $(\rR_j,h-2)$ composed of two edges of length one. Then we remove any edge from a vertex $(\rR_i,h)$ to a vertex $(\rR_j,h-3)$ of length three whenever there is a path from $(\rR_i,h)$ to $(\rR_j,h-3)$ composed of the remaining edges of lengths one and two. Continuing the process, we obtain the graph $G(\rV)$.

\begin{defi}
    We call the graph $G(\rV)$ the strong graph of $\rV$. We also call the strong graph $G(\rV)$ the submodule graph of $\rV$.
\end{defi}

The strong graph $G(\rV)$ as a non-directed graph has no cycles.

The main property of the strong graph $G(\rV)$ of $\rV$ is the following lemma.

\begin{lemma}  Let $G(\rV)$ be the strong graph of a module $\rV$.  Let $G$ be any graph of $\rV$. Then if there is directed path from $(\rR_i,h_i)$ to $(\rR_j,h_j)$ in $G$, there is directed path from $(\rR_i,h_i)$ to $(\rR_j,h_j)$ in $G(\rV)$. 

In particular, if $H\subset G$ is a connected subgraph, then the subgraph in $G(\rV)$ formed by vertices in $H$ is connected. Similarly, if $H\subset G(\rV)$ is a closed subgraph then  the subgraph in $G$ formed by vertices in $H$ is closed.
\end{lemma} 
\begin{proof}
If such path in $G$ exists, then the edge from $(\rR_i,h_i)$ to $(\rR_j,h_j)$  is admissible and therefore it exists in $\tilde G(\rV)$. It follows that either it exists in $G(\rV)$ or it was removed and then there is a path of edges in $G(\rV)$ of smaller length connecting these two vertices.
\end{proof} 

By Lemma \ref{trivial lem}, part (\ref{item:submod_subgraph}), any  submodule of $\rW\subset \rV$ corresponds to a unique closed subgraph $H\subset G(\rV)$ such that $\chi_{q,s}(\rW)=\chi_{q,s}(H)$.

\begin{defi}
    A socle-multiplicity free module $\rV$ is called understandable if for any close subgraph of $G(\rV)$ there is a submodule $\rW\subset \rV$  such that $\chi_{q,s}(\rW)=\chi_{q,s}(H)$.
\end{defi}

In other words, we call a module understandable if there is a bijection between closed subgraphs $H\subset G(\rV)$  and distinct graded $q$-characters of submodules $\rW\subset \rV$.

Thus, a strong graph of an understandable module describes the socles of all submodules of $\rV$.  In particular, it describes the classes of all submodules of $\rV$  in the Grothendieck ring.

If a multiplicity free module $\rV$ is semi-simple then $\rV$ is understandable, and the strong graph $G(\rV)$ is just the graph with no edges. 

In our examples of graphs below, all modules are understandable.

\subsection{The case of multiplicity free.}\label{subsec:graph_mult_free} In this section we assume that in the (non-graded) $q$-character $\chi_q(\rV)$ each dominant monomial has coefficient at most one. In this case the strong graph has good properties which are easy to prove.

Clearly, if each dominant monomial in $\chi_q(\rV)$ has coefficient at most one, then all composition factors $\rR_1,\dots, \rR_k$ of $\rV$ are distinct, that is $\rV$ is multiplicity free.

Let $m_i$ be the dominant top monomials of $\rR_i$. By the assumption $m_i$ has multiplicity one in $\chi_q(\rV)$. Recall that $m_i$ is a rational function of $z$. For each $1\le i \le k$, the module $\rV$ has a unique up to a scalar non-zero vector $v_i$ of $\ell$-weight $m_i$,  $\psi^\pm(z) v_i =m_i v_i$. 

Let $\overline{\rR}_i=\Uqa\, v_i \subset \rV$ be submodules of $\rV$ generated by $v_i$. 

We now construct the  graph $\tilde G(\rV)$ of $\rV$. The vertices of $\tilde G(\rV)$ are labelled by $\rR_i$. We drop the degrees since they are uniquely determined by $\rR_i$. We draw a directed edge from $\rR_i$ to $\rR_j$ if and only if $\rR_j$ is a composition factor of $\overline{\rR}_i$ or, equivalently, if and only if $\overline{\rR}_j\subset \overline{\rR}_i$. 

As before, we remove excessive edges to eliminate non-directed cycles and obtain the  graph $\bar G(\rV)$.

Given a subgraph of $H\subset \bar G(\rV)$ define the submodule $\rW(H)\subset \rV$ by $\rW=\mathop{+}\limits_{\rR_i\in H} \overline{\rR}_i$. 

We have $\rW(H_1\cup H_2)=\rW(H_1)+\rW(H_2)$.

Note that for a general socle-multiplicity free modules there is no natural map from subgraphs to submodules. A subgraph may correspond to a family of modules with the same $q$-character. In particular, one may have two submodules $\rW_1$ and $\rW_2$ such that the graph of $\rW_1+\rW_2$ is larger then the union of graphs of $\rW_1$ and $\rW_2$, see Example \ref{ex:soc_not_additive}.

\begin{prop}
    There is a bijection between closed subgraphs of $\bar G(\rV)$ and submodules of $\rV$ sending $H\subset\bar  G(\rV)$ to $\rW(H)$. 

    Moreover, if $H_1,H_2\subset\bar  G(\rV)$ are closed subgraphs of $\bar G(\rV)$, then $H_1\cap H_2$ is a closed subgraph and $\rW(H_1\cap H_2)=\rW(H_1)\cap \rW(H_2)$. 
\end{prop}
\begin{proof}
A submodule $\rW\subset \rV$ has a composition factor  $\rR_i$ if and only if $\chi_q(\rW)$ contains $m_i$. In particular, a submodule $\rW\subset \rV$ has a composition factor  $\rR_i$ if and only if $\rW$ contains a vector of $\ell$-weight $m_i$.

Therefore, a submodule $\rW\subset \rV$ has a composition factor  $\rR_i$ if and only if $v_i\in \rW$. 

It follows that $\rW$ has composition factor $\rR_i$ if and only if $\overline{\rR}_i\subset \rW$. Therefore, $\rW \supset \mathop{+}\limits_i \bar\rR_i$, where the sum is taken over all $\rR_i$ in composition series of $\rW$. This inclusion is equality since  $\mathop{+}\limits_i \bar\rR_i$ exhausts composition series of $\rW$. 
Therefore, any submodule $\rW$ has the form $\rW(H)$, where  $H$ is the subgraph with vertices corresponding to the composition factors of $W$.

The subgraph $H$ is closed. Indeed, let  $\rR_i$ be a composition factor of $\rW$. If there is an edge from the $i$-th vertex to the $j$-th vertex, then $\overline{\rR}_i \supset \overline{\rR}_j$ and $\rR_j$ is also a composition factor of $\rW$.

 The subgraph $H$ is unique since vertices are recovered by the composition series of $\rW$.

\end{proof}

In particular, we obtain the following corollary.
\begin{cor}
    The graph $\bar G(\rV)$ is the strong graph of $\rV$, $\bar G(\rV)=G(\rV)$ and $\rV$ is understandable.
\end{cor}
Note that under our assumptions, any two distinct submodules of $\rV$ are not isomorphic and even have distinct $q$-characters. In particular, there is a finite number of submodules.

We illustrate our construction in an example.
\begin{example}\label{ex:weight_graph}
    Consider the module $w = 0246$. From the $q$-characters, the composition factors and the corresponding dominant monomials are 
    \begin{align*}
    (\rR_1,\rR_2,\rR_3,\rR_4,\rR_5)&=([0,6], [0,2], [4,6], 06, \mbC), \\
    (m_1,m_2,m_3,m_4,m_5)&=(1_01_21_41_6,1_01_2, 1_41_6,1_01_6,1).
    \end{align*}
    Computing $\Hom_{\Uqa} (\rR_i,w)$, we see that $\soc(w)=\rR_4\oplus \rR_5$. Taking the tensor product of submodule $\mbC \subset 02$ with $46$, we obtain $\mbC \subset 46 \subset w$. Analogously $\mbC \subset 02 \subset w$. Therefore, quotient $w/\mbC$  contains $[0,2] \oplus [4,6]$, which implies $\soc_2(w) \supset [0,2] \oplus [4,6]$, moreover, we have $\overline{\rR}_2=02$ and $\overline{\rR}_3=46$. Since $w$ is a Weyl module, $w$ is cyclic from the highest $\ell$-weight vector (see \cite{chari2002braid}), $\soc_3(w) = [0,6]$ and $\bar{R}_1 = w$. Comparing to the composition series, we obtain $\soc_2(w) = [0,2] \oplus[4,6]$. Therefore, the degrees of $\rR_i$ are  $\{3,2,2,1,1\}$, respectively. Reducing, we remove the edge from $\rR_1$ to $\rR_5$.   Thus, we have the graph $G(w)$ given by the following picture.
    
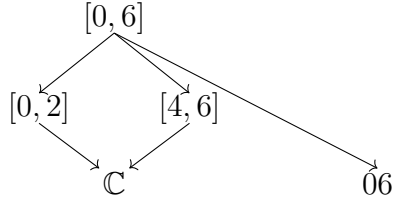
\begin{center}
    \begin{tikzpicture}
\node at (0,0.2) {$[0,6]$};
\node at (-1,-1) {$[0,2]$};
\node at (3.5,-2) {$06$};
\node at (1,-1) {$[4,6]$};
\node at (0,-2) {$\mbC$};
\draw[->] (0,0) to (-1,-0.8);
\draw[->] (0,0) to (1,-0.8);
\draw[->] (1,-1.2) to (0.2,-1.8);
\draw[->] (-1,-1.2) to (-0.2,-1.8);

\draw[->] (0,0) to (3.5,-1.8);
\end{tikzpicture}
\end{center}
We have four edges of length 1 and one edge of length 2.

The only proper non-trivial indecomposable submodules of $w$ are $\overline{\rR}_2, \overline{\rR}_3, \rR_4, \rR_5$ and $\overline{\rR}_2+\overline{\rR}_3$.
\end{example}

We note that all dominant monomials of a word $w$ have coefficient at most one if and only if for each letter $a\in 2\mbZ$ occurring in $w$ more than once, the letter $a+2$ is not present in $w$ at all. So, most interesting words do not have this property. 

\subsection{One-strong graphs.}
Practically, it is not clear how to compute strong graph of an $\Uqa$-module. In this section we give an alternative construction of a subgraph of strong graph obtained by removing all edges of the length greater than $1$. This construction can be used even by a computer program.

First, we study  modules of height two in more detail. We use  Krull-Schmidt theorem for $\Uqa$ modules. The Krull-Schmidt theorem asserts that any module is isomorphic to  a direct sum of indecomposable modules, and such  decomposition is unique up to permutation of summands, see \cite{etingof2011introduction}.
\begin{lemma}\label{lemma:nonss_compl}
Assume that $\rV/\soc(\rV)$ is simple. Let  $\rW \subset \rV$ be a non-semi-simple submodule. Then $$\rV = \rW \oplus \rM,$$ where $\rM$ is semi-simple.
\end{lemma}
\begin{proof}
    Let $\pi: \rV \twoheadrightarrow \rV/\soc(\rV)$  be the canonical projection. Then since $\rW \not \subset \soc(\rV)$, we have $\pi(\rW) \neq 0$, which due to Shur lemma gives $\pi(\rW) \cong \rV/\soc(\rV)$. This implies $\rW + \soc(\rV) = \rV$. Since $\soc(\rV)$ is semi-simple, we have $\soc(\rV) = \rW \cap \soc(\rV) \oplus \rM$ for some semi-simple $\rM$. This gives $\rV = \rW \oplus \rM$.
\end{proof}

\begin{cor}\label{cor:prop_sssubmod}
    Let $\rV$ be indecomposable and let $\rV/\soc(\rV)$ be simple. Then any proper submodule $\rM \subsetneqq \rV$ is contained in $\soc(\rV)$.
\end{cor}
\begin{proof}
    Assume the opposite, let $\rM$ be a proper non-semi-simple submodule. Then by Lemma \ref{lemma:nonss_compl} $\rV$ is a direct sum which is impossible since $\rV$ is indecomposable.
\end{proof}
\begin{lemma}\label{lemma:ind_submod}
    Let $\rV$ be indecomposable and let $\rV/\soc(\rV)$ be simple. Let $\rM\subsetneqq \rV$ be a submodule. If $\rV/\rM$ is semi-simple then $\rM = \soc(\rV)$.
\end{lemma}
\begin{proof}
     By Corollary \ref{cor:prop_sssubmod}, $\rM\subset \soc(\rV)$.  Assume $\rM\neq \soc(\rV)$, then $\rV/\rM=\rM_1\oplus \rM_2$ is a non-trivial direct sum. Let $\pi:\rV\to \rV/\rM$ be the canonical projection. Then by Corollary \ref{cor:prop_sssubmod}, $\pi^{-1}(\rM_i)\subset \soc(\rV)$. It follows that $\rV=\pi^{-1}(\rM_1)+ \pi^{-1}(\rM_2)\subset \soc (\rV)$ which is a contradiction.
\end{proof}

\begin{lemma}\label{lemma:intssqoutss}
    Assume that $\rV/\soc(\rV)$ is simple. Let $\rM_1, \rM_2 \subset \soc(\rV)$ be semi-simple submodules of $\rV$ such that $\rV/\rM_1$ and $\rV/\rM_2$ are semi-simple. Then $\rV/(\rM_1\cap \rM_2)$ is semi-simple.
\end{lemma}
\begin{proof}

By Krull-Schhmidt theorem for modules, $\rV=\rW\oplus \rN$ where $\rN\subset \soc(\rV)$ is semi-simple, and $\rW$ is an indecomposable module such that $\rW/\soc(\rW)\cong\rV/\soc(\rV)$ is simple. Taking quotients of all modules by $\rN$, we are reduced to the case $\rN=0$. 

Then by Lemma \ref{lemma:ind_submod}, $\rM_1=\rM_2=\soc(\rV)$. Therefore, $\rM_1\cap \rM_2=\soc(\rV)$.

\end{proof}

\begin{defi}
    For a module $\rV$ such that $\rV/\soc(\rV)$ is simple denote by $\rM(\rV)$ the minimal semi-simple submodule of $\rV$ such that $\rV/\rM(\rV)$ is semi-simple.  
\end{defi}
By Lemma \ref{lemma:intssqoutss}, $\rM(\rV)$ is unique. The module $\rM(\rV)$ can be described as follows.

\begin{prop}\label{prop:mv_as_icomp}
    Assume that $\rV/\soc(\rV)$ is simple and $\soc(\rV)$ is multiplicity free. Then $\rV$ admits a decomposition $\rV = \rW \oplus \rN$, where $\rN$ is semi-simple and $\rW$ is indecomposable of height $2$. Then $\rM(\rV) = \soc(\rW)$.
\end{prop}
\begin{proof}
    By the Krull-Schmidt theorem, there is a decomposition $\rV = \bigoplus\limits_{j = 1}^{N}\rW_{j}$, where $\rW_j$ are indecomposable modules, and classes of isomorphisms of $\rW_j$ are defined uniquely up to a permutation.
(
    Denote $\rR = \rV/\soc(\rV)$. Then, $\rR = \bigoplus\limits_{j = 1}^{N}\rW_j/\soc(\rW_{j})$, and since $\rR$ is simple all $\rW_j$ except for one are irreducible. This implies the decomposition $\rV = \rW \oplus \rN$.

    Note that $\rV/\soc(\rW) = \rW/\soc(\rW) \oplus \rN$, is semi-simple, which gives $\rM(\rV) \subset \soc(\rW)$. The module $\rW/\rM(\rV) \subset \rV/\rM(\rV)$, is semi-simple, hence by Lemma \ref{lemma:ind_submod} $\rM(\rV) = \soc(\rW)$.
    
\end{proof}

Note that in the setting of Proposition \ref{prop:mv_as_icomp} the submodule $\rW$ may be not unique, however, the submodule $\soc(\rW)$ does not depend on this choice, see Example \ref{ex:soc_not_additive}, in that example, $\rW\cong 20$.

\medskip

Recall, that for a socle-multiplicity free module $\rV$ with composition factors $\rR_i$ of degree $h_i$,  we consider directed graphs with vertices $(\rR_i,h_i)$.

\begin{defi}\label{def:one_strong_graph}
Let $\rV$ be a socle-multiplicity free module of height two.
Define the one-strong graph  $\Gamma(\rV)$ as follows. Given a vertex $(\rR_i, 2)$, let $\rW_i \subset\rV$ be the unique submodule such that $\soc(\rW_i)= \soc(\rV)$ and $\rW_i/\soc(\rV) \cong \rR_i$. The  graph $\Gamma(\rV)$ has an edge connecting vertex $(\rR_i, 2)$ to vertex $(\rR_j,1)$  if and only if $\rR_j\subset \rM(\rW_i)$.

Let $\rV$ be an arbitrary socle-multiplicity free module.
Define the one-strong graph $\Gamma(\rV)$ as follows. 

The graph $\Gamma(\rV)$ has an edge connecting vertex  
$(R_i,h_i)$ to vertex $(R_j,h_j)$ if and only if $h_j=h_i-1$ and 
if there is an edge connecting these two vertices in one-strong graph $\Gamma(\mcls_{h_i}(\rV)/\mcls_{h_i-2}(\rV))$ of the height two module $\mcls_{h_i}(\rV)/\mcls_{h_i-2}(\rV)$.
\end{defi}

Note that $ht(\mcls_{h_i}(\rV)/\mcls_{h_i-2}(\rV))=2$ since by $\mcls_{h_i}(\rV)/\mcls_{h_i-2}(\rV))\cong \mathcal{S}_2(\rV/\mcls_{h_i-2})$ by \eqref{eq:filtr_quot1}.
The module $\rW_i \subset \mcls_{h_i}(\rV)/\mcls_{h_i-2}(\rV)$ which decides the presence of edges from $(\rR_i,h_i)$ in the graph $\Gamma(\rV)$ is given by $\pi_{h_i}^{-1}(\rR_i)$, where 
$$\pi_k: \mcls_{k}(\rV)/\mcls_{k-2}(\rV) \twoheadrightarrow \mcls_{k}(\rV)/\mcls_{k-1}(\rV) = \soc_{k}(\rV)$$
is canonical projection. 

\begin{prop}\label{prop:submodsubgraph}
 Let $\rV$ be a socle-multiplicity free module. Let $\rW \subset \rV$ be a submodule. Then $\Gamma(\rW) \subset \Gamma(\rV)$ is a closed subgraph and embedding preserves the grading.
\end{prop}
\begin{proof}
    
Recall that by Corollary \ref{cor:subsoc}, vertices of  $\Gamma(\rW)$ are also vertices of $\Gamma(\rV)$.
    
   The edges between $k$th and $(k+1)$st components of $\Gamma(\rV)$ depend on module $\mcls_{k+1}(\rV)/\mcls_{k-1}(\rV)$ only. Since $\mcls_{k+1}(\rW)/\mcls_{k-1}(\rW) \subset \mcls_{k+1}(\rV)/\mcls_{k-1}(\rV)$ it is sufficient to reduce to the case $k = 1$ and $ht(\rV) = 2$. In that case $\mcls_{2}(\rV) = \rV$.

    Let $ht(\rV) = 2$. Let  $\pi_{\rW}: \rW \twoheadrightarrow \rW/\soc(\rW)\subset \rV/\soc(\rV),\;\;\; \pi_{\rV}: \rV \twoheadrightarrow \rV/\soc(\rV)$ be the canonical projections.  Then $\pi_\rW = \pi_\rV|_{\rW}$. Let  $\rR_i \subset \rW/\soc(\rW)$  be an irreducible submodule. Then $\pi_{\rW}^{-1}(\rR_i) \subset \pi_{\rV}^{-1}(\rR_i)$.  By Proposition \ref{prop:mv_as_icomp}, we have decomposition $\pi_{\rW}^{-1}(\rR_i) = \rU \oplus \rM$, where $\rU$ is indecomposable, $ht(\rU)=2$, and $\rM$ is semi-simple. By Lemma \ref{lemma:nonss_compl}  we have decomposition $\pi_{\rV}^{-1}(\rR_i) = \pi_{\rW}^{-1}(\rR_i)\oplus \rN=\rU \oplus \rM\oplus  \rN$, where $\rN$ is semi-simple. Hence, by Proposition \ref{prop:mv_as_icomp}, the edges from the vertex $(\rR_i,2)$ in both $\Gamma(\rV)$ and $\Gamma(\rW)$ go to the set of vertices of degree one labelled by irreducible components of $\soc(\rU)$.
\end{proof}

The following proposition explains our choice of the name "one-strong" for $\Gamma(\rV)$.
\begin{prop}
    The set of edges of one-strong graph $\Gamma(\rV)$ coincides with the set of all admissible edges of length one. In other words, edges of $\Gamma(\rV)$ are exactly the edges of length one in $\tilde{G}(\rV)$ or, 
    in $G(\rV)$.
\end{prop}
\begin{proof}
    We prove that an edge $e$ from $(\rR_{i}, h)$ to $(\rR_j, h-1)$ is not admissible if and only if it is not contained in $\Gamma(\rV)$. 
    
    Assume that the edge $e$ is not admissible. Then there is a submodule $\rW$, such that $\soc_{h}(\rW) \supset \rR_{i}$, but $\rR_{j} \not\subset \soc_{h-1}(\rW)$. Then $e$ is not an edge $\Gamma(\rW)$.
    By Proposition \ref{prop:submodsubgraph}  $e$ is not an edge in $\Gamma(\rV)$.

    Assume that $e$ is not in $\Gamma(\rV)$. Let $\hat{\rW}$ be an indecomposable component of $\pi_{h}^{-1}(\rR_{i})$ of height $2$. The submodule $\hat{\rW}$ is unique up to isomorphism. Note that $\soc(\hat{\rW})$ is unique and does not contain $\rR_j$. Let $\hat{\pi}: \mcls_{h}(\rV) \twoheadrightarrow \mcls_{h}(\rV)/\mcls_{h-2}(\rV)$ be the canonical projection and let $\rW = \hat{\pi}^{-1}(\hat{\rW}) \subset \rV$. We have $\rR_{j} \not\subset  \soc(\hat{\rW})= \soc_{h-1}(\rW)$ and $\soc_{h}(\rW) = \rR_{i}$, thus the edge $e$ is not admissible.
\end{proof}

\begin{cor}\label{cor:one_strong_strong}
Suppose for any two vertices $(\rR_i,h_i)$, $(\rR_j,h_j)$ with $h_i-h_j\geq 2$, there is a directed path in $\Gamma(\rV)$ from $(\rR_i,h_i)$ to $(\rR_j,h_j)$, 
then the one-strong graph coincides with the strong graph,   $\Gamma(\rV)=G(\rV)$. \qed
\end{cor}
Below we give examples of one-strong graphs. By Corollary \ref{cor:one_strong_strong}, all of them except for $w=0246$ are strong graphs.

By construction, $\Gamma(\rV)$ is an invariant of a socle multiplicity-free module $\rV$, which can be thought of as a refinement of the socle filtration $(\soc_i(\rV))_{i=1}^{\infty}$. The next example shows that there is a pair of modules which has the same socle filtrations $(\soc_i(\rV))_{i=1}^{\infty}$, but different one-strong submodule graphs.
\begin{example}\label{example}
    Consider the words $w_1 = 0024$ and $w_2 = 0204$. The composition factors of both words are $(\rR_1, \rR_2, \rR_3) = (0[0,4], 04 ,0^2)$

    Computing $\Hom_{\Uqa}(\rR_j,w_i)$, we get $\soc_1(w_1) = \soc_1(w_2) = \rR_2 \oplus \rR_3$. It follows that $\soc_2(w_1) = \soc_2(w_2) = \rR_1$. 
    
    We compute $\Hom_{\Uqa}(w_1, \rR_2)=\Hom_{\Uqa}(w_1, \rR_3)=0$. Thus, $\rR_2$ and $\rR_3$ are not direct summands of $w$. Therefore, both edges of length one are in $\Gamma(w_1)$. Indeed, let $\rW\subset w$ be a submodule such that $\soc_2(\rW)=\rR_1$. Then by Lemma \ref{lemma:nonss_compl},   $\rW=\rV$.

   We compute $\Hom_{\Uqa}(w_2, \rR_3)=0$. Therefore, by a similar argument, the edge from $\rR_1$ to $\rR_3$ is admissible in $\Gamma(w_2)$. We have $020=0[0,2]\oplus 0$, see Proposition \ref{prop:020_compl}. Therefore, $w_2 = 04 \oplus 0[0,2]4$ and the edge from $\rR_1$ to $\rR_2$ is not in $\Gamma(w_2)$.

    Thus, we obtain graphs $\Gamma(w_i)=G(w_i)$.
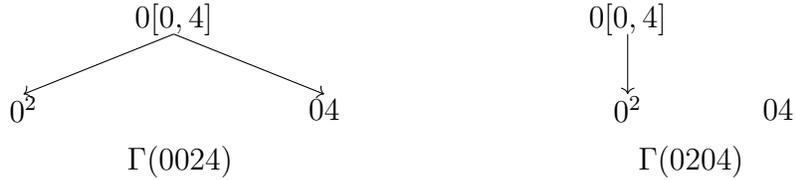
\begin{center}
\begin{tabular}{cc}
\begin{tikzpicture}
\node at (0,0.2) {$0[0,4]$};
\node at (-2,-1) {$0^2$};
\node at (2,-1) {$04$};
\draw[->] (0,0) to (-2,-0.8);
\draw[->] (0,0) to (2,-0.8);
\end{tikzpicture}
\qquad\qquad\qquad\qquad &

\begin{tikzpicture}
\node at (0,0.2) {$0[0,4]$};
\node at (0,-1) {$0^2$};
\node at (2,-1) {$04$};
\draw[->] (0,0) to (0,-0.8);
\end{tikzpicture}\\
$\Gamma(0024)$\qquad\qquad\qquad\qquad & $\Gamma(0204)$\\

\end{tabular}
\end{center}

In particular, $0024\ncong 0204$.
\end{example}

The next example shows that there is a pair of non-isomorphic words, which have identical one-strong module graphs.
\begin{example}\label{ex:graphs_0220_2002}
       Let $w_1 = 0220$ and $w_2 = 2002$. We argue as in Example \ref{example} to compute the graphs. It turns out that one-strong graphs and strong graphs all coincide, $\Gamma(w_1)=G(w_1)=\Gamma(w_2)=G(w_2)$ and are given by the following picture.

  
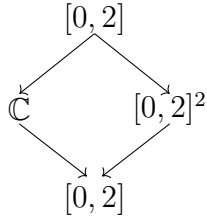
\begin{center}  
    \begin{tikzpicture}

\node at (0,0.2) {$[0,2]$};
\node at (-1,-1) {$\mbC$};
\node at (1,-1) {$[0,2]^2$};
\node at (0,-2.2) {$[0,2]$};
\draw[->] (0,0) to (-1,-0.8);
\draw[->] (0,0) to (1,-0.8);
\draw[->] (1,-1.2) to (0.1,-1.9);
\draw[->] (-1,-1.2) to (-0.1,-1.9);
\end{tikzpicture}
\end{center}

However, the modules are not isomorphic, $0220 \ncong 2002$.
Indeed,  $\dim(\Hom_{\Uqa}(0220, 0220)) = h(02202442) =2$, since $|\iconf(02202442)| = |\sconf(02202442)| = 2$, and  $$\dim(\Hom_{\Uqa}(2002, 0220)) = h(02204224) = h(s^2(02204224)) = h(20422446) = 1$$ as $|\sconf(20422446)|=1$.
\end{example}

The next example gives a module such that the one-strong and strong graphs are different. In this example, the one-strong and strong graphs of the dual module coincide.

\begin{example}\label{ex:0246}
    Let $w_1 = 0246$ and $w_2 = 6420$. Up to a shift the modules $w_1$ and $w_2$ are dual of each other. The one strong graphs are give by the following pictures.
    \begin{center}
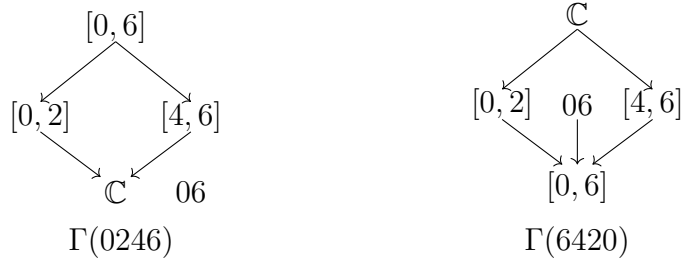
\begin{tabular}{cc}
\begin{tikzpicture}
\node at (0,0.2) {$[0,6]$};
\node at (-1,-1) {$[0,2]$};
\node at (1,-2) {$06$};
\node at (1,-1) {$[4,6]$};
\node at (0,-2) {$\mbC$};
\draw[->] (0,0) to (-1,-0.8);
\draw[->] (0,0) to (1,-0.8);
\draw[->] (1,-1.2) to (0.2,-1.8);
\draw[->] (-1,-1.2) to (-0.2,-1.8);
\end{tikzpicture}
\qquad\qquad\qquad\qquad &

\begin{tikzpicture}
\node at (0,0.2) {$\mbC$};
\node at (-1,-1) {$[0,2]$};
\node at (0,-1) {$06$};
\node at (1,-1) {$[4,6]$};
\node at (0,-2.1) {$[0,6]$};
\draw[->] (0,0) to (-1,-0.8);
\draw[->] (0,0) to (1,-0.8);
\draw[->] (1,-1.2) to (0.2,-1.8);
\draw[->] (-1,-1.2) to (-0.2,-1.8);
\draw[->] (0,-1.2) to (0,-1.8);
\end{tikzpicture}\\
$\Gamma(0246)$\qquad\qquad\qquad\qquad & $\Gamma(6420)$\\

\end{tabular}
\end{center}
 The strong graph $G(0246)$ is given in Example \ref{ex:weight_graph}.
The graph $\Gamma(0246)$ is a proper subgraph of $G(0246)$. The graph $\Gamma(6420)$ is strong by Corollary \ref{cor:one_strong_strong}.
\end{example}

Before we give next examples, we recall some symmetries. 

First, the graphs are compatible with shifts. Namely, given a word $w$, the strong graph of $w_{\tau_a}$  is obtained from the strong graph of $w$ by applying $\tau_a$ to each letter.

Second, let $w$ be a word and let $\omega(w)$ be the word obtained from $w$ by replacing each letter $a$ by $-a$ and writing letters in the opposite order, see Corollary \ref{cor:reverse_word}. Then the strong graph of $\omega(w)$ is obtained from the string graph of $w$ by replacing each letter $a$ by $-a$.

Third, let $w^\times$ be the word obtained from $w$ by writing it from right to left. Then $w^\times=(w^*)_{\tau_{-2}}$ is a shift of the module dual to $w$. In many cases the strong graph of $w^\times$ is obtained from the strong graph of $w$ by reversing all arrows. This happens in all examples we present except for $w=0246$ (see Example \ref{ex:0246}), and $w=2024$ when $w^\times=4202$ is not socle-multiplicity free. Note that $4202$ is reducible, and each summand is multiplicity free, therefore one can define the graph for each summand and that graph is obtained from $w=2024$ by reversing all arrows. 

\medskip

The strong graphs for words of length $3$ are discussed in Example \ref{ex:3letters_graphs}.

In the next example we draw all strong graphs of words of length $4$ with content $\{0,2,2,4\}$ and $\{0,2,4,6\}$ up to the symmetries (we do not repeat the graphs already given in Example \ref{ex:0246}).

\begin{example}\label{ex:mod_graphs_4let}\hfill

\begin{center}
\begin{tabular}{c c c c c }
    \begin{tikzpicture}
    \node at (0,0.2) {$2[0,4]$};
    \node at (1,-1) {$[2,4]$};
    \node at (0,-1) {$\mbC$};
    \node at (1,-2) {$\mbC$};
    \node at (0,-2) {$[0,2]$};
    \draw[->] (0,0) to (0,-0.8);
    \draw[->] (0,-1.2) to (0,-1.8);
    \draw[->] (1,-1.2) to (1,-1.8);
    \end{tikzpicture} 
     & 
    \begin{tikzpicture}
    \node at (0,0.2) {$\mbC$};
    \node at (-1,-1) {$[0,2]$};
    \node at (0,-1) {$[2,4]$};
    \node at (1.2,-1) {$2[0,4]$};
    \node at (0,-2) {$\mbC$};
    \draw[->] (0,0) to (-1,-0.8);
    \draw[->] (0,0) to (1,-0.8);
    \draw[->] (0,0) to (0,-0.8);
    \draw[->] (0,-1.2) to (0,-1.8);
    \draw[->] (1,-1.2) to (0.2,-1.8);
    \draw[->] (-1,-1.2) to (-0.2,-1.8);
    \end{tikzpicture} 
    &
     \begin{tikzpicture}
    \node at (0,0.0) {$[2,4]$};
    \node at (0,-0.75) {$\mbC$};
    \node at (0,-1.5) {$2[0,4]$};
    \node at (0,-2.25) {$\mbC$};
    \node at (0,-3) {$[0,2]$};
    \draw[->] (0,-0.2) to (0,-0.55);
    \draw[->] (0,-0.95) to (0,-1.3);
    \draw[->] (0,-1.7) to (0,-2.05);
    \draw[->] (0,-2.45) to (0,-2.8);
    \end{tikzpicture} 
     & 
     \begin{tikzpicture}
    \node at (0,0.2) {$[4,6 ]$};
    \node at (0,-1) {$\mbC$};
    \node at (1.3,-1) {$[0,6]$};
    \node at (0,-2.05) {$[0,2]$};
    \node at (1.3,-2.1) {$06$};
    \draw[->] (0,0) to (1.3,-0.8);
    \draw[->] (0,0) to (0,-0.8);
    \draw[->] (0,-1.2) to (0,-1.8);
    \draw[->] (1.3,-1.2) to (0.1,-1.8);
    \draw[->] (1.3,-1.2) to (1.3,-1.8);
    \end{tikzpicture} 
     &
     \begin{tikzpicture}
    \node at (0,0.2) {$\mbC$};
    \node at (-1.5,-0.75) {$[0,2]$};
    \node at (1.5,-0.75) {$[4,6]$};
    \node at (0,-1.7) {$[0,6]$};
    \node at (0,-2.5) {$06$};
    \draw[->] (0,0) to (-1.5,-0.5);
    \draw[->] (0,0) to (1.5,-0.5);
    
    \draw[->] (0,-1.9) to (0,-2.3);
    \draw[->] (1.5,-0.9) to (0.2,-1.4);
    \draw[->] (-1.5,-0.9) to (-0.2,-1.4);
    \end{tikzpicture} 
     
     \\ 
    $G(2024)$ & $G(2042)$ & $G(2204)$ & $G(2046)$ & $G(2064)$ \\
    
\end{tabular}
\end{center}
\end{example}
 We are interested in the structure of the modules $0^n2^n0^n$, cf. Example \ref{ex:002200}. As one possible step in this direction, we study modules of the form $0^n2^n$. 
\begin{example}
    The strong module graphs for $0^22^2, 0^32^3$ are given by the following pictures.
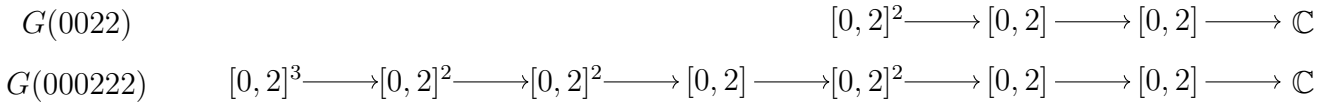
\begin{figure}[H]
    \begin{tikzpicture}
    \node at (-3.5,-0.05) {$G(0022)$};
    \node at (7,0) {$[0,2]^2$};
    \node at (9,0) {$[0,2]$};
    \node at (11,0) {$[0,2]$};
    \node at (12.8,0) {$\mbC$};
    \draw[->] (7.5,0) to (8.5,0);
    \draw[->] (9.5,0) to (10.5,0);
    \draw[->] (11.5,0) to (12.5,0);
    \node at (-3.5,-0.85) {$G(000222)$};
    \node at (-1,-0.8) {$[0,2]^3$};
    \node at (1,-0.8) {$[0,2]^2$};
    \node at (3,-0.8) {$[0,2]^2$};
    \node at (5,-0.8) {$[0,2]$};
    \node at (7,-0.8) {$[0,2]^2$};
    \node at (9,-0.8) {$[0,2]$};
    \node at (11,-0.8) {$[0,2]$};
    \node at (12.8,-0.8) {$\mbC$};
    \draw[->] (-0.5,-0.8) to (0.5,-0.8);
    \draw[->] (1.5,-0.8) to (2.5,-0.8);
    \draw[->] (3.5,-0.8) to (4.5,-0.8);
    \draw[->] (5.5,-0.8) to (6.5,-0.8);
    \draw[->] (7.5,-0.8) to (8.5,-0.8);
    \draw[->] (9.5,-0.8) to (10.5,-0.8);
    \draw[->] (11.5,-0.8) to (12.5,-0.8);
    \end{tikzpicture}
\end{figure}
    (We drew these graphs horizontal way.)
\end{example}

\begin{conj}\label{conj:0n2n} The module $0^n2^n$ is socle-multiplicity free. Moreover,
    $$\soc_{i}(0^n2^n) = \begin{cases} [0,2]^{s_2(i-1)},& i\leq 2^{n},\\ \{0\},&\textit{ otherwise.}\end{cases}$$
    where $s_2(i-1)$ is the sum of digits of the binary representation of the number $i-1$.
\end{conj}

\begin{lemma}
    If $\soc_{i}(0^n2^n)$ is simple for $i=1,\dots, 2^n$, then Conjecture \ref{conj:0n2n} is true.
\end{lemma}
\begin{proof} 
We use induction on $n$. By the assumption $ht(0^n2^n)=2^{n}$. The module $0^n2^n$ contains a submodule $0^{n-1}2^{n-1}$. Due to the hypothesis assumption this submodule should coincide with the component of the maximal filtration $\mcls_{2^{n-1}}(0^n2^n)$. The quotient-module  $0^n2^n/\mcls_{2^{n-1}}(0^n2^n) \cong 0^{n-1}2^{n-1}[0,2]$ has height $2^{n-1}$ by \eqref{eq:filtr_quot1}. On the other hand, the quotient module $0^{n-1}2^{n-1}[0,2]$ has a filtration with $2^{n-1}$ simple quotients obtained from the filtration of $0^{n-1}2^{n-1}$ by multiplication by $[0,2]$. Therefore, this filtration is maximal. It follows that $\soc_{2^{n-1}+i}(0^n2^n) = \soc_{i}(0^{n-1}2^{n-1})[0,2]$.
\end{proof}

\medskip

\begin{remark}\label{remark:edge_extension}
An admissible edge of length one actually carries additional information - an isomorphism class of extensions of the modules it connects. 

Namely, let an edge connect vertices $(\rR_{i}, h)$ and $(\rR_{j}, h-1)$. As before the submodule $\pi_{h}^{-1}(\rR_{i})\subset \mcls_{h}(\rV)/\mcls_{h-2}(\rV)$ contains a unique up to isomorphism indecomposable component $\rW$ of height $2$ (see Lemma \ref{def:one_strong_graph}).  Then $\soc(\rW) = \rR_{j} \oplus \rN$ for a unique semi-simple module $\rN$. Then $\rW/\rN$ is a non-trivial extension of $\rR_{i}$ by $\rR_{j}$. 

Note that this additional information is still not sufficient to uniquely characterize the module. For instance, one can check that in Example \ref{ex:graphs_0220_2002} all non-trivial extension for modules $0220$ and $2002$ are isomorphic. 

We expect that similar data has to be used in the study of modules which are not socle-multiplicity free. 
\end{remark}

\subsection{Graphs and tensor products.}
We study how the strong graph of a tensor product of two modules is related to the strong graphs of these modules.  In general, a tensor product of two socle-multiplicity free modules may not be non-socle-multiplicity free, and then we do not have definition of a module graph.
\begin{example}
    Consider the module $\rV = 2020 \cong 20 \oplus 20[0,2]$ which is socle-multiplicity free with $\soc(2020)=[0,2]\oplus[0,2]^2$ and $\soc_2(2020)=\mbC \oplus [0,2]$.
    
    However, $0 \otimes \rV$ is semi-simple (see Example \ref{ex:bn_bound}) and not socle-multiplicity free.
\end{example}
 But in the cases when the tensor product is still socle-multiplicity free, the examples look interesting.
 
The first natural expectation is the following conjecture.

\begin{conj}\label{conj:graph_product}
    Let $\rV$, $\rW$ be $q$-character separated. Let $\rV$, $\rW$, $\rV\otimes\rW$ be socle-multiplicity free. Then the strong graph of the tensor product is the Cartesian product of graphs, $G(\rV\otimes \rW) = G(\rV)\times G(\rW)$.
\end{conj}

We prove this conjecture in a special case. 

\begin{lemma}
     Let $\rV$ and $\rW$ be thin and $q$-character separated. Assume that the action of universal R-matrix $\uR$ is well-defined on $\rV \otimes\rW$.
     
    Then the strong graph of the tensor product $\rV\otimes \rW$ is the Cartesian product of graphs, $G(\rV\otimes \rW) = G(\rV)\times G(\rW)$.
\end{lemma}
\begin{proof}
    By Proposition \ref{prop:equiv_drinf}, we have $\rV \otimes \rW \cong \rV \otimes_D \rW$. Since $\rV$ and $\rW$ are thin and $q$-character separated, $\rV \otimes_D\rW$ is thin. Denote the set of $a\in \mbC$ such that $1_a$ appears in a monomial of the $q$-character of $\rU$ in any non-zero degree by $S_{\rU}$.

    We claim that if $\rU\subset \rV \otimes_D \rW$ is a submodule, then $(\Uqa\otimes\Uqa)(\rU) = \rU$. Let $u \in \rU[\phi]$ be an $\ell$-weight vector, of an $\ell$-weight $\phi$. Then $\phi = \phi_1\phi_2$, where $\phi_1, \phi_2$ are $\ell$-weights in $\rV$ and $\rW$ respectively and $u \in \rV[\phi_1]\otimes \rW[\phi_{2}]$. Write $u = v\otimes w$, where $v\in \rV[\phi_1], w\in \rW[\phi_2]$. Then we have
    \begin{equation*}
        x_{+}(z)u = x_{+}(z)v\otimes \psi^{-}(z)w + v\otimes x_{+}(z)w = \phi_2(z)(x_{+}(z)\otimes 1)u + (1\otimes x_{+}(z))u.
    \end{equation*}
    
   By Proposition \ref{prop:ellroot_action}, $\phi_2(z)(x_{+}(z)\otimes 1)u$ belongs to a sum 
    $$
    \bigoplus_{a\in S_{\rV}}\rV[A_{a+1}\phi_1]\otimes \rW[\phi_2] = \bigoplus_{a\in S_{\rV}}\big(\rV\otimes_D \rW\big)[A_{a+1}\phi_1\phi_2],
    $$
    and $(1\otimes x_{+}(z))u$ is in 
    $$
    \bigoplus_{b\in S_{\rW}}\rV[\phi_1]\otimes \rW[A_{b+1}\phi_2] = \bigoplus_{b\in S_{\rW}}\big(\rV\otimes_D \rW\big)[A_{b+1}\phi_1\phi_2].
    $$
    Here, as always, $A_{a+1} = 1_a1_{a+2}$.
    
    Since $\rV$ and $\rW$ are $q$-character separated, for all $a\in S_{\rV}, b\in S_{\rW}$, we have $A_{a} \neq A_{b}$. The submodule $\rU$ is thin, therefore each $\ell$-weight component of $x_{+}(z)(u)$ belongs to $\rU$. This implies $(x_{+}(z)\otimes 1)u, \phi_2(z)(1\otimes x_{+}(z))u \in \rU[[z, z^{-1}]]$. Note that we can divide $\phi_2(z)(x_{+}(z)\otimes 1)u$ by $\phi_2(z)$ since poles and zeroes of $1_{b}$ with $b\in S_{\rW}$ do not belong to the support of delta-functions appearing in the action of $x_{+}(z)$ on $\rV$.
    
    Similarly, $(x_{-}(z)\otimes 1)u, (1\otimes x_{-}(z))u \in \rU[[z^{\pm 1}]]$.

    Therefore, we obtain $(\Uqa \otimes \Uqa) (\rU) \subset \rU$.

    Let $\{\rR_i\}_{i\in I}, \{\rR_{j}\}_{j\in J}$ be the sets of irreducible factors of $\rV$ and $\rW$ respectively.
    
    The set of irreducible factors of $\rV \otimes \rW$ is given by $\{\rR_{i} \otimes \rR_{j}\}_{(i,j)\in I\times J}$ since the products $\rR_{i} \otimes \rR_{j}$ are irreducible. Indeed, since $\rV$ and $\rW$ are $q$-character separated, $\rR_i$ and $\rR_j$ are $q$-character separated for any $(i,j)\in I\times J$. Both $\rR_i$ and $\rR_j$ are products of evaluation modules in general position. Let $[\al, \bt]$ and $[\al^{\prime}, \bt^{\prime}]$ be two factors in $\rR_i$ and $\rR_j$ respectively. Then variables $1_{\al},\dots ,1_{\bt}$ and $1_{\al+2},\dots, 1_{\bt+2}$ appear in highest and lowest monomials of $\rR_{i}$ respectively. Therefore, $\al^{\prime} > \bt+2$ or $\al > \bt^{\prime}+2$, therefore strings $[\al, \bt]$ and $[\al^{\prime}, \bt^{\prime}]$ are in general position.

    Since $\rV\otimes \rW$ is thin, we have identification between vertices of $G(\rV\otimes \rW)$ and vertices of $G(\rV) \times G(\rW)$. Moreover, the graph $G(\rV\otimes\rW)$ can be constructed as described in Section \ref{subsec:graph_mult_free}.
    
  Let $\phi$ be a dominant $\ell$-weight corresponding to an irreducible factor $\rR_i\otimes \rR_j$. Then $\phi = \phi_i\phi_j$, where $\phi_i, \phi_j$ are dominant $\ell$-weights in $\chi_q(\rV), \chi_q(\rW)$ respectively. Let  $v_i \in \rV[\phi_i], v_j \in \rV[\phi_j]$ be non-zero vectors.  Then $v_i\otimes v_j\in (\rV\otimes_D\rW)[\phi]$.
  
  The submodule $\overline{\rR_{i}\otimes\rR_j}\subset\rV\otimes_D\rW$ is generated by $v_i \otimes v_j$. Clearly, $\overline{\rR_{i}\otimes\rR_j} \subset \overline{\rR}_i\otimes \overline{\rR}_j$. By the above argument, $\Uqa \otimes \Uqa (v_i\otimes v_j) = \overline{\rR}_i\otimes \overline{\rR}_j \subset \overline{\rR_{i}\otimes\rR_j}$. Henceforth, $\overline{\rR_{i}\otimes\rR_j} = \overline{\rR}_{i}\otimes\overline{\rR}_j$.

    Therefore, the edge between $\rR_i\otimes \rR_j$ and $\rR_k\otimes\rR_l$ is admissible if and only the edge between $\rR_{i}$ and $\rR_k$ and the edge between $\rR_j$ and $\rR_l$ are admissible or $i = k$ and the edge between $\rR_{j}$ and $\rR_{l}$ is admissible or $j = l$ and the edge between $\rR_{i}$ and $\rR_{k}$ is admissible. In other words, the graph $\tilde{G}(\rV\otimes \rW)$ is the tensor product of graphs  $\tilde{G}(\rV)$ and $\tilde{G}(\rW)$ (not to confuse with the Cartesian product). 
    
    Let the edge $e$ connecting $\rR_i\otimes \rR_j \rightarrow \rR_k\otimes\rR_l$ be admissible.
    
    If $i\neq k$ and $j\neq l$, there is a path $\rR_i\otimes \rR_j \to \rR_i \otimes \rR_l \to \rR_k\otimes \rR_l$ in $\tilde{G}(\rV\otimes \rW)$, therefore in the graph $G(\rV\otimes\rW)$ the edge $e$ is removed.

    If $i = k$ and $j\neq l$, then the edge $e$ is in $G(\rV\otimes\rW)$ if and only if the edge  $\rR_j \to \rR_l$ is in  $G(\rV)$. Indeed, paths connecting $\rR_i\otimes \rR_j$  and  $\rR_i \otimes \rR_l$ pass only vertices of type $\rR_i\otimes \rR_s$ and therefore bijectively correspond to paths in $\tilde{G}(\rV)$ connecting $\rR_j$ and  $\rR_l$. Similarly, if $i \neq k$ and $j = l$ the edge $e$ is in $G(\rV\otimes\rW)$ if and only if the edge  $\rR_i \to \rR_k$ is in  $G(\rW)$.
\end{proof}

\begin{example}
    The strong graph of module $0268$ is the Cartesian product of strong graphs of modules $02$ and $68$ and is given by the following picture.
    $$
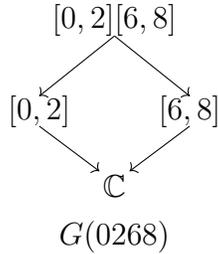
\begin{tabular}{c}
\begin{tikzpicture}
\node at (0,0.22) {$[0,2][6,8]$};
\node at (-1,-1) {$[0,2]$};
\node at (1,-1) {$[6,8]$};
\node at (0,-2) {$\mbC$};
\draw[->] (0,0) to (-1,-0.8);
\draw[->] (0,0) to (1,-0.8);
\draw[->] (1,-1.2) to (0.2,-1.8);
\draw[->] (-1,-1.2) to (-0.2,-1.8);
\end{tikzpicture}
\\
$G(0268)$
\end{tabular}
$$
\end{example}

In general, the strong graph of a tensor product is obtained from the strong graph of factors in a rather non-trivial way. We study the case when one of the factors is a word of length 1.

In the next example we give all non-trivial graphs of words of length $3$. These words (up to the symmetries) are obtained by multiplying $w=02$ by $a\in \{0,2,4\}$ on the left and on the right.

\begin{example}\label{ex:3letters_graphs}
The strong graph of the module $02$ is $[0,2]\longrightarrow \mbC$ (written in a horizontal way). We have the following strong graphs obtained by tensor multiplication of $02$ with modules $0$ and $4$.
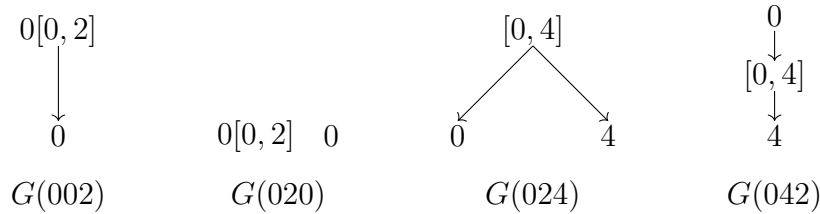
\begin{center}    
\begin{tabular}{cccc}
\begin{tikzpicture}
\node at (0,0.2) {$0[0,2]$};
\node at (0,-1.2) {$0$};
\draw[->] (0,0) to (0,-1);
\node at(0,-2) {$G(002)$};
\end{tikzpicture}
& \qquad
\begin{tikzpicture}

\node at (0,-1.2) {$0[0,2]$};
\node at (1,-1.2) {$0$};
\node at (0.3,-2) {$G(020)$};
\end{tikzpicture}
& \qquad
\begin{tikzpicture}
\node at (0,0.2) {$[0,4]$};
\node at (-1,-1.2) {$0$};
\node at (1,-1.2) {$4$};
\node at (0,-2) {$G(024)$};
\draw[->] (0,0) to (-1,-1);
\draw[->] (0,0) to (1,-1);

\end{tikzpicture}
& \qquad
\begin{tikzpicture}
\node at (0,0.4) {$0$};
\node at (0,-0.4) {$[0,4]$};
\node at (0,-1.2) {$4$};
\draw[->] (0,0.2) to (0,-0.2);
\draw[->] (0,-0.6) to (0,-1);
\node at (0,-2) {$G(042)$};
\end{tikzpicture}
\end{tabular}
\end{center}
\end{example}

Let $\rV$ be a socle-multiplicity free module with composition factors $\rR_1,\dots,\rR_k$ of degrees $h_1,\dots,h_k$. Let $\rR$ be
an irreducible module such that for all $i$ the tensor product $\rR_{i} \otimes \rR$ is irreducible.  Assume that  $\rV \otimes \rR$ is socle-multiplicity free. Recall that  $\{(\rR_i, h_i)\}_{i=1,\dots,k}$ is the set of vertices of the strong graph 
$G(\rV)$. 

Note that under our assumptions, $0\subset \mcls_1(\rV)\otimes \rR \subset \mcls_2(\rV)\otimes \rR \subset \dots$ is a semi-simple filtration of $\rV\otimes \rR$. By Lemma \ref{lemma:TP_factors_drop} this implies that the set of vertices of $\Gamma(\rV)$ is given by $\{(\rR_i\otimes\rR,\tilde{h}_i )\}_{i=1,\dots,k}$ with $\tilde{h}_i \leq h_i$. However, in the case of socle-multiplicity free case, that statement can be improved as follows.

\begin{lemma}\label{lemma:tp_graph_map}
    For each composition factor $\rR_i\subset \soc_{h_i}(\rV)$, there exists a unique $\tilde h_i$ such that $\rR_i\otimes \rR\subset \soc_{\tilde h_i} (\mathcal S_{h_i}(\rV)\otimes \rR)$ and  $\rR_i\otimes \rR\not\subset \soc_{\tilde h_i} (\mathcal S_{h_i-1}(\rV)\otimes \rR)$.  Moreover,  $\tilde h_i\leq h_i$, and the map $\iota$ sending $(\rR_i,h_i)\mapsto  (\rR_i\otimes \rR,\tilde h_i)$ is a bijection between vertices of $G(\rV)$ and vertices of $G(\rV\otimes\rR)$.
\end{lemma}
\begin{proof}
    Under our assumptions, we have isomorphism of semi-simple modules,
    $$
    \soc_{h_i} (\rV)\otimes \rR\cong (\mathcal S_{h_i}(\rV)\otimes \rR)/(\mathcal S_{h_i-1}(\rV)\otimes \rR)\cong\mathop\oplus_{j=1}^{h_i} \soc_j (\mathcal S_{h_i}(\rV)\otimes \rR)/\soc_j (\mathcal S_{h_i-1}(\rV)\otimes \rR).
    $$
    Here, the second isomorphism follows from semi-simplicity of modules and equality of their composition series. Therefore, there exists a unique $\tilde h_i$ such that  $\rR_i\otimes\rR\subset
    \soc_{\tilde h_i} (\mathcal S_{h_i}(\rV)\otimes \rR)/\soc_{\tilde{h}_i} (\mathcal S_{h_i-1}(\rV)\otimes \rR)$. In particular, $\rR_i\otimes \rR\subset \soc_{\tilde h_i} (\mathcal S_{h_i}(\rV)\otimes \rR)$. Since $\soc_{\tilde h_i}(\mathcal S_{h_i}(\rV)\otimes \rR)$ is multiplicity free  $\rR_i\otimes \rR\not\subset \soc_{\tilde h_i} (\mathcal S_{h_i-1}(\rV)\otimes \rR)$.
    In particular, $\rR_i\otimes \rR\subset\soc_{\tilde h_i}(\rV\otimes \rR)$.
    
    We have another isomorphism of semi-simple modules,
$$
\soc_{\tilde{h}_i}(\rV\otimes\rR) \cong \mathop\oplus_{k=1}^{ht(\rV)}\soc_{\tilde{h}_{i}}(\mcls_{k}(\rV)\otimes \rR)/\soc_{\tilde{h}_{i}}(\mcls_{k-1}(\rV)\otimes \rR).
$$

It follows that $\iota$ is injective.
Indeed, if $(\rR_i\otimes \rR,\tilde h_i)$ is an image of $(\rR_i,h_i)$, then $h_i$ is the unique $k$ such that $\soc_{\tilde{h}_{i}}(\mcls_{k}(\rV)\otimes \rR)/\soc_{\tilde{h}_{i}}(\mcls_{k-1}(\rV)\otimes \rR)$. The map $\iota$ is surjective since graphs of $G(\rV)$ and $G(\rV\otimes \rR)$ have the same number of vertices.

\end{proof}

We expect that the bijection $\iota$ is compatible with edges in the following way.
\begin{conj}\label{conj:irr_prod_graph_erase}
If the edge connecting $(\rR_i\otimes\rR,\tilde h_i)$ and $(\rR_j\otimes\rR,\tilde h_j)$ in a graph of $\rV\otimes \rR$ is admissible, then the edge connecting $(\rR_i,h_i)$ and $(\rR_j,h_j)$ in a graph of $\rV$ is admissible.
\end{conj} 
In other words, under the bijection $\iota$ some edges in $G(\rV)$ are erased but no new edge is created.

\begin{remark}\label{remark:extension_remove}
 We expect that the admissible edge of length one is erased if and only if the corresponding extension of modules, see Remark \ref{remark:edge_extension}, becomes completely reducible after tensor multiplication by $\rR$. 
\end{remark}

We illustrate the picture in the following example.

\begin{example}
We give strong graphs of words $0022$, $0220$, $2200$ on the top row. We add admissible edges of length two. Then we multiply our words by $0$ to obtain strong graphs of $00220$, $02200$ on the second row and then one more time to get $002200$.  
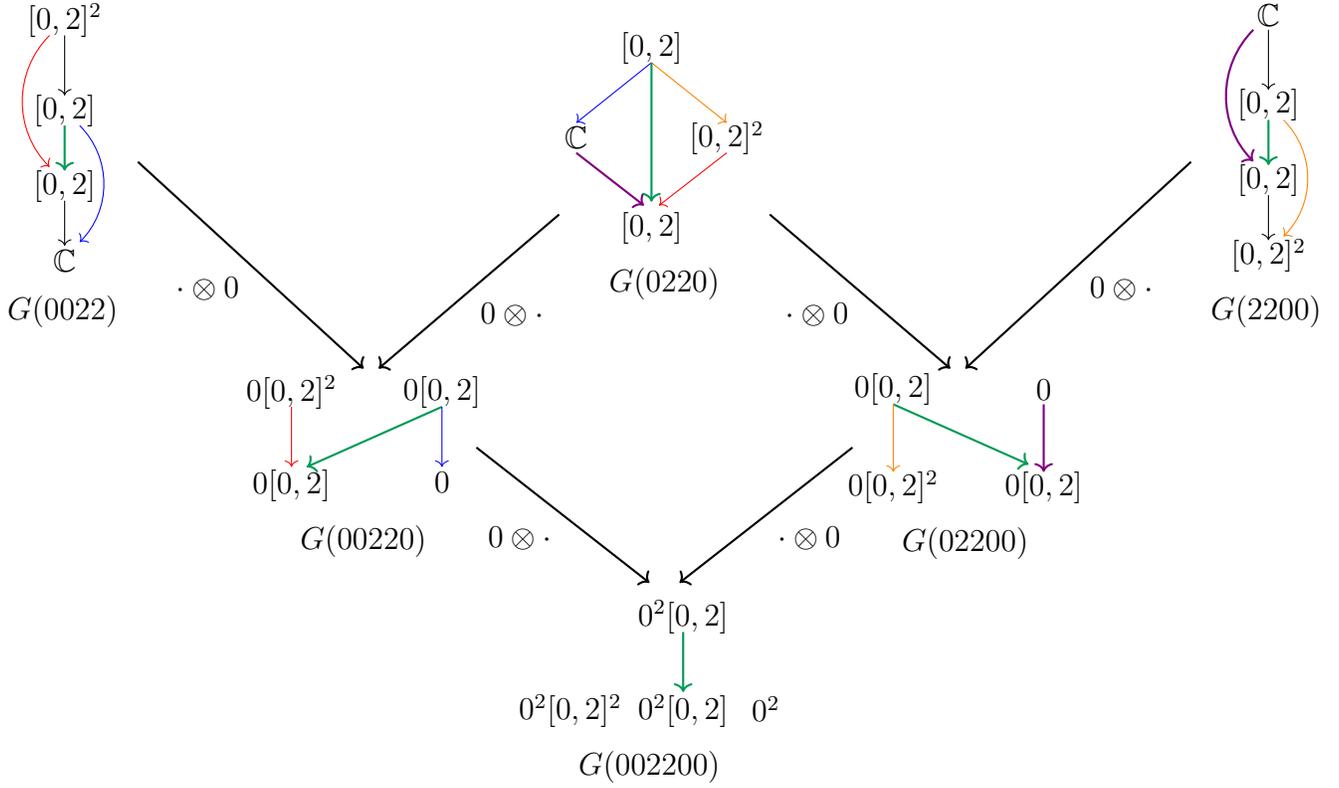
\begin{figure}[H]
    \begin{tikzpicture}
        \node at (-8, 0) {
    \begin{tabular}{c}
    \begin{tikzpicture}
    \node at (0,0.2) {$[0,2]^2$};
    \node at (0,-1) {$[0,2]$};
    \node at (0,-2) {$[0,2]$};
    \node at (0,-3) {$\mbC$};
    \draw[->] (0,0) to (0,-0.8);
    \draw[->,red] (-0.2 ,0) to [out=-135,in=135] (-0.2,-1.75);
    \draw[->,ForestGreen, thick] (0,-1.2) to (0,-1.8);
    \draw[->,blue] (0.2 ,-1.2) to [out=-45,in=45] (0.2,-2.75);
    \draw[->] (0,-2.2) to (0,-2.8);
    
    \end{tikzpicture}\\
    $G(0022)$
    \end{tabular}};
        \node at (8, 0) {
    \begin{tabular}{c}
    \begin{tikzpicture}
    \node at (0,0.2) {$\mbC$};
    \node at (0,-1) {$[0,2]$};
    \node at (0,-2) {$[0,2]$};
    \node at (0,-3) {$[0,2]^2$};
    \draw[->] (0,0) to (0,-0.8);
    \draw[->,violet, thick] (-0.2 ,0) to [out=-135,in=135] (-0.2,-1.75);
    \draw[->,ForestGreen, thick] (0,-1.2) to (0,-1.8);
    \draw[->,orange] (0.2 ,-1.2) to [out=-45,in=45] (0.2,-2.75);
    \draw[->] (0,-2.2) to (0,-2.8);
    \end{tikzpicture}\\
    $G(2200)$
    \end{tabular}};
    \node at (0, 0) {
    \begin{tabular}{c}
    \begin{tikzpicture}
    \node at (0,0.2) {$[0,2]$};
    \node at (-1,-1) {$\mbC$};
    \node at (1,-1) {$[0,2]^2$};
    \node at (0,-2.2) {$[0,2]$};
    \draw[->,blue] (0,0) to (-1,-0.8);
    \draw[->,orange] (0,0) to (1,-0.8);
    \draw[->,red] (1,-1.2) to (0.1,-1.9);
    \draw[->,violet, thick] (-1,-1.2) to (-0.1,-1.9);
    \draw[->, ForestGreen, thick] (0,0) to (0,-1.85);
\end{tikzpicture}\\
    $G(0220)$
    \end{tabular}};
    \node at (-4, -4) {
    \begin{tabular}{c}
    \begin{tikzpicture}
    \node at (0,0.2) {$0[0,2]$};
    \node at (-2,0.2) {$0[0,2]^2$};
    \node at (-2,-1.05) {$0[0,2]$};
    \node at (0,-1) {$0$};
    \draw[->,ForestGreen, thick] (0,0) to (-1.8,-0.8);
    \draw[->,blue] (0,0) to (0,-0.8);
    \draw[->,red] (-2,0) to (-2,-0.8);
\end{tikzpicture}\\
    $G(00220)$
    \end{tabular}};
    \node at (4, -4) {
    \begin{tabular}{c}
    \begin{tikzpicture}
    \node at (0,0.2) {$0$};
    \node at (-2,0.2) {$0[0,2]$};
    \node at (-2,-1.1) {$0[0,2]^2$};
    \node at (0,-1.1) {$0[0,2]$};
    \draw[->,ForestGreen, thick] (-2,0) to (-0.2,-0.8);
    \draw[->, violet, thick] (0,0) to (0,-0.9);
    \draw[->,orange] (-2,0) to (-2,-0.9);
\end{tikzpicture}\\
    $G(02200)$
    \end{tabular}};
    \node at (-0.2, -7) {
    \begin{tabular}{c}
    \begin{tikzpicture}
    \node at (0,-1.05) {$0^2[0,2]$};
    \node at (1.1,-1.02) {$0^2$};
    \node at (-1.5,-1.05) {$0^2[0,2]^2$};
    \node at (0,0.2) {$0^2[0,2]$};
    \draw[->,ForestGreen, thick] (0,0) to (0,-0.8);
\end{tikzpicture}\\
    $G(002200)$
    \end{tabular}};

    \draw[->, thick] (-7,0) -- (-4,-2.75) node[midway,below left] {$\cdot\otimes 0$};
    \draw[<-, thick] (-0.2,-5.6) -- (-2.5, -3.8) node[midway,below left] {$0\otimes \cdot$};
    \draw[<-, thick] (0.2,-5.6) -- (2.5, -3.8) node[midway,below right] {$\cdot\otimes 0$};
    \draw[->, thick] (7,0) -- (4,-2.75) node[midway,below right] {$0\otimes \cdot$};
    \draw[<-, thick] (-3.8,-2.75) -- (-1.4,-0.7) node[midway,below right] {$0\otimes \cdot$};
    \draw[<-, thick] (3.8,-2.75) -- (1.4,-0.7) node[midway,below left] {$\cdot\otimes 0$};
    \end{tikzpicture}
\end{figure}

We color edges which correspond each other by the map $\iota$ of Lemma \ref{lemma:tp_graph_map} by the same color.
    
Observe that after each multiplication the degree of each vertex may decrease. Some edges disappear and some remain. Also, some edges become shorter in length.
\end{example}

\hfill
\section{Further directions.}\label{sec:further}
\subsection{Factorization by content.}\label{subsec:fact_by_cont}
We further explore the factorization of the space $H(w)$ when $w$ is disconnected in a proper sense. We recall three such statements, Propositions \ref{prop:hom_factorize_lattice}, \ref{prop:no_gaps} and Conjecture \ref{conj:drop_adm}.  
\begin{prop}\label{prop:no_conf_factor}
    Assume that $w_1 = (a_1,\dots, a_k)\in W_{k}, w_2 = (b_1,\dots, b_l) \in W_{l}$ with
    \begin{equation}\label{eq:prop_factor}
        a_i < b_j,
    \end{equation}    for all $i, j$. Let $w_3$ be a shuffle of words $w_1$ and $w_2$. Assume that in any arc configuration of the word $w_3$ no letter of $w_1$ is connected to a letter of $w_2$.
    
    Then
    \begin{equation*}        h_{\textit{irr}}(w_1)h_{\textit{irr}}(w_2) = h_{\textit{irr}}(w_3) \leq h(w_3) \leq  h(w_1)h(w_2).
    \end{equation*}\end{prop}
\begin{proof}
    First, we prove that $h_{\textit{irr}}(w_1)h_{\textit{irr}}(w_2) = h_{\textit{irr}}(w_3)$.
    
    Clearly, any irreducible arc configuration of $w_3$ is a shuffle of irreducible arc configurations of $w_1$ and $w_2$. Given an irreducible arc configuration of $w_1$ and an irreducible arc configuration of $w_2$, their shuffle is again an irreducible arc configuration since only new intersections appearing in shuffle are between arcs connecting pairs of letters $(a_i, a_j)$ and $(b_r, b_s)$. But by the hypothesis $b_r, b_s \notin \{a_1,\dots, a_k\}$. Hence the intersection is irreducible. Therefore, $h_{\textit{irr}}(w_1)h_{\textit{irr}}(w_2) = h_{\textit{irr}}(w_3)$.

    Second, by Theorem $\ref{thm:lower_bound_irr}$,  we have $h_{\textit{irr}}(w_3) \leq h(w_3)$.

    Third, we prove $h(w_3) \leq  h(w_1)h(w_2)$. Assume that $w_3 = (b_1,\dots, b_{l}, a_1, \dots, a_{k})$. Then applying slides, see Lemma \ref{lemma:slide_isom}, we obtain $h(w_3) = h((a_{1}-4, a_{2}-4, \dots, a_{k}-4, b_1,\dots, b_{l}))$. By Proposition \ref{prop:no_gaps},
    \begin{align*}
    h(w_3)&=h((a_{1}-4, a_{2}-4, \dots, a_{k}-4, b_1,\dots, b_{l})) = \\  & =h((a_{1}-4, a_{2}-4, \dots, a_{k}-4))h((b_1,\dots, b_l)) = h((a_1,\dots, a_{k}))h((b_1,\dots, b_l)).
    \end{align*}
   Let $d(w_3)$ be the minimal number of elementary transpositions of letters of $w_1$ with the letters of $w_2$ required to change $w_3$ to $(b_1,\dots, b_{l}, a_1, \dots, a_{k})$. Then $d(w_3) = 0$ if and only if $w_3=(b_1,\dots, b_{l}, a_1, \dots, a_{k})$.
     
    We use induction on $d(w_3)$. If $d(w_3) > 0$ then $w_3$ has the form $(c_1,\dots, c_m, a, b, c_{m+3},\dots, c_{k+l})$ where $a$ is in $w_1$ and $b$ is in $w_2$. Let $\tilde w_3=(c_1,\dots, c_m, b, a, c_{m+3},\dots, c_{k+l})$.
    
    We apply the intertwiner $\check{R}_{m+1}(a, b) = \Id_{(\mbC^2)^{\otimes m}}\otimes \check{R}(a, b)\otimes \Id_{(\mbC^2)^{\otimes (k+l-m-2)}}$, which induces the map $\check{R}_H: H(w_3) \longrightarrow H(\tilde{w}_3)$. In case $b \neq a + 2$, $\check{R}(a, b)$ is non-degenerate, therefore $\check{R}_H$ is non-degenerate. In case $b = a + 2$, the kernel of  $\check{R}_{m+1}(a, b)$ is isomorphic to $(c_1,\dots, c_m, c_{m+3},\dots, c_{k+l})$, therefore the kernel of $\check{R}_H$ is isomorphic to $H((c_1,\dots, c_m, c_{m+3},\dots, c_{k+l}))$. But $\conf((c_1,\dots, c_m, c_{m+3},\dots, c_{k+l})) = \varnothing$, otherwise by addition of an arc we obtain an arc configuration of $w_3$ in which $a$ and $b$ are connected by an arc, which contradicts the hypothesis. Therefore, $H((c_1,\dots, c_m, c_{m+3},\dots, c_{k+l})) = 0$ and $R_{H}$ is injective.

    Hence $\dim(H(w_3)) \leq \dim(H(\tilde{w}_3))$.

    If $w_3$ satisfies the assumption of the proposition, then $\tilde{w}_3$ also does since $b>a$. Moreover, $d(w_3)<d(\tilde{w}_3)$.

\end{proof}
\begin{conj}
    Under the assumptions of Proposition \ref{prop:no_conf_factor}, $h(w_3) =  h(w_1)h(w_2)$.
\end{conj}

We expect an even stronger statement. 

\begin{conj}\label{conj:hom_fact_confconn}
For a word $w$ the space $H(w)\cong H(w_1)\otimes H(w_2)\otimes\dots \otimes  H(w_k)$, where $w_i$'s are subwords corresponding to conf-connected components of $w$.
\end{conj}

\begin{remark}\label{rem:h_def_by_conf}
If Conjecture \ref{conj:hom_fact_confconn} is true, 
then the number $h(w)$ is completely determined by the set of arc configurations of $w$. Indeed, Conjecture \ref{conj:hom_fact_confconn} reduces the computation of $h(w)$ to the case of conf-connected words $w$. A conf-connected word is determined by arc configurations up to a shift, see Corollary \ref{cor:conf_conn_determ}. 
\end{remark}

In Appendix \ref{app:tables} we list the answers for conf-connected words up to length $10$ (modulo symmetries $ab=ba$, $|b-a|>2$, slides, and anti-involution $\omega$).

\begin{remark}
    Let $w$ be a word such that $\conf(w) \neq \varnothing$. Let $w^{(1)},\dots, w^{(k)}$ be the subwords of $w$ corresponding to its conf-connected components. As discussed in Section \ref{subsec:deg_graphs}, there is a natural identification $\conf(w) = \conf(w^{(1)})\times \dots \times \conf(w^{(k)})$. We note that under this identification the standard arc configuration of $w$ corresponds to the product of standard arc configurations of $w^{(1)}, \dots, w^{(k)}$.
\end{remark}

\subsection{Statistics of reducible intersections.}\label{subsec:int_poly}
The Theorem \ref{thm:lower_bound_irr} shows that for a given word $w \in W_{2n}$ some information about $H(w)$ can be deduced from the set $\iconf(w)$. The elements in  $\iconf(w)$ are distinguished in $\conf(w)$ by the condition  that all intersections are irreducible. It could happen that the number $h(w)$ is related to the statistics of reducible intersections for all elements of $\conf(w)$.

\begin{defi}\label{def:int_poly}
    For a given $w\in W_{2n}$ and an arc configuration $C\in \conf(w)$, let $J(C)$ be the number of reducible intersections of arcs in $C$. We call the polynomial 
    \begin{equation}
        p_{w}(x) = \sum_{C\in \conf(w)}x^{J(C)},
    \end{equation}
    intersection polynomial of $w$.
\end{defi}

We have $p_w(0)=|\iconf(w)|\leq h(w)$.

Intersection polynomials are invariant with respect to several maps which preserve $h(w)$.

\begin{prop}\label{prop:int_poly_symm}
For a given word $w\in W_{2n}$, intersection polynomial is invariant with respect to common shift of all letters of $w$, slides (cf. Lemma \ref{lemma:slide_isom}), the anti-involution $\omega$  (cf. Corollary \ref{cor:reverse_word}), and permutations of adjacent letters $a,b$ in $w$  whenever $|a-b|\neq 2$.
\end{prop}
\begin{proof}
    The part about slides follows from Remark \ref{rem:irr_int_slides_preserved}. The rest is clear.
\end{proof}

\begin{prop}
    Let $w\in W_{2n}$ be a word, let $C\in\conf(w)$, and let $(i_1, j_1), (i_2, j_2)\in C$ be two intersecting arcs. If the intersection is reducible, then $i_1, j_1, i_2, j_2$ are in the same conf-connected component of $\{1,\dots,2n\}$. 
\end{prop}
\begin{proof}
    Let $w = (a_1, \dots, a_{2n})$. Without loss of generality, $i_1 < i_2 < j_1 < j_2$.
    
    If $a_{i_2} = a_{i_1}$, then $C\backslash\{(i_1,j_1),(i_2, j_2)\} \sqcup \{(i_1,j_2), (i_2, j_1)\}$ connects $i_1$ with $j_2$.

    If $a_{i_2} = a_{j_1}$, then $C\backslash\{(i_1,j_1),(i_2, j_2)\} \sqcup \{(i_1,i_2), (j_1, j_2)\}$ connects $i_1$ with $i_2$.
\end{proof}
\begin{cor}\label{cor:int_poly_factor}
    Let $w$ be a word. Let $(w^{(1)},\dots,w^{(k)})$ be the subwords of $w$ corresponding to conf-connected components of $w$. Then $p_{w}(x) = \prod_{j=1}^{k}p_{w_{j}}(x)$.
\qed
\end{cor}

We note that  Corollary \ref{cor:int_poly_factor}  is consistent with Conjecture \ref{conj:hom_fact_confconn}.

\begin{prop}
    For a word $w \in W_{2n}$ such that $\conf(w)\neq \varnothing$ the intersection polynomial $p_w(x)$ is monic. Moreover, the arc configuration of $w$ which has the maximal number of reducible intersections is the standard arc configuration of $w$.
\end{prop}
\begin{proof}
    Due to Lemma \ref{lemma:std_conf_slide} we can assume that the first letter of $w$ is smaller than all other letters of $w$. Without loss of generality we can assume that this letter equals to $0$.

    We use induction on $n$. Let $j$ be the position of the rightmost letter $2$ in $w$. Let $C$ be an arc configuration and let $(1, j^{\prime}), (j,k)\in C$ for some $j^{\prime}<j<k$. We claim that the number of reducible intersections of $\tilde{C} = (C\backslash\{(1,j^{\prime}), (j,k)\})\sqcup\{(1,j), (j^{\prime},k)\}$ is strictly larger.

    Assume that an intersection of arcs $(l,m)$ and $(1,j^{\prime})$ in $C$ is reducible. Assume that $(l,m)$ does not intersect $(j,k)$. Then $l < j^{\prime} < m$ and the $l$-th letter is $2$. If $j^{\prime} < m < j$, then the intersection of arcs  $(l,m)$, $(j^{\prime}, k)$ is reducible. If $m > k$, then the intersection of arcs $(l,m)$ $(1,j)$ is reducible. 

     Assume that the intersection of arcs $(l,m)$,  $(j,k)$  in $C$ is reducible. Assume that $(l,m)$ does not intersect $(1,j^{\prime})$. If $l < j$, then the intersection of arcs  $(l,m)$, $(1,j)$ in $\tilde{C}$  is reducible. If $m > k$, then the intersection of arcs $(l,m)$, $(j^{\prime},k)$ in $\tilde{C}$ is reducible.

     In the case an arc $(l,m)$ intersects both $(1,j^{\prime})$ and $(j,k)$ and at least one of these intersections is reducible, both are reducible and in $\tilde{C}$ the intersections of arc $(l,m)$ with both $(1,j)$ and $(j^{\prime},k)$ are reducible.
     
     We conclude that the number of reducible intersections of arc $(1,j), (j^{\prime}, k)$ with an arc $(l,m) \in C\backslash\{(1,j^{\prime}), (j,k)\}$ is not smaller than the number of intersections of the arc $(l,m)$ with arcs $(1,j^{\prime}), (j, k)$.

     Note that arcs $(1,j^{\prime}), (j,k)$ do not intersect and the intersection of $(1,j)$ and $(j^{\prime}, k)$ is reducible. Therefore, the number of reducible intersections in $\tilde{C}$ is strictly larger.

     We conclude that any arc configuration with the maximal number of reducible intersections includes the arc $(1,j)$. Note that the number of reducible intersections of the arc $(1,j)$ with other arcs cannot exceed $\min(|I_{w}(4)\cap \{j+1,\dots, 2n\}|, m_3)$ (recall that $m_3$ is the number of arc with the left end equal to $2$). This number is achieved if the arc configuration $\mathring{C}_{(1,j)}$ is the standard arc configuration of $\mathring{w}_{(1,j)}$. By induction on $n$ this is the unique arc configuration which maximizes the number of reducible intersections between arcs different from $(1,j)$.

\end{proof}

We list all configuration-connected words (up to equivalences described in Proposition \ref{prop:int_poly_symm}) of lengths $2,4,6,8,10$ together with corresponding intersection polynomials in Appendix \ref{app:tables}.

These examples exhibit several patterns which we formulate as a conjecture.

\begin{conj}\label{conj:int_poly} Let $w$ be a word.
\begin{enumerate}
    \item\label{conj:int_poly_1} The number $h(w)$ depends only on $p_{w}(x)$, namely if  $p_{w_1}(x)=p_{w_2}(x)$, then $h(w_1)=h(w_2)$.
    \item The number $h(w) = 1$ if and only if the intersection polynomial of $w$ is a product of polynomials of the form $\frac{x^n-1}{x-1}$.
    
\end{enumerate}
\end{conj}
Part \eqref{conj:int_poly_1} of Conjecture \ref{conj:int_poly} implies Conjecture \ref{conj:drop_adm}.

\subsection{Structure of singular vectors.}
In this section we formulate a conjecture on the structure of singular vectors in a tensor product corresponding to a word $w\in W_{2n}$. 
Recall that in proof of Theorem \ref{thm:lower_bound_irr} for each irreducible arc configuration $C$ we constructed a singular vector with leading term having $"+"$'s at positions $\len(C)$. We can generalize this as follows. 

\begin{defi}
We call a sequence $m = (\epsilon_1,\dots, \epsilon_{2n})\in \{+,-\}^{\times 2n}$ a pivot of a word $w$ if there exists a vector $v_m\in H(w)$ of the form $v_m=m+\mathop{\sum}\limits_{m'<m}a_{m,m'}m'$, $a_{mm'}\in\mbC$. We denote the set of all pivots of $w$ by  $\pvt(w)$.
\end{defi}
Clearly $|\pvt(w)|=h(w)$ and $\{v_m\}_{m\in\pvt(w)}$ form a basis in $H(w)$.

Write vectors  $v_1,\dots,v_{h(w)}$ in the basis $(\epsilon_1,\dots, \epsilon_{2n})$ as a matrix  of size  $h(w)\times 2^{2n}$. Then the elements of $\pvt(w)$ are the labels of the columns  which contain pivots.

Let  $\mu_{w}:\conf(w) \to \{+,-\}^{\times 2n}$ be the map such that $(\mu(C))_i=+$ if and only if $i\in \operatorname{LC}(C)$, cf. the proof of Lemma \ref{lemma:sconf_to_ends}.

\begin{conj}\label{conj:pivots}
    For any $w\in W_{2n}$, we have
    \begin{equation*}
        \pvt(w) \subset \mu_{w}(\sconf(w)).
    \end{equation*}
\end{conj}

By the proof of Theorem \ref{thm:lower_bound_irr}, $\mu_{w}(\iconf(w)) \subset \pvt(w)$. Therefore, Conjecture $\ref{conj:pivots}$ implies the following purely combinatorial statement.

\begin{conj}\label{conj:irr_ends_subset_steady}
Let $w\in W_{2n}$. Then,
\begin{equation*}
    \mu_{w}(\iconf(w)) \subset \mu_{w}(\sconf(w)).
\end{equation*}
That is the set of left ends of an irreducible arc configuration of a word $w$ is always the set of left ends of a steady arc configuration of the word $w$.
\end{conj}

    Let $C\in\conf(w)$ be either steady or irreducible arc configuration. By Lemmas \ref{lemma:nconf_to_ends}, \ref{lemma:sconf_to_ends}, $C$ is uniquely determined by the set $\len(C)$. Therefore, Conjecture \ref{conj:irr_ends_subset_steady} would give an injective map
    \begin{equation*}
    \iota:\ \iconf(w) \hookrightarrow \sconf(w).
    \end{equation*}

 Note than an irreducible arc configuration does not have to be steady.
\begin{example} 
Consider the word $w = 02204242$. Then we have
\begin{equation}
\iota :
\vcenter{\hbox{\begin{tikzpicture}
    \node at (0, 0) {$0$};
    \node at (0.6,0) {$2$};
    \node at (1.2,0) {$2$};
    \node at (1.8,0) {$0$};
    \node at (2.4,0) {$4$};
    \node at (3.0,0) {$2$};
    \node at (3.6,0) {$4$};
    \node at (4.2,0) {$2$};
    \draw[-] (0 ,0.3) to [out=30,in=150] (4.2,0.3);
    \draw[-] (0.6 ,0.3) to [out=30,in=150] (3.6,0.3);
    \draw[-] (1.2 ,0.3) to [out=30,in=150] (2.4,0.3);
    \draw[-] (1.8 ,0.3) to [out=30,in=150] (3.0,0.3);
\end{tikzpicture}}}
\longmapsto 
\vcenter{\hbox{\begin{tikzpicture}
    \node at (0, 0) {$0$};
    \node at (0.6,0) {$2$};
    \node at (1.2,0) {$2$};
    \node at (1.8,0) {$0$};
    \node at (2.4,0) {$4$};
    \node at (3.0,0) {$2$};
    \node at (3.6,0) {$4$};
    \node at (4.2,0) {$2$};
    \draw[-] (0 ,0.3) to [out=30,in=150] (4.2,0.3);
    \draw[-] (0.6 ,0.3) to [out=30,in=150] (2.4,0.3);
    \draw[-] (1.2 ,0.3) to [out=30,in=150] (3.6,0.3);
    \draw[-] (1.8 ,0.3) to [out=30,in=150] (3.0,0.3);
\end{tikzpicture}}}.
\end{equation}
\end{example}

\subsection{Questions.}\label{subsec:questions}
In this section we collect the questions and conjectures which we hope will be a subject of future studies.

We start with those which were formulated as conjectures or discussed in text. Here we repeat them in an informal style and refer to the relevant parts of the text. 

\medskip

{\it Questions discussed in text.}
\begin{enumerate}    
    \item Given a word $w\in W_{2n}$ what is $h(w)$? (See Sections \ref{sec:triv_submod},\ref{sec:bounds}, \ref{sec:degs_limits}, \ref{subsec:exact_comps}, \ref{subsec:implicit_bounds},\ref{subsec:fact_by_cont}.)
    \item Is that true that if the Drinfeld tensor product of modules is well-defined, then it is isomorphic to the usual tensor product? (See Conjecture \ref{conj:drinf_prod}.)
    \item Is that true that for any word all singular vectors can be obtained as degenerations of singular vectors of words with generic parameters? (See Conjecture \ref{conj:deg_onto}.)
    \item Is that true that for any word all singular vectors can be obtained as degenerations of singular vectors of a word with generic parameters corresponding to the standard configuration of the initial word? (See Examples \ref{ex:degeneration_different_ways} and \ref{ex:degs_long} and Conjecture \ref{conj:deg_std_onto}.)  
    \item Is that true all degeneracy graphs are connected? (See Conjecure \ref{conj:deg_graph_conn}.)
    \item Is that true that if a word is conf-connected, then removing the ends of the rightmost arc of the smallest color gives a conf-connected word? (See Conjecture \ref{conj:vertex_over}.)
    \item How to tell if two words are isomorphic? (See Section \ref{subsec:sep_inv}.)
    \item 
    What is the number of non-isomorphic words of given length and support? (See Section \ref{subsec:class_count}.)

    \item Is that true that two direct sums of words are isomorphic if and only if after a suitable permutation, the summands are isomorphic? (See the end of Section \ref{subsec:class_count}.)

    \item Are there relations between words with letters in $\{0,2,4\}$ which do not follow from two letter relations and if there are, then what are they? (See the end of Section \ref{subsec:class_count}.)
    
    \item Words consisting of $0$'s and $2$'s only are isomorphic if and only if they are related by a sequence of isomorphisms $00^n2^n0^n \cong 0^n2^n0^n0, 22^n0^n2^n \cong 2^n0^n2^n2$? (See Conjecture \ref{conj:02isoms}.)
    \item What are extensions between a pair of irreducible modules? (See Section \ref{sec:extensions} and Theorem \ref{thm:eval_ext}.) 
        
    \item Is that true that $00^n2^n0^n \cong 0^n2^n0^n0$? (See Section \ref{subsec:comm}, Conjecture \ref{conj:commute}.)
    \item What are submodule graphs of $0^n2^n$ and $0^n2^n0^n$? (See Sections \ref{subsec:comm}, \ref{sec:mod_graphs}, and Conjecture \ref{conj:0n2n}.)

    \item How to compute the submodule graphs of a word in combinatorial terms? (See Section \ref{sec:mod_graphs}.)
    \item Is that true that if a pair of modules is $q$-character separated and all submodule graphs are defined, then the submodule graph of tensor product is the Cartesian product of graphs? (See Conjecture \ref{conj:graph_product}.) 
    \item Is that true that multiplication of a module by an irreducible module such that the product with any  irreducible subquotient of the module is irreducible may only erase edges of the submodule graph of the module? (See Conjecture \ref{conj:irr_prod_graph_erase}.)
    \item Is there  a closed subgraph of the submodule graph of a module such that there is no submodule corresponding to this subgraph? (See Section \ref{subsec:graphs_of_mods}.)
    
    \item How to generalize submodule graphs to all indecomposable modules? (See Seciton \ref{sec:mod_graphs}.)
    \item What is the data to add to a submodule graph to characterize a module?
    (See Section \ref{sec:mod_graphs} and Remark \ref{remark:extension_remove}.)
    \item Is that true that the value of $h(w)$ is uniquely determined by the intersection polynomial of $w$? (See Conjecture \ref{conj:int_poly}.)
    \item Is that true that the value of $h(w)$ is uniquely determined by $\conf(w)$? (See Conjecture \ref{conj:drop_adm})
    
    \item How does the intersection polynomial change as we apply isomorphisms $0200 \cong 0020$ and $2022 \cong 2202$? (See Section \ref{subsec:comm}, \ref{subsec:int_poly}.)
    \item Is that true that $h(w)$ is product of $h(w^{(i)})$, where $w^{(i)}$'s are maximal conf-connected subwords of $w$? (See Conjecture \ref{conj:hom_fact_confconn}.)
    \item What is the number of conf-connected words of a given length (up to a shift)? (See Section \ref{subsec:deg_graphs}.)
    \item Are pivots of a singular vector necessarily  pivots of an arc configuration? (See Conjecture \ref{conj:pivots}.)
    \item  Is that true that left ends of an irreducible arc configuration are always left ends of a steady arc configuration for the same word? (See Conjecture \ref{conj:irr_ends_subset_steady}.)
\end{enumerate}

\medskip 

We add a few more questions which we did not discuss in the text.

\medskip 

{\it Questions not discussed in text.}
\begin{enumerate}
    \item What is the number of non-isomorphic conf-connected words of a given length up to a shift? 
    \item Is a given word indecomposable?
    \item Does $\mbC$ appear as a direct summand of a word?
    \item Is there a word with exactly $p$ arc configurations for each prime $p$? 
    \item Is the number of non-isomorphic 
    submodules of a word finite? 

    \item Is every indecomposable module a subquotient of a word?

    \item Is there a generalization of arc configurations to the general products of evaluation modules?
    \item What is the maximal height of a word of a given length?
    \item What is the maximal number of composition factors of a word of a given length?
    \item Is it possible for a submodule graph of a submodule to have more edges than in the corresponding subgraph?
    \item Is there a more convenient combinatorics of "thick" arcs joining more than two letters which computes $h(w)$?
    \item Is there a generalization of arc configurations to the products of evaluation vector representations of $U_q\widehat{\mathfrak{sl}}_n$?
\end{enumerate}

\appendix

\section{Proofs.}\label{app:etc}
\begin{proof}[Proof of Proposition \ref{prop:Rmat_lambdas}]
        Recall the notation $a=\frac{\al_1+\bt_1}{2}, b=\frac{\al_2+\bt_2}{2}, m = \frac{\bt_1-\al_1}{2}+1, n = \frac{\bt_2 -\al_2}{2}+1$.
        For generic $a,b$ tensor products $[\alpha_{1},\beta_{1}][\alpha_{2},\beta_{2}]$ and $[\alpha_{2},\beta_{2}][\alpha_{1},\beta_{1}]$ are irreducible and isomorphic since $q$-characters of both modules have the same unique dominant monomial. Therefore, a $\Uqa$-homomorphism exists and is unique up to a multiplicative constant in that case. It follows that such homomorphism exists for all $a, b$. Indeed, $\check{R}$ is an $\Uqa$-map if and only if variables $\la_0,\dots,\la_{\min\{m,n\}}$ satisfy a linear system of equations $[\check{R},e_0]=[\check{R},f_0]=0$.
        If a homogeneous system of linear equations depends on parameters algebraically and has a non-trivial solution for generic values of these parameters, it has  a non-trivial solution for all values of parameters.

        Now, we consider equation $[\check{R},f_0]u_k=0$ for variables $\la_i$.

        Since $f_0$ has weight $2$ and preserves singular vectors, we have
        \begin{equation}\label{eq:rmat_pf1}
            f_{0}u_{k} = c_{k}u_{k-1},    
        \end{equation}        for some $c_k \in \mbC$. Compute
        \begin{equation}\label{eq:rmat_pf2}
        f_{0}u_{k} = (q^{(n{-}k{+}1)k {+} m {-} b} {-} q^{(n{-}k{+}2)(k-1) {-} a})({-}1)^{k} [m]_{q}! [n{-}k{+}1]_{q}!v_0\otimes f^{(k{-}1)}w_{0} + \dots,
        \end{equation}        where the dots denote vectors of the form $v\otimes w$ such that the weight of $v$ is less than $m$.
    
        Collecting equations \eqref{eq:rmat_pf1}, \eqref{eq:rmat_pf2} and \eqref{eq:ev_mod_prod_sl2_hw}, we obtain
        \begin{equation}\label{eq:rmat_pf3}
        c_{k} = q^{-a} - q^{n+m - 2(k-1) - b}.
        \end{equation}        Let $c_{k}^{\text{op}}=q^{-b} - q^{n+m - 2(k-1) - a}$  be obtained from $c_k$ by exchanging $a$ with $b$ and $m$ with $n$. Then 
        \begin{equation*}
           \lambda_{k-1}c_{k}u^{\text{op}}_{k-1} = \check{R} f_0 u_{k} = f_0 \check{R} u_{k} = \lambda_{k}c^{\text{op}}_{k}u^{\text{op}}_{k-1},
        \end{equation*}
        which gives
        \begin{equation}\label{eq:rmat_pf4}
            (q^{-a} - q^{n+m - 2(l-1) - b})\lambda_{k-1} = (q^{-b} - q^{n+m - 2(k-1) - a})\lambda_{k}.
        \end{equation}        Among the numbers $c_k,c_k^{\text{op}}$, $k=1,\dots,\min\{m,n\}$, at most one is zero. Therefore, we have at most one solution of \eqref{eq:rmat_pf4} and this solution is given by \eqref{eq:Rmat_lambdas}.
\end{proof}
\begin{proof}[Proof of Lemma \ref{lemma:hom_move}]
    Define a linear map
    \begin{equation}
    \begin{aligned}
        \varphi: \rW^* \otimes \rV^{*} &\longrightarrow (\rV\otimes \rW)^*,\\
                   \mu \otimes \lambda & \longmapsto \left(v\otimes w \mapsto \lambda(v)\mu(w)\right).
    \end{aligned}
    \end{equation}
    This is an isomorphism of vector spaces. It remains to check that it is a homomorphism of modules.

    Let $x\in \rH$. Then write $\Delta(x) = \sum_{i = 1}^{N}x_{i}^{(1)} \otimes x_{i}^{(2)}$. Let $\mu \in \rW^{*}, w\in \rW, \lambda\in \rV^{*}, v\in \rV$. Then
\begin{multline*}
 \varphi(\rho_{\rW^* \otimes \rV^{*}}(x) (\mu \otimes \lambda))(v \otimes w) =\sum_{i=1}^{N} \varphi\Big(\mu \circ \rho_{\rW}\big(S(x_{i}^{(1)})\big)  \otimes \lambda \circ \rho_{\rV}\big(S(x_{i}^{(2)})\big)\Big) (v \otimes w)
 = \\ = \sum_{i=1}^{N} \varphi(\mu \otimes \lambda)\big(S(x_{i}^{(2)}) v \otimes S(x_{i}^{(1)}) w\big) =\varphi ( \mu \otimes \lambda ) \Big( \rho _{\rV}\otimes \rho _{\rW}\big( \underbrace{S\otimes S \circ \Delta^{\text{op}} ( x)}_{\Delta\circ S(x)} \big) ( v\otimes w) \Big) =\\=  \rho_{(\rV \otimes \rW)^*}(x) (\varphi(\mu \otimes \lambda))(v \otimes w).
\end{multline*}
This proves the first equation. The proof of the second is similar.

    Note that for any $\rH$-module $\rW$, maps
    \begin{align*}
        \pi_{\rW} : \rW^{*} \otimes \rW &\longrightarrow \mbC, & \iota_{\rW} : \mbC &\longrightarrow \rW \otimes \rW^{*},
    \end{align*}
    which correspond to convolution and inclusion of the identity operator, are homomorphisms of $\rH$-modules.

    Then one constructs
    \begin{equation*}
        \begin{aligned}
            \alpha: \mathrm{Hom}_{\rH}\left(\rV\otimes \rW, \rU\right) &\longrightarrow \mathrm{Hom}_{\rH}\left(\rV, \rU\otimes \rW^{*}\right),\\
                            \xi   &\longmapsto \left(\xi\otimes \mathrm{Id}_{\rW^*}\right) \circ \left(\mathrm{Id}_{\rV}\otimes \iota_{\rW}\right),\\
            \beta: \mathrm{Hom}_{\rH}\left(\rV, \rU\otimes \rW^{*}\right) &\longrightarrow \mathrm{Hom}_{\rH}\left(\rV\otimes \rW, \rU\right),\\
                            \nu   &\longmapsto \left(\mathrm{Id}_{\rU}\otimes \pi_{\rW}\right) \circ \left(\nu \otimes \mathrm{Id}_{\rW}\right).\\
        \end{aligned}
    \end{equation*}
    A straightforward check shows that $\alpha$ and $\beta$ are inverse to each other. This proves the third equation.

    For the fourth equation, the corresponding isomorphisms are given by 
     \begin{equation*}
        \begin{aligned}
            \alpha^{\prime}: \mathrm{Hom}_{\rH}\left( \rW\otimes \rV, \rU\right) &\longrightarrow \mathrm{Hom}_{\rH}\left(\rV, {}^{*}\rW\otimes \rU\right),\\
                            \xi   &\longmapsto  \left(\mathrm{Id}_{{}^{*}\rW} \otimes \xi \right) \circ \left(\iota_{{}^*\rW} \otimes \mathrm{Id}_{\rV} \right),\\
            \beta^{\prime}: \mathrm{Hom}_{\rH}\left(\rV, {}^*\rW \otimes \rU\right) &\longrightarrow \mathrm{Hom}_{\rH}\left(\rW\otimes \rV, \rU\right),\\
                            \nu   &\longmapsto \left(\pi_{{}^{*}\rW} \otimes \mathrm{Id}_{\rU} \right)\circ \left(\mathrm{Id}_{\rW}\otimes \nu\right), 
        \end{aligned}
    \end{equation*}
    where we use canonical identification ${}^*(\rW^*) \cong ({}^{*}\rW)^* \cong \rW$.
\end{proof}
{\it Formulas for $\uR_{+}, \uR_0, \uR_{-}$.}
    \begin{align}
        \uR_{+} = \overleftarrow{\prod_{r\geq0}}\exp_q\big((q-q^{-1})x_r^+\otimes &x_{-r}^-\big),\;\;\;\uR_{-} = \overrightarrow{\prod_{r\geq 1}}\exp_q\big((q-q^{-1})(K^{-1}x_r^-)\otimes (x_{-r}^+K)\big),\label{eqn:uRmat}\\
        \uR_{0} &= \exp\left(\sum_{r>0}\frac{r(q-q^{-1})}{[2r]}h_r\otimes h_{-r}\right).\label{eqn:uRmat1}
    \end{align}
Here $\exp_{q}(x) = \sum_{n = 0}^{\infty}\frac{x^n}{[n]_q!}.$

Decomposition \eqref{eq:Rmat_gauss} together with formulas \eqref{eqn:uRmat}, \eqref{eqn:uRmat1} are found in \cite{tolstoy1992universal}.

\begin{proof}[Proof of Proposition \ref{prop:ellroot_action}] We essentially follow the proof of Proposition 3.9 in \cite{mukhin2014affinization}.
    We have decomposition $\rV = \mathop{\oplus}\limits_\theta \rV[\theta]$ into generalized $\ell$-weight spaces. For a vector $v\in \rV$ denote the projection of $v$ to $\rV[\theta]$ in this decomposition by $(v)_{\theta}$. 
    
    It is sufficient to prove that for an $\ell$-weight $\tilde{\phi}$ such that  $(x^{\pm}(w)v)_{\tilde{\phi}} \neq 0$, there exists $a\in \mbC$ such that $\tilde{\phi} = A_{a}\phi$. Let $\{v_1,\dots,v_k\}$ and $\{u_1,\dots,u_l\}$ be bases of $\rV[\phi]$ and $\rV[\tilde{\phi}]$ respectively such that the action of the currents $\psi^{+}(w)$ in these bases is lower-triangular, that is $$(\psi^{+}(z)- \phi^{+}(z))v_j = \mathop{\sum}_{j^{\prime}<j}\xi_{j,j^{\prime}}(z)v_{j^{\prime}},\;\; (\psi^{+}(z)- \phi^{+}(z))u_r = \mathop{\sum}_{r^{\prime}<r}\zeta_{r,r^{\prime}}(z)u_{r^{\prime}}.$$

    Assume that $(x^{+}(z)v)_{\tilde{\phi}}\neq 0$. Let $j_0$ be the minimal such that $(x^{+}(z)v_{j_0})_{\tilde{\phi}} \neq 0$. We have a decomposition $(x^{+}(w)v_{j_0})_{\tilde{\phi}} = \mathop{\sum}\limits_{s=1}^{l}\lambda_{s}(w)u_s$.\\ 
    From defining relations \eqref{eq:def_new_drinf} we have 
    \begin{equation}\label{eq:ell_root_pf_tmp0}
    (q^2w - z)x^{+}(w)(\psi^{+}(z) - \phi^{+}(z))v_{j_0} = ((w - q^2 z)\psi^{+}(z) -(q^2w -z)\phi^{+}(z))x^{+}(w)v_{j_0}.
    \end{equation}
    Project the left hand side and the right hand side of \eqref{eq:ell_root_pf_tmp0}  to $\rV[\tilde{\phi}]$. We obtain
    \begin{equation*}
        ((q^2w - z)x^{+}(w)(\psi^{+}(z) - \phi^{+}(z))v_{j_0})_{\tilde{\phi}} = \mathop{\sum}_{j^{\prime}<j_0}(q^2w - z)\xi_{j,j^{\prime}}(z)(x^{+}(w)v_{j^{\prime}})_{\tilde{\phi}} = 0,
    \end{equation*}
    and
    \begin{multline}\label{eq:ell_root_pf_tmp}
        (((w - q^2 z)\psi^{+}(z) -(q^2w -z)\phi^{+}(z))x^{+}(w)v_{j_0})_{\tilde{\phi}} = ((w - q^2 z)\psi^{+}(z) -(q^2w -z)\phi^{+}(z))(x^{+}(w)v_{j_0})_{\tilde{\phi}} =\\= \mathop{\sum}_{s=1}^{l}\mathop{\sum}\limits_{s^{\prime}<s}\lambda_s(w)\zeta_{s,s^{\prime}}(z)u_{s^{\prime}} + ((w-q^2z)\tilde{\phi}^{+}(z) - (q^2w-z)\phi^{+}(z))\mathop{\sum}_{s=1}^{l}\lambda_s(w)u_s.
    \end{multline}
    Let $s_0$ be the maximal such that $\lambda_{s_0}(w)\neq 0$, then taking the coefficient of $u_{s_0}$ in \eqref{eq:ell_root_pf_tmp} we get $$((w-q^2z)\tilde{\phi}^{+}(z) - (q^2w-z)\phi^{+}(z))\lambda_{s_0}(w) = 0,$$ where $\lambda_{s_0}(w)\neq 0$. Rewriting this equation gives
    $$\Big(w\frac{\tilde{\phi}^{+}(z)-q^2\phi^{+}(z)}{q^2\tilde{\phi}^{+}(z)-\phi^{+}(z)} - z\Big)\lambda_{s_0}(w) = 0.$$ Write $\lambda_{s_0}(w) = \mathop{\sum}\limits_{p\in\mbZ}\lambda_{s_0,p}w^p$. Taking coefficients of $w^{p+1}$ we get
    $$\lambda_{s_0,p}\frac{\tilde{\phi}^{+}(z)-q^2\phi^{+}(z)}{q^2\tilde{\phi}^{+}(z)-\phi^{+}(z)} - \lambda_{s_0,p+1}z = 0.$$ This system admits a non-zero solution if and only if there exists $a\in \mbC$ such that $\frac{\tilde{\phi}^{+}(z)-q^2\phi^{+}(z)}{q^2\tilde{\phi}^{+}(z)-\phi^{+}(z)} = q^az$, or, equivalently, $\tilde{\phi}^{+}(z) = q^2\frac{1 - zq^{a-2}}{1-zq^{a+2}}\phi^{+}(z) = A_{a}(z)\phi^{+}(z)$. This implies $\tilde{\phi} = A_a\phi$. 
    
    The case of $(x^{-}(w)v)_{\tilde{\phi}}\neq 0$ is similar. 
\end{proof}

\begin{proof}[Proof of Proposition \ref{prop:equiv_drinf}]
    We denote the set $S_{\rV}= \{a\in \mbC | 1_{a} \textit{ appears in }\chi_q(\rV)\}$. Let $B_{\rV}=\{v_1,{\dots}, v_{d_1}\}$ be a basis of $\rV$ consisting of $\ell$-weight vectors such that $\Uq$-weight of $v_{i}$ is not less than $\Uq$-weight of $v_{i+1}$ for all $i$. Let $B_{\rW} = \{w_1,\dots, w_{d_2}\}$ be a basis of $\rW$ consisting of $\ell$-weight vectors and such that $\Uq$-weight of $w_{i}$ is not greater than $\Uq$-weight of $w_{i+1}$ for all $i$.
    
    First, note that $\rV \otimes \rW$ is thin, since all $1_a$'s appearing in the $q$-character of $\chi_q(\rV)$ are different from the ones appearing in the $q$-character of $\chi_q(\rW)$. 
     We can assume without loss of generality that $\rV$ and $\rW$ are indecomposable. By Proposition \ref{prop:ellroot_action} this implies that any two $\ell$-weights in $\rV$ (respectively in $\rW$) are related by a sequence of multiplications and divisions by $A_{a+1} = 1_a1_{a+2}$ where $a$ is in $S_{\rV}$ (respectively in $S_{\rW}$). 

    We prove that the action of $\uR_0$ is well-defined and all matrix elements on the diagonal are non-zero. Note that after an appropriate rescaling one can assume that $\uR_0(v_1\otimes w_1) = v_1\otimes w_1$.

    Assume that $\uR_0(v_i\otimes w_j) = \lambda_{i,j}v_{i}\otimes w_j$. Let $v_k$ be a basis vector such that the $\ell$-weight of $v_k$ differs from one of $v_i$ by multiplication by $A_{a+1}$, where $a\in S_{\rV}$. This shifts the eigenvalue of $h_r$, $r>0$, by the coefficient of $z^r$ in $\frac{1}{q-q^{-1}}\log(\frac{1-q^{a-1}z}{1-q^{a+3}z})$ which is  equal to $\frac{[2r]q^{r(a+1)}}{r}$. Then we have 
    
    $$\uR_0(v_k\otimes w_j) = \lambda_{i,j}\exp\left(\sum_{r>0}(q-q^{-1})q^{r(a+1)} (1\otimes h_{-r})\right)(v_k\otimes w_j) = \frac{\lambda_{i,j}q^{-\mu_j}}{m_{2,j}(q^{-a-1})}v_k\otimes w_j.$$
    Here $\mu_j$ is $\Uq$-weight of $w_j$ and $m_{2,j}(q^{-a-1})$ is the $\ell$-weight of $w_j$ where we replaced $z\mapsto q^{-a-1}$. 
    
    The function $m_{2,j}(z)$ is a monomial in $1_b$'s with $b \in S_{\rW}$. The only zero (respectively, pole) of $1_b$ is $z = q^{-b+1}$ (respectively, $z = q^{-b-1}$). Note that $A_{a+1}=1_a1_{a+2}$ is a ratio of $\ell$-weights in $\rV$, therefore both $1_a, 1_{a+2}$ appear in the $q$-character of $\rV$. By the assumption, $b\notin \{a, a+2\}$, therefore, $m_j(q^{-a-1})$ is well-defined and is non-zero.

    The case when the $\ell$-weight of $w_l$ differs from the one of $w_j$ by multiplication by $A_{b+1}$ is treated similarly.

    Therefore, we conclude that the action of $\uR_0$ is well-defined and in the basis $B_{\rV}\otimes B_{\rW} = \{v_i \otimes w_j\}$ the action of $\uR_0$ is diagonal with non-zero entries.

    Since $\uR_{\pm} \in 1 +U^{\pm}\otimes U^{\mp}$, in the basis $B_{\rV}\otimes B_{\rW}$, it acts as an upper-triangular (respectively, lower-triangular) matrix with  the diagonal matrix elements one. Then the entries of $\uR_{\pm}$ can be uniquely recovered by the LDU decomposition of $\uR$ as rational functions of entries of $\uR$. Note that denominators of these functions are elements on the diagonal of $\uR_0$, therefore no division by zero appears.    

    The proposition follows from \eqref{eq:Rp_vs_drinf}.
    
\end{proof}

\section{Tables of conf-connected words of lengths $4,6,8, 10$ up to symmetries.}\label{app:tables}

For the cases of $4,6,8, 10$ letters we list conf-connected words up to slides, permutations of adjacent letters which do not differ by $\pm 2$, anti-involution $\omega$ (see Corollary \ref{cor:reverse_word}). For each word we write the number $h(w)$, the number of arc configurations, and the intersection polynomial (see Definition \ref{def:int_poly}). 
\begin{table}[H]
\begin{tabular}{c}
$\begin{array}{|c|c|c|c|}
\hline
w & h(w) & |\conf(w)| & p_{w}(x)\\
\hline
 0022 & 1 & 2 & x+1\\
 \hline
\end{array}$
\bigskip
\bigskip

\\

$\begin{array}{|c|c|c|c|}
\hline
w & h(w) & |\conf(w)| & p_{w}(x)\\
\hline
 020242 & 2 & 3 & x^2+2\\
 000222 & 1 & 6 & (x+1)(x^2+x+1)\\
 002022 & 1 & 4 & (x+1)^2\\
 \hline
\end{array} $
\bigskip
\bigskip

\\
$\begin{array}{|c|c|c|c|}
\hline
w & h(w) & |\conf(w)| & p_{w}(x)\\
\hline
02020242 & 3 & 7 & x^3+3 x^2+3 \\
 00202242 & 2 & 14 & (x+1) \left(x^4+2 x^2+3 x+1\right) \\
 00202422 & 2 & 10 & (x+1) \left(x^3+x^2+x+2\right) \\
 02002422 & 2 & 8 & (x+1)^2 \left(x^2-x+2\right) \\
 02020422 & 2 & 6 & (x+1) \left(x^2+2\right) \\
 20204242 & 2 & 5 & x^3+2 x+2 \\
 02402462 & 2 & 5 & x^3+2 x+2 \\
 00002222 & 1 & 24 & (x+1)^2 \left(x^2+1\right) \left(x^2+x+1\right) \\
 00020222 & 1 & 18 & (x+1) \left(x^2+x+1\right)^2 \\
 00022022 & 1 & 12 & (x+1)^2 \left(x^2+x+1\right) \\
 00220422 & 1 & 8 & (x+1)^3 \\
 00202022 & 1 & 8 & (x+1)^3 \\
 \hline
\end{array}$
\end{tabular}
\end{table}

\begin{table}[H]
\vskip-30pt
$
\begin{array}{|c|c|c|c|}
\hline
w & h(w) & |\conf(w)| & p_{w}(x)\\
\hline
 0202024242 & 6 & 31 & x^7+4 x^5+3 x^4+6 x^3+11 x^2+6 \\
 0202420242 & 5 & 17 & x^5+2 x^4+2 x^3+7 x^2+5 \\
 0202020242 & 4 & 15 & x^4+4 x^3+6 x^2+4 \\
 2420246424 & 4 & 11 & x^5+2 x^3+2 x^2+2 x+4 \\
 2420426424 & 4 & 9 & \left(x^2+2\right)^2 \\
 0020202242 & 3 & 46 & (x+1) \left(x^6+3 x^5+3 x^4+3 x^3+7 x^2+5 x+1\right) \\
 0020202422 & 3 & 38 & (x+1) \left(x^6+x^5+2 x^4+5 x^3+4 x^2+3 x+3\right) \\
 0020220242 & 3 & 32 & (x+1)^3 \left(x^3+x^2+2\right) \\
 0200202422 & 3 & 28 & (x+1)^2 \left(x^4+x^3+2 x^2+3\right) \\
 0202042242 & 3 & 24 & (x+1)^2 \left(x^2+1\right) \left(x^2-x+3\right) \\
 0202024422 & 3 & 24 & (x+1)^2 \left(x^4+2 x^2+3\right) \\
 0200220242 & 3 & 22 & (x+1) \left(x^4+3 x^3+3 x^2+x+3\right) \\
 0022020242 & 3 & 22 & (x+1) \left(x^4+3 x^3+3 x^2+x+3\right) \\
 0202002422 & 3 & 20 & (x+1)^2 \left(x^3+2 x^2-x+3\right) \\
 0202042422 & 3 & 18 & (x+1) \left(x^4+2 x^3+x^2+2 x+3\right) \\
 2002420242 & 3 & 14 & (x+1) \left(x^4+2 x^2+x+3\right) \\
 0202420422 & 3 & 14 & (x+1) \left(x^3+3 x^2+3\right) \\
 0202024462 & 3 & 14 & (x+1) \left(x^3+3 x^2+3\right) \\
 0202020422 & 3 & 14 & (x+1) \left(x^3+3 x^2+3\right) \\
 0024202462 & 3 & 14 & (x+1) \left(x^4+2 x^2+x+3\right) \\
 2020204242 & 3 & 13 & x^4+4 x^3+x^2+4 x+3 \\
 0202042462 & 3 & 13 & x^4+4 x^3+x^2+4 x+3 \\
 2020420242 & 3 & 11 & x^4+2 x^3+2 x^2+3 x+3 \\
 0204202462 & 3 & 11 & x^4+2 x^3+2 x^2+3 x+3 \\
 2042042642 & 3 & 8 & x^4+2 x^2+2 x+3 \\
 0424026462 & 3 & 8 & x^4+2 x^2+2 x+3 \\
 0246024682 & 3 & 8 & x^4+2 x^2+2 x+3 \\
 0002022242 & 2 & 78 & (x+1) \left(x^2+x+1\right) \left(x^6+2 x^4+3 x^3+4 x^2+2 x+1\right) \\
 0002022422 & 2 & 60 & (x+1)^2 \left(x^2+x+1\right) \left(x^4+x^2+2 x+1\right) \\
 0020022422 & 2 & 48 & (x+1)^2 \left(x^6+x^4+3 x^3+3 x^2+2 x+2\right) \\
 0002202422 & 2 & 48 & (x+1)^2 \left(x^2+x+1\right) \left(x^3+x^2+2\right) \\
 0002024222 & 2 & 42 & (x+1) \left(x^2+x+1\right) \left(x^4+x^3+2 x^2+x+2\right) \\
 0020024222 & 2 & 36 & (x+1)^2 \left(x^2+x+1\right) \left(x^3+2\right) \\
 0022002242 & 2 & 32 & (x+1)^3 \left(x^3+x^2+2\right) \\
 0200024222 & 2 & 30 & (x+1) \left(x^2+x+1\right) \left(x^4+x^2+x+2\right) \\
 0020204222 & 2 & 30 & (x+1) \left(x^2+x+1\right) \left(x^3+x^2+x+2\right) \\
 0022002422 & 2 & 28 & (x+1)^2 \left(x^4+x^3+x^2+2 x+2\right) \\
 0020220422 & 2 & 28 & (x+1)^2 \left(x^4+2 x^2+3 x+1\right) \\
  0200204222 & 2 & 24 & (x+1)^2 \left(x^2-x+2\right) \left(x^2+x+1\right) \\
 2002042242 & 2 & 20 & (x+1)^2 \left(x^4-x^3+2 x^2+x+2\right) \\
 0220024422 & 2 & 20 & (x+1)^2 \left(x^4+2 x+2\right) \\
 0202204422 & 2 & 20 & (x+1)^2 \left(x^3+x^2+x+2\right) \\

 \hline
\end{array}
$
\end{table}

\begin{table}[H]
\vskip-30pt
$\begin{array}{|c|c|c|c|}
\hline
w & h(w) & |\conf(w)| & p_{w}(x)\\
\hline
  0200220422 & 2 & 20 & (x+1)^2 \left(x^3+x^2+x+2\right) \\
 0024022462 & 2 & 20 & (x+1)^2 \left(x^4-x^3+2 x^2+x+2\right) \\
 0022024462 & 2 & 20 & (x+1)^2 \left(x^3+x^2+x+2\right) \\
 0022020422 & 2 & 20 & (x+1)^2 \left(x^3+x^2+x+2\right) \\
 0020242022 & 2 & 20 & (x+1)^2 \left(x^3+x^2+x+2\right) \\
 2002204242 & 2 & 18 & (x+1) \left(x^4+x^3+2 x^2+3 x+2\right) \\
 0202004222 & 2 & 18 & (x+1) \left(x^2+2\right) \left(x^2+x+1\right) \\
 0022042462 & 2 & 18 & (x+1) \left(x^4+x^3+2 x^2+3 x+2\right) \\
 2020402242 & 2 & 16 & (x+1)^3 \left(x^2-x+2\right) \\
 2002042422 & 2 & 16 & (x+1)^3 \left(x^2-x+2\right) \\
 0220042422 & 2 & 16 & (x+1)^3 \left(x^2-x+2\right) \\
 0200242022 & 2 & 16 & (x+1)^3 \left(x^2-x+2\right) \\
 0024024262 & 2 & 16 & (x+1)^3 \left(x^2-x+2\right) \\
 2020042422 & 2 & 14 & (x+1) \left(x^4+2 x^2+2 x+2\right) \\
 0244024662 & 2 & 14 & (x+1) \left(x^4+2 x^2+2 x+2\right) \\
 0204024262 & 2 & 14 & (x+1) \left(x^4+2 x^2+2 x+2\right) \\
 2020204422 & 2 & 12 & (x+1)^2 \left(x^2+2\right) \\
 2002420422 & 2 & 12 & (x+1)^2 \left(x^2+2\right) \\
 0220420422 & 2 & 12 & (x+1)^2 \left(x^2+2\right) \\
 0202044262 & 2 & 12 & (x+1)^2 \left(x^2+2\right) \\
 0202042022 & 2 & 12 & (x+1)^2 \left(x^2+2\right) \\
 0024204262 & 2 & 12 & (x+1)^2 \left(x^2+2\right) \\
 2042402426 & 2 & 10 & (x+1) \left(x^3+2 x+2\right) \\
 2020420422 & 2 & 10 & (x+1) \left(x^3+2 x+2\right) \\
 0244026462 & 2 & 10 & (x+1) \left(x^3+2 x+2\right) \\
 0204204262 & 2 & 10 & (x+1) \left(x^3+2 x+2\right) \\
  0000022222 & 1 & 120 & (x+1)^2 \left(x^2+1\right) \left(x^2+x+1\right) \left(x^4+x^3+x^2+x+1\right) \\
 0000202222 & 1 & 96 & (x+1)^3 \left(x^2+1\right)^2 \left(x^2+x+1\right) \\
 0000220222 & 1 & 72 & (x+1)^2 \left(x^2+1\right) \left(x^2+x+1\right)^2 \\
 0002020222 & 1 & 54 & (x+1) \left(x^2+x+1\right)^3 \\
 0000222022 & 1 & 48 & (x+1)^3 \left(x^2+1\right) \left(x^2+x+1\right) \\
 0002204222 & 1 & 36 & (x+1)^2 \left(x^2+x+1\right)^2 \\
 0002200222 & 1 & 36 & (x+1)^2 \left(x^2+x+1\right)^2 \\
 0002022022 & 1 & 36 & (x+1)^2 \left(x^2+x+1\right)^2 \\
 0022004222 & 1 & 24 & (x+1)^3 \left(x^2+x+1\right) \\
 0020022022 & 1 & 24 & (x+1)^3 \left(x^2+x+1\right) \\
 0002242022 & 1 & 24 & (x+1)^3 \left(x^2+x+1\right) \\
 0002202022 & 1 & 24 & (x+1)^3 \left(x^2+x+1\right) \\
 2002204422 & 1 & 16 & (x+1)^4 \\
 0022044262 & 1 & 16 & (x+1)^4 \\
 0022042022 & 1 & 16 & (x+1)^4 \\
 0020202022 & 1 & 16 & (x+1)^4 \\
 \hline
\end{array}$
\end{table}

\bigskip 

{\bf Acknowledgments.\ }
The authors are partially supported by Simons Foundation grant number \#709444. The second author is grateful to C. Stroppel for many stimulating discussions and to Oberwolfach and Hausdorff Research Institutes for Mathematics for hospitality. We thank Indiana University Indianapolis for providing computational resources.

\end{document}